\documentclass[11p t,a4paper]{report}

 \usepackage{tocloft}
\usepackage[affil-it]{authblk}
\usepackage{amsmath, amssymb, amsthm}
\usepackage{pdfpages}
\usepackage{graphicx,color}
\usepackage[left=1.2in, right=1.2in, top=1in, bottom=1in, includefoot, headheight=13.6pt]{geometry}
\usepackage{hyperref}
\usepackage{CJKutf8}
\usepackage[initials,sorted,sorted-cites]{amsrefs}
\usepackage[toc,style=super4col]{glossaries}
\usepackage[makeroom]{cancel}

\usepackage{hhline,booktabs}
\usepackage{amsmath}
\usepackage{amsthm}
\usepackage{amssymb,amsfonts}
\usepackage{MnSymbol,mathrsfs}
\usepackage[utf8]{inputenc}
\usepackage[makeroom]{cancel}
\usepackage{xcolor}
\usepackage[shortlabels]{enumitem}
\usepackage[T1]{fontenc}

\usepackage{enumitem}

\newlist{myitemize}{itemize}{3}
\setlist[myitemize]{label=\textbullet,leftmargin=0.8in}

\newlist{myitemize2}{itemize}{3}
\setlist[myitemize2]{label=\textbullet,leftmargin=1.2in}

\newtheorem{thm}{Theorem}[section]
\newtheorem{cor}[thm]{Corollary}
\newtheorem{prop}[thm]{Proposition}
\newtheorem{lem}[thm]{Lemma}

\theoremstyle{definition}
\newtheorem{defn}[thm]{Definition}

\newtheorem{est}[thm]{Estimate} 
\newtheorem{exmp}[thm]{Example} 
\newtheorem{remk}[thm]{Remark}

\newcommand{\prnt}[1]{\left( #1 \right)}
\newcommand{\Prnt}[1]{\left[ #1 \right]}
\newcommand{\cprnt}[1]{\left\{ #1 \right\}}

\newcommand{\Avg}[2]{\left< #1,\, #2 \right>}
\newcommand{\Abs}[1]{\left| #1\right|}
\newcommand{\tl}[1]{\tilde{ #1} }

\DeclareMathOperator*\uplim{\overline{lim}}

\newcommand{\Poinc}{Poincar\'{e} }

\newcommand{\wrm}{WRM}
\newcommand{\cwrm}{CWRM}
\newcommand{\mesi}{MESI}

\newcommand{\deriv}[2]{\frac{d#1}{d#2}}
\newcommand{\dderiv}[2]{\frac{d^2#1}{d#2^2}}
\newcommand{\pd}[2]{\frac{\partial#1}{\partial#2}}

\newcommand{\conv}[2]{t #1 + (1-t) #2}
\newcommand{\convar}[2]{(1-t) #1 + t #2}
\newcommand{\convars}[2]{(1-s) #1 + s #2}

\newcommand{\R}{\mathbb{R}}
\newcommand{\B}{\mathcal{B}}
\newcommand{\Rcal}{\mathcal{R}}

\newcommand{\Z}{\mathcal{Z}}

\newcommand{\Hcal}{\mathcal{H}}

\newcommand{\Ical}{\mathcal{I}}
\newcommand{\Ncal}{\mathcal{N}}
\newcommand{\N}{\mathbb{N}_0}
\newcommand{\Nbb}{\mathbb{N}}
\newcommand{\NN}{\mathcal{N}}
\renewcommand{\P}{\mathcal{P}}
\newcommand{\F}{\mathcal{F}}
\newcommand{\tlF}{\tilde{\mathcal{F}}}

\newcommand{\D}{\mathcal{D}}

\newcommand{\A}{\mathcal{A}}
\newcommand{\U}{\mathcal{U}}
\newcommand{\Hbb}{\mathbb{H}}
\newcommand{\Sbb}{\mathbb{S}}
\newcommand{\si}{\mathfrak{s}} 
\newcommand{\co}{\mathfrak{c}}

\newcommand{\gfrak}{\mathfrak{g}} 
\newcommand{\zfrak}{\mathfrak{z}}

\newcommand{\nab}{\nabla}
\newcommand{\half}{\frac{1}{2}}
\newcommand{\on}{\frac{1}{n}}
\newcommand{\CC}{\alpha_{K,N,D}}

\newcommand{\Tc}{\mathcal{T}}
\newcommand{\lamp}{\lambda^{(1,p)}}
\newcommand{\lampc}{\lambda^{(1,2)}}
\newcommand{\lam}{\lambda}
\newcommand{\lamM}{\lambda^{\star}}

\newcommand{\g}{\gamma}
\newcommand{\pip}{\frac{\pi_p}{2}}
\newcommand{\hfrak}{\mathfrak{h}}

\newcommand{\hChe}{\mathfrak{h}_{Che}}
\newcommand{\lLed}{\mathfrak{l}_{Led}}
\newcommand{\Exr}[1]{\mathfrak{E}xr\prnt{#1}}
\newcommand{\Ext}[1]{\mathfrak{E}x\prnt{#1}}
\newcommand{\beq}{\begin{equation}}
\newcommand{\eeq}{\end{equation}}
\newcommand{\bea}{\begin{eqnarray}}
\newcommand{\eea}{\end{eqnarray}}
\newcommand{\eq}[1]{\begin{align*}#1\end{align*}}
\newcommand{\eql}[1]{\begin{align}#1\end{align}}
\newcommand{\slim}{\sum\limits}

\newcommand{\Ry}[2]{Ray[#1](#2)}
\newcommand{\Ryb}[2]{Ray_b[#1](#2)}

\definecolor{aogreen}{rgb}{0.0, 0.5, 0.0}
\definecolor{deeppink}{rgb}{1.0, 0.08, 0.58}
\definecolor{deeppink_a}{rgb}{1.0, 0.56, 0.98}

\newcommand{\pink}[1]{\textcolor{black}{#1}}
\newcommand{\red}[1]{\textcolor{black}{#1}}
\newcommand{\blue}[1]{\textcolor{black}{#1}}
\newcommand{\green}[1]{\textcolor{black}{#1}}
\newcommand{\pinka}[1]{\textcolor{black}{#1}}

\newcommand{\isupp}{\mathfrak{supp}}
\newcommand{\Cinf}{C^{\infty}}
\newcommand{\M}{\mathcal{M}}

\newcommand{\bM}{\mathfrak{ M}}

\newcommand{\Pkinfd}{\mathcal{P}_{[K,\infty,D]}}
\newcommand{\FF}{\mathcal{M}}
\newcommand{\Fknd}{\FF_{[K,N,D]}}

\newcommand{\Fkndf}{\FF_{[K,N,D],h}}

\newcommand{\Fkn}{\FF_{[K,N]}}
\newcommand{\Fkndg}[3]{\FF_{[#1,#2,#3]}}
\newcommand{\FkndgM}[3]{\FF^M_{[#1,#2,#3]}}

\newcommand{\Fkinfd}{\FF_{[K,\infty,D]}}
\newcommand{\Fkninf}{\FF_{[K,N,\infty]}}
\newcommand{\Fkndreg}{\FF^{M,reg}_{[K,N,D]}(\R)}

\newcommand{\Pknd}{\mathcal{P}_{[K,N,D]}}
\newcommand{\Pkn}{\mathcal{P}_{[K,N]}}
\newcommand{\Pkndf}{\mathcal{P}_{[K,N,D],h}}
\newcommand{\Pkndhf}{\mathcal{P}_{[K,N,D],h_f}}
\newcommand{\Pkndhff}{\hat{\mathcal{P}}_{[K,N,D],h_f}}
\newcommand{\Pkndff}{\hat{\mathcal{P}}_{[K,N,D],h}}

\newcommand{\Pknds}{\mathcal{P}^*_{[K,N,D]}}
\newcommand{\Pkndfs}{\mathcal{P}^*_{[K,N,D],h}}
\newcommand{\Pkndffs}{\hat{\mathcal{P}}^*_{[K,N,D],h}}

\newcommand{\Tr}{\mathscr{T}\prnt{\mu_{+},\mu_{-}}}  %\textfrak{T}
\newcommand{\Pir}{\Pi(\mu_{+},\mu_{-})}
\newcommand{\Lip}{Lip_1}

\newcommand{\Zc}{\mathcal{Z}}
\newcommand{\ff}{G}
\newcommand{\fs}{\ff_s(t)}
\newcommand{\fss}{\ff_s(t_*)}
\newcommand{\fsvar}{\tilde{\ff}_s(t)}
\newcommand{\fssvar}{\tilde{\ff}_s(t_*)}

\newcommand{\UU}{\mathcal{B}}
\newcommand{\Us}{\UU^{(s)}}
\newcommand{\Ust}{\UU^{(s)}(t)}
\newcommand{\Uss}{\UU^{(s)}(t_*)}
\newcommand{\Ustvar}{\tilde{\UU}^{(s)}(t)}

\newcommand{\Ussvar}{\tilde{\UU}^{(s)}(t_*)}
\newcommand{\At}{\A^{(s)}(t)}
\newcommand{\Ats}{\A^{(s)}(t_*)}

\def\Xint#1{\mathchoice
{\XXint\displaystyle\textstyle{#1}}%
{\XXint\textstyle\scriptstyle{#1}}%
{\XXint\scriptstyle\scriptscriptstyle{#1}}%
{\XXint\scriptscriptstyle\scriptscriptstyle{#1}}%
\!\int}
\def\XXint#1#2#3{{\setbox0=\hbox{$#1{#2#3}{\int}$}
\vcenter{\hbox{$#2#3$}}\kern-.5\wd0}}

\def\dashint{\Xint-}

\usepackage{float}

\makeglossaries
 \newglossaryentry{CD(K,N)}%
{%
  name={\ensuremath{CD(K,N)}},
  description={Curvature-Dimension conditions},
  sort={e00a}
}
\newglossaryentry{CDD(K,N,D)}%
{%
  name={\ensuremath{CDD(K,N,D)}},
  description={Curvature-Dimension-Diameter conditions},
  sort={e00b}
}
\newglossaryentry{CD_b(K,N)}%
{%
  name={\ensuremath{CD_b(K,N)}},
  description={\red{$CD(K,N)$ and finite measure conditions}},
  sort={e00a1}
}
\newglossaryentry{CDD_b(K,N,D)}%
{%
  name={\ensuremath{CDD_b(K,N,D)}},
  description={\red{$CDD(K,N,D)$ and finite measure conditions}},
  sort={e00b1}
}

\newglossaryentry{zeta_max}%
{%
  name={\ensuremath{\zfrak_{+}}},
  description={\blue{First zero of a function, which is non-negative at 0, to the right of 0}},
  sort={c23}
}
\newglossaryentry{zeta_min}%
{%
  name={\ensuremath{\zfrak_{-}}},
  description={\blue{First zero of a function, which is non-negative at 0, to the left of 0}},
  sort={c23}
}

\newglossaryentry{f_vee}%
{%
  name={\ensuremath{f_{\vee}}},
  description={\blue{Defined as $\max\{ f, 0\}$}},
  sort={c23a}
}
\newglossaryentry{delta}%
{%
  name={\ensuremath{\delta}},
  description={The expression $\frac{K}{N-1}$},
  sort={e01}
}
\newglossaryentry{l_{delta}}%
{%
  name={\ensuremath{l_{\delta}}},
  description={Automatic diameter upper bound, namely $\frac{\pi}{ \sqrt{\delta}}$ if $\delta > 0$ , $+\infty$ otherwise },
  sort={e01b}
}
\newglossaryentry{Delta_g}%
{%
  name={\ensuremath{\Delta_{\gfrak}}},
  description={\pinka{The Laplace-Beltrami operator}},
  sort={e001aa}
}
\newglossaryentry{Delta_g_mu}%
{%
  name={\ensuremath{\Delta_{\gfrak,\mu}}},
  description={\pinka{The weighted-Laplacian associated with a measure $\mu$}},
  sort={e001aab}
}
\newglossaryentry{Delta_p}%
{%
  name={\ensuremath{\Delta_{p,\xi}}},
  description={\pinka{The weighted p-Laplacian associated with a measure $\xi$ supported in $\R$}},
  sort={e001aac}
}
\newglossaryentry{L_g_mu}%
{%
  name={\ensuremath{L_{\gfrak,\mu}}},
  description={\pink{The diffusion operator $-\Delta_{\gfrak,\mu}$}},
  sort={e001ab}
}

\newglossaryentry{D}%
{%
  name={\ensuremath{D}},
  description={Diameter upper bound},
  sort={e01a}
}
\newglossaryentry{D_{delta}}%
{%
  name={\ensuremath{D_{\delta}}},
  description={Effective diameter upper bound, namely  $\min\{D, l_{\delta}\}$},
  sort={e01c}
}
\newglossaryentry{N}%
{%
  name={\ensuremath{N}},
  description={Generalized/Effective dimension upper bound},
  sort={e01}
}

\newglossaryentry{K}%
{%
  name={\ensuremath{K}},
  description={Generalized Ricci lower bound},
  sort={e01}
}
\newglossaryentry{gmetric}%
{%
  name={\ensuremath{\gfrak}},
  description={Riemannian Metric},
  sort={e000}
}
\newglossaryentry{Ric}%
{%
  name={\ensuremath{Ric_{\gfrak}}},
  description={Ricci tensor },
  sort={e001a}
}
\newglossaryentry{Ric_N}%
{%
  name={\ensuremath{Ric_{\gfrak,\mu,N}}},
  description={Generalized Ricci tensor},
  sort={e001a0}
}

\newglossaryentry{Lambda_Poi}%
{%
  name={\ensuremath{\Lambda_{Poi}}},
  description={The \Poinc constant},
  sort={a010}
}

\newglossaryentry{Lambda_p_Poi}%
{%
  name={\ensuremath{\Lambda^{(p)}_{Poi}}},
  description={The p-\Poinc constant},
  sort={a010}
}
\newglossaryentry{Lambda_LS}%
{%
  name={\ensuremath{\Lambda_{LS}}},
  description={The log-Sobolev constant},
  sort={a010}
}

\newglossaryentry{si_fun}%
{%
  name={\ensuremath{\si_{\delta}(x)}},
  description={Generalized sine function},
  sort={c235}
}
\newglossaryentry{co_fun}%
{%
  name={\ensuremath{\co_{\delta}(x)}},
  description={Generalized cosine  function},
  sort={c235}
}

\newglossaryentry{Jknh}%
{%
  name={\ensuremath{J_{K, N, \hfrak}(x)}},
  description={The maximally supported model density of slope $\hfrak$},
  sort={c23}
}

\newglossaryentry{lambda_p_knd}%
{%
  name={\ensuremath{\lamp_{K,N,D}}},
  description={The sharp lower bound for the p-\Poinc constant},
  sort={b}
}

\newglossaryentry{rho_knd}%
{%
  name={\ensuremath{\rho_{K,N,D}}},
  description={The sharp lower bound for the  log-Sobolev constant},
  sort={b}
}

\newglossaryentry{sigma_knd}%
{%
  name={\ensuremath{\sigma^{(t)}_{K,\NN}(\theta)}},
  description={Distortion coefficients},
  sort={ee1}
}
\newglossaryentry{tau_knd}%
{%
  name={\ensuremath{\tau^{(t)}_{K,N}(\theta)}},
  description={Distortion coefficients},
  sort={ee1}
}

\newglossaryentry{M_knd}%
{%
  name={\ensuremath{M_{K,\NN}^{(t)}[d](a,b)}},
  description={Distorted means},
  sort={ee1a}
}
\newglossaryentry{tlM_knd}%
{%
  name={\ensuremath{\tl{M}_{K,N}^{(t)}[d](a,b)}},
  description={Distorted means},
  sort={ee1b}
}
\newglossaryentry{theta_K}%
{%
  name={\ensuremath{\theta_{K}(A_0,A_1)}},
  description={Distorted distance function},
  sort={ee1}
}

\newglossaryentry{M_cl_knd}%
{%
  name={\ensuremath{\Fknd(I)}},
  description={The `Synthetic' \red{$CDD_b(K,N,D)$} class of measures supported in $I\subset \R$},
  sort={f00}
}
\newglossaryentry{M_cl_M_knd}%
{%
  name={\ensuremath{\Fknd^M(I)}},
  description={The subset of  $\Fknd(I)$ of model measures},
  sort={f030}
}
\newglossaryentry{Phi_Ray}%
{%
  name={\ensuremath{\Phi_{u^*,v^*}(\xi)}},
  description={Generalized Rayleigh quotient},
  sort={g13a}
}
\newglossaryentry{alpha_knd}%
{%
  name={\ensuremath{\CC}},
  description={General optimal lower bound},
  sort={b}
}
\newglossaryentry{F_fun}%
{%
  name={\ensuremath{\F_*(\xi)}},
  description={Main function space},
  sort={L12}
}
\newglossaryentry{F_fun_aux}%
{%
  name={\ensuremath{\F^a_*(\xi)}},
  description={Auxiliary function space},
  sort={L12a}
}
\newglossaryentry{Ext_poi}%
{%
  name={\ensuremath{\Ext{A}}},
  description={The extreme points of a set $A$},
  sort={f0451}
}
\newglossaryentry{Ext_ray}%
{%
  name={\ensuremath{\Exr{K}}},
  description={The extreme rays of a cone $K$},
  sort={f0451}
}
\newglossaryentry{conic_hull}%
{%
  name={\ensuremath{co(x)}},
  description={The conic hull of $x$},
  sort={f0452}
}
\newglossaryentry{probabilities}%
{%
  name={\ensuremath{\P}},
  description={The Borel probability measures on $\R$},
  sort={f0}
}
\newglossaryentry{Pknd}%
{%
  name={\ensuremath{\Pknd(I)}},
  description={The set of measures  $\xi\in \P\cap \Fknd(I)$},
  sort={f041}
}
\newglossaryentry{finite_Radon}%
{%
  name={\ensuremath{\M_b}},
  description={The set of non-negative finite Radon measures on $\R$},
  sort={f0}
}

\newglossaryentry{Mknd_ac}%
{%
  name={\ensuremath{\Fknd^{ac}(I)}},
  description={The subset of $\Fknd(I)$ of a.c. measures},
  sort={f030}
}

\newglossaryentry{Mknd_Ck}%
{%
  name={\ensuremath{\Fknd^{C^{k}}(I)}},
  description={The subset of $\Fknd^{ac}(I)$ with $k$-smooth densities},
  sort={f0300}
}

\newglossaryentry{Mknd_s}%
{%
  name={\ensuremath{\Fknd^{s}(I)}},
  description={The subset of $\Fknd(I)$ of singular measures},
  sort={f030}
}
\newglossaryentry{hfrak}%
{%
  name={\ensuremath{\hfrak}},
  description={The slope of $J_{K,N,\hfrak}$ at 0},
  sort={a2}
}
\newglossaryentry{Interval_h}%
{%
  name={\ensuremath{I_h}},
  description={Convex hull of $supp(h)$ },
  sort={f0410}
}
\newglossaryentry{Pknd_h}%
{%
  name={\ensuremath{\Pkndfs(I_h)}},
  description={The set of measures  $\xi \in \Pknd^*(I_h)$ s.t. $\int hd\xi=0$ \red{( for $*=$ `ac', `s',`M', or `\,\,\,\,')}},
  sort={f041a}
}
\newglossaryentry{hatPknd_h}%
{%
  name={\ensuremath{\Pkndffs(I_h)}},
  description={The set of measures  $\xi\in \Pkndfs(I_h)$ s.t. $\int_{-\infty}^x hd\xi\neq 0 \,\,\forall x\in int(supp(\xi))$ },
  sort={f041a}
}

\newglossaryentry{supp}%
{%
  name={\ensuremath{supp}},
  description={Functions/measures support},
  sort={a000}
}

\newglossaryentry{ssupp}%
{%
  name={\ensuremath{\isupp}},
  description={The modified support},
  sort={a0000}
}

\newglossaryentry{Yknh}%
{%
  name={\ensuremath{Y_{K,N,\hfrak}(x)}},
  description={General model density function},
  sort={c234}
}

\newglossaryentry{D_knd}%
{%
  name={\ensuremath{\D^{reg}_{(K,N,D)}}},
  description={The parametric domain of regularity},
  sort={g02}
}

\newglossaryentry{xi_knd_rc}%
{%
  name={\ensuremath{\xi_{(K,N),r,c}}},
  description={Measure map to the model class},
  sort={g02}
}
\newglossaryentry{xi_knd_01}%
{%
  name={\ensuremath{\xi_{(K,N)}}},
  description={Measure map to the model class},
  sort={g02}
}
\newglossaryentry{M_knd_reg}%
{%
  name={\ensuremath{\Fkndreg}},
  description={The regularity domain},
  sort={g02}
}

\newglossaryentry{Ray_quo}%
{%
  name={\ensuremath{Ray[f](\xi)}},
  description={The Rayleigh quotient},
  sort={g13}
}

\newglossaryentry{lam_star}%
{%
  name={\ensuremath{\lamM_{K,N,D}}},
  description={The sharp lower bound for  the \Poinc constant},
  sort={b}
}

\newglossaryentry{f_p_1}%
{%
  name={\ensuremath{f^{(p-1)}(x)}},
  description={Defined as the assignment $|f(x)|^{p-2}f(x)$},
  sort={c247}
}
\newglossaryentry{pi_p}%
{%
  name={\ensuremath{\pi_p}},
  description={The constant  $\frac{2\pi}{p\sin(\frac{\pi}{p})}$},
  sort={a5}
}
\newglossaryentry{alpha_lambda}%
{%
  name={\ensuremath{\alpha}},
  description={The constant  $\prnt{\frac{\lambda}{p-1}}^{\frac{1}{p}}$},
  sort={a010}
}
\newglossaryentry{B_plus}%
{%
  name={\ensuremath{\B_{+}}},
  description={Bobkov-G\"{o}tze constant},
  sort={a010}
}
\newglossaryentry{B_minus}%
{%
  name={\ensuremath{\B_{-}}},
  description={Bobkov-G\"{o}tze constant},
  sort={a010}
}

\newglossaryentry{Isoper}%
{%
  name={\ensuremath{\Ical}},
  description={The isoperimetric profile},
  sort={a04}
}

\newglossaryentry{isop_flat}%
{%
  name={\ensuremath{\Ical^{\flat}}},
  description={Bobkov's rays isoperimetric profile},
  sort={a04}
}
\newglossaryentry{Cheg_const}%
{%
  name={\ensuremath{\hChe}},
  description={Cheeger's constant},
  sort={a010}
}
\newglossaryentry{Ledo_const}%
{%
  name={\ensuremath{\lLed}},
  description={Ledoux's constant},
  sort={a010}
}
\newglossaryentry{Bus_Led_fcn}%
{%
  name={\ensuremath{E_{Poi}}},
  description={Buser-Ledoux minimum expression},
  sort={m2}
}
\newglossaryentry{Led_fcn}%
{%
  name={\ensuremath{E_{LS}}},
  description={Ledoux's minimum expression},
  sort={m2}
}

 \newglossaryentry{cos_p}%
{%
  name={\ensuremath{\cos_p(x)}},
  description={$p$-cosine function},
  sort={c246}
}
 \newglossaryentry{sin_p}%
{%
  name={\ensuremath{\sin_p(x)}},
  description={$p$-sine function},
  sort={c246}
}

\begin{document}

 \begin{titlepage}
    \begin{center}
        \vspace*{1cm}
 
        {\Huge
        Functional Inequalities on Weighted Riemannian Manifolds Subject to Curvature-Dimension Conditions}
 
        \vspace{0.5cm}
        {\Large
        Research Thesis}
 
        \vspace{1.5cm}
        
         {\LARGE Eran Calderon}
 
 \vspace{1.5cm}
 
 {\Large A thesis submitted in partial fulfillment of the requirements for the degree of Doctor of Philosophy. }
        \vfill

        \vspace{0.8cm}

        \Large
        Submitted to the Senate of the\\
        Technion, Israel Institute of Technology, Haifa\\
        %Date: April, 2019, Tamuz, 5779
        Date: November, 2018, Kislev 5779
 
    \end{center}
    \newpage 
    The research was done under the supervision of Prof. Emanuel Milman, in the department of Mathematics of the Technion. 
    
    \vspace{1.5cm}
    
    {\Large \bf Acknowledgement }

    The research was done in the Technion, Haifa; the services and the facilities of the institute were valuable for this research. 
    
    The generous financial help of the Crown Family Doctoral Fellowship and the Technion Is gratefully acknowledged.
    
    The research leading to these results is part of a project that has received funding from the European Research Council (ERC) under the European Union's Horizon 2020 research and innovation programme (grant agreement No. 637851).
    \vspace{1.1cm}

    I wish to express my sincere gratitude to Prof. Emanuel Milman for being my professional mentor. In particular, for the exposure I was given to an interesting topic, the helpful words of advice during our discussions, the constructive criticism, and the review as well as proofreading of the thesis, which has greatly improved it.  It was a great honor to be his student. 
    
    Above all I express my sincere gratitude and appreciation to my beloved family who has always been to my side, and supported me at anytime and in anyplace.
    
\end{titlepage}
 
\begin{abstract}
In this work we establish new sharp inequalities of \Poinc or log-Sobolev type, 
on geodesically-convex weighted Riemannian manifolds $(M,\gfrak,\mu)$ whose (generalized) Ricci curvature $Ric_{\gfrak,\mu,N}$ with effective dimension parameter $N\in (-\infty,\infty]$ is 
bounded from below by a constant $K\in\R$, and whose diameter is bounded above by $D\in (0,\infty]$. When this condition holds we say that $M$ satisfies the $CDD(K,N,D)$ condition (CDD for Curvature-Dimension-Diameter).  
Specifically, we derive lower bounds for the Poincar\'{e}, p-Poincar\'{e} and log-Sobolev constants, depending on the parameters $K$, $N$ and $D$.

To this end we establish a general method which complements 
the `localization' Theorem which has recently been established by B. Klartag. 
Klartag's theorem is based on optimal transport techniques, leading to a disintegration 
of the manifold measure into marginal measures supported on geodesics of the manifold. 
This leads to a reduction of the problem of proving a n-dimensional inequality into 
an optimization problem over a class of measures with 1-dimensional supports. In this work we firstly develop a general approach which leads to a reduction of this optimization problem into a simpler optimization problem, on a subclass of measures, which will be referred to as `model measures'. This reduction is based on functional analytic techniques, in particular a classification of extreme points of a specific subset of measures, and showing that the solution to the optimization problem (which is over a non-linear function) is attained on this set of extreme points. 
This reduction is not restricted to the optimization problems associated with the three inequalities mentioned; it is general and can be in principle applied to many other functional inequalities not studied here.

By employing ad-hoc analytical techniques we solve the optimization problems 
associated with the Poincar\'{e}, p-Poincar\'{e} and the log-Sobolev inequalities subject to 
specific Curvature-Dimension-Diameter conditions. 
Notably, we prove new sharp \Poinc inequalities for $N\in (-\infty,0]$, and \pink{quite} unexpectedly we find that for $N\in (-1,0]$ the characterization of the sharp lower bound on the \Poinc constant is of completely different nature,  an observation which hints on a new phenomena; in addition we derive new lower bounds on the log-Sobolev constant under $CDD(K,\infty,D)$ conditions where $K\in\R$ and $D\in (0,\infty]$, which up to numeric constants are best possible.  
\end{abstract}
 \tableofcontents 
	
 \renewcommand{\cftfignumwidth}{6em}
 \renewcommand{\cftfigpresnum}{Figure }

 \renewcommand\cfttabaftersnum{:}
 \renewcommand\cfttabpresnum{Table~}
 \cftsetindents{tab}{1.5em}{5em}

\newpage
 \listoffigures
 \listoftables
	
\pagenumbering{arabic}

 \newpage
{\bf Common Abbreviations:}\\
\red{a.c. - Absolutely Continuous}\\
\red{a.e. - Almost Everywhere}\\
BC - Boundary Conditions\\ 
BVP - Boundary Value Problem\\
C.S. - Cauchy-Schwarz\\
\cwrm  - Convex Weighted Riemannian Manifold\\
diam - Diameter\\
IVP - Initial Value Problem\\
LS - Log-Sobolev\\
\mesi - Monotonic Exhausting Sequence of Intervals\\
resp. - Respectively\\
RHS/LHS - Right/Left Hand Side\\
s.t. - Such That\\
supp - Support\\
t.v.s. - Topological Vector Space\\
\red{u.s.c. - Upper semi-continuous}\\
w.l.o.g. - Without Loss of Generality\\
\wrm  - Weighted Riemannian Manifold\\
w.r.t. - With Respect To\\

\subsection*{Remarks about Notation :}
{\footnotesize 
\begin{itemize}
%\item Given a closed set $A$,  the statement $f\in C(\bar{A})\cap C^2(\mathring{A})$ will be written simply as $f\in C^2(A)$. 
\item We denote by $\R_+$ (resp. $\R_+^*$) the set of non-negative (resp. positive) real numbers. We denote by $\Nbb$ the set of natural numbers, and by $\N$ the set $\Nbb\cup\{0\}$. 
\item We denote by $B_x(\epsilon)$ the open ball of radius $\epsilon$ around $x$ (when the metric should be clear from the context). 
\item We denote the interior of a set $A$ either by $\mathring{A}$ or $int(A)$. The complement of $A$ is denoted by $A^c$.
\item Given two sets $A,B\neq \emptyset$, we denote by $d(A,B)$ the distance between them (the metric should be understood from the context). 
%\item If $A$ is closed (or non-open) then the set $C^k(A)$ (of $k$-times continuously differentiable functions) should be interpreted as $C^k(int(A))\cap C(A)$.  
\item Given $\epsilon>0$ we will write $A_{\epsilon}$ for the set $\{x: d(x,A)< \epsilon\}$ (for points $x,y$ on a Riemannian manifold, $d(x,y)$ stands for the Riemannian distance). 
\item $\nu$ in the proper context stands for the unit outer normal on the boundary of a manifold.
\item $UTM$ stands for the unit tangent bundle of a manifold $M$.
\item In the proper context, the symbols $D_s$ and $\partial_s$ stand for derivative $\deriv{}{s}$ and $\pd{}{s}$ w.r.t. the variable $s$. 
\item Given a domain $\Omega$ we denote by $C(\Omega)$ (resp. $C^{k}(\Omega)$/$C^{\infty}(\Omega)$) the continuous (resp. $k$-smooth/$\infty$-smooth) functions $f:\Omega\to \R$. We denote by $C(\Omega; F)$ the functions of $C(\Omega)$ whose image is contained in $F$. 
\item We denote by $C_c^{\infty}(\Omega)$ the functions $f\in C^{\infty}(\Omega)$ s.t $supp(f)$ is a compact subset of $\Omega$. We denote by $C_b(\Omega)$ the set of bounded functions on $\Omega$. We interpret $C^k(\overline{\Omega})$ as the set $C(\overline{\Omega})\cap C^k(\Omega)$. 
\item We denote by $AC(\Omega)$ (resp. $AC_{loc}(\Omega)$) the space of a.c. (resp. locally a.c.) real-valued functions on $\Omega$. 
\item We denote by $\nabla_g f$ the gradient of a function $f$ associated with the Riemannian metric $g$. We denote by $Hess_g[f]$ or $\nab_g^2 f$ the Hessian of $f$ (calculated w.r.t. the metric $g$).
	\item We define $I(x; \epsilon)$ to be the open interval $(x-\epsilon , x+\epsilon )$.
		\item Throughout this work we mostly use Greek letters for measures. We denote the Lebesgue measure by $\mu_{Leb}$, {\bf however} for the Lebesgue measure on $\R$ we exclusively use the letter $m$.
	\item When $\mu$ is a measure and $f$ is a $\mu$ measurable function, we define $\mu(f):=\int fd\mu$.
    \item We denote by $\M(I)$ the set of all non-negative Radon measures on an interval $I$, and define the set of bounded (resp. probability) measures $\xi\in \M(I)$ by $\M_b(I)$ (resp. $\P(I)$). 
	\item The letter $\xi$ will exclusively be used to denote measures on $\R$. 
	\item The support of a \blue{(signed)} measure $\xi$ on $\R$ is defined by
	\blue{
\[ supp(\xi)=\{x\in \R: |\xi|(I(x; r))>0,\qquad \forall r>0\}\,. \]
}
It is always a closed set. Similarly for a weighted Riemannian manifold with a measure $\mu$ we define \[ supp(\mu)=\{x\in M: \mu(B_x(r))>0,\qquad \forall r>0\}\,. \]
 \item We say $\xi$ is supported `on' (resp. `in') $I$ if $I=supp(\xi)$ (resp. $supp(\xi)\subset I$).
\item The support of a continuous function $f:\R\to \R$ is defined as $\overline{\{|f|>0\}}$. The support of a measurable \blue{function $f:\R\to\R\cup\{\pm \infty\}$,} denoted by $supp(f)$, is defined as the support of the measure $\xi$ defined by $d\xi=fdm$.
\item For a measurable function $f:\R\to \R\cup\{+\infty\}$ we define $\isupp(f):=supp(f \cdot m)\cap supp(f^{-1}\cdot m)$. 
\item We write $1_A(x)$ for the indicator function associated with a set $A$.
\item Integration without any specification of domain should be interpreted as integration over $\R$. 
\item Given two points $x_0,x_1\in \R$ (resp. sets $A_0,A_1\subset \R$), for any $t\in [0,1]$ we define $x_t:=\convar{x_0}{x_1}$ (resp.  $A_t=\convar{A_0}{A_1}:=\{(1-t)x_0+tx_1:\,\, x_0\in A_0,\, x_1\in A_1\}$).
\item $\P(\R)$ stands for the class of Borel probability measures on $\R$.
\item Given two functions $f,g:B\to \R_{+}^*$ defined on some domain $B$, we write $f\lesssim g$ (resp. $f\gtrsim g$) if for some constant $c>0$ it holds that that $f\leq c g$ (resp. $f\geq c g$) on $B$. We write $f\eqsim g$ if $f\lesssim g$ and $g\lesssim f$. 
\end{itemize}
}
\pagebreak
%
% \setglossarystyle{glslistdottedwidth}
 
 \printglossary[title={List of Symbols}]
%\nopagebreak

\chapter{Introduction}

\section{Curvature-Dimension-Diameter conditions}

We denote by $(M^n,\gls{gmetric})$ a smooth connected complete $n$-dimensional Riemannian manifold, with (possibly empty) boundary $\partial M$.
We say $\partial M$ is locally convex if the second fundamental form $\Pi_{\partial M}$ is non-negative on $\partial M$. $M$ is called geodesically convex if any two points in $M$ are connected by some length minimizing geodesic of $M$. We remark that geodesic convexity of $int(M)$ implies geodesic convexity of $M$, and hence (according to \cite{Bis}) local-convexity of $\partial M$. We denote by $\mu_{\gfrak}$ and by $\gls{Ric}$ the natural Riemannian measure and the Ricci tensor associated with $\gfrak$ respectively. 
\bigskip

% Throughout we assume $X,Y,Z$ are smooth vector fields on $M$. For simplicity of notation we omit specifying the point $p\in M$ at which tensors are evaluated. 

\bigskip

%%\Avg{R(e_i,X)Y}{e_i}
%
%which in terms of an orthonormal frame $\{e_i\}_{i=1}^n$ $Ric(X,Y)=\sum_{i=1}^n\Avg{R(e_i,X)Y}{e_i}$, and its components are $R_{ij}:=Ric(\partial_i,\partial_j)$. 

%
%
%Recall that the 2-tensor $Ric_{\gfrak}$ is defined by $Ric_{\gfrak}(X,Y)=Trace\prnt{Z\mapsto R(Z,X)Y}$, where $X,Y,Z$ are smooth vector fields and $R$ is the Riemann (1,3) tensor , which in coordinates can be expressed as $$R(\partial_i,\partial_j)\partial_k=\nabla_{\partial_i}\nabla_{\partial_j}\partial_k-\nabla_{\partial_j}\nabla_{\partial_i}\partial_k=\nab_{\partial}R_{kij}^l\partial_l\,.$$
%If we fixed $u\in T_pM$ then 
%$$Ric_{\gfrak}(Z,Z)=\frac{1}{n-1}\sum_{k=1}^{n-1}K_p(u,e_k)\qquad \text{where } \, K_p(Z,e_k):=R(Z,e_k,Z,e_k)\,.$$
% 

 Lower bound on the Ricci tensor, i.e. a bound of the form $Ric_{\gfrak}\geq K$ for some $K\in \R$ (which means $Ric_{\gfrak}\geq K \cdot \gfrak$ as 2 tensors), has many geometric, topological as well as analytic implications. In this work we focus on the analytic aspects, specifically functional inequalities on manifolds, under the additional assumption that the manifold's diameter $diam(M)$ is bounded above by $D\in (0,\infty]$. The setting of Riemannian manifolds has been extensively studied; in this work we consider the more general setting of Weighted Riemannian Manifolds (\wrm{}), which for the purposes of this introduction are defined to be triples $(M,\gfrak,\mu)$, where $(M,\gfrak)$ is a smooth Riemannian manifold and $\mu=e^{-V}\mu_{\gfrak}$ is a finite measure on $M$ where $V\in C^{\infty}(M; \R)$. Many results pertinent to the class of Riemannian manifolds with $Ric_{\gfrak}\geq K$ extend mutatis-mutandis to the class of weighted Riemannian manifolds which satisfy $Ric_{\gfrak,\mu,N}\geq K$, where $Ric_{\gfrak,\mu,N}$ is the generalized Ricci tensor associated with a parameter $N\in (-\infty,0]\cup [n,\infty]$. In general for $N\in (-\infty,\infty]$ we define $\gls{Ric_N}$ as follows:
\[ Ric_{\gfrak,\mu,N}:=Ric_{\gfrak}+Hess_{\gfrak}[V]-\frac{1}{N-n}\nab_{\gfrak} V\otimes \nab_{\gfrak} V \,, \]
where the last term is zero by definition when $N=\infty$, and in addition the only case where $N=n$ is when $V=\text{const}$. This object encapsulates data which is attributed to the metric as well as to the measure. The development of this object as well as the realization of its role is attributed to Bakry and \'{E}mery \cite{BE1,BE2}, who where inspired by previous works of Lichnerowicz \cite{Lic1, Lic2}.
\bigskip

We say that $(M^n,\gfrak,\mu)$ satisfies $CD(K,N)$ (Curvature-Dimension conditions) if $Ric_{\gfrak,\mu,N}\geq K$. If in addition $diam(M)\leq D$ where $D\in (0,\infty]$, we say $(M^n,\gfrak,\mu)$ satisfies $CDD(K,N,D)$ (Curvature-Dimension-Diameter conditions).
\subsubsection{Examples of spaces satisfying $CD(K,N)$}
\begin{enumerate}
	\item $\pmb{CD(0,n)}$: $(M,|\cdot|,dx)$ which corresponds to a bounded open  convex set $M\subset \R^n$ with Lebesgue measure $dx$. 
	\item $\pmb{CD(K,\infty)}$: $(\R^n,|\cdot|,\prnt{\frac{K}{2\pi}}^{\frac{n}{2}}e^{-\half K|x|^2}dx)$ ($K>0$) $\R^n$ equipped with a Gaussian measure.
	\item $\pmb{CD(n-1,n)}$: $(\Sbb^n,\gfrak_{\Sbb^n},\mu_{\gfrak_{\Sbb^n}})$ the canonical $n$-sphere. 
	\item $\pmb{CD(-(n-1),n)}$: $(\Hbb^n,\gfrak_{\Hbb^n},\mu_{\gfrak_{\Hbb^n}})$ the hyperbolic n-space. Notice that $\mu_{\gfrak_{\Hbb^n}}$ is not a finite measure (an assumption which we included in the definition of a weighted Riemannian manifold), yet the $CD(K,N)$ condition is still meaningful.  
	\item $\pmb{CD(0,-\alpha)}$: $(\R^n, |\cdot |,J_{n,\alpha}(x)dx)$ where $J_{n,\alpha}(x):=\frac{1}{(1+|x|^2)^{\frac{n+\alpha}{2}}}$ (with $\alpha>0$). These are `heavy tailed measures' (for which $\frac{1}{\prnt{J_{n,\alpha}(x)}^{\frac{1}{n-N}}}$ is convex). 
	\item $\pmb{CD(n-1-\frac{n+1}{4},-1)}$: $(\Sbb^n,\gfrak_{\Sbb^n},  \frac{1-|x_0|^2}{|y-x_0|^{n+1}}d\mu_{\gfrak_{\Sbb^n}}(y))$ where $|x_0|<1$ and $n\geq 2$. `Harmonic measures'. See \cite{Mil7} for a proof. 
\end{enumerate}

\subsection{Functional inequalities}\label{subsec:intro_func_ineq}

Subject to these conditions we study the following 3 types of functional inequalities: 
{\small 
\begin{itemize}
	\item {\bf \Poinc inequality:} For some $C>0$:
	\eql{\label{defn:PoincareIneq_intro}C \int_Mf(x)^2d\mu(x)\leq \int_M |\nab_{\gfrak} f(x)|^2d\mu(x) \qquad\qquad\qquad \forall f\in\mathcal{F}_{Poi}(M,\mu)\,,}
	where $\F_{Poi}(M,\mu):=\cprnt{0\not\equiv f\in C_c^{\infty}(M)\quad \text{s.t.}\quad \int_Mf(x)d\mu(x)=0}$. 
	\item {\bf $p$-\Poinc inequality:} For $p\in (1,\infty)$, for some $C^{(p)}>0$:
	\eql{ C^{(p)} \int_M|f(x)|^pd\mu(x)\leq \int_M |\nab_{\gfrak} f(x)|^pd\mu(x) \qquad\qquad\qquad \forall f\in \mathcal{F}^{(p)}_{Poi}(M,\mu)\,,} 
	where $\F^{(p)}_{Poi}(M,\mu):=\cprnt{0\not\equiv f\in C_c^{\infty}(M)\quad \text{s.t.}\quad \int_M |f(x)|^{p-2}f(x)d\mu(x)=0}$. 
	\item {\bf Log-Sobolev inequality:} For some $C_{LS}>0$:
	\eql{\label{defn:LSIneq_intro} \frac{C_{LS}}{2}\cdot \int_M 
	f(x)^2\log\prnt{f(x)^2}d\mu(x)\leq \int_M|\nab_{\gfrak} f(x)|^2d\mu(x) \qquad \forall f\in\mathcal{F}_{LS}(M,\mu)\,,}
	where
	$$\F_{LS}(M,\mu):=\cprnt{ f\in\Cinf(\R):\quad f^2=1+g \quad \text{where }\, g\in \F_{Poi}(M,\mu)}\,.$$
	We may equivalently express it as 
		\eql{ \frac{C_{LS}}{2}Ent_{\mu}(f^2)\leq \int_M|\nab_{\gfrak} f(x)|^2d\mu(x) \qquad \forall f\in\tl{\mathcal{F}}_{LS}(M,\mu)\,,}
		where 
		$$Ent_{\mu}(f^2) :=\int_M f(x)^2\log(f(x)^2)d\mu(x)-\int_Mf(x)^2d\mu(x)\log\prnt{\int f(x)^2d\bar{\mu}(x)}\qquad \bar{\mu}:=\frac{1}{\mu(1)}\mu$$ 
		is the Entropy of $f^2$ w.r.t. $\mu$, and 
		$$\tlF_{LS}(M,\mu):=\{ f\in\Cinf(M):\quad  f^2=c+g \quad \text{where}\quad c>0,\,\, g\in \F_{Poi}(M,\mu)\}\,.$$
	 Note that whenever $M$ is compact there is no distinction between $\tlF_{LS}(M,\mu)$  and $\{f\in \Cinf(M):\, \int f^2d\mu>0\}$. 
	
		%\[ \F_{LS}(M,\mu) :=\cprnt{ f\in \tlF_{LS}(M,\mu):\,\, \int f^2(x)d\bar{\mu}(x)=1}\,.\]
	%
	%\text{const}\not\equiv f\in C_c^{\infty}(M)\quad \text{s.t}\quad \frac{1}{\mu(M)}\int_M f(x)^2d\mu(x)=1\}$. 
\end{itemize}
}

We define the \blue{Poincar\'{e}, p-Poincar\'{e} and log-Sobolev} constants :
\eql{ 
 \label{dfn:Constants1} \gls{Lambda_Poi}(M,\gfrak,\mu)&:=\inf_{f\in \F_{Poi}(M,\mu)}\frac{\int_M |\nab_{\gfrak} f(x)|^2d\mu(x)}{\int_Mf(x)^2d\mu(x)}\,,\\
 \label{dfn:Constants2}\gls{Lambda_p_Poi}(M,\gfrak,\mu)&:=\inf_{f\in \F^{(p)}_{Poi}(M,\mu)}\frac{\int_M |\nab_{\gfrak} f(x)|^pd\mu(x)}{\int_M|f(x)|^pd\mu(x)}\,,\\
 \label{dfn:Constants3}\gls{Lambda_LS}(M,\gfrak,\mu)&:=\inf_{f\in \F_{LS}(M,\mu)}\frac{2\int_M |\nab_{\gfrak} f(x)|^2d\mu(x)}{\int f(x)^2\log(f(x)^2)d\mu(x)}\,.}
%\inf_{f\in \tlF_{LS}(M,\mu)}\frac{2\int_M |\nab_{\gfrak} f(x)|^2d\mu(x)}{Ent_{\mu}(f^2)}=
These are the best (meaning largest possible) constants $C, C^{(p)}$ and $C_{LS}$ for which the \blue{above corresponding} inequalities are satisfied. 
\begin{remk} The following abbreviations will be assumed throughout:
\begin{itemize}
	\item We write $\Lambda_{*}$ when we generally refer to any of the above constants.
	\item Whenever $M=\R$ and $\gfrak$ is the Euclidean metric $|\cdot |$, we abbreviate and write $\Lambda_{*}(\mu)$ instead of $\Lambda_{*}(\R,\gfrak,\mu)$. When $(M,\gfrak)$ is a Riemannian manifold and the measure is the Riemannian volume measure $\mu_{\gfrak}$, we abbreviate and write $\Lambda_{*}(M,\gfrak)$ instead of $\Lambda_{*}(M,\gfrak,\mu_{\gfrak})$. 
\end{itemize} 
\end{remk}
%%
%Positivity of either of the above constants implies that the corresponding inequality holds on $M$ with constant which is at least $\Lambda_{*}$. 

For many purposes it is beneficial to know how large are these constants, since these constants can quantify properties such as stability, or the rate of convergence to equilibrium values of certain processes (such as the variance or the entropy under the heat flow; more on that in the next chapter), measure concentration, etc. 
%Our main goal in this work is to derive new, sharp or asymptotic, estimates for these constants which characterizes one of the following type of inequalities: \Poinc inequalities, p-\Poinc inequalities and log-Sobolev inequalities. We study all three inequalities on  weighted Riemannian manifolds $(M^n, g,\mu)$ ($n$ dimensional Riemannian manifolds with a metric $g$ equipped with a possibly non-trivial measure $\mu$), subject to Curvature-Dimension-Diameter conditions. These conditions are characterized by 3 numbers $K,N$ and $D$, and we abbreviate $CDD(K,N,D)$ when we refer to these conditions. The parameter $D$ corresponds to an upper bound on the diameter of $M$, and the couple $K$ and $N$ corresponds to the condition $Ric_{g,N}\geq K\cdot g$ (as 2 tensors), where $Ric_{g,N}$ is a  $N$ dependent generalization of the Ricci tensor $Ric_g$ to Riemannian manifolds with weight;  it coincides with $Ric_g$ when $N=n$. The latter condition, which depends on the parameters $K$ and $N$ will be

In this work our main goal is to derive lower bounds for these three constants, depending on the parameters $K$, $N$ and $D$.  As it turns out, sharp lower bounds for these constants can be identified as the respective constants $\Lambda_{*}(\xi)$ of measures $\xi$ supported on $\R$. In order to clarify this statement we consider the following two estimates for $\Lambda_{Poi}(M,\gfrak)$ (when $\mu=\mu_{\gfrak}$):

\begin{thm}[Li-Yau 1980, Zhong-Yang 1984] \label{thm:ZhongYang} Let $(M^n,\gfrak)$ be a compact connected Riemannian manifold of dimension $n\geq 1$, with (possibly empty) locally convex boundary, and assume $Ric_{\gfrak}\geq 0$ and that $diam(M)\leq D\in (0,\infty)$; then $\Lambda_{Poi}(M,\gfrak)\geq \frac{\pi^2}{D^2}$.
\end{thm}
\smallskip 

\begin{thm}[Lichnerowicz 1958] \label{thm:Lichnerowicz} Let $(M^n,\gfrak)$ be a compact connected Riemannian manifold of dimension $n>1$, with (possibly empty) locally convex boundary, and assume $Ric_{\gfrak}\geq K>0$ ; then $\Lambda_{Poi}(M,\gfrak)\geq K\frac{n}{n-1}$.
\end{thm}

Denote by $m$ the Lebesgue measure on $\R$. The Euler-Lagrange equation which is associated with the calculation of the \Poinc constant $\Lambda_{Poi}(\R, \xi)$ for measures $\xi=p\cdot m$ supported on $[a,b]\subset \R$ gives rise to the Sturm-Liouville problem $(pf')'=-\lambda pf$ with BC $f'(a)=f'(b)=0$. For reasons which will be immediately clarified we recast these two estimates into an equivalent form:
\begin{itemize}
	\item The Li-Yao/Zhong-Yang estimate can be written as $\Lambda_{Poi}(M,\gfrak)\geq \lam_{0, n, D}$, where $\lam_{0, n, D}=\Lambda_{Poi}(1_{[\frac{D}{2},\frac{D}{2}]}\cdot m)$; indeed $\frac{\pi^2}{D^2}$ is the first non-zero eigenvalue of the Sturm-Liouville equation $f''=-\lambda f$ with BC $f'(-\frac{D}{2})=f'(\frac{D}{2})=0$.
	\item The Lichnerowicz estimate can be written as $\Lambda_{Poi}(M,\gfrak)\geq \lam_{K, n, \pi}$, where $\lam_{K, n, \pi}=\Lambda_{Poi}(\cos^{n-1}(\sqrt{\delta}x) 1_{[-\frac{\pi}{2\sqrt{\delta}},\frac{\pi}{2\sqrt{\delta}}]}(x)\cdot m)$ with $\delta:=\frac{K}{n-1}$; indeed $\frac{Kn}{n-1}$ is the the first non-zero eigenvalue of the Sturm-Liouville problem $\prnt{\cos^{n-1}(\sqrt{\delta}x)f'}'=-\lambda \cos^{n-1}(\sqrt{\delta}x) f$ with BC $f'(-\frac{\pi}{2\sqrt{\delta}})=f'(\frac{\pi}{2\sqrt{\delta}})=0$ (as can be verified by substitution of the eigenfunction $f_1(x):=\sin(\sqrt{\delta}x)$). 
	We remark that although no diameter bound was specified in the Lichnerowicz estimate (which amounts to $D=\infty$), a diameter bound essentially exists as implied by the Bonnet-Myers theorem, which states that under $Ric_{\gfrak}\geq K$ with $K>0$ it holds that $diam(M)\leq \frac{\pi}{\sqrt{\delta}}$.
\end{itemize}

\subsection{A unified framework: comparison with 1-dimensional problems}

These results can be incorporated into a unified general framework. To this end we define for $K\in \R$, $N\in (-\infty,0]\cup (1,\infty)$ and $D\in (0,\infty]$: 
\[  \delta:=\frac{K}{N-1} \qquad \text{and} \qquad \gls{D_{delta}}:=\begin{cases} \min\{D, \frac{\pi}{\sqrt{\delta}}\} & \mbox{if } \delta>0 \\ D & \mbox{otherwise} \end{cases} \,,\]
and
\[ \gls{si_fun}:=\begin{cases} \sin(\sqrt{\delta}x)/\sqrt{\delta} & \delta>0 \\
x & \delta = 0 
\\
\sinh(\sqrt{-\delta}x)/\sqrt{-\delta} & \delta<0 \end{cases} \quad \text{and} \quad
\gls{co_fun}:=\si_{\delta}(x)'=\begin{cases} \cos(\sqrt{\delta}x) & \delta>0 \\
1 & \delta = 0 
\\
\cosh(\sqrt{-\delta}x) & \delta>0 \end{cases}\quad\,. \]
Notice that $\co_{\delta}^{n-1}(x)$ coincides with the function 1 when $\delta =0$ and with $\cos^{n-1}(\sqrt{\delta}x)$ when $\delta=\frac{K}{N-1}>0$; these are the densities of the measures on $\R$ which we encountered in the equivalent formulations of the Li-Yau/Zhong-Yang and the Lichnerowicz estimates. The following theorem unifies the previous estimates into a single framework, which incorporates also weighted Riemannian manifolds. It was firstly proved for compact Riemannian manifolds (i.e. $N=n$) by P. Kr\"{o}ger \cite{Kro}, and was later extended to the setting of weighted Riemannian manifolds which satisfy $CDD(K,N,D)$ with $N\in (1,\infty]$ by D. Bakry and Z. Qian \cite{BaQi}.

\begin{thm}[Kr\"{o}ger 1992, Bakry-Qian 2000]
Let $(M^n,\gfrak,\mu)$ be a compact \wrm{} of dimension $n\geq 1$, with (possibly empty) locally-convex boundary, and which verifies $CDD(K,N,D)$ where $K\in \R$, $1<N\in [n,\infty]$ and $D\in (0,\infty)$. Then
\[ \Lambda_{Poi}(M,\gfrak)\geq \lam_{K,N,D}\,,\]
 where 
\begin{itemize}
	\item If $N\in [n,\infty)$ then $\lam_{K,N,D}=\Lambda_{Poi}(\co_{\delta}^{N-1}(x) 1_{[-\frac{D_{\delta}}{2},\frac{D_{\delta}}{2}]}(x)\cdot m)$, 
	\item If $N=\infty$ then $\lam_{K,\infty,D}=\Lambda_{Poi}(e^{-\frac{Kx^2}{2}} 1_{[-\frac{D}{2},\frac{D}{2}]}(x)\cdot m)$\,.
\end{itemize}
Moreover, these estimates are sharp. 
\end{thm}
We remark that the proofs of these theorems rely on gradient estimates of the eigenfunctions. Our first goal in this work is to fill the gap regarding the range $N\in (-\infty, 0]$. 
Throughout we will assume that $D<\frac{\pi}{\sqrt{\delta}}$ if $K<0$ and $N\leq 0$; this proviso will be clarified later on in this work (note that $D\leq \frac{\pi}{\sqrt{\delta}}$ automatically if $K>0$ and $N\geq n>1$ by a generalization of the Bonnet-Meyers theorem \cite{KTS3}). 
Since in general we permit $D=\infty$, we prefer not to work with eigenfunctions estimates (the spectral-gap might not correspond to an eigenfunction). We employ a fundamentally different approach, using tools which are based on optimal transport techniques and functional analysis. The optimal transport core of this approach has been developed in a recent work of B. Klartag \cite{Kla}, and we complement his work by a development of the functional analysis ingredient (specifically, characterization of extreme points of a certain set of probabilities). The first main result we aim to prove is the following theorem : 

\begin{thm}\label{thm:Poin_Intro}
Let $(M,\gfrak,\mu)$ be a \wrm{} of dimension $n\geq 1$, s.t. $int(M)$ is geodesically-convex, and which verifies $CDD(K,N,D)$ where $K\in \R$, $N\in (-\infty,0] \cup [\max(n,2),\infty]$ and $D\in (0,\infty]$.  Then:
\[ \Lambda_{Poi}(M,\gfrak)\geq \lam_{K,N,D}\,,\]
 where 
\begin{itemize}
	\item If $N\in (-\infty, -1]\cup [max(n,2),\infty)$ then $\lam_{K,N,D}=\Lambda_{Poi}(\co_{\delta}^{N-1}(x) 1_{[-\frac{D_{\delta}}{2},\frac{D_{\delta}}{2}]}(x)\cdot m)$, 
	\item If $N=\infty$ then $\lam_{K,\infty,D}=\Lambda_{Poi}(e^{-\frac{Kx^2}{2}} 1_{[-\frac{D}{2},\frac{D}{2}]}(x)\cdot m)$,
	\item If $N\in [-1,0]$ then $\lam_{K,N,D}=\lim_{\epsilon\to 0+}\Lambda_{Poi}(\si_{\delta}^{N-1}(x) 1_{[\epsilon+0,\epsilon+D]}(x)\cdot m)$\,,
	subject to the proviso that $D<\frac{\pi}{\sqrt{\delta}}$  when $K<0$.
\end{itemize}
Moreover, these estimates are sharp. 
\end{thm}
While for $N\in (-\infty, -1)$ our estimates are in the spirit of Kr\"{o}ger-Bakry-Qian (KBQ), for $N=-1$, the threshold point,  we have seemingly two different characterizations of the same sharp estimate, and when $N\in (-1,0]$ the KBQ type of estimates are {\bf no-longer valid}, what indicates on a new-phenomena, and we refer to the latter domain  as anomalous. Some hints for irregularity in the domain $N\in (-1,0]$ actually could be found even prior to this result. The Lichnerowicz estimate, which states that when a compact Riemannian manifold $(M,\gfrak)$ of dimension $n\geq 2$ satisfies  $Ric_{\gfrak}\geq K$ then $\Lambda_{Poi}(M,\gfrak)\geq \frac{Kn}{n-1}$, was generalized to the setting of weighted Riemannian manifolds; for a proof when $N\in (1,\infty]$ the reader is referred to \cite[p.215]{BGL}; extension to $N\in (-\infty, 0)$ can be found in \cite{Oht1} and \cite{Mil5}; specifically it is proved that if $(M,\gfrak,\mu)$ is compact with either empty or locally-convex boundary which satisfies $Ric_{\gfrak, \mu,N}\geq K$ with $N\in (-\infty,0)\cup [n,\infty]$ then $\Lambda_{Poi}(M,\gfrak,\mu)\geq \frac{KN}{N-1}$. However in \cite{Mil5} sharpness could be proved only for $N\in (-\infty,-1]\cup [n,\infty]$. It was shown in \cite{Mil3} that the Lichnerowicz estimate is no longer sharp as $N<0$ tends to 0.

\bigskip

The main tools which we employ for the proof of this theorem are not restricted to the problem of estimating $\Lambda_{Poi}(M,\gfrak,\mu)$, but will apply also to the estimates of $\Lambda^{(p)}_{Poi}(M,\gfrak,\mu)$ and $\Lambda_{LS}(M,\gfrak,\mu)$, and can easily extend to many other types of functional inequalities of similar flavor.
\smallskip

More general than the \Poinc inequality is the $p$-\Poinc inequality, with $p\in(1,\infty)$ (where $p=2$ corresponds to the classical \Poinc inequality). 
Our derivation of the sharp lower bounds relies on Sturm-Liouville properties of the eigenvalue problem associated with the $p$-\Poinc constant in dimension 1, similar to the case $p=2$. In Chapter \ref{chp:pPoinc} we list relevant properties of the `$p$--Sturm-Liouville' theory'. These properties are crucial for our analysis; most of them were verified for our specific `$p$-Sturm--Liouville problem', however we have not found a reference for one of the properties detailed in Subsection \ref{Properties:pLap_egns}. Therefore the following result is stated in a conditioned manner.

%At present there are many results supporting that  for various boundary conditions and weights, however we were unable to find a reference which supports our specific problem, hence we acknowledge that these properties need to be verified. Conditioned that the $p$ Sturm-Liouville theory associated with the 1 dimensional p-\Poinc eigenvalue problem (which is stated in Chapter \ref{chp:pPoinc}) is valid, the following result is proved:
\begin{thm}\label{thm:pPoin_Intro}
Let $(M,\gfrak,\mu)$ be a \wrm{} of dimension $n\geq 1$, s.t. $int(M)$ is geodesically-convex, and which verifies $CDD(K,N,D)$ where $K\in \R$ and $N\in [\max(n,2),\infty]$ and $D\in (0,\infty]$. Then under the technical assumption detailed in Chapter 
\ref{chp:pPoinc}, for all $p\in (1,\infty)$:
\[ \Lambda^{(p)}_{Poi}(M,\gfrak)\geq \lamp_{K,N,D}\,,\]
 where 
\begin{itemize}
	\item If $N\in [n,\infty)$ then $\lamp_{K,N,D}=\Lambda^{(p)}_{Poi}( \co_{\delta}^{N-1}(x) 1_{[-\frac{D_{\delta}}{2},\frac{D_{\delta}}{2}]}(x)\cdot m)$,
	\item If $N=\infty$ then $\lamp_{K,\infty,D}=\Lambda^{(p)}_{Poi}( e^{-\frac{Kx^2}{2}} 1_{[-\frac{D}{2},\frac{D}{2}]}\cdot m)$\,.
\end{itemize}
Moreover, these estimates are sharp. 
\end{thm}
This type of estimate is reminiscent to the lower bounds we proved for $\Lambda_{Poi}(M,\gfrak,\mu)$, yet by the time constraints of this work we have not managed to get into conclusions pertinent to the domain $N\in (-\infty, 0]$. The case $N=n$ and $K\leq 0$ has already been proved by A. Naber and D. Valtorta \cite{Val, NaVa}. 
\smallskip

For the log-Sobolev constant we derive the following  estimates:
\begin{thm}\label{thm:LS_Intro}
Let $(M,\gfrak,\mu)$ be a \wrm{} of dimension $n\geq 1$, s.t. $int(M)$ is geodesically-convex, and which verifies $CDD(K,\infty,D)$ where $D\in (0,\infty]$. Then
\[ \Lambda_{LS}(M,\gfrak,\mu)\gtrsim \Lambda_{LS}(e^{-\frac{Kx^2}{2}}1_{[-\frac{D}{2},\frac{D}{2}]}(x)\cdot m)\eqsim \begin{cases} \max\{ K, \frac{1}{D^2}\} & K>0\\  \frac{1}{D^2}& K=0 \\ \max\cprnt{ \sqrt{|K|}, \frac{1}{D}}\prnt{\frac{|K|D}{e^{\frac{|K|D^2}{8}-1}}} & K<0 \,. \end{cases}\]
Moreover, these estimates are best possible, up to universal numeric constants. 

\end{thm}
We remark that for the log-Sobolev constant we only provide estimates up-to numeric constants since the Euler-Lagrange equation, which is associated with the log-Sobolev constants $\Lambda_{LS}( \xi)$ for measures $\xi$ supported on $\R$, is not linear nor even half-linear, hence imposes more difficulties in comparison to the calculation of $\Lambda_{Poi}(\xi)$ and $\Lambda^{(p)}_{Poi}(\xi)$. 
\bigskip

\subsection{The localization paradigm}
All these results are proved by following a procedure which begins with the following reduction: rather  than estimating from below the constant $\Lambda_{*}(M,\gfrak,\mu)$ of the respective inequality on a manifold of dimension $n$, estimate from below $\inf_{\xi\in \Pknd}\Lambda_{*}(\xi)$ where $\Pknd$ is a class of probability measures $\xi$ supported in $\R$ and satisfy certain generalized concavity conditions. 
Conceptually this shares resemblance to the Kr\"{o}ger-Bakry-Qian proofs which were also relying on bounding from below $\Lambda_{Poi}(M,\gfrak,\mu)$ by $\inf_{\xi\in \Pknd}\Lambda_{Poi}(\xi)$; however on a technical level the methods which lead to this comparison are fundamentally different. We approach the problem via a method which conceptually dates back to the 1960's, in the work of Payne and Weinberger on estimating the \Poinc constant on convex domains in Euclidean space, where they applied iterative bisections of a convex body in order to reduce the problem from dimension $n$ to dimension 1 (`localization'). The localization paradigm gained further development notably in the works of Gromov and V. Milman \cite{GM}, and Kannan-Lovasz-Simonovits \cite{LS,KLS}. Recently B. Klartag \cite{Kla} has established an extension of the localization paradigm to the (weighted) Riemannian setting using optimal transport techniques. 
\bigskip 

We firstly state the Payne-Weinberger result and provide an elaborated sketch of its proof, since it serves as the prototype problem which inspired our approach. 
\begin{thm}[Payne-Weinberger \cite{PaWe}] Let $K\subset \R^n$ be a convex open set of diameter $D\in (0,\infty)$, and {\bf $\mu$  the Lebesgue measure} on $K$, then
for any $f\in C^1(K;\R)$ s.t. $\int_K fd\mu=0$ it holds that:
\[ \frac{\pi^2}{Diam(K)^2}\int_K f^2d\mu \leq \int_K |\nab f|^2d\mu\,.\]
\end{thm}
\begin{proof} The proof relies on iterative bisections defined inductively, and a statement regarding the \Poinc constant of the limit objects.
Goal: Show that for any $f\in \Cinf(K; \R)$ s.t. $\int_K fd\mu=0$ holds: 
\[ \frac{\pi^2}{Diam(K)^2}\int_K f^2d\mu \leq \int_K |\nab f|^2d\mu \,. \]
\begin{itemize} 
	\item {\bf Step 1:} Find a hyperplane $H$ passing through the barycenter of $K$ s.t. $\int_{K\cap H_{+}}fd\mu = \int_{K\cap H_{-}}fd\mu=0$. 
	\item It suffices now to prove $\frac{\pi^2}{Diam(K\cap H_i)^2}\int_{K\cap H_i} f^2d\mu \leq \int_{K\cap H_i} |\nab f|^2d\mu $ for $i\in\{+,-\}$. Indeed if this holds then:
	\eq{&\frac{\pi^2}{Diam(K)^2}\int_K f^2d\mu \leq \min_{i\in\{+,-\}}\cprnt{\frac{\pi^2}{Diam(K\cap H_i)^2}}\sum_{i\in\{+,-\}} \int_{K\cap H_i} f^2d\mu\\&\leq 
	\sum_{i\in\{+,-\}} \frac{\pi^2}{Diam(K\cap H_i)^2}\int_{K\cap H_i} f^2d\mu
	\leq \sum_{i\in\{+,-\}}\int_{K\cap H_i} |\nab f|^2d\mu= \int_{K} |\nab f|^2d\mu\,. }
	\item {\bf Step $k$:} bisect the current $2^{k-1}$ convex bodies through their barycenter; obtain $2^k$ convex bodies $\{K_j\}_{j=1..2^k}$ s.t. $\int_{K_j}fd\mu=0$ for all $j\in \{1,...,2^{k}\}$. 	It is now sufficient to prove $\Lambda_{Poi}(K_j, \gfrak_{Euc})\geq \frac{\pi^2}{diam(K_j)^2}$ for each of the bodies $\{K_j\}_{j=1..2^k}$.	

		\item In the limit $k\to \infty$ we get :
	\begin{itemize}
		\item A partition of $K$ into segments: $\{K_q\}_{q\in Q}$ ($Q$ is a set of representative points).
		\item A disintegration of the measure: probability measures $\{\mu_q\}_{q\in Q}$ and a measure $\zeta$ on $Q$ s.t. $\mu=\int_{q\in Q}\mu_q d\zeta(q)$.  
		\item For $\zeta$ a.e. $q\in Q$: $\mu_q$ is supported on $K_q$, $d\mu_q=p_q dm$ with $p_q^{\frac{1}{n-1}}$ a concave function, and $\int_{K_q}fd\mu_q =0$. 
	\end{itemize}
\end{itemize} 
Indeed, if we fix coordinates so that $K_q$ coincides with the line $(0, l)$ along the $\bar{e}_1$ axis of $\R^n$, then 
\begin{itemize}
	\item We can identify $p_q(t)$ as the limit (as $j\to\infty$) of $p_q^{(j)}(t):=\frac{\mu_{n-1}(K_j\cap \{x\in \R^n:\, x_1=t\})}{\mu_{n}(K_j)}$.\\
  \item $p_q^{(j)}(t)^{\frac{1}{n-1}}$ is concave by the Brunn-Minkowski theorem.
\end{itemize}

The final ingredient is a 1-dimensional lemma. We define by $\P_{n,D}$ the class of probability  measures $\xi=p\cdot m$ supported on $I=[a,b]\subset \R$ s.t. $|b-a|\leq D$ and s.t. $p^{\frac{1}{n-1}}$ is concave on its support $I$. We define $\P^M_{n,D}$ to be the subclass of $\P_{n,D}$ of `model' measures $\xi=p\cdot m$ s.t. $p^{\frac{1}{n-1}}$ is affine on $I$ (i.e. convex as well as concave).

\begin{lem} If $\xi\in \P_{n,D}$ then for any smooth function $f$ on $I$ s.t. $\int_Ifd\xi=0$: 
\[ \frac{\pi^2}{D^2}\int_I f(x)^2p(x)dm(x) \leq \int_I |\nab f(x)|^2p(x)dm(x) \]
\end{lem}
In order to prove the lemma one shows that the lowest non-zero eigenvalue of any Sturm-Liouville Neumann problem
   \[ (pf')'+\lambda pf=0 \qquad f'(a)=f'(b)=0\,,\]
with $b-a\in (0,D]$, which can also be identified as $\Lambda_{Poi}([a,b],|\cdot |, pdm)$, is at least $\frac{\pi^2}{D^2}$, 
 the lowest non-zero eigenvalue of the Sturm-Liouville problem
	\[ (f')'+\lambda f=0 \qquad f'(-\frac{D}{2})=f'(\frac{D}{2})=0\,, \]
	which is associated with $p\equiv \frac{1}{D}$ and $b=-a=\frac{D}{2}$. 
	This can be shown by simple ODE arguments, which we do not give here. 
 Clearly the normalized Lebesgue measure $d\xi_{*} := \frac{1}{D}  dm$ is in the class $\mathcal{P}^M_{n,D}$. 
 Therefore the lemma literally states that 
\[ \inf_{\xi\in \mathcal{P}_{n,D}}\Lambda_{Poi}(\xi)=\inf_{\xi\in \mathcal{P}^M_{n,D}}\Lambda_{Poi}(\xi)=\Lambda_{Poi}(\xi_{*})=\frac{\pi^2}{D^2}\,,\]

and we conclude that $\Lambda_{Poi}(K,|\cdot |, m)\geq \inf_{\xi\in \P_{n,D}}\Lambda_{Poi}(\xi)= \frac{\pi^2}{D^2}$.
\end{proof}
\bigskip

This proof provides the conceptual basis for our approach. Since we consider general weighted Riemannian manifolds, hyperplane bisections are no longer meaningful, and the localization is established by optimal transport methods using Klartag's `Localization Theorem' from which it follows that 
\eql{\label{KlartagsOptimization}\Lambda_{*}(M,\gfrak,\mu)\geq \inf_{\xi\in \Pknd}\Lambda_{*}(\xi)\,,} 
where the densities of $\xi\in \Pknd$ satisfy a concavity property analogous to the $\frac{1}{n-1}$ concavity of the densities of the measures $\xi\in \P_{n,D}$. 

%This would imply that the sharp bound can be expressed as:
%$$\Lambda_{*}(M,\gfrak,\mu)\geq \inf_{\xi\in \Pknd^M}\Lambda_{*}(\xi)\,.$$

%
%
%We obtain the main results \ref{thm:Poin_Intro}, \ref{thm:pPoin_Intro} and \ref{thm:LS_Intro} by solving this optimization problem using different methods. For the \Poinc constant we achieve this . For the $p$-\Poinc constant we use Pruffer-transformations (following Naber-Valtorta \cite{Val,NaVa}). For the log-Sobolev constant we could only perform the analysis up-to numeric constants, by employing a criterion for the log-Sobolev inequality due to Bobkov-G{\"o}tze. 
%
%
%
%
%
%Moreover this estimate is sharp. 
%
%
\subsection{The solution to the optimization problem \eqref{KlartagsOptimization}}

Our main results concentrate around the reduction of problem \eqref{KlartagsOptimization} into a simpler optimization problem, and eventually obtaining 
the sharp constants associated with the $3$ functional inequalities mentioned above (i.e. Theorems \ref{thm:Poin_Intro}, \ref{thm:pPoin_Intro} and \ref{thm:LS_Intro}).
\begin{enumerate}
	\item We show that \eql{\label{KlartagsOptimization2} \inf_{\xi\in \Pknd}\Lambda_{*}(\xi)=\inf_{\xi\in \Pknd^M}\Lambda_{*}(\xi)\,,}
	where the relation between  $\Pknd^M$, the subclass of `model' measures,  and $\Pknd$, is analogous to the relation between  $\P^M_{n,D}$ and $\P_{n,D}$. The proof of this step is within the lion-share of this work. We achieve it by establishing an abstract framework, which will make it feasible to apply functional analytic techniques.  Most importantly, we identify the extreme points of (a certain subset of) $\Pknd$ with (a certain subset of) $\Pknd^M$, and conclude from standard theorems (specifically the Krein-Milman/Bauer's theorems) the identity \eqref{KlartagsOptimization2}. This approach is inspired by the work of M. Fradelizi and O. Gu\'{e}don \cite{FG}, who treated the case $K=0$ on $\R^n$. Below we further elaborate on this step.
	\item We solve the following optimization problems:
	  \begin{itemize}
			\item $\inf_{\xi\in \Pknd^M}\Lambda_{Poi}(\xi)$ (using primarily results from Sturm-Liouville theory);
			\item $\inf_{\xi\in \Pknd^M}\Lambda_{Poi}^{(p)}(\xi)$ (using Pr\"{u}ffer transformation, specific ODE techniques, and some `$p$--Sturm-Liouville' theory as detailed in Chapter \ref{chp:pPoinc}; we remark that the validity of a specific result of the theory to our specific problem needs to be verified)\pinka{;}
			\item $\inf_{\xi\in \Pknd^M}\Lambda_{LS}(\xi)$ (using a result of S. Bobkov and F. G{\"o}tze about estimates of the log-Sobolev constant).
		\end{itemize} 
		As we mentioned above the solution to the log-Sobolev problem is only up to universal numeric constants. 
	\item It turns out that up to scaling and translations operations (i.e.  $J(x)dm\mapsto c\,J(x+r)dm$ with $r\in \R$ and $c>0$) a measure $\xi=J\cdot m\in\Fknd^M$ can be represented as $d\xi_{\hfrak,d}=J_{K,N,\hfrak}(x)1_{[-\frac{d}{2},\frac{d}{2}]}(x)dx$
where $J_{K,N,\hfrak}:\R\to\R_{+}\cup\{+\infty\}$, and $[-\frac{d}{2},\frac{d}{2}]\subset \isupp(J_{K,N,\hfrak})$, is such that $J_{K,N,\hfrak}(0)=0$ and $J_{K,N,\hfrak}'(0)=\hfrak\in\R$. 
	Therefore due to invariance of $\Lambda_{*}(\xi)$ under the stated scaling and translation operations, we can parametrize the image set $\Lambda_*(\Fknd^M)$ with only two parameters: $(\hfrak,d)\mapsto \Lambda_*(\xi_{(\hfrak,d)})$. While the dependence $d\mapsto \Lambda_*(\xi_{(\hfrak,d)})$ is rather clear, it is highly non-trivial for the parameter $\hfrak$. We complement the results of KBQ pertaining to $\Lambda_{Poi}(\xi)$ and of Naber-Valtorta pertaining to $\Lambda_{Poi}^{(p)}(\xi)$ by showing that $\hfrak \mapsto \Lambda_{Poi}(\xi_{(\hfrak,d)})$ and $\hfrak\mapsto \Lambda^{(p)}_{Poi}(\xi_{(\hfrak,d)})$ depend monotonically on $|\hfrak|$. From this monotonicity we conclude the explicit solutions to the optimization problems, which are essentially Theorems \ref{thm:Poin_Intro} and \ref{thm:pPoin_Intro}. The anomalous nature of the range $N\in [-1,0]$ manifests in a reversal of the direction of monotonicity, when compared to the range $N\in (-\infty,-1]\cup [n,\infty]$ (where $N=-1$ signifies a threshold point, at which  $\hfrak \mapsto \Lambda_{Poi}(\xi_{(\hfrak,d)})$ and $\hfrak\mapsto \Lambda^{(p)}_{Poi}(\xi_{(\hfrak,d)})$ are independent of $\hfrak$). 
\end{enumerate}

\subsubsection{Characterization of extreme points}
%The , which constitutes the basis for the derivation of results \ref{thm:Poin_Intro}, \ref{thm:pPoin_Intro} and \ref{thm:LS_Intro}
%is the methodic reduction from optimization over the class $\Pknd$ to the smaller class $\Pknd^M$. This reduction is not restricted to the 3 functional inequalities specified above; it can be applied to a myriad of functional inequalities of similar form. Specifically we used the following properties:
We stress that the identity \eqref{KlartagsOptimization2} is very general, it is not exclusive to the 3 inequalities which were mentioned above. It basically relies on the following properties:
\begin{enumerate}
	\item The constant $\Lambda_{*}(\xi)$ is defined by an optimization problem of the form
 $$\text{Find: }\qquad \inf_{f\in \F_{*}(\xi)}\Phi_*[f](\xi)\,,$$
where $\Phi_*[f](\xi)$ is the ratio between two non-negative bounded linear functionals in $\xi$ which depend on $f$: $\xi\mapsto u_f^*(\xi):=\int u_fd\xi$ and $\xi\mapsto v_f^*(\xi):=\int v_fd\xi$ (for example for the \Poinc constant $\Lambda_{Poi}(\xi)$ we set $u_f^*(\xi)=\int f'(x)^2d\xi$ and $v_f^*(\xi)=\int f(x)^2d\xi$). 
  \item The function space $\F_*(\xi)$ can be expressed as $\F_*(\xi)=\{f\in \F_*^a:\quad h_f^*(\xi)=0\,\}$, where $\F_*^a$ is an auxiliary function space independent of $\xi$, and $\xi\mapsto h_f^*(\xi):=\int h_fd\xi$, for some $h_f\in \Cinf_c(\R)$ associated with $f$ (for example, for the calculation of $\Lambda_{Poi}(\xi)$ we set $\F_*^a:=\{0\not\equiv f\in \Cinf_c(\R)\}$ and $h_f^*(\xi)=\int fd\xi$\,\,). 
\end{enumerate}
Conditioned that certain relations hold between the supports of $u_f, v_f$ and $h_f$, these properties allow us to make the following minimization reordering: 
\eql{\label{eqn:reordering}\inf_{\xi\in \Pknd}\Lambda_*(\xi)=\inf_{\xi\in \Pknd}\,\,\inf_{f\in \F_{*}(\xi)}\Phi_*[f](\xi)=\inf_{f\in \F_{*}^a}\,\,\inf_{\substack{\xi\in \Pknd(I_h)\\h_f^*(\xi)=0}}\Phi_*[f](\xi)\,,}
where $I_{h_f}:=conv(supp(h_f))$ and $\Pknd(I_{h_f}):=\{\xi\in\Pknd:\,\, supp(\xi)\subset conv(supp(h_f)) \}$. 

Given a (not necessarily convex) set $A$ in a linear space, we define $\Ext{A}$ to be the set of extreme points of $A$. The main result on which all estimates are based is the following characterization theorem, which is formulated in a simplified form, prioritizing clarity over preciseness:
\begin{thm}\label{thm:exPointsIntro} For $K\in \R$, $N\in (-\infty,0]\cup [2,\infty]$ and $h_f\in \Cinf_c(\R)$: 
{\footnotesize 
$$\Ext{\cprnt{\xi\in \Pknd(I_{h_f}):\quad h_f^*(\xi)=0}}=\cprnt{\xi\in \Pknd^M(I_{h_f}) \quad s.t.\,\, \int_{-\infty}^x h_fd\xi\neq 0 \,\,\text{for any } \,x\in int(supp(\xi))} \,.$$ 
}
%$\substack{\xi\in \Pknd^M\\ supp(\xi)\subset \\h_f^*(\xi)=0}
\end{thm}
As we previously mentioned, the case $K=0$ was proved by Fradelizi-Gu\'{e}don in \cite{FG} and we extend their approach. Using functional analysis arguments we conclude the following corollary  from Theorem \ref{thm:exPointsIntro}:
\begin{cor}\label{cor:from:toModel} For $K\in \R$, $N\in (-\infty,0]\cup [2,\infty]$ and $h_f\in \Cinf_c(\R)$:
$$\inf_{\substack{\xi\in \Pknd(I_h)\\h_f^*(\xi)=0}}\Lambda_{*}(\xi)=\inf_{\substack{\xi\in \Pknd^M(I_h)\\h_f^*(\xi)=0}}\Lambda_{*}(\xi)\,.$$
\end{cor}
This corollary followed by a second reordering of \eqref{eqn:reordering} yields the identity \eqref{KlartagsOptimization2}. 
\subsubsection{The structure of this work}
\begin{enumerate}
  \item[Chapter \ref{chp:Background}:] We provide background on consequences of lower bounds on the Ricci curvature. We discuss about the $3$ functional inequalities mentioned above, in particular implications and previous results;  lastly we give mathematical background to some of the main tools which underlie the derivation of our results, specifically localization via optimal transport and the Bobkov-G{\"o}tze estimates.

	\item[Chapter \ref{chp:ExPoints}:]  We develop the abstract framework which is necessary for the extreme points characterization theorem. To this end we provide background to $CDD(K,N,D)$ conditions, formulate Klartag's Localization theorem, then after a discussion about equivalent forms of the $CDD(K,N,D)$ conditions we formulate 'synthetic conditions' which are essentially  
 $CDD(K,N,D)$ conditions in the absence of smoothness; we define the synthetic classes, prove important properties, then we reformulate the problem in an abstract
manner; lastly we prove the extreme points characterization theorem.

\item[Chapter \ref{chp:GeneralSetting}:] This chapter can be considered as a prelude to the following 3 chapters.  In this chapter we recast the main problems into a convenient general setting, and also prove a general diameter monotonicity lemma which we will need for Chapters \ref{chp:Poinc}-\ref{chp:LS}.

\item[Chapter \ref{chp:Poinc}:] We derive sharp lower bounds for the \Poinc constant of a \wrm{}. We will mostly rely on classical results from the theory of Sturm-Liouville boundary value problems, which arise from the Euler-Lagrange equation associated with $\Lambda_{Poi}(\xi)$. 
\item[Chapter \ref{chp:pPoinc}:] We derive sharp lower bounds for the p-\Poinc constant of a \wrm{}. We will again approach the optimization problem via the Euler-Lagrange equation associated with $\Lambda^{(p)}_{Poi}(\xi)$; however since the equation is not linear but only half-linear, we will prove the claims using Pr\"{u}ffer-transformations into polar functions.
\item[Chapter \ref{chp:LS}:] We derive lower bounds for the log-Sobolev constant under $CDD(K,\infty,D)$ conditions, which are sharp up to universal numeric constants. To this end we use a method developed by Bobkov-G{\"o}tze \cite{BobG}, which is fundamentally different from the methods we used for $\Lambda_{Poi}(\xi)$ and $\Lambda^{(p)}_{Poi}(\xi)$. 
Except for establishing this lower bound we prove new results regarding the existence of minimizers (i.e. functions $f\in \F_{LS}(\xi)$) realizing $\Lambda_{LS}(\xi)$), as well as equivalences of the worst \Poinc and log-Sobolev constants under $CDD(K,\infty,D)$.
\item[Chapter \ref{chp:epilogue}:] We briefly discuss about  the main contributions of this work.

\end{enumerate}

 \chapter{Background}
\label{chp:Background}
\section{Bounded Ricci Curvature}

The goal of this section is to give a panoramic overview of  how measure and geometry are intertwined in certain inequalities, and how a single notion - `generalized Ricci curvature', embodies these two seemingly distinct notions. Our presentation is rather classical, and ignores the modern `synthetic' definitions of bounded Ricci curvature which applies not only to manifolds but also to general metric measure spaces without assuming any differential structure. This introduction is based mainly on \cite{Oht2}, and partially on \cite{Cha1} and \cite{Vil}. The interested reader is referred to these sources for a deeper exposition to the subject. Throughout this brief review basic knowledge of Riemannian geometry is assumed. 

\bigskip

The classical Brunn-Minkowski (BM) inequality in $\R^n$ states that for any two compact sets $A,B\subset \R^n$:
\eq{ m(\convar{A}{B})^{\on}\geq \convar{m(A)^{\on}}{m(B)^{\on}}\,. }
More generally it is possible to show that a weighted Euclidean space $(\R^n, \mu=e^{-V}m)$ (here $V\in \Cinf(\R^n)$) satisfies the following BM analogue for sets $A,B\subset\R^n$:
\eq{ \mu(\convar{A}{B})^{\frac{1}{N}}\geq \convar{\mu(A)^{\frac{1}{N}}}{\mu(B)^{\frac{1}{N}}}\,,}
with $N\in (n,\infty)$ iff for all $v\in T_x(\R^n)$ 
\eql{\label{HessCond} Hess[V](x)(v,v)-\frac{\Avg{\nabla V(x)}{v}^2}{N-n}\geq 0  \,, }
where $Hess[V]$ stands for the Hessian of $V$. We refer to $N$ as generalized (or effective) dimension since it turns-out that this is the right parameter which characterizes also similar inequalities, replacing the role of the dimension $n$.

Let us consider now Riemannian manifolds, or even more generally weighted Riemannian manifold $(M^n,\gfrak,\mu=e^{-V}\mu_{\gfrak})$, where $(M^n,\gfrak)$ is a $n$-dimensional Riemannian manifold, and $\mu$ is a finite measure on $M$ with a smooth strictly positive density $e^{-V}$. The information about its curvature is encapsulated in the Riemann tensor $R_{\gfrak}$. This is the (3,1) tensor defined by $$R_{\gfrak}(X,Y)Z=\nab_X\nab_YZ-\nab_Y\nab_XZ-\nab_{[X,Y]}Z\,,$$
for every three smooth vector fields $X,Y,Z$ on $M$ (in local  coordinates  $R_{\gfrak}(\partial_i,\partial_j)\partial_k=\nab_{\partial_i}\nab_{\partial_i}\partial_k-\nab_{\partial_j}\nab_{\partial_i}\partial_k:=R_{ijk}^l\partial_l$). Accordingly, the Ricci tensor is the $(2,0)$ tensor defined by
$$Ric_{\gfrak}(X,Y)=Trace\prnt{Z\mapsto R_{\gfrak}(Z,X)Y}$$
(in local coordinates $R_{ij}:=Ric_{\gfrak}(\partial_i,\partial_j)=\sum_{i=1}^nR_{ijk}^i$), or equivalently if $e_1,...,e_{n-1}$ is an orthonormal frame spanning $X^{\bot}$: 
$$Ric_{\gfrak}(X,X)=\sum_{i=1}^{n-1} K(X,e_i)$$
where $K(X,Y):=	\frac{\Avg{R_{\gfrak}(X,Y)Y}{X}}{|X|^2|Y|^2-\Avg{X}{Y}}$  are the sectional curvatures of the 2-plane $Span(X,Y)$.
Intuitively while sectional curvatures quantify distance convergence/divergence of neighboring geodesics, Ricci curvature quantifies volume distortion along a geodesic. Indeed for a Riemannian manifold $(M^n,\gfrak,\mu=\mu_{\gfrak})$ it emerges as the coefficient of the first non-constant term in the expansion of $\sqrt{det(\gfrak_{ij})}$ in geodesic polar coordinates. The simple condition 
\eql{\label{bddRiccicond} Ric_{\gfrak}\geq K\gfrak}
(in the sense of quadratic forms) is known to have many implications of either analytic, geometric or even topological nature. For a general weighted Riemannian manifold $(M^n,\gfrak,\mu=e^{-V}\mu_{\gfrak})$ it was realized that an object which encapsulates the interaction of curvature with measure
is the generalized Ricci tensor defined by 
\eql{\label{def:ric_gmu} Ric_{\gfrak,\mu,N}(v):=\begin{cases} Ric_{\gfrak}(v)+Hess_{\gfrak}[V](v,v)-\frac{\Avg{\nab_{\gfrak} V}{v}^2}{N-n}&\mbox{ if } N\in (n,\infty)\\
Ric_{\gfrak}(v)&\mbox{ if } N= n\end{cases}
\,,}
where the case $N=n$ is possible only for $V=const$ (i.e. a non-weighted Riemannian manifold). 
As it turns out, the two conditions \eqref{HessCond} and \eqref{bddRiccicond} can be unified into a single condition
\eq{ Ric_{\gfrak,\mu,N}\geq K\gfrak \,,}
which applies to general weighted Riemannian manifolds, and not just to the marginal cases of Riemannian manifolds (signifying pure geometric effects) or weighted Euclidean spaces (signifying pure measure effects).
Here as before we interpret $N$ as a generalized  dimension. This condition reduces to conditions \eqref{HessCond} and \eqref{bddRiccicond} in the respective mentioned marginal cases; however, when we consider the general setting of weighted Riemannian manifolds, this condition leads to   extensions of many previously known results which were known to be correct for the marginal cases. 

As an instructive example of such a mixed measure-geometry result, one can show that the following generalized form of BM inequality holds on $(M^n,\gfrak,\mu=e^{-V}\mu_{\gfrak})$ if $Ric_{\gfrak,\mu,N}\geq K\gfrak$ with $N\in (n,\infty)$:  if $A,B\subset M$ are such that $\mu(A)\mu(B)>0$, $t\in [0,1]$, and $Z_t(A,B)$ is the set of points $\g(t)$, where $\g:[0,1]\to M$ is a length minimizing geodesic with $\g(0)\in A$ and $\g(1)\in B$, then
\eq{ \mu(Z_t(A,B))^{\frac{1}{N}}\geq \tau_{K,N}^{(1-t)}(\theta_K(A,B))^{\frac{1}{N}}\mu(A)^{\frac{1}{N}}+\tau_{K,N}^{(t)}(\theta_K(A,B))^{\frac{1}{N}}\mu(B)^{\frac{1}{N}}\,. }

Here $\theta_K(A,B)\in [0,diam(M)]$ are functions which depend on $K$ and the geometry of the sets, and $\tau_{K,N}^{(t)}(d)$ are some functions which for fixed $K\in \R$ and $N\in (n,\infty)$ depend continuously on $t\in [0,1]$ and $d\in [0,diam(M)]$; in addition $\tau_{0,N}^{(t)}(d)^{\frac{1}{N}}=t$, $\tau_{K,N}^{(0)}(d)=1-\tau_{K,N}^{(1)}(d)=0$. This suggests that the tensor $Ric_{\gfrak,\mu, N}$ accounts for both the measure $\mu$ and the curvature volume distortion effects. 
Moreover, one should notice that definition \eqref{def:ric_gmu} actually makes perfect sense for any $N\in (-\infty, \infty]$ as long as we consider $N=n$ only for measures $\mu=e^{-V}\cdot \mu_{\gfrak}$ where $V$ is constant, and we identify $\frac{1}{\infty}$ as $0$. We remark that it is common to refer to the 2-tensor $Ric_{\gfrak,\mu,\infty}:=Ric_{\gfrak}+Hess_{\gfrak}[V]$ as the Bakry-\'{E}mery tensor \cite{BE2}. 

We provide one more example (of historical importance) where the role of $Ric_{\gfrak,\mu,N}$ is manifested. We define $L_{\gfrak,\mu}$ to be the linear operator for which the following integration by parts formula holds
\eql{\label{eq:intByParts} \int_M L_{\gfrak,\mu}(f)hd\mu=\int_{M}\gfrak(\nab_{\gfrak} f, \nab_{\gfrak} h)d\mu \qquad \forall f,h\in C_0^{\infty}(M)\,.}
If $\mu=e^{-V}\mu_{\gfrak}$ then \pink{$\gls{L_g_mu}=-(\Delta_{\gfrak}-\gfrak(\nabla_{\gfrak} V, \nabla_{\gfrak}))$} where \green{$\Delta_g$} is \green{(the non-positive)} second-order diffusion operator known as the Laplace-Beltrami operator. In the particular case of $\mu=\mu_{\gfrak}$ we get $L_{\gfrak,\mu_{\gfrak}}=-\gls{Delta_g}$  which is known to satisfy the Bochner identity
\eq{ \half \Delta_{\gfrak}(|\nab_{\gfrak} f|^2)={\gfrak}(\nab_{\gfrak} f, \nab_{\gfrak}(\Delta_{\gfrak} f))+||Hess_{\gfrak}( f)||_{HS}^2+Ric_{\gfrak}(\nab_{\gfrak} f, \nab_{\gfrak} f)\,. }
Using this identity Lichnerowicz \cite{Lic} proved that if $(M,{\gfrak})$ is a compact and connected \pink{manifold}, $Ric_{\gfrak}(\cdot, \cdot)\geq K {\gfrak}(\cdot,\cdot)$ with $K>0$, and $\Lambda_{Poi}(M, {\gfrak},\mu_{\gfrak})$ is the first non-zero eigenvalue of \pinka{(minus)} the Laplacian $-\Delta_{\gfrak}$, then $\Lambda_{Poi}(M, {\gfrak},\mu_{\gfrak})\geq \frac{Kn}{n-1}$, which amounts to the \Poinc inequality 
\eq{ \int_M \prnt{f-\dashint_{M}fd\mu_{\gfrak}}^2d\mu_{\gfrak}\leq \frac{n-1}{Kn}\int_{M}|\nab_{\gfrak} f|^2d\mu_{\gfrak}\qquad \forall f\in C^{\infty}(M)\,. }
His estimate generalizes naturally to the framework of compact weighted manifolds $(M,{\gfrak},\mu)$ satisfying $Ric_{\gfrak,\mu,N}(\cdot, \cdot)\geq K {\gfrak}(\cdot,\cdot)$. We will briefly show how the condition on $Ric_{\gfrak,\mu,N}$ emerges. 

%The Poincare constant $\Lambda_{Poi}(M, {\gfrak},\mu_{\gfrak})$ is then identified as the first non-zero eigenvalue of $L_{\gfrak,\mu}=\Delta_{\gfrak}-\gfrak(\nab_{\gfrak}, \nab_{\gfrak})$. 

With respect to $L_{\gfrak,\mu}$ the following form of the Bochner identity holds 
\green{
\eq{ -\half L_{\gfrak,\mu}(|\nab_{\gfrak} f|^2)=-\gfrak(\nab_{\gfrak} f, \nab_g(L_{\gfrak,\mu} f))+||Hess_{\gfrak}( f)||_{HS}^2+Ric_{\gfrak,\mu,\infty}(\nab_{\gfrak} f, \nab_{\gfrak} f) \,.} }
Since $H_{ij}(f):=[Hess_{\gfrak}(f)]_{ij}$ satisfies the inequality \green{$Tr(H(f))^2=\prnt{\slim_{i=1}^n H_{ii}(f)}^2\leq n||H(f)||_{HS}^2 $}, and \green{$ Tr(H(f))=\Delta_{\gfrak}(f)=-L_{\gfrak,\mu}(f)+\gfrak(\nab_{\gfrak}V,\nab_{\gfrak} f)$} satisfies $\frac{(\Delta_{\gfrak} f)^2}{n}\geq \frac{(L_{\gfrak,\mu}f)^2}{N}-\frac{\gfrak(\nab_{\gfrak}f,\nab_{\gfrak} V)^2}{N-n}$ 
due to the inequality

\[ \frac{(a+b)^2}{n}=\frac{a^2}{N}-\frac{b^2}{N-n}+\frac{N(N-n)}{n}\prnt{\frac{a}{N}+\frac{b}{N-n}}^2\geq \frac{a^2}{N}-\frac{b^2}{N-n}\,, \] 
we can conclude the inequality 
\green{
\eq{ -\half L_{\gfrak,\mu}(|\nab_{\gfrak} f|^2)\geq -\gfrak(\nab_{\gfrak} f, \nab_{\gfrak}(L_{\gfrak,\mu} f))+\frac{(L_{\gfrak,\mu}f)^2}{N}+Ric_{\gfrak,\mu,N}(\nab_{\gfrak} f, \nab_{\gfrak} f)  \,.} }
integration over $M$ followed by usage of identity \eqref{eq:intByParts} gives: 
\eql{\label{LicIneq} 0\geq -\prnt{1-\frac{1}{N} }\int_M (L_{\gfrak,\mu}f)^2d\mu+\int_M Ric_{\gfrak,\mu,N}(\nab_{\gfrak} f, \nab_{\gfrak} f)d\mu\,.}
Let $f_1$ be a normalized eigenfunction of $L_{\gfrak, \mu}$, corresponding to its first non-zero eigenvalue $\Lambda_{Poi}(M, \gfrak,\mu)$. For $f=f_1$ holds the identity $$\int_M (L_{\gfrak,\mu}f_1)^2d\mu=\Lambda_{Poi}(M, \gfrak,\mu)\int_M f_1\cdot  L_{\gfrak,\mu}f_1 \,d\mu=\Lambda_{Poi}(M, \gfrak,\mu)\int_M \gfrak(\nab_{\gfrak} f_1, \nab_{\gfrak}f_1)d\mu\,.$$
Therefore if $Ric_{\gfrak,\mu,N}\geq K\cdot \gfrak$ substitution of 
 $f=f_1$ into \eqref{LicIneq} gives 
\eq{\prnt{\frac{N-1}{N}}\Lambda_{Poi}(M, \gfrak,\mu)\int_M |\nab_{\gfrak} f_1|^2d\mu\geq K\int_M |\nab_{\gfrak} f_1|^2d\mu\,,}
from which follows the Lichnerowitz' estimate: $\Lambda_{Poi}(M, \gfrak,\mu)\geq  \frac{KN}{N-1}$.

%$(M,g)$ is compact with $Ric_{g,\mu_g}(\cdot, \cdot)\geq K g(\cdot,\cdot)$ with $K>0$ and  is the first-non-zero eigenvalue of $L$, then $\Lambda_{Poi}(M, g,\mu)\geq \frac{KN}{N-1}$, which amounts to the Poincare inequality 
%\eq{ \int_M \prnt{f-\int_{M}fd\mu}^2d\mu\leq \int_{M}|\nab_g f|^2d\mu\qquad \forall f\in C^{\infty}(M)\,. }

\bigskip
In order to get a better feeling of Ricci curvature it is worth to complement the previous presentation by a local one, which concentrates on infinitesimally small volumes. Let $p\in M$ and let $e_1,...,e_n$ be an orthonormal basis for $T_pM$; it determines a chart $\Ncal$, referred to as `Riemannian normal coordinates', given by $\Ncal^j(q)=\Avg{exp_p^{-1}(q)}{e_j}$. Denote by $\g(t)=exp_p(tv)$ the geodesic with $\g(0)=p$ and $\g'(0)=v=\slim_{j=1}^nv^je_j\in UT_pM$. In these coordinates $\g$'s $j$'th coordinate component is $\Ncal^j\circ\g(t)=tv^j$ and $\g'(t)=\slim_{j}v^j\partial_{j}|_{\g(t)}$. Assume $\g:[0,l]\to M$ where $l$ is taken so that $\g(t)$ is not conjugate to $p$ for $t\in (0,l]$. w.l.o.g. we may assume that $e_n=v=\dot{\g}(0)$. 
Along $\g$ there is a Jacobi field $X$  with $X(0)=0$ and $\nab_{\dot{\g}}X(0)=w$. A geodesic variation $\g_s(t)$ which generates $X$ can be chosen to be $\g_s(t)=exp_p(t(v+sw))$ with $s\in (-\epsilon,\epsilon)$ for some $\epsilon>0$, and $X(t)$ is explicitly given by $\pd{\g_s}{s}(0,t)=t(d\, exp_p)_{tv}w$ (see \cite{Cha1} p.88-89).  The $n-1$ vectors $e_1,...,e_{n-1}$ determine $n-1$ Jacobi fields $\{X_j(t)\}_{j=1..n-1}$ along $\g(t)$ determined by the condition $X_j(0)=0$ and $\nab_{\dot{\g}}X_j(0)=e_j$. These are given explicitly by the relation
\eq{\partial_{j}|_{exp_p(tv)}=(d\,exp_p)|_{tv}(e_j)=t^{-1}X_j(t)\,.}
Along $\g(t)$ it holds that $\Avg{X_j}{\dot{\g}}=0$ (due to Gauss' lemma) and $\Avg{\nab_{\dot{\g}}X_j}{\dot{\g}}=0$ (since $\nab_{\dot{\g}}\Avg{\nab_{\dot{\g}}X_i}{\dot{\g}}=\Avg{\nab_{\dot{\g}}^2X_i}{\dot{\g}}=-\Avg{R_{\gfrak}(X_i,\dot{\g})\dot{\g}}{\dot{\g}}$ by the Jacobi equation, and the latter term is identically zero).
%then by the foregoing $\partial_j|_{exp_p(tv)}=\frac{1}{t}X_j(t)$.
 %Assuming $p$ is not conjugate to $\g(t)$ for $t\in (0,1)$, 
For $t\in (0,l)$ the Jacobi fields $X_1,..,X_{n-1},X_n:=\dot{\g}(t)$ form a basis for $T_{\g(t)}M$. We get the following relation between the metric and the Jacobi fields correlation matrix $\tilde{\A}_{jk}:=\Avg{X_j}{X_k}_{1\leq j,k\leq n-1}$ :
\eq{ \gfrak_{jk}(exp(tv))=t^{-2}\tilde{\A}_{jk}(t)\qquad j,k=1..n\,. }
The following expansion is easily verified; it follows from standard Taylor expansion of Jacobi fields (see \cite{Cha1} p.90):
\eq{ \tilde{\A}_{ij}(t)=t^2\Avg{e_j}{e_k}-\frac{t^4}{3}\Avg{R_{\gfrak}(v,e_j)v}{e_k}+O(t^5)\,. }
Therefore $\gfrak_{jk}(exp_p(tv))=t^{-2}\tilde{\A}_{ij}(t)=\delta_{jk}-\frac{t^2}{3}\Avg{R(v,e_j)v}{e_k}+O(t^3)$. Thus using the matrix identity $det(\tilde{A}(t))=1+tTr(\tilde{A}'(0))+O(t^2)$ we conclude that as $t\to 0$:
\eq{det(\gfrak_{jk}(exp_p(tv)))=1-\frac{t^2}{3}Ric_{\gfrak}(v,v)+O(|t|^3)  \,.}
Thus in the specified coordinates we can express the volume form around 0 as $d\mu_g=\sqrt{\det(\gfrak_{ij})}d\mu_{Eucl}=\prnt{1-\frac{1}{6}R_{jk}v^jv^k+O(|v|^3)}d\mu_{Eucl}$, where $R_{jk}$ stand for the components of $Ric_{\gfrak}$. Therefore $Ric_{\gfrak}$ quantifies to leading order the distortion of the volume form from being Euclidean. However for applications this relation is not very useful, and it is preferable to study the $(n-1)\times (n-1)$ matrix valued function $\A_{jk}(t):=\Avg{X_i}{X_j}(t)$; in addition we introduce two more matrix valued functions: $ \U_{jk}:=\half (\A'\A^{-1})_{jk}$ and $\Rcal_{jk}:=\Avg{R_{\gfrak}(J_j,\dot{\g})\dot{\g}}{J_k} $. 

 Starting from the identity $(det \A)'=Tr(adj(\A)A')=(det \A) Tr\prnt{\A^{-1}A'}$  we conclude that 
\eq{ [(det \A)^{\frac{1}{2(n-1)}}]'=\frac{1}{n-1}(det(\A))^{\frac{1}{2(n-1)}}Tr(\U) \,.}
By taking a second derivative we get the identity: 
\eq{ [(det \A)^{\frac{1}{2(n-1)}}]''=\frac{1}{(n-1)^2}(det(\A))^{\frac{1}{2(n-1)}}Tr(\U)^2+\frac{1}{n-1}(det(\A))^{\frac{1}{2(n-1)}}Tr(\U') \,.}
One can show that $\U$ is symmetric \cite{Oht2}; then by the Cauchy-Schwarz inequality $Tr(\U^2)\geq \frac{Tr(\U)^2}{n-1}$ it follows that 
\eql{ \label{diffEqnA}[(det \A)^{\frac{1}{2(n-1)}}]''\leq \frac{1}{n-1}(det(\A))^{\frac{1}{2(n-1)}}\prnt{Tr(\U^2)+Tr(\U')}\,.}

To interpret the RHS, note that it can be verified \cite{Oht2} that $\U$ satisfies the Riccati type matrix equation  $\U'+\U^2+\Rcal \A^{-1}=0$.
 $\A$ is the matrix representation for the restriction $\bar{\gfrak}$ of $\gfrak$ to $\dot{\g}^{\bot}$, which is spanned by the basis $\{X_j\}_{j=1..n-1}$. Hence we can identify $Tr(\Rcal(t)\A^{-1}(t))=Tr(\bar{\gfrak}^{-1}\Rcal_{jk})=Ric_{\gfrak}(\dot{\g})$, whence $(tr\, \U)'+tr(\U^2)+Ric_{\gfrak}(\dot{\g})=0$. 

  %One can show that $\U$ is symmetric, $\U \A=\A\U^t$, and it satisfies  
%From the inequality $tr(\U^2)\geq \frac{(tr\, \U)^2}{n-1}$ it follows that
%\eq{ (tr\,\U)'+\frac{(tr\,\U)^2}{n-1}+Ric(\dot{\g})\leq 0}
If $Ric_{\gfrak}\geq K\cdot \gfrak$ then according to \eqref{diffEqnA} the function $J(t):=\sqrt{det{A}}(t)$ satisfies the following differential inequality 
\eql{\label{eq:Aineq} \dderiv{}{t}\Prnt{J^{\frac{1}{n-1}}}\leq -\frac{Ric_{\gfrak}(\dot{\g})}{n-1}J^{\frac{1}{n-1} }\leq -\frac{K}{n-1}J^{\frac{1}{n-1} }\qquad J(0)=0,\, J'(0)=1 \,.}
If $s_{K,n}$ is the solution to the equation
\eq{ \dderiv{}{t}\Prnt{s_{K,n}}= -\frac{K}{n-1}s_{K,n} \qquad s_{K,n}(0)=0,\, s_{K,n}'(0)=1\,, }
then we can conclude that $\frac{J^{\frac{1}{n-1}}}{s_{K,n}}$ is non-increasing, which is a way of expressing the change in $\sqrt{det(\gfrak_{ij})}$ along $\g$ (geometrically - expansions/contractions orthogonal to $\dot{\g}$).  Many other insights can be obtained from \eqref{eq:Aineq}, in particular one can derive the Brunn-Minkowski inequality. 

A more general analysis (see \cite[p.383]{Vil} or \cite{KTS3}) which incorporates a general measure $\mu=e^{-V}\mu_{\gfrak}$ and effective dimension $N\in (-\infty, \infty]$ (with the usual interpretation when $N=n$), shows that a variant \green{$J_{\mu}(t):=J(t)\cdot e^{-V(\g(t))+V(\g(0))}$}  of the previous determinant satisfies the following analogous inequality:
\eql{\label{DetermIneq} \dderiv{}{t}\Prnt{J_{\mu}^{\frac{1}{N-1}}}\leq -\frac{K}{N-1}J_{\mu}^{\frac{1}{N-1} } \,,}
where $J_{\mu}(0)=0,\,J_{\mu}'(0)=Id$. These determinants and their ODEs are paramount to this work, however in subsequent chapters they will appear in a rather axiomatic form, without referring to their origins and their geometric interpretation.

\subsection{Functional Inequalities}
Before embarking with the principle methods and results we provide some background regarding the functional inequalities which will be studied in this manuscript. Our objective is to survey past developments, interesting implications and finally our contributions.

Throughout this section we assume $(M^n,\gfrak)$ denotes a connected, smooth n-dimensional Riemannian manifold whose boundary $\partial M$ is either empty or locally-convex  (i.e., the second fundamental form \green{
$\mathrm{I\!I}(X,Y):=\gfrak(\nab_{X} \nu,Y)$, where $\nu$ is the exterior unit normal to $\partial M$,} is positive semi-definite on $\partial M$), and $\mu$ is a finite measure on $M$, with smooth positive density $e^{-V}$  w.r.t. the Riemannian measure $\mu_{\gfrak}$. As we stated in the introduction the main results of this work concern with three types of functional inequalities on weighted manifolds $(M,\gfrak,\mu)$ (and respectively their optimal constants):
\begin{enumerate}
	\item \Poinc inequality\pink{;}
	\item $p$-\Poinc inequality for $p\in (1,\infty)$\pink{;}
	\item Log-Sobolev inequality.
\end{enumerate}

The \Poinc inequality is a particular case of a $p$-\Poinc inequality, yet for general $p\in (1,\infty)\setminus \{2\}$ certain methods do not apply, hence we employ different directions, which are still productive but at the price of less generality.
As we previously mentioned to each of these inequalities we can associate constants $\Lambda_{Poi}(M,\gfrak,\mu), \Lambda_{Poi}^{(p)}(M,\gfrak,\mu)$ and $\Lambda_{LS}(M,\gfrak,\mu)$ (the \green{Poincar\'{e}, $p$-\Poinc and Log-Sobolev constants} respectively) which are maximal among the constants for which these inequalities hold. From a panoramic viewpoint the main goals can be simply stated: given $K$ and $N$ s.t. $Ric_{\gfrak,\mu,N}\geq K$ and diameter $Diam(M)\leq D$ where $0<D\leq \infty$, find the maximal lower bounds for the $\Lambda_{*}$ constants. These lower bounds will naturally be functions  of $K$, the effective dimension $N$ and the diameter bound $D$.  The discussion below will render this goal into a more concrete problem. The main sources for this presentation are \cite{Che1}, \cite{BGL} and also \cite{Val}. 

\subsubsection{\Poinc Inequalities}
%if \partial M non empty then convex boundary
We say that $(M,\gfrak,\mu)$ satisfies a \Poinc (or spectral-gap) inequality with constant $C$, and denote this by  $Poi(C)$, if
\eql{ \label{eqn:defPoiIneq} C\cdot Var_{\bar{\mu}}(f)\leq \int_M |\nab f(x)|^2d\bar{\mu}(x) \qquad \forall f\in C_c^{\infty}(M)\,, }
where $\bar{\mu}:=\frac{1}{\mu(M)}\mu$ and $Var_{\bar{\mu}}(f):=\int_Mf^2(x)d\bar{\mu}(x)-\prnt{\int_Mf(x)d\bar{\mu}(x)}^2$. Equivalently
\eql{\label{eqn:defPoiIneq2} C \int_Mf^2(x)d\mu(x)\leq \int_M |\nab f(x)|^2d\mu(x) \qquad \forall f\in \F_{Poi}(M,\mu)}
where 
$$\F_{Poi}(M,\mu):=\cprnt{0\not\equiv f\in \Cinf_c(M)\quad \text{s.t.}\quad \int_Mf(x)d\mu(x)=0}\,. $$

If $C_1>C_2$ then clearly $Poi(C_1)$ implies $Poi(C_2)$. We define $\Lambda_{Poi}(M,\gfrak,\mu)$ to be the least upper bound for the set of $C$ s.t. $Poi(C)$ is valid. It is characterized as the solution to the following variational problem:
\[ \Lambda_{Poi}(M,\gfrak,\mu)=\inf\cprnt{ \frac{\int_M|\nabla_{\gfrak} f(x)|^{2}d\mu(x)}{\int_M |f(x)|^{2} d\mu(x)}:\,\,\, f\in\F_{Poi}(M,\gfrak,\mu) }\,.\]

The name spectral-gap inequality stems from its relation to the bottom of the spectrum of the natural operator $L_{\gfrak,\mu}=-\gls{Delta_g_mu}:=-\prnt{\Delta_{\gfrak}-\Avg{\nab_{\gfrak} V}{\nab_{\gfrak}}}$ presented before; the Euler-Lagrange equation for \eqref{eqn:defPoiIneq} is $L_{\gfrak,\mu}f=\lambda f$ ($\lambda$ can be identified as a Lagrange multiplier). It is a positive semi-definite symmetric operator on the dense subset $\Cinf_c(M)\subset \D(L_{\gfrak,\mu})\subset L^2_{\mu}(M)$, where $\D(L_{\gfrak,\mu})$ is the maximal domain for $L_{\gfrak,\mu}$. It admits a self-adjoint extension, but this need not be unique.
%p.221!,223, 476
 When the extension is unique, namely the closure (w.r.t. the graph norm), we refer to it as essentially-self-adjoint (e.s.a). With $\D_0=\Cinf_c(M)$ e.s.a is guaranteed when $\partial M =\emptyset$ and $M$ is complete; the same is true if $M$ is compact with boundary $\partial M\neq \emptyset$ with $\D_0=\Cinf_{c,Neu}(M):=\{f\in C^{\infty}(M):\partial_{\nu}f=0\,\,\text{on}\,\,\partial M   \}$ (cf. \cite{Tay1} ch.8). Being e.s.a its self-adjoint extension is unique, and therefore we can unequivocally speak of its spectrum $\sigma$. However in the latter case the minimizer of \eqref{eqn:defPoiIneq}, due to the free boundary conditions, belongs to $\Cinf_{c,Neu}(M)$, hence the spectrum of the operator $\overline{L_{\gfrak,\mu}}$ encapsulates the best constant in the inequality \eqref{eqn:defPoiIneq} despite that the test functions are taken from $C_c^{\infty}(M)$.  
$\overline{L_{\gfrak,\mu}}$ is self-adjoint and non-negative hence $\sigma(\overline{L_{\gfrak,\mu}})\subset\R_{+}$. Under $Poi(C)$ it holds that $\sigma(\overline{L_{\gfrak,\mu}})\subset \{0\}\cup [C,\infty)$; this is why we interpret $Poi(C)$ as a gap in the spectrum of $\overline{L_{\gfrak,\mu}}$. The spectrum is divided into two disjoint sets $\sigma_{d}(\overline{L_{\gfrak,\mu}})$ and $\sigma_{ess}(\overline{L_{\gfrak,\mu}})$. The former is the discrete spectrum, constituted of isolated eigenvalues with corresponding finite-dimensional eigenspaces; the latter, which contains all other points of the spectrum, is the essential spectrum. The spectrum is said to be discrete if $\sigma_{ess}(\overline{L_{\gfrak,\mu}})=\emptyset$. This is the case for compact manifolds, where we can identify $C$ with $\lambda_1$, the first non-zero eigenvalue of $\overline{L_{\gfrak,\mu}}$. Due to our assumption on  finiteness of the measure $\mu$, the constant function $1$ is an eigenvector for $\overline{L_{\gfrak,\mu}}$ corresponding to the eigenvalue $0$, and therefore $Poi(C)$ indeed implies a gap in the spectrum.  

We mention a few consequences of $Poi(C)$ with $C>0$ (following \cite{BGL}):
\begin{enumerate}
	\item {\bf Exponential measure concentration} The following concentration inequality holds: for any Lipschitz function $f$ with Lipschitz constant $||f||_{Lip}$:
	\eq{ \bar{\mu}\prnt{|f-\int_M fd\bar{\mu}|\geq r}\leq 6e^{ -r\sqrt{C}||f||_{Lip} } \qquad \forall r\geq 0\,.}
	Therefore Lipschitz functions `mostly' don't deviate much from their expectation. 
\item {\bf Exponential integrability} For every 1-Lipschitz function $f$ and every $s<\sqrt{4C}$: $\int_Me^{sf}d\bar{\mu}<\infty$ .
\item {\bf Exponential variance decay} If $(P_t)_{t\geq 0}$, where \green{ $P_t:=e^{-tL_{\gfrak,\bar{\mu}}}$}, is the semi-group of operators  \green{ generated by}  $L_{\gfrak,\bar{\mu}}$ then for every $f\in L^2(\bar{\mu})$ and every $t\geq 0$: 
\[ Var_{\bar{\mu}}(P_tf)\leq e^{-2Ct}Var_{\bar{\mu}}(f)\,. \]
Therefore, $\Lambda_{Poi}(M,\gfrak,\bar{\mu})$ quantifies the convergence rate to `equilibrium'. 

\end{enumerate}

\red{We mention a few examples of known \Poinc inequalities (\cite{BGL}): }
\begin{enumerate}
  \item On $(\R_{+}, \mu(x)=e^{-x}m)$ holds $Poi(\frac{1}{4})$.
	\item On $(\R^n,\mu=e^{-V}m)$ with $\mu$ a log-concave probability measure, holds $Poi(C)$ for some $C>0$. In particular if $Ric_{\gfrak_{Euc},\infty}=Hess[V]\geq K\cdot  Id$ with $K>0$ then $Poi(K)$ holds. 
%	\item On the  Hyperbolic space $\Lambda_{Poi}(\Hcal^n)=\frac{(n-1)^2}{4}$ ($\mu_{\gfrak_{\Hbb^n}}$ is not a finite measure, yet formulation \eqref{eqn:defPoiIneq2} of the \Poinc inequality is still meaningful).
	\item \blue{On $S^n(r_{K,n})$ (the $n$-dimensional sphere in $\R^{n+1}$ of radius $r_{K,n}:=\sqrt{\frac{n-1}{K}}$, so that $Ric_{\gfrak}=K\cdot \gfrak$) holds $Poi(\frac{Kn}{n-1})$, and this is sharp. Since on $S^n(1)$ all sectional curvatures equal $1$, so that $Ric_{\gfrak}=(n-1)\cdot\gfrak$, the Lichnerowicz estimate implies $\Lambda_{Poi}(S^n(1), \gfrak_{S^n(1)})\geq n$. This shows that the Lichnerowicz estimate is sharp. Moreover a theorem of Obata states that up to isometries the sphere is the only compact manifold for which the \Poinc constant and the Lichnerowicz lower bound coincide.}
	\item In general if $Ric_{\gfrak,\mu,N}\geq K$ with $K>0$, $1\neq N\in (-\infty, 0)\cup [n,\infty]$,  then the generalized Lichnerowicz estimate $\Lambda_{Poi}(M, \gfrak,\mu)\geq \frac{KN}{N-1}$ holds \cite{Mil5, Oht1}; in \cite{Mil5} it is shown that sharpness of the constant holds also when $N\in (-\infty, -1]$. As we mentioned for $1<n\in \Nbb$ the number $\frac{Kn}{n-1}$ stands for $\Lambda_{Poi}(S^n(r_{K,n}))$. 	
	
\end{enumerate}

The parameters which determine the estimates above are essentially the Ricci lower bound $K$ and the effective dimension $N$. With additional data finer estimates can be derived. This motivates the inclusion of a third parameter, $D$, which stands for an upper bound on the diameter of $M$. This is the most natural parameter to add, considering our knowledge of the spectrum of a 1d string or a 2d membrane. More generally, a classical result of Payne and Weinberger from the 60's \cite{PaWe} states that if $A\subset \R^n$ is a convex domain whose diameter is bounded by $D>0$, then $\Lambda_{Poi}(A, \gfrak_{Euc}, \mu_{Leb})\geq \frac{\pi^2}{D^2}$; it is optimal since it is approached by a sequence of parallelepipeds $A_i$, which degenerate into an interval by the shrinkage along all their dimensions but one.

We remark that we could have concluded $Poi(C)$ on $A$ for some $C>0$ even without the theorem just by considering the log-concave measure $\mu=e^{-V}m$ with $V(x)=1$ if $x\in A$, and $V(x)=\infty$ otherwise, however this gives no information on how large can the spectral-gap be. 

For general compact manifolds $(M^n, \gfrak, \mu_{\gfrak})$ (i.e. $N=n$) such that $Ric_{\gfrak}\geq K \gfrak$ lower bounds for $\Lambda_{Poi}(M,\gfrak,\mu_{\gfrak})$ are abundant. Table \ref{table:2} should give \pink{quite} a comprehensive overview of the development of these estimates  under different assumptions on the pertinent parameters.

\begin{changemargin}{0.0cm}{0cm} 
\begin{table}[H]
\begin{tabular}{ |p{3cm}||p{1.2cm}||p{3.5cm}|p{2.6cm}|p{3.9cm}|  }
 %\hline
% \multicolumn{5}{|c|}{Noteable lower bounds for $\Lambda_{Poi}(M^n,g,\mu_g)$} \\
 \hline
 Author & Year & Lower Bound& Conditions& Parameters\\
 \hline
A. Lichnerowicz \cite{Lic}  & 1958   &$K_n:=\frac{K}{1-\frac{1}{n}}=\frac{nK}{n-1}=\Lambda_{Poi}(S^n(1))$&  $K> 0$& \\
\hline
P.H. Berard, G. Besson,S. Gallot &  1985  & $n\prnt{\frac{\int_0^{\pi/2}\cos^{n-1}(x)dx}{\int_0^{D/2}\cos^{n-1}(x)dx}}^{\frac{2}{n}}$ & $K=n-1>0$   &  \\
\hline
 P. Li,S.T.Yau \cite{LY}& 1980 & $\frac{\pi^2}{2D^2}$ &  $K\geq 0$ &\\
\hline
 J.Q. Zhong,H.C. Yang \cite{ZhY}  & 1984 & $\frac{\pi^2}{D^2}$&  $K\geq 0$ &\\
\hline
D.G. Yang&   1999  & $\frac{\pi^2}{D^2}+\frac{K}{4}$& $K\geq 0$ &\\
\hline\hline
 P. Li,S.T.Yau \cite{LY}& 1980  & $\frac{1}{D^2(n-1)exp[1+\sqrt{1+16\alpha^2}]}$   &$K\leq 0$&$\alpha=\half D\sqrt{|K|(n-1)}$\\
\hline
 K.R. Cai & 1991  & $\frac{\pi^2}{D^2}+K$ &$K\leq 0$ & \\
\hline
 D. Zhao& 1999  & $\frac{\pi^2}{D^2}+0.52K$ & $K\leq 0$ & \\
\hline
H.C. Yang,F. Jia & 1990/1 & $\frac{\pi^2}{D^2}e^{-\alpha}$ & $n\geq 5,\, K\leq 0$ & $\alpha=\half D\sqrt{|K|(n-1)}$\\
\hline
H.C. Yang,F. Jia &1990/1 & $\frac{\pi^2}{2D^2}e^{-\alpha'}$ & $2\leq n\leq 4,\,K\leq 0$&$\alpha'=\half D\sqrt{\min\{|K|(n-1),2\}}$\\
 \hline\hline
M. Chen,E. Scacciatelli,L. Yao \cite{Che3} &2002 & $\frac{\pi^2}{D^2}+\frac{K}{2}$ & $K\in\R$& \\
 \hline
\end{tabular}
\caption{Notable lower bounds for $\Lambda_{Poi}(M^n,\gfrak,\mu_{\gfrak})$.}\label{table:2}
\end{table}
\end{changemargin}
The Lichnerowicz estimate is sharp but is independent of $D$. The Zhong-Yang estimate is also sharp but is independent of $K\geq 0$ and the dimension. Clearly an estimate which incorporates all the available data is ideal.
The last estimate, which clearly improves some of the previous results, was derived (with the assistance of a computer) by  substitution of very complicated test functions  into the Chen-Wang formula (1997) \cite{ChWa} (for manifolds with convex boundary and Neumann boundary conditions):
\[\lambda_1\geq \sup_{f\in \F}\inf_{r\in (0,D)}\frac{4f(r)}{\int_0^r \co(s)^{-1}ds\int_s^D \co(x)f(x)dx}:=\beta_1\,, \]
where $\co(x)=\cosh^{n-1}\prnt{\frac{r}{2}\sqrt{\frac{-K}{n-1}}}$ ($r\in (0,D), K\in \R$) and $\F=\{f\in C([0,D]): \, f>0 \text{ on } (0,D)\}$. 
Surprisingly, the representative test function has a \pink{quite}  simple form \cite{Che1}: $f(r)=\prnt{\int \co(s)^{-1}ds}^{\gamma}$ ($\gamma\geq 0$). 
If $Ric_{\gfrak}\geq K$, with no further assumption on $K$ (i.e. $K\in \R$), the coefficient $\half$ in the last estimate corresponds to the optimal estimate we can get from estimates of the form $\lambda_1\geq \frac{\pi^2}{D^2}+t K$ with $t\in \R$ \cite{Che1} (a linear interpolation between the Lichnerowicz estimate to the Zhong-Yang estimate) . 

In addition Chen (2000) \cite{Che2} showed that $4A^{-1}\geq\beta_1\geq A^{-1}$ where
\[ A=\sup_{r\in (0,D)}\prnt{\int_0^r\co(x)^{-1}dx}\prnt{\int_r^D\co(x)dx}\,. \]
 The lower and upper bound correspond to $\gamma=\half$ and $\gamma=1$ respectively.

%The previously mentioned Payne-Weinberger estimate for convex bodies in $\R^n$ relied on comparison with an eigenvalue of an associated Sturm-Liouville problem. 
In 1992 a major progress in deriving estimates was established by P. Kr\"{o}ger \cite{Kro} for compact manifolds $(M,\gfrak,\mu_{\gfrak})$ satisfying $Ric_{\gfrak}\geq K$. To understand this approach we go back to the work of Payne and Weinberger (1960) \cite{PaWe}. Their estimate $\Lambda_{Poi}(K, m)\geq \frac{\pi^2}{D^2}$ for convex bodies $K$ in $\R^n$ emerged by comparison with an eigenvalue of a 1d vibrating string with fixed ends at two points a distance $D$ apart (the ends are fixed because they applied a Liouville transformation turning the problem from Neumann to Dirichlet). 
Using a gradient comparison approach as in \cite{LY} Kr\"{o}ger showed that (sharp) lower bounds for $\Lambda_{Poi}(M,\gfrak,\mu_{\gfrak})$ can be identified as eigenvalues of associated SL problems. In 2000 Kr\"{o}ger's results were extended by Bakry and Qian \cite{BaQi} to weighted manifolds $(M,\gfrak,\mu)$ with $N\in (n,\infty]$ satisfying $Ric_{\gfrak,\mu,N}\geq K$. These new estimates are sharp, i.e. best possible, and can all be expressed as
$\Lambda_{Poi}(M,\gfrak,\mu)\geq \lambda_{K,N,D}$ where $\lambda_{K,N,D}$ stands for the first positive eigenvalue of a Sturm-Liouville eigenvalue problem 
\[  f''(x)+\frac{J_{K,N}'(x)}{J_{K,N}(x)}f'(x) =-\lambda f(x)\qquad f'(-\frac{D_{K}}{2})=f'(\frac{D_{K}}{2})=0\,,\]
where $D_{K}=D$ if $K\leq 0$ and to $\min\{\frac{\pi}{\sqrt{\frac{K}{N-1}}}, D\}$ if $K>0$, and $J_{K,N}(x)$ is a smooth positive function on $(-\frac{D_{K}}{2},\frac{D_{K}}{2})$ which for $N\in [n,\infty)$   satisfies the equation (cf. \eqref{DetermIneq})
\eql{\label{eqn:J} \dderiv{}{t}\Prnt{J_{K,N}^{\frac{1}{N-1}}}= -\frac{K}{N-1}J_{K,N}^{\frac{1}{N-1} }  \qquad J_{K,N}(0)=1\,.}
The case $N=\infty$ also fits naturally to this framework, but for simplicity we don't discuss it in this introductory section. 
For each $K\in \R$ and $N\in [n,\infty)$ there are many solutions to \eqref{eqn:J}, distinguished by the specification of $J_{K,N}'(0)$; the `right' specification of  $J_{K,N}'(0)$ for the spectral-gap problem is the outcome of the Kr\"{o}ger and Bakry-\'{E}mery analysis, who showed that the `worst' case (i.e. a minimal gap) we can obtain from  the class of densities which satisfy \eqref{eqn:J}, is attained when $J_{K,N}'(0)=0$; a  case which is characterized by that $J_{K,N}(x)$ being symmetric around the origin ($J(x)=J(-x)$ for every $x\in (-\frac{D_{K}}{2},\frac{D_{K}}{2})$).

As we mentioned this viewpoint incorporates the previous sharp estimates of Lichnerowicz and Zhong-Yang according to the identifications specified in the introduction section (specifically the discussion after stating Theorems \ref{thm:ZhongYang} and \ref{thm:Lichnerowicz}). Notice that when $Ric_{\gfrak}\geq K$ the Lichnerowicz classical estimate emerges as the first non-zero eigenvalue of the SL problem above with $J_{K,n}(x)=\cos^{n-1}(\frac{x}{r_{K,n}})$ where $r_{K,n}=\sqrt{\frac{n-1}{K}}$ (as defined above), which solves \eqref{eqn:J} with initial conditions $J_{K,n}(0)=1$, $J_{K,n}'(0)=0$ and $D\geq \pi r_{K,n}=\pi\sqrt{\frac{n-1}{K}}$  (the maximal diameter according to the Bonnet-Myers theorem, whence $D_K=\pi \sqrt{\frac{n-1}{K}}$), while if $K\geq 0$ the Zhong-Yang estimate emerges with $J_{0,n}(x)=1$ which solves \eqref{eqn:J} with initial conditions $J_{0,n}(0)=1$, $J_{0,n}'(0)=0$, and $D_K=D$.
At present a generalization of these estimates to weighted manifolds with $Ric_{\gfrak,\mu,N}\geq K$ and $N\in (-\infty,0]$ is missing. As we mentioned in the introduction, we close this gap by introducing a general method which will also be applied to study the two other functional inequalities.

 %To clarify this point notice that
 %when $K>0$ a theorem of Bonnet and Myers states that $diam(M)\leq \frac{\pi}{\sqrt{\delta}}$ where $\delta:=\frac{K}{N-1}$; 

%We remark that one of the achievements of our work is complementing their results, by deriving new (sharp) estimates pertaining to measures for which $N\leq 0$. 

\subsubsection{p-\Poinc Inequalities}
The classical \Poinc inequality discussed before is a particular case of a general class of inequalities to which we refer as $p$-\Poinc inequalities, whose `optimal' constant $\Lambda_{Poi}^{(p)}(M,\gfrak,\mu)$ is related to the `eigenvalues' of a non-linear operator known as the $p$-Laplacian. Let $p\in (1,\infty)$, then 
we refer to the inequality 
\eql{\label{eqn:defPoiIneq2} C \int_M|f(x)|^pd\mu(x)\leq \int_M |\nab_g f(x)|^pd\mu(x) \qquad \forall f\in C_c^{\infty}(M)\quad \text{s.t.}\quad \int_M |f(x)|^{p-2}f(x)d\mu(x)=0\,,} 
as a $p$-\Poinc inequality with constant $C$, and denote it by $Poi^{(p)}(C)$. The $p$-\Poinc constant $\Lambda_{Poi}^{(p)}(M,\gfrak,\mu)$ is defined as 
\[ \Lambda_{Poi}^{(p)}(M,\gfrak,\mu)=\inf_{f\in\F_{Poi}^{(p)}(M,\mu) }\ \frac{\int_M|\nabla f(x)|^{p}d\mu(x)}{\int_M |f(x)|^{p} d\mu(x)}  \,,\]
where
\[ \F^{(p)}_{Poi}(M,\mu) :=\cprnt{0\not\equiv f\in\Cinf_c(M)\quad \text{ s.t. }\quad \int_M |f(x)|^{p-2}f(x)d\mu(x) =0}\,.\]
 When $p=2$ we will refer to the $p$-\Poinc constant and to the domain simply by $\Lambda_{Poi}(M,\gfrak,\mu)$ and $\F_{Poi}(M,\mu)$ respectively. In case $(M,\gfrak,\mu)=(\R,|\cdot |,\xi)$ (i.e. $\mu=\xi$ is a measure supported in $\R$) we abbreviate and write $\Lambda_{Poi}(\xi)$ and $\F_{Poi}(\xi)$.

By considering the Euler-Lagrange equation, $\Lambda_{Poi}^{(p)}(M,\gfrak,\mu)$ can be identified as the smallest positive number $\lambda$ for which there is a non-zero $f\in W^{1,p}$ which satisfies the following half-linear equation in the weak sense 
\eq{  &\Delta_{p,\mu} f(x) :=p(x)^{-1}\nabla \cdot (p(x) |\nabla f|^{p-2}\nabla f(x))= -\lambda|f(x)|^{p-2}f(x) \qquad \text{ on M}\\
       &\nabla f(x)\in \prnt{T_x(\partial M)}^{\bot} }
where $p(x)=e^{-V(x)}=\deriv{\mu}{\mu_g}$. 

Unlike the case $p=2$, explicit lower bounds for the $p$-\Poinc under $Ric_{\gfrak}\geq K$ are scarce. In table \ref{table:3} below \red{several important  known} results are listed.

\begin{changemargin}{0.0cm}{0cm} 
\begin{table}[H]
\centering
\begin{tabular}{ |p{3cm}||p{0.9cm}||p{3.5cm}|p{3.4cm}|p{2cm}|  }
 %\hline
 %\multicolumn{4}{|c|}{Noteable lower bounds for $\Lambda^{(p)}_{Poi}(M^n,\gfrak,\mu_{\gfrak})$ where $p>1$} \\
 \hline
 Author & Year & Lower Bound& Conditions\\
 \hline
A.M.Matei \cite{Mat}   & 2000  &$\Lambda_{Poi}^{(p)}(S^n(r_{K,n}))$&  $K>0$\\
\hline
S. Kawai, N. Nakauchi \cite{KN} & 2003 & $\frac{1}{p-1}\prnt{\frac{\pi_p}{4D}}^p$ & $K=0, p\geq 2$  \\
\hline
H. Zhang \cite{Zha} & 2007 & $(p-1)\prnt{\frac{\pi_p}{2D}}^p$ & $K= 0$\, \text{but} \, $\exists x\in M: \, Ric_{\gfrak}(x)>0$\\
\hline
Valtorta \cite{Val}& 2012 & $(p-1)\frac{\pi_p^p}{D^p}$ & $K=0$\\
 \hline
\end{tabular}
\caption{Notable lower bounds for $\Lambda^{(p)}_{Poi}(M^n,\gfrak,\mu_{\gfrak})$ where $p>1$.}\label{table:3}
\end{table}
\end{changemargin}
where $\pi_p=\frac{2\pi}{p\sin(\pi/p)}$ and $r_{K,n}:=\sqrt{\frac{n-1}{K}}$. We remark that in \cite{Mat} it is also proved that for $K>0$ a generalized Obata theorem holds, which states that \\$\Lambda_{Poi}^{(p)}(M,\gfrak,\mu_{\gfrak})= \Lambda_{Poi}^{(p)}(S^n(r_{K,n}),\gfrak_{S^n(r_{K,n})},\mu_{\gfrak})$  iff $M$ is isometric to $S^n(r_{K,n})$. In the same work it is also proved that  $\Lambda_{Poi}^{(p)}(M,\gfrak,\mu_{\gfrak})\geq \prnt{\frac{\hfrak_{Che}}{p}}^p$, where $\hfrak_{Che}$ stands for the Cheeger isoperimetric constant of $M$  defined by $\hfrak_{Che}:=\inf_{E\subset M:\, 0\leq \mu_{\gfrak}(E)\leq \half\mu_{\gfrak}(M)}\frac{\mu^{n-1}_{\gfrak}(\partial E)}{\mu_{\gfrak}(E)}$ (the classical Cheeger inequality corresponds to $p=2$ ). 

The last estimate by Valtorta is also sharp (i.e. best possible). 
More generally it was proved by Naber and Valtorta \cite{NaVa} that 
 for $K\leq 0$ and $D\in (0,\infty)$ the sharp estimate 
$\Lambda_{Poi}^{(p)}(M,\gfrak,\mu_{\gfrak})\geq \lambda_{K,n,D}$ holds, where $\lambda_{K,n,D}$ stands for the first positive Neumann eigenvalue on $[-\frac{D}{2},\frac{D}{2}]$ of the eigenvalue problem 
\[ \deriv{}{x}(|f|^{p-2}(x)f(x))+\frac{J_{K,n}'(x)}{J_{K,n}(x)}|f'|^{p-2}(x)f'(x)=-\lambda |f|^{p-2}(x)f(x)\,, \]
where $J_{K,n}(x)=\cosh\prnt{\sqrt{-\frac{K}{n-1}} x}^{n-1}$. For $K< 0$ this lower bound (despite being sharp) is not attained \cite{NaVa}.  The formulation of these estimates clearly resembles the Kr\"{o}ger-Bakry-Qian estimates.

Generalization of these estimates to weighted manifolds with $N\in (-\infty,0]\cup [n,\infty]$ is still missing. We present an approach which is useful to study the range $N\in (n,\infty]$ (with the potential of being generalized to $N\leq 0$). 
%Currently our result, specifically the theorem regarding monotonicity properties of the first eigenvalue $\lambda$ as we modify the boundary conditions defining the densities $J_{K,n}$, applies only to $N\in [n,\infty]$. We remark that this project was of least importance, therefore the proof is given in significant less detail than the ordinary Poincare inequality. 

\subsubsection{Log-Sobolev inequalities (LSI)}
These type of inequalities were firstly introduced by L.Gross \cite{Gro} in order to simplify the proof of previous  hypercontractive estimates derived by E. Nelson \cite{Nel}. 

Define by $\bar{\mu}:=\frac{1}{\mu(M)}\cdot \mu$ the normalization of $\mu$ to a probability measure. 
For all non-negative $\mu$-integrable functions $f$ on $M$ such that $\int_M |f\log f| d\mu<\infty$ we define the Entropy of $f$:
\eql{\label{defn:back:ent} Ent_{\mu}f:=\int_M f(x)\cdot \log f(x)d\mu(x)-\int_M f(x) d\mu(x)\cdot \log\prnt{\int_M f(x)d\bar{\mu}(x)}\qquad (f\geq 0)\,. }
We take the following as the definition of a Logarithmic Sobolev inequality: a triple $(M, \gfrak,\mu)$ satisfies a logarithmic Sobolev inequality $LS(C)$ with constant $C\geq 0$ if 
\eql{\label{eqn:LSI} \int_M|\nabla_{\gfrak} f(x)|^2d\mu(x) \geq \frac{C}{2}\cdot  Ent_{\mu}(f(x)^2)\,, }
for all $f\in \tlF_{LS}(M)$, where:
\[ \tlF_{LS}(M,\mu) :=\cprnt{f\in\Cinf(M):\,\, f^2=c+g\quad \text{where}\quad c>0,\,\, g\in \F_{Poi}(M,\mu) }\,.\]
Equivalently by 2-homogeneity of $f\to Ent_{\mu}(f^2)$:
\eql{\int_M|\nabla_{\gfrak} f(x)|^2d\mu(x) \geq \frac{C}{2}\cdot  \int f(x)^2\log(f(x)^2)d\mu(x) \,, }
for all $f\in \F_{LS}(M,\mu)$, where:
\[ \F_{LS}(M,\mu) :=\cprnt{ f\in \tlF_{LS}(M,\mu):\,\, \int f^2(x)d\bar{\mu}(x)=1}\,.\]
%\eql{\label{eqn:LSI} \int_M|\nabla_g f(x)|^2d\mu(x) \geq \frac{C}{2}\cdot  Ent_{\mu}(f(x)^2)\,. }
%
%
%for all $f\in C^{\infty}(M)$ such that $\int_M f^2\log(1+f^2)d\bar{\mu}<\infty$:

We denote by $\Lambda_{LS}(M,\gfrak,\mu) >0$ the maximal (`optimal') constant for which the inequality holds for all $f$, and refer to it as the LS constant (which makes the inequality sharp). When $(M,\gfrak,\mu)=(\R,|\cdot|,\xi)$, where $\mu=\xi$ is a measure supported in $\R$, then we abbreviate and refer to the log-Sobolev constant by $\Lambda_{LS}(\xi)$ and to the functions domain by $\tlF_{LS}(\xi)$ and $\F_{LS}(\xi)$. 
%We stress that in the definition of $\Lambda_{LS}(M,g,\mu)$ we always consider the normalized form of $\mu$ in \eqref{eqn:LSI}, i.e $\bar{\mu}$, and not $\mu$ itself. 

We remark that there are several other alternative, yet equivalent, ways to introduce LSIs. For conciseness we avoid discussing them here, and refer the interested reader to \cite{BGL}.

Proving LSI is more delicate than \Poinc inequalities, as the latter are spectral, and as such spectral methods (such as eignenfunction expansions ) can be applied in order to prove them. Therefore examples of LSI together with their LS constants are more scarce.

Always $LS(C)$ implies \Poinc inequality $Poi(C)$. This follows simply by applying the LSI to $f=1+\epsilon g$, where $g\in C_c^{\infty}(\R)$ such that $\int gd\mu=0$. As $\epsilon\to 0$ 
\eql{\label{Implication_LS_Poinc} Ent_{\mu}(f^2)=2\epsilon^2\int g^2d\mu+o(\epsilon^2)\Rightarrow \qquad C \int g^2d\mu\leq \int g'(x)^2d\mu\,. }

We mention a few consequences of $LS(C)$ with $C>0$ (following \cite{BGL}):
\begin{enumerate}
	\item \green{{\bf Gaussian concentration}}  For every Lipschitz function $f$ and every $r\geq 0$
	\[ \bar{\mu}(|f-\int_Mfd\bar{\mu}|\geq r)\leq 2e^{-\frac{Cr^2}{2||f||^2_{Lip}}}\,. \]
	Thus LSI implies stronger concentration than \Poinc inequalities. 
	\item \green{{\bf Gaussian integrability}} For every 1-Lipschitz function $f$ and every $0\leq \sigma<C$: \green{$\int_M e^{\half \sigma f^2}d\mu<\infty$}. This result gives a necessary condition for a LSI to hold. 
	\item {\bf Exponential decay in entropy} For every $f\in L^1(\mu)$ positive with finite entropy 
	\[ Ent_{\mu}(P_tf)\leq e^{-2Ct}Ent_{\mu}(f)\,. \] 
	\item {\bf Gross' hypercontractivity} For every $f\in L^p(\mu)$ ($1< p< \infty$) it holds that $||P_tf||_{q}\leq ||f||_p$ for every $t\geq 0$ whenever $1<p<q<\infty$ and \green{ $\sqrt{\frac{q-1}{p-1}}\leq e^{Ct}$}, so the operators $P_t$ `improve integrability'. \green{In fact this property is equivalent to LS(C).}
\end{enumerate}

The exponential measure is known to satisfy a \Poinc
inequality; according to (2) it cannot satisfy a logarithmic Sobolev inequality. We also remark that the standard Gaussian measure satisfies $LS(1)$ hence the integrability condition in the proposition is optimal (consider the function $f(x)=x$ to justify the statements).

The following result is reminiscent of the Lichnerowicz estimate for $\Lambda_{Poi}(M,\gfrak,\mu)$ under $CD(K,N)$.
\begin{thm}[Bakry-\'{E}mery \cite{BE2}]\label{thm:BakryEmeryLS} If $M$ is compact and $Ric_{\gfrak,\mu,N}\geq K$ with $K>0$ and $N>1$ then $\Lambda_{LS}(M,\gfrak,\mu)\geq \frac{KN}{N-1}$. 
\end{thm}
For example the standard sphere $S^n(1)\subset \R^{n+1}$ satisfies $CD(n-1, n)$ hence $LS(n)$; however since the optimal \Poinc constant is $\Lambda_{Poi}(S^n, \gfrak_{S^n}, \mu_{\gfrak})=n$, it follows that $\Lambda_{LS}(S^n, \gfrak_{S^n}, \mu_{\gfrak})=n$ as well. Mind that the case $n=1$ corresponds to the 1-dimensional torus (or equivalently to the circle $[0,2\pi]_{0\sim 2\pi}$) which satisfies sharp $Poi(1)$ (as can be verified by considering the corresponding Sturm-Liouville operator) as well as sharp $LS(1)$ \cite{Wei}; in this case $K=0$ by definition, and the convention $\frac{KN}{N-1}=1$ makes sense.  

Up to this point we can summarize that when $M$ is compact, $K>0$ and $N> 1$ then $\Lambda_{Poi}(M,\gfrak,\mu)\geq \Lambda_{LS}(M,\gfrak,\mu)\geq K_N>0$, where $K_N:=\frac{K}{1-\frac{1}{N}}$. Rothaus established in \cite{Rot4} an interesting interpolation, by showing that when $Ric_{\gfrak}\geq K$ ($N=n$) where $K\in \R$, then \cite{Rot4,Led2}
\[\Lambda_{LS}(M,\gfrak,\mu_{\gfrak})\geq t_n\Lambda_{Poi}(M,\gfrak,\mu_{\gfrak})+(1-t_n)K_n\,,\]
where $t_n=\frac{4n}{(n+1)^2}$. Thus in the case of non-negative curvature $\Lambda_{Poi}(M,\gfrak,\mu_{\gfrak})$ and $\Lambda_{LS}(M,\gfrak,\mu_{\gfrak})$ are of the same (dimension-dependent) order.

In \cite{SaC} Saloff-Coste proved an important rigidity result for Riemannian manifolds, showing that if $(M,\gfrak,\mu_{\gfrak})$ has a finite volume, $Ric_{\gfrak}\geq K$ for some $K\in \R$, and $\Lambda_{LS}(M,\gfrak,\mu_{\gfrak})>0$,  then 
$M$ must be compact. The analogue statement for $\Lambda_{Poi}(M,\gfrak,\mu_{\gfrak})>0$ is false.   

In this work we will study LSI on weighted manifolds with $Ric_{\gfrak,\mu,N}\geq K$, where $K\in \R$, $N=\infty$ and $D\in (0,\infty]$. The lower bounds which we prove will depend on $K$ and $D$ (but not on $N$) and will apply to $N\in (n,\infty]$. 

%Our results are sharp in the following sense: they represent the sharp lower bound up to universal numeric constants; hence except for improvement of the numeric constants, it is the 

The previous best known estimate was derived by E. Milman in \cite{Mil2}, where it was shown that under $CDD(K,\infty,D)$:
\[ \Lambda_{LS}(M,\gfrak,\mu)\geq \frac{c_1}{\prnt{\int_0^{c_2D}e^{-\frac{K}{2}t^2}dt}^2}  \,,\]
where $c_1,c_2$ are some positive constants. Dimension independent lower bounds of different functional form were also proved in \cite{Wan2} by F.Wang. 

At present no sharp estimates are known. Our results close the gap being sharp in the following sense: up to universal numeric constants they are equal to the sharp lower bound. 

 \section{Main Tools}

In this section we discuss two powerful methods which underlie the derivation of our results.
This presentation attempts to balance simplicity with completeness, and we refer the reader to the relevant references at the end of each topic for further information.

\subsection{The localization method}

This method dates back to the work of Payne and Weinberger from 1960 \cite{PaWe} where a sharp  lower bound for the \Poinc constant was derived by comparison with the first non-zero eigenvalue of a vibrating string. As explained in the introduction, the reduction to comparison with a 1d problem is achieved by iterated bisections of a convex body, such that the limit convex body is effectively a 1d interval equipped with a log-concave measure. Such a technique was also 
implemented and developed by Gromov and Milman (1987) \cite{GM} and Kannan, Lovasz, Simonovits (1993/5) \cite{LS,KLS}. We refer to this dimensional reduction and its application to the study of inequalities as `localization'. Recently there has been a major progress in the theory and implementation of optimal transport techniques. Inspired by these developments B.Klartag \cite{Kla} established in 2014 a general localization theorem which applies to weighted Riemannian manifolds. Essentially the theorem guarantees a partition of the manifold measure into marginal measures supported on geodesics (`needles'); moreover under the condition $Ric_{\gfrak,\mu, N}\geq K$ the densities of these measures verify inequalities of a familiar form \eqref{DetermIneq}, extending the notion of log-concavity. We remark that the localization theorem was also generalized to Finsler manifolds \cite{Oht3} and to general metric measure spaces \cite{CM1}. We will present a brief introduction to the method; a formal formulation of the theorem will be given in the subsequent chapter. Our introduction is based on \cite{Cav}, \cite{Oht3},\cite{CFM} and \cite{Eva2}.

\subsection{Localization via optimal transport}

In order to provide a simple overview of how optimal transport ideas fuse into a localization theorem we consider the most basic problem, so called the Monge problem, where the manifold is $\R^n$ and the cost function is the Euclidean distance.

{\bf Problem 1 (Monge)} 
Let $\mu_{+}=f_{+}d\mu_{Leb}$ and $\mu_{-}=f_{-}d\mu_{Leb}$ be two Borel probability measures. 

Denote by $\Tr$ the set of 1-1 Borel maps $T:\R^n\to\R^n$ such that $T_{\sharp}d\mu_{+}=d\mu_{-}$ (also known as the set of rearrangements or transport maps), i.e. $\int_{\R^n}\phi(T(x))d\mu_{+}(x)=\int_{\R^n}\phi(y)d\mu_{-}(y)$ for every $\phi\in C(\R^n)$. Consider the following problem: find $T_{*}\in \Tr$ which minimizes the distance-cost, that is 
\[ I[T_{*}]=\min_{T\in \Tr} I[T] \qquad \text{ where } I[T]=\int_{\R^n}||T(x)-x||d\mu_{+}(x)\,.\]
 $\Tr$ was defined by the condition $T_{\sharp}d\mu_{+}=d\mu_{-}$, which is equivalent to $f_{+}(x)=f_{-}(T(x))det DT(x)$, and this constraint is highly non-linear, and \green{therefore with the exception of a few particular cases}, $\Tr$ is a set which is very difficult to work with. In the 1940's Kantorovich proposed several pioneering ideas, which significantly relaxed the minimization problem. 
%Define $P(\R^n\times \R^n)$ to be the set of Radon probability measures on $\R^n\times \R^n$. 
Rather \green{than} considering the set $\Tr$ he decided to consider the set $\Pir$ of `transport plans' defined by:
\[ \Pi(\mu_0,\mu_1):=\cprnt{\pi\in P(\R^n\times \R^n):\, p_{1\sharp}\pi=\mu_{+},\, p_{2\sharp}\pi=\mu_{-}}\,,\]
where $p_i:\R^n\times \R^n\to \R^n$ is the projection onto the $i$-th component ($\pi(A\times \R^n)=\mu_{+}(A)$ and \pinka{$\pi(\R^n\times B)=\mu_{-}(B)$} for Borel sets $A,B\subset \R^n$). By definition $\pi_T:=(Id, T)_{\sharp}\mu_{+}\in \Pir$; indeed for a Borel set $E\subset \R^n\times \R^n$ : $\pi_T(E)=\int_{\cprnt{x:\,\, (x,T(x))\in E }}d\mu_{+}(x) $  and
\[ \int_{\R^n}||x-T(x)||d\mu_{+}(x)=\int_{\R^n\times\R^n}||x-y||d\pi_T(x,y)\,.\]

%Equivalently, $\pi \in \Pir$ iff it is a non-negative measure on $\R^n\times \R^n$ such that for all $(\phi, \psi)\in L^1(d\mu_0)\times L^1(d\mu_1)$ 
%\[ \int_{\R^n\times \R^n}[\phi(x)+\psi(y)]d\pi(x,y)=\int_X\phi d\mu_{+}(x)+\int_Y\psi d\mu_{-}(y)\,.\] 
Rather \green{than} considering \green{P}roblem 1 he considered the following (infinite dimensional) linear programming minimization problem over the weakly compact convex set $\Pir$ : 

{\bf Problem 2 (Kantorovich)} Find a measure $\pi_{*}\in \Pir$ solving 
\[F[\pi_{*}]=\min_{\pi\in \Pir} J[\pi] \qquad \text{ where } J[\pi]=\int_{\R^n}\int_{\R^n}||x-y||d\pi_{*}(x,y)\,.\]
The first important observation is that such a problem admits at least one solution. Furthermore he also observed that such a problem admits an associated dual maximization problem: 

{\bf Problem 3} Consider $f:=f_{+}-f_{-}$. Find a function $\phi_*\in \Lip$ such that 
\green{
\[ K[\phi_*]=\sup_{\phi\in \Lip }K[\phi]\qquad \text{ where } K[\phi]=\int_{\R^n}\phi(x)f(x)d\mu_{Leb}(x)\,,\]}
\green{where $Lip$ stands for the space of Lipschitz functions on $\R^n$, and $Lip_1$ stands for the} \pinka{subspace} \green{ of Lipschitz functions } \pinka{ having Lipschitz constant at most 1.} \green{Under appropriate conditions on $\mu_{+}$ and $\mu_{-}$, a maximizer is known to exist  \cite{CFM}}. 

The optimality condition translates into the identity
\[ \min_{\pi \in \Pir}J[\pi]=\max_{\phi\in \Lip}K[\phi]\,. \]
\begin{defn}
Any function $\phi\in \Lip$ which maximizes $K[\phi]$ will be referred to as a Kantorovich potential. 
\end{defn} 
Let $\phi$ be a Kantorovich potential. In addition Kantorovich concluded that :
\[ \pi\in \Pir \text{ is optimal } \iff \pi\prnt{ \Gamma  }=1 \text{ where } \Gamma:=\cprnt{(x,y)\in \R^{n}\times \R^n: \phi(x)-\phi(y)=|x-y|} \,.\]
This says something about the support of the transport plan, and \green{under appropriate conditions} one can conclude that it is actually concentrated on the graph of a Borel map T, which is equivalent to $\pi=(Id,T)_{\sharp}\mu_{+}$. This is the content of the following proposition (proved in \cite{CFM}) which applies to $f_{+}$ and $f_{-}$  being compactly supported: 
\begin{prop}
Fix $\phi\in Lip_1(\R^n, d)$ and let $T\in \Tr$. If
\eql{\label{eqn:duality} \phi(x)-\phi(T(x))=||x-T(x)|| \text{  for  } \mu_{+} \, \text{ a.e. } x\in X\,, }
then 
\begin{enumerate}
	\item $\phi$ is a Kantorovich potential maximizing $K$ (\green{P}roblem 3).
	\item $T$ is an optimal map in the Monge \green{P}roblem 1.
	\item $\inf_{T\in \Tr} I[T]=\sup_{u\in \Lip} K[u]$.
	\item Every other optimal map $\tilde{T}$ and Kantorovich potential $\tilde{\phi}$ will also satisfy \eqref{eqn:duality}.
\end{enumerate}
\end{prop}

\begin{defn} A set $S\subset \R^n\times \R^n$ is $|| \cdot ||$ cyclically monotone iff for any finite subset $\cprnt{ (x_1,y_1),...,(x_N,y_N)}\subset S$ the following holds
\[ \sum_{1\leq i\leq N}||x_i-y_i||\leq \sum_{1\leq i\leq N}||x_i-y_{i+1}||\qquad (y_{N+1}=y_1)\,.\]
\end{defn} 
The set $\Gamma$ is cyclically monotone, and whenever $(x,y)\in \Gamma$, then considering $z_t=\conv{x}{y}$ it holds that $(z_s,z_t)\in \Gamma$ for any $s\leq t$. This suggests that $\Gamma$ partitions into equivalence classes which are lines $\cprnt{L_q}_{q\in Q}$ and points $\cprnt{z_a}_{a\in A}$ in $\R^n$. One calls the set $\Tc:=\bigcup_{q\in Q}L_q$ the transport set (the points which move under the optimal map) and $\Zc:=\bigcup_{a\in A}z_a$. 

Via measure disintegration theorem \cite{EvGa} we have a decomposition of the marginal measures $\mu_{+}=\int_{Q}\mu_{+q}d\zeta(q)$ and $\mu_{-}=\int_{Q}\mu_{-q}d\zeta(q)$. When $Q$ satisfies a certain measurablity condition \cite{Cav}, the conditional measures $\mu_{+q}$ and $\mu_{-q}$ are supported on the straight lines $L_q$.

The Monge problem can be stated and solved in a much broader framework of separable metric measure spaces $X$. One distinct difference between the general case and $\R^n$, is that the space is decomposed into sets $\Tc\cup \Zc$, where $\Zc$ are as before, the points stabilized by the transport map, and $\Tc$ is the transport set which is partitioned, up to a measure zero, into a family of geodesics $\cprnt{ L_q}_{q\in Q}$. 

In $\R$ there is a simple construction for the optimal transport map, which is the monotone rearrangement transport. By the measure decomposition the original Monge problem reduces to a family of 1-dimensional problems: for each $q\in Q$ find \\$\min_{\pi\in \Pi(\mu_{+q},\mu_{-q})}\int_{L_q\times L_q}d(x,y)d\pi(x,y)$, which is essentially \pink{equivalent to the Euclidean problem}, since the geodesics are isometric to intervals in $\R$, and since $\mu_{+q}$ contains no atoms, the monotone rearrangement $T_q$ exists, and one defines $T$ to be $T_q$ on $L_q$. 
This overview gives the conceptual basis for the optimal transport approach to the needle decomposition. With a conceptually similar approach follows the needle decomposition on weighted manifolds $(M,\gfrak,\mu)$ (or even general measure spaces); under the assumption that $M$ is complete or at least geodesically-convex we get in an analogous manner a partition of $M$ into length minimizing geodesics $\{\g_q\}_{q\in Q}$, and measures $\{\mu_{q}\}_{q\in Q}$ on $M$ supported on these geodesics. Moreover given that $Ric_{\gfrak,\mu}\geq K$ one can also conclude additional important properties of the marginals $d\mu_{q}$. 
For applications (e.g functional and isoperimetric inequalities), we are usually given a guiding function - a non-zero $f\in C_c(M)$ such that $\int_M f(x)d\mu(x)=0$. A Monge problem of finding an optimal rearrangement between the measures $d\mu_{+}:=f_{+}d\mu$ and $d\mu_{-}:=f_{-}d\mu$ is naturally associated with $f$. $f$ satisfies the condition $\int_M|f(x)|d(x,x_0)d\mu(x)<\infty$ which guarantees a Kantorovich potential $\phi$ (\cite{Kla,Oht3}). Following the concepts we have previously introduced, we obtain a decomposition of the measure into marginals $\cprnt{\mu_q}_{q\in Q}$ (the `needles'). A complete formulation of Klartag's needle decomposition theorem on weighted Riemannian manifolds will be given in the next chapter.  
\bigskip

For further reading: \cite{Kla,Oht3,CM1,Cav,Eva2}.

\subsection{Estimates of the \Poinc and the log-Sobolev constants on $\R$}

In subsequent chapters we will see that application of the localization theorem to functional inequalities, \green{reduces} the problem of characterizing the sharp constant of the functional  inequality, \green{ to a one-dimensional analogous problem for}  a class of measures supported \green{in} $\R$. The Muckenhoupt condition  \cite{Muck, BGL} provides a useful necessary and sufficient condition for \green{the validity of the} \Poinc \pinka{inequality} for such measures, as well as some important quantitative estimates of their  \Poinc constant. The condition as well as the estimates were extended to Log-Sobolev inequalities by Bobkov and G\"{o}tze \cite{BobG}. It will play a very important role in subsequent chapters, therefore we provide an elaborated discussion about its origins. Specific steps of its proof allow some flexibility; small modifications to these steps can yield equivalent formulations of the condition (or even further extensions) .

\subsection{\Poinc Inequalities on $\R$}
The classical Hardy inequality on the line
\eq{ \int_0^{\infty}\Abs{\frac{f(x)}{x}}^pdx\leq \prnt{\frac{p}{p-1}}^{p}\int_0^{\infty}\Abs{f'(x)}^pdx\qquad  f\in C_c^1([0,\infty)) \text{ s.t. } f(0)=0,\, p>1\,,}
was extended to the following form, called weighted Hardy inequality\green{,} by B.Muckenhoupt (1972) \cite{Muck}: 
\blue{
\eql{\label{eqn:HardyWeights} \int_0^{\infty}|f(x)|^pd\xi\leq A\int_0^{\infty}|f'(x)|^pd\nu , \qquad f\in C_c^1([0,\infty)) \text{ s.t. }f(0)=0\,, } }
where $\nu,\mu$ are non-negative Borel measures on $\R_+$ satisfying the Muckenhoupt condition. This is its most general form in dimension 1. 

\subsection{The Muckenhoupt Condition}
%\eq{ B_{+}:&=\sup_{x>0} r_{+}(x)^2 h_{+}(x) \quad \text{where}\quad  \quad\text{and}\quad h_{+}(x):=\int_{\eta}^x \frac{1}{p(t)}dt\quad\text{defined for } x\geq \eta \\
 %B_{-}:&=\sup_{x<0} r_{-}(x)^2h_{-}(x) \quad \text{where}\quad \quad\text{and}\quad h_{-}(x):=\int_x^{\eta} \frac{1}{p(t)}dt$\quad\text{defined for } x\leq \eta }
The following result on the optimal constant in the weighted Hardy inequality with $p=2$ is due to M. Artola, G. Talenti, and G. Tomaselli (cf. Muckenhoupt \cite{Muck}); it is valid for general non-negative Borel measures $\xi$ and $\nu$, which for our purposes we assume $d\xi(x)=p_{\xi}(x)dm(x)$ and $d\nu(x)=p_{\nu}(x)dm(x)$. We will write that $Har(A)$ is verified if \eqref{eqn:HardyWeights} is verified with constant $A=A(\xi,\nu)$. We will show the relation of this type of inequality to the \Poinc and LS inequalities. To this end it will be useful to introduce new definitions:
\begin{itemize}
	\item We denote by $\eta$ the median of $\xi$. 
	\item $I_{+}:=[0,\infty)$ and $I_{-}:=(-\infty,0]$.
	\item $h_{+}:=\int_0^x\frac{1}{p_{\nu}(t)}dt$ and $h_{-}:=\int_x^0\frac{1}{p_{\nu}(t)}dt$.
	\item $r_{+}(x):=\sqrt{\xi([x,+\infty))}$ and $r_{-}(x):=\sqrt{\xi((-\infty,x])}$. 
 \item Given $i\in\{-,+\}$ and a function $f$ we define the $\pm$ components of $f$ by: $f_{i}(x):=f(x)1_{I_i}(x)$. 
\end{itemize}

\begin{thm}\label{thm:Muck2} If $\{A_{i}(\xi,\nu)\}_{i\in \{+,-\}}$ are the optimal constants in the inequalities 
\eql{ \int_{I_{\pm}}f^2(x)d\xi(x)\leq A_i(\xi,\nu)\int_{I_{\pm}}f'(x)^2d\nu(x) , \qquad f\in C_c^1(I_i), f(0)=0\,, }
then \blue{$  B_i(\xi,\nu)\leq A_i(\xi,\nu)\leq 4B_i(\xi,\nu) $ }
where for $i\in \{+,-\}$: 
\blue{
$$B_+:=\sup_{x>0}\, r_{+}(x)^2 h_{+}(x)\,,\qquad B_{-}:=\sup_{x<0}\, r_{-}(x)^2 h_{-}(x)\,.$$
}
\end{thm}

%\eqref{eqn:HardyWeights} with $p=2$ and $d\nu=d\xi=p(x)dm(x)$ then  where $B(\xi)=\sup_{x>0}r_{+}(x)^2\int_0^x\frac{dt}{p(t)}$. 
%\end{thm}

\begin{proof}[Proof of Theorem \ref{thm:Muck2}]
	We firstly show that \blue{$B_{+}<\infty$} implies \blue{$Har(4B_{+})$}. 
	Starting with the CS inequality :
\eq{ f(x)^2=\prnt{\int_{0}^{x}f'(t)dt}^2=\prnt{\int_{0}^{x}f'(t)k(t)^{\half}k(t)^{-\half} dt}^2\leq \int_{0}^{x}f'(t)^2k(t)dt\int_{0}^{x}\frac{1}{k(t)}dt\,,}
where $k(x)$ is a yet undetermined positive function on $[0, x]$. Write the RHS as \\$\prnt{\int_{0}^xf'(t)^2k(t)dt}2 g(x)$ where $g(x):=\half \int_{0}^{x}\frac{1}{k(t)}dt$; then by Fubini's theorem
\eq{ &\int_{0}^{\infty}f(x)^2d\xi=\int_{0}^{\infty}f(x)^2p_{\xi}(x)dx\leq 2\int_{0}^{\infty}p_{\xi}(x)g(x)\prnt{\int_{0}^{x}f'(t)^2k(t)dt}dx\\&\stackrel{Fubini}{=}
2\int_{0}^{\infty}f'(t)^2k(t)
\prnt{\int_{t}^{\infty}g(x)p_{\xi}(x)dx}dt \,.}
Notice that by definition $g'(x)=\frac{1}{2k(x)}$, so $k(x)$ is determined directly by $g'(x)$. Take $g(x)=\sqrt{h_{+}(x)}$ then $g(x)\leq \frac{\sqrt{B_{+}}}{r_{+}(x)}$ by definition of $B_{+}$, and since $-2r_{+}'(x)=\frac{p_{\xi}(x)}{r_{+}(x)}$ we get the inequality
{\small 
\eq{ &\int_{0}^{\infty}f(x)^2p_{\xi}(x)dx\leq -4\sqrt{B_{+}}\int_{0}^{\infty}f'(t)^2k(t)\prnt{\int_{t}^{\infty}r_{+}'(x)dx}dt=4\sqrt{B_{+}}\int_{0}^{\infty}f'(t)^2 k(t)r_{+}(t)dt\,.}}
Notice that $k(t)=\frac{1}{2g'(t)}=\sqrt{h_{+}(t)}p_{\nu}(t)$ implying that $k(t)r_{+}(t)=\sqrt{h_{+}(t)}p_{\nu}(t)r_+(t)\leq \sqrt{B_{+}}p_{\nu}(t)$. By definition $r_{+}(t)\sqrt{h_{+}(t)}\leq\sqrt{B_{+}}$ whence
\eql{\label{eqn:ComponentPoi} &\int_{I_{+}}f(x)^2p_{\xi}(x)dx\leq 4B_{+}\int_{I_{+}}f'(x)^2p_{\nu}(x)dx \,,}
implying that $A_{+}\leq 4B_{+}$. We apply the same argument to the term $\int_{I_{-}}f(x)^2d\xi$ using $B_{-}$ instead of $B_{+}$ to conclude $A_{-}\leq 4B_{-}$. 

The opposite direction, i.e. \blue{$Har(A_{+})$ implies $B_{+}\leq A_{+}<\infty$}, is proved by substitution of proper test functions (\cite[p. 197]{BGL}).

\end{proof}
The following theorem, known as the Muckenhoupt condition, is a corollary of the previous theorem:
\begin{thm}\label{thm:Muck1} Let $\xi$ be a Borel probability measure on $\R$ with $d\xi(x)=p_{\xi}(x)dm(x)$. Let $\eta$ be a median of $\xi$ (i.e. $\min\prnt{\xi([\eta,+\infty)),\xi([-\infty,\eta))}\geq \half$). Define 
\eq{ B_{+}:=\sup_{x>\eta} \xi([x,+\infty)) \int_{\eta}^x \frac{1}{p_{\xi}(t)}dt\qquad 
 B_{-}:=\sup_{x<\eta} \xi((-\infty,x])\int_x^{\eta} \frac{1}{p_{\xi}(t)}dt\,. }
Then $\Lambda_{Poi}(\xi)>0$ iff $B=B_{+}+B_{-}<\infty$. Moreover we can then estimate $\frac{1}{4}B\leq \frac{1}{\Lambda_{Poi}(\xi)}\leq 4B$. 
\end{thm}
\begin{proof}
\blue{
Notice that $Var_{\xi}(f)\leq \int[f(t)-f(\eta)]^2d\xi(t)$ since the variance minimizes distance to constants in $L^2(\xi)$, while the integral $\int f'(x)^2d\xi(x)$ is invariant under $f(t)\mapsto f(t)-f(\eta)$. By a change of coordinates we may further assume that $\eta=0$ and then $f(0)=0$ as well. Hence in order to show $\Lambda_{Poi}(\xi)>0$ it is sufficient to show that for some $A<\infty$: $\int f(t)^2d\xi(t)\leq A\int f'(t)^2d\xi(t)$ for any $f\in \Cinf_c(\R)$ s.t. $f(\eta)=0$.
According to Theorem \ref{thm:Muck2} any such $f$ satisfies $\int_{I_{\pm}}f(x)^2d\xi(x)\leq A_{\pm}(\xi)\int_{I_{\pm}}f'(t)^2 d\xi(x)$, where the constants $A_{+}(\xi):=A_+(\xi,\xi)$ and $A_{-}(\xi):=A_{-}(\xi,\xi)$ verify the following estimates $B_i(\xi)\leq A_i(\xi)\leq 4B_i(\xi)$. Thus in particular  $\Lambda_{Poi}(\xi)^{-1}\leq A:=A_{+}+A_{-}$. 
We conclude that $\xi$ verifies a \Poinc  inequality with constant $[4(B_{-}+B_{+})]^{-1}$ (i.e. $Poi(\frac{1}{4B})$). }

% 

%Notice that $Var_{\xi}(f)\leq \int[f(x)-f(\eta)]^2d\xi(x)$ since the variance minimizes distance to constants in $L^2(\xi)$, therefore if $A$ is a constant for which the inequality $\int[f(x)-f(\eta)]^2d\xi\leq A\int f'(x)^2d\xi$ holds, then $\Lambda_{Poi}(\xi)^{-1}\leq A$. Therefore if $A$ is the sharp constant for the latter  inequality then $\Lambda_{Poi}(\xi)^{-1}=A$. Therefore if $Poi(\Lambda_{Poi}(\xi))$ is valid, it is completely equivalent to prove that $A$ satisfies the stated estimates.
%% is equal to the optimal constant $A$ in the original \Poinc inequality $Var_{\xi}(f)\leq\int f'(x)^2d\xi$ . 
%
%By a change of coordinates we may assume w.l.o.g that $\eta=0$. 

For the converse, still under the assumption that $\eta=0$,  one uses the fact that under $Poi(\Lambda_{Poi}(\xi))$ for any function which vanishes outside $I_{+}$ (alternatively outside $I_{-}$): $\int f^2(x)d\mu(x) \leq \frac{2}{\Lambda_{Poi}(\xi)}\int f'(x)^2d\mu(x)$ (\cite[p.182]{BGL}). Substitution of proper test functions shows $B_{+},B_{-}\leq \frac{2}{\Lambda_{Poi}(\xi)}$, hence $B=B_{+}+B_{-}\leq \frac{4}{\Lambda_{Poi}(\xi)}$ (for the details see \cite[p.197]{BGL}). 
\end{proof}

\bigskip

One should notice that whenever $B<\infty$ then $\frac{1}{p_{\xi}}$ must be locally integrable  and $supp(\xi)$ must be a connected set. In addition the median is unique, since the measure of any open interval inside $supp(\xi)$ is strictly positive.

\subsection{Log-Sobolev inequalities on $\R$}

\subsubsection{LS inequalities as \Poinc inequalities in Orlicz spaces \cite{BobG}}

Let $(\Omega, \mu)$ be a probability space. Given a Young function (i.e. an even convex function $N:\R\to [0,+\infty)$ with $N(0)=0, N(x)>0$ for $x>0$), the Orlicz space $L_N(\Omega, \mu)$ consists of all measurable functions with 
\pinka{$||f||_N=\sup\cprnt{\lambda>0:\,\, \int N(f/\lambda)d\mu\geq 1 }<\infty$}. Since $N$ strictly increases on $[0,\infty)$ it admits an inverse $N^{-1}:[0,\infty)\to [0,\infty)$. When $N(x)=|x|^p$ $(1\leq p<\infty)$, $L_N(\Omega, \mu)$ is the usual Lebesgue \green{$L^p(\mu)$} space which corresponds to the norm $||f||_p$.  We will now discuss how LS inequalities can be interpreted as \Poinc-type inequalities. 
Recall \eqref{defn:back:ent} for the definition of the Entropy. 

The LS inequality 

\eq{ Ent_{\mu}(f^2)\leq \Lambda_{LS}^{-1}(\Omega, \mu)\int f'(x)^2d\mu(x)\,, }

can clearly be strengthened to
\[ L(f):=\sup_{a\in \R} Ent_{\mu}((f+a)^2)\leq \Lambda_{LS}^{-1}(\Omega, \mu)\int_{\Omega} f'(x)^2d\mu(x)\,, \]
since $(f+a)'=f'$ for all $a\in \R$. In \cite{BobG} Bobkov and G\"{o}tze proved that with respect to the Young function $N_1(x):=x^2\log(1+x^2)$, for any \green{$f\in L_{N_1}(\Omega, \mu)$}:

\[ ||f-\int_{\Omega} fd\mu||_{N_1}^{2}\eqsim L(f)\,, \]

where $A\eqsim B$ iff there are numeric constants $c_1,c_2>0$ s.t. $c_2B\leq A\leq c_1B$. 
This implies that the LS inequality can be written as a Poincar\'{e}-type inequality\footnote{We remark that one can also identify LSI as limits of Sobolev inequalities (or more specifically Beckner-type inequalities) with norm $|| \cdot ||_{2+\epsilon}$ where $\epsilon\to 0$; the reader is referred to \cite[p.312]{BGL}  for further details.} with constant $\Lambda_{LS}'$

\eql{\label{eqn:LSasPoin} ||f-\int_{\Omega} fd\mu||_{N_1}\leq \Lambda_{LS}^{'-1}\int_{\Omega} f'(x)^2d\mu\,.}
 
In addition they observed that when \green{$(\Omega,\mu)=(\R,\xi)$, where $\xi$ is a probability measure supported in $\R$},  \eqref{eqn:LSasPoin} can be reduced into a Hardy-type inequality for the Orlicz space norm $||\cdot ||_{N_2}$ where $N_2(x):=|x|\log(1+|x|)$. Specifically they showed that for any $f\in L_{N_1}(\R,\xi)$ : 
\begin{enumerate}
	\item \blue{$||f||_1\leq ||f||_2\leq C_{2,N_1} ||f||_{N_1}$} (with \green{ $0<C_{2,N_1}\leq \frac{\sqrt{5}}{2}$ some constant)}.
	\item $L(f)\eqsim ||f-\int fd\xi||_{N_1}\lesssim ||f||_{N_1}$. 
\end{enumerate}
However for the $\pm$ components of a function $f$, $(2)$ can be improved: \green{ $||f_{i}-\int f_{i}d\xi||_{N_1}\eqsim ||f_{i}||_{N_1}$}. Indeed by the Cauchy-\pinka{Schwarz} inequality \pink{$|\int f_id\xi|=|\int f_i1_{I_i}d\xi|\leq ||f_i||_2 \frac{1}{\sqrt{2}} \leq \frac{C_{2,N_1}}{\sqrt{2}} ||f_{i}||_{N_1}$}. Using the triangle inequality
\pinka{
\eq{
||f_i||_{N_1}&\leq ||f_i-\int f_id\xi||_{N_1}+|\int f_id\xi|=
||f_i-\int f_id\xi||_{N_1}+\frac{|\int f_id\xi|}{N_1^{-1}(1)}\\&\leq
||f_{i}-\int f_{i}d\xi||_{N_1}+\frac{\sqrt{5}}{2\sqrt{2}}||f_i||_{N_1}\,,}
}
\green{whence $||f_i||_{N_1}\lesssim ||f_i-\int f_id\xi||_{N_1}$.
}
In view of this inequality, if we apply \eqref{eqn:LSasPoin} to the $\pm$ components of a function $f$, and use the relation $||f||_{N_1}^2=||f^2||_{N_2}$, we get the inequalities
\eql{\label{IneqNormLS} ||f_{i}^2||_{N_2}=||f_{i}||^2_{N_1}\lesssim ||f_{i}-\int f_{i}d\xi||_{N_1}^2\eqsim L(f_i)\lesssim \Lambda_{LS}^{-1}\int_{I_i}f'(x)^2d\xi(x)\qquad i\in\{+,-\}\,.}

\subsubsection{The Bobkov-G\"{o}tze condition for LS inequalities}

Denote by $X_{i}$ ($i\in \{+,-\}$) the Banach spaces \green{$L_{N_2}(I_i, \xi)$}. 
Let $\A(X_i, \xi)$ be the `optimal' (smallest) constants for the inequalities
\[ ||f^2||_{N_2}\leq \A\int_{I_i}f'(x)^2d\xi(x),\qquad f\in C^{1}(I_i),\,f(0)=0,\,\,i\in\{+,-\}\,.\]

By \eqref{IneqNormLS} we can conclude that $\Lambda_{LS}^{-1}(\xi)\geq \A(X_{\pm},\xi)$. In particular\\ $\Lambda_{LS}^{-1}(\xi)\geq\half\prnt{ \A(X_{+},\xi)+\A(X_{-},\xi)}$. 
On the other hand 
\eq{ L(f)&\eqsim ||f-\int fd\xi ||_{N_1}^2\leq \prnt{\sum_{i=\pm}||f_i-\int f_i d\xi ||}^2\lesssim ||f^2_{+}||_{N_2}+||f^2_{-}||_{N_2}\\&\lesssim \prnt{\A(X_{+},\xi)+\A(X_{-},\xi)}\int f'(x)^2 d\xi(x)\,, }
implying that $\A(X_{+},\xi)+\A(X_{-},\xi)\gtrsim \Lambda_{LS}^{-1}$. Therefore we may write $\Lambda_{LS}^{-1}\eqsim \A(X_{+},\xi)+\A(X_{-},\xi)$. The constants $\A(X_{\pm},\xi)$ can be easily estimated using the following observation: the spaces $X_i$ 
 belong to a class of Banach spaces which are called `ideal' (see \cite{KaAk} for further details); the norms of these spaces admit a useful representation formula: 
\green{
\eql{\label{eqn:norm}  ||f||=\sup_{g\in \mathcal{G}_i}\int_{I_i}|f(x)|g(x)d\xi(x) \qquad i\in\{+,-\}\,,}
}
where \blue{$\mathcal{G}_i$} is some family of non-negative Borel measurable functions $g$ on $I_{i}$ with Borel measure \green{$\xi$}. 
Now, apply Theorem \ref{thm:Muck2} to the pairs \blue{$(\xi_g, \xi)$ where $d\xi_g(x):=g(x)d\xi(x)$} \pink{for $g\in \mathcal{G}_{\pm}$}. The theorem shows that  \pink{$A_{\pm}(\xi_g, \xi)$} is equivalent to \green{$B_{\pm}(\xi_g,\xi)$}; specifically: $B_{\pm}(\xi_g, \xi)\leq A_{\pm}(\xi_g, \xi)\leq 4B_{\pm}(\xi_g, \xi)$. 
%If $\A(X_i, \xi)$ ($i\in \{+,-\}$) is the optimal constant in the inequality
%\[ ||f^2||_{N_2}\leq \A\int_{I_i}f'(x)^2d\xi(x)\,, \]
%where $f\in C^{\infty}(\R)$ such that $f(0)=0$, 
$\A(X_i, \xi)$ ($i\in \{+,-\}$) (being optimal) can be expressed as
\pinka{$\A(X_{i},\xi)=\sup_{g\in \mathcal{G}_i} A_{i}(\xi_g, \xi)$} in view of the representation formula \eqref{eqn:norm}. If
\pinka{$\B(X_{i},\xi):=\sup_{g\in \mathcal{G}_i} B_{i}(\xi_g, \xi)$}, then
\[ \B(X_i,\xi)\leq A(X_i,\xi)\leq 4\B(X_i,\xi)\,. \]

\blue{Recall that  $X_{i}$ ($i\in \{+,-\}$) were defined as the Banach spaces} \green{ $L_{N_2}(I_i, \xi)$}; \blue{ their norm \\ $||\cdot||:=||\cdot||_{N_2}$ verifies  \eqref{eqn:norm}; by this representation formula we may write:}
\blue{
\eq{ \B(X_+,\xi)&=\sup_{g\in \mathcal{G}_{+}}\sup_{x>0}\prnt{\int 1_{[x, +\infty)}d\xi_g} h_{+}(x)=\sup_{x>0}||1_{[x, +\infty)}|| h_{+}(x)\,.} 
} 
Since \blue{$||1_{[x,\infty)}||_{N_2}=\frac{1}{N_2^{-1}(1/\xi([x,\infty))}$}, and since $N_2^{-1}(t)\eqsim \frac{t}{\log t}$ (see \cite[p.25]{BobG}) we conclude that 
\[ \B(X_{+},\xi) \eqsim \sup_{x>0}\frac{1}{N_2^{-1}(1/\xi([x,\infty))}h_{+}(x)\eqsim \sup_{x>0}(1-F(x))\log\prnt{\frac{1}{1-F(x)}}h_{+}(x)\,, \]
where $F(x):=\xi((-\infty,x])$ stands for the distribution function of $\xi$. Similar arguments show that 
\[ \B(X_{-},\xi) \eqsim \sup_{x<0}\frac{1}{N_2^{-1}(1/\xi((-\infty, x])}h_{-}(x)\eqsim \sup_{x<0}F(x)\log\prnt{\frac{1}{F(x)}}h_{-}(x) \,.\]
These estimates give the Bobkov-G\"{o}tze condition.
\begin{thm}\label{thm:Muck1} Let $\xi$ be a Borel probability measure on $\R$ with $d\xi(x)=p_{\xi}(x)dm(x)$. Let $\eta$ be a median of $\xi$ (i.e. $\min\prnt{\xi([\eta,+\infty)),{\xi([-\infty,\eta)}\geq \half}$. Define 
\eq{ \B_{+}:&=\sup_{x>\eta} \xi([x,+\infty))\log\prnt{\frac{1}{\xi([x,+\infty))}} \int_{\eta}^x \frac{1}{p_{\xi}(t)}dt\quad \text{and}\\
 \B_{-}:&=\sup_{x<\eta} \xi((-\infty,x])\log\prnt{\frac{1}{\xi((-\infty,x])}}\int_x^{\eta} \frac{1}{p_{\xi}(t)}dt \,.}
Then $\Lambda_{LS}(\xi)>0$ iff $\B=\B_{+}+\B_{-}<\infty$. Moreover $\frac{1}{\Lambda_{LS}(\xi)}\eqsim \B$. 
\end{thm}
For the exact multiplicative constants which give this equivalence the reader is referred to \cite{BobG}. As we previously mentioned, modifications/improvements to this condition are possible. For example from  \cite{BaRo, BMZ} it follows that $\Lambda_{LS}^{-1}(\xi)\leq 4\max(B_{-}^{\theta}, B_{+}^{\theta})$, where for $\theta\in\R$:
\eq{
\B_{+}^{\theta}:&=\sup_{x>\theta} \xi([x,+\infty))\log\prnt{1+\frac{e^2}{\xi([x,+\infty))}} \int_{\theta}^x \frac{1}{p_{\xi}(t)}dt\quad \text{and}\\
 \B_{-}^{\theta}:&=\sup_{x<\theta} \xi((-\infty,x])\log\prnt{1+\frac{e^2}{\xi((-\infty,x])}}\int_x^{\theta} \frac{1}{p_{\xi}(t)}dt \,,}
yet for $\theta=\eta$ we also have $\max(B_{-}^{\eta}, B_{+}^{\eta})\leq 4\Lambda_{LS}^{-1}(\xi)$. 
\bigskip

For further reading see \cite{BobG,Muck,BGL}.

 \chapter{The extreme points characterization theorem}
\label{chp:ExPoints}
\section{Definitions}
 
The following definitions will be utilized through the rest of this work. Mind that these definitions were adapted for this work, however the same terminology might be used differently by other authors, in other contexts or for different purposes.   

\begin{defn}[Convex Weighted Riemannian manifolds]\label{defn:WRM} A weighted Riemannian manifold (\wrm) is a triple $(M^n,\gfrak,\mu)$, where 
\begin{itemize}
	\item $(M^n,\gfrak)$ is a smooth, connected, complete $n$-dimensional Riemannian manifold (with or without boundary $\partial M$) where $n\geq 1$.
 \item $\mu$ is a measure on $M$ having density $U(x)=e^{-V(x)}$ with respect to $\mu_{\gfrak}$ (the standard Riemannian measure), which is smooth and positive on $M$.
\end{itemize}
We say $(M^n,\gfrak,\mu)$ is a convex weighted Riemannian manifold (\cwrm{}) if in addition 
\begin{itemize}
	\item $int(M)$ is geodesically-convex.
\end{itemize}
\end{defn}
\smallskip

We remind the reader that a set $A$ is said to be geodesically convex if any two points $x,y\in A$ are connected by at least one distance minimizing geodesic in $A$.	

Our results pertain to convex weighted Riemannian manifolds, therefore from this point on whenever we refer to a general triple $(M^n,\gfrak,\mu)$ with no further specification, we will assume it is a \cwrm{}. 
\bigskip

%; by locally-convex we mean that the second fundamental form evaluated at the boundary $\Pi_{p}(X,Y):=\gfrak(\nabla_X \nu,Y)$ ($X,Y\in T_p\partial M$) is positive semi-definite for all $p\in \partial M$. Similar convexity assumptions where also assumed in \cite{Mil4}, however here we leave no alternatives to geodesic convexity of the interior .

\begin{defn}[Curvature-Dimension-Diameter conditions] \label{dfn:CDD} Given $K\in \R$ and $N\in (-\infty, \infty]$, we say that $(M^n,\gfrak,\mu)$ satisfies 
\begin{itemize}
	\item $\pmb{ \gls{CD(K,N)}}$ (Curvature-Dimension conditions) if 
\[ Ric_{\gfrak,\mu,N}\geq \gls{K}\cdot \gfrak \qquad \text{ (as  symmetric 2-tensors on $M$)} \,,\]
where $Ric_{\gfrak,\mu, N}$ is the generalized Ricci tensor which is defined by
\eq{ Ric_{\gfrak,\mu,\gls{N}}:&= Ric_{\gfrak}+\nab_{\gfrak}^2V-\frac{1}{N-n}\nab_{\gfrak} V\otimes \nab_{\gfrak} V=Ric_{\gfrak}-(N-n)\frac{\nab_{\gfrak} ^2 U^{\frac{1}{N-n}}}{U^{\frac{1}{N-n}}}\,,}
with the conventions $\frac{1}{\infty}=\infty\cdot 0=0$.
Here $Ric_{\gfrak}$ denotes the ordinary Ricci tensor. Following the above conventions the case $N=n$ is possible only when $U$ is constant ($\nab_{\gfrak} V=0$).  Evidently $Ric_{\gfrak,\mu_{\gfrak},n}=Ric_{\gfrak}$ and $Ric_{\gfrak,\mu,\infty}=Ric_{\gfrak}+\nab_{\gfrak}^2 V$.
\item $\pmb{ \gls{CDD(K,N,D)}}$ (Curvature-Dimension-Diameter conditions) if in addition to $CD(K,N)$ also $diam(M):=\sup_{x,y\in M}d(x,y)\leq \gls{D}$ where $D\in (0,+\infty]$.
\item $\pmb{\gls{CD_b(K,N)}}$ (resp. $\pmb{\gls{CDD_b(K,N,D)}}$) if in addition to $CD(K,N)$ (resp. $CDD(K,N,D)$) also $\mu(M)<\infty$. 
\end{itemize}
\end{defn}
For early works which motivate this definition the reader is referred to \cite{BE1, BE2, Bus, Mil0}.
For most of this work the pertinent range of $N$ is $(-\infty, 0]\cup [n,\infty]$; in this range occurrences of $\frac{1}{N}$  (resp. $\frac{1}{N-n}$) when $N=0$ (resp. $N=n$) should be interpreted as $-\infty$ (resp. $+\infty$), as one would expect by considering limits. 
%It is then should be clear how to interpret the notation $N=0_{+}$ or $N=0_{-}$. 
For many purposes it could have been more natural to accept $\frac{1}{N}$ as the significant parameter; for example, one should observe that the $CD(K,N)$ condition is monotone in $\frac{1}{N}$ when $\frac{1}{N}\in [-\infty, \frac{1}{n}]$ in the following sense 
\eql{\label{monoton:CDKN} CD(K,N_1)\Rightarrow CD(K,N_2) \qquad\mbox{ if }\qquad \frac{1}{N_2}\leq \frac{1}{N_1}\,.} 
In particular this shows that $CD(K,0)$ is the weakest. Notwithstanding, considering other purposes, as well as tradition, we have accepted $N$ as the significant parameter for the formulation of statements and definitions.
The range $[n,\infty]$ has been extensively studied and is well understood, while at present much less is known about the range $(-\infty, 0]$. Nevertheless, this gap is quickly filled; first results on the subject appeared in  \cite{Oht1, OT1, Mil3, Mil5,Mil6}. 

\bigskip
% and $D_{\delta}$
\begin{defn}[The symbols $\gls{delta}$ and $\gls{l_{delta}}$] \label{dfn:deltaSymbols} Whenever $K\in\R,\, N\in(-\infty,\infty]$, we define

%{\small 

%\eq{ 
\[\delta=\delta(K,N):=\begin{cases}\frac{K}{N-1}& \mbox{ if } N\neq 1\\ 0 &\mbox{ otherwise } \end{cases}\,, \quad l_{\delta}=l_{\delta}(K,N):=\begin{cases} \frac{\pi}{\sqrt{\delta}} &\mbox{if} \qquad \delta>0\\
+\infty &\mbox{if} \qquad \delta\leq 0
\end{cases}\,.\]
 %\text{and}\,\, D_{\delta}=D_{\delta}(K,N,D):=\min\{D, l_{\delta}\}
%\,.
%} }
\end{defn}
\begin{remk}\label{remk:obDiameter}
The dependence of $l_{\delta}$ on $\frac{1}{N}$ will be important for certain statements. Consider the variable $\frac{1}{N}$ on the domain $\frac{1}{N}\in [-\infty, 1)$ (i.e. $N\in (-\infty,0]\cup (1,\infty]$). It is immediate to verify that on $[-\infty, 1)$, for $K>0$ the function $\frac{1}{N}\to l_{\delta}$ is non-increasing, while for $K<0$ it is non-decreasing. Notice that for $\frac{1}{N}=0$ it holds that $l_{\delta}=\infty$ no matter what is the sign of $K$. 
\end{remk}
Throughout this work we assume the following definition of the support: 
\begin{defn}[The supports $\gls{supp}$ and $\gls{ssupp}$]\,\qquad\\
\begin{itemize}
	\item Given a \blue{(signed)} measure $\xi$ on $\R$ we define its support $supp(\xi)$ to be the set defined by:
	\blue{
\[ supp(\xi)=\{x\in \R: |\xi|(I(x;r))>0,\qquad \forall r>0\}\,, \]
}
where $I(x; r):=(x-r,\,x+r)$.
  \item Given a $[m]$-measurable function \blue{$f:\R\to \R\cup \{\pm \infty\}$, we define $supp(f):=supp(f\,dm)$ and  $\isupp(f):=supp(f\,dm)\cap supp(f^{-1}\,dm)$. }
\end{itemize}

%$\isupp(f)$ to be the set
%\[ X\setminus \bigcup \cprnt{A\subset \Omega: \, A \text{ is open and } f(x)=0 \text{ or } f(x)=+\infty \text{ for } [m]-a.e \,\,x\in A } \]
\end{defn}

\begin{defn}[The class $\Fknd^{C^{\infty}}(I)$]\label{dfn:Z_KND} Given $K\in \R$, $N\in (-\infty,\infty]$, $D\in (0,\infty]$, and $I\subset \R$ a closed interval, we say that a measure $\xi$ on $\R$ is of class $\Fknd^{C^{\infty}}(I)$, if it satisfies the following conditions:
\begin{enumerate}
	\item $\xi$ is finite.
	\item $supp(\xi)\subset I$.
	\item $supp(\xi)$ is an interval $I_{\xi}$ s.t. $diam(I_{\xi})\leq D$.
	\item \label{cd_kn_property}$d\xi=Jdm$, where 
	\begin{enumerate}
		\item $J\in \Cinf(int(I_{\xi}))$ is strictly-positive on $int(I_{\xi})$.
		\item $J$ satisfies the following differential inequality on $int(I_{\xi})$
	\eql{\label{eqn:CD(k,N):1}  -LogHess_{N-1}J\geq K\,,}
	where
	{\footnotesize 
	\[ LogHess_{N-1}J:=(\log J)''+\frac{1}{N-1}((\log J)')^2 \qquad \prnt{\iff (N-1)\frac{(J^{\frac{1}{N-1}})''}{J^{\frac{1}{N-1}}}  \,\,,\text{ for } N\notin\{1,\infty\}}\,, \]
	}
	with the usual interpretation when $N=1$ or $N=\infty$. 
	\end{enumerate}
	When $I=\R$ we refer to this class simply by $\Fknd^{C^{\infty}}$.
\end{enumerate}
%We define $\Z_{(k,N,D)}(I)$ to be the subset of $\Z_{(k,N)}(I)$ whose members $\rho$ satisfy the diameter condition : $Diam(supp(\rho))\leq D$. 
\end{defn}

\begin{remk}
 The condition \eqref{eqn:CD(k,N):1} should be compared with Definition \ref{dfn:CDD}, as it can be interpreted as a
 $CDD_b(K,N,D)$ condition of the 1 dimensional space $(\R, |\cdot |,\mu=J\cdot m)$. We remark that this condition is closely related to the Heintze-Karcher theorem and its generalizations (Heintze-Karcher \cite{HeKa}, Bayle \cite{Bay}, Morgan \cite{Mor}, and E.Milman \cite{Mil3} for negative $N$) for smooth hypersurfaces in $S\subset M$.  For $N=n$ the condition can be conceptually interpreted as the equation satisfied by the square-root of a matrix determinant of Jacobi-fields orthogonal to a specified geodesic emanating from the hypersurface. For further details about generalizations of the Heintze-Karcher theorem to the setting of \wrm{} the reader is referred to \cite{Mil2, Mil3}.
\end{remk}

Following \cite{Mil2, Mil3} we introduce:
\begin{defn}[\blue{The $\gls{zeta_min}$ and $\gls{zeta_max}$ delimiters, $f_{+}$}]\label{dfn:zeta_min_max}
Given a continuous function $f\in C(\R)$ with $f(0)\geq 0$, we set $\zfrak_{-}(f):=\sup\{x\leq 0:\,f(x)=0\}$ and $\zfrak_{+}(f):=\inf\{x> 0:\,f(x)=0\}$ and we define $f_+:= f\cdot 1_{[\zfrak_{-}(f), \,\,\zfrak_{+}(f)]}$. 
\end{defn}
\begin{defn}[$\gls{f_vee}$]\label{defn:vee} \blue{Given a function $f:\R\to \R\cup \{\pm\infty\}$ we define $f_{\vee}:=\max\{ f, 0\}$.}
\end{defn}

\begin{defn}[Model-space densities]\label{defn:Jknh}
Given $K\in \R$, $N\in (-\infty,\infty]$ and 
 $\gls{hfrak}\in\R$, we define $J_{K,N,\hfrak}:\R\to \R_+\cup\{+\infty\}$ as the functions \green{
\eql{\label{eqn:JhkN} \gls{Jknh}=
\begin{cases}   
(\co_{\delta}(x)+\frac{\hfrak}{N-1} \si_{\delta}(x))_+^{N-1} &\mbox{ if } N \notin\{1,\infty\}\\
\exp(\hfrak x-\frac{K}{2}x^2) &\mbox{ if } N=\infty\,\\
1 & \mbox{ if }N=1\,,
\end{cases}
}
}
where 
\eql{\label{SinCos} \si_{\delta}(x):=\begin{cases} \sin(\sqrt{\delta}x)/\sqrt{\delta} & \delta>0 \\
x & \delta = 0 
\\
\sinh(\sqrt{-\delta}x)/\sqrt{-\delta} & \delta<0 \end{cases} \qquad
\co_{\delta}(x):=\begin{cases} \cos(\sqrt{\delta}x) & \delta>0 \\
1 & \delta = 0 
\\
\cosh(\sqrt{-\delta}x) & \delta<0\,, \end{cases}
}
where $\delta=\delta(K,N)$ is defined in Definition \ref{dfn:deltaSymbols}. 
Notice that $J_{K,N,\hfrak}(0)=1$ and for $N\neq 1$ $J_{K,N,\hfrak}'(0)=\hfrak$. 
\end{defn}
\begin{remk}\label{remk:ModelDensitiesProp} One should notice the following:
\begin{itemize}
	\item For $N\neq 1$ $J_{K,N,\hfrak}$ satisfies the equation $-LogHess_{N-1}J_{K,N,\hfrak}= K$ on $int(\isupp(J_{K,N,\hfrak}))$ (cf. \eqref{eqn:CD(k,N):1} with $J(0)=1$,\, $J'(0)=\hfrak$).
	\item $diam(\isupp(J_{K,N,\hfrak}))$ is exactly $l_{\delta}$ (see Definition \ref{dfn:deltaSymbols}). 
	\item When $N<0$ and $K<0$ then $J_{K,N,\hfrak}$ is not integrable at the boundary points of $\isupp(J_{K,N,\hfrak})$. 
	\item It is easy to check that for fixed $K, \, \hfrak\in\R$ it holds that $J_{K, N, \hfrak}(x)\stackrel{N\to\infty}{\longrightarrow} J_{K, \infty, \hfrak}(x)$.  

\end{itemize}
\end{remk}
\bigskip

\begin{defn}[The model class $\Fknd^M$]\label{defn:CD_Model_First}
We define $\Fknd^M$ as the set of finite absolutely continuous measures $d\xi=Jdm$, s.t. \blue{$J(x+r)=cJ_{K,N,\hfrak}(x) 1_{I}(x)$}, where $c>0$, \green{$r\in \R$} and $I\subset \isupp(J_{K,N,\hfrak})$ is an interval s.t. $0<diam(I)\leq D$.  
\end{defn}
Notice that $\Fknd^M\subset \Fknd^{\Cinf}$.

 \section{The localization theorem and functional inequalities}

As discussed in the introduction, the localization theorem \cite{Kla} of B. Klartag is the main tool we use to estimate the constants associated with functional inequalities on manifolds. We present it here in a convenient form. Notice that throughout we use the notion of Lebesgue measurability on a manifold $M$, meaning measurability with respect to the completion of the Borel $\sigma$-field (see \cite{Kla} for further details).

\begin{thm}[The localization theorem with a guiding function, B. Klartag \cite{Kla}; see also \cite{CM1,Oht3}] \label{thm:localization} Assume $(M^n, \gfrak, \mu=e^{-V}\cdot \mu_{\gfrak})$ (with $n\geq 2$) is a geodesically-convex \wrm{}  with $\mu(M)<\infty$ and s.t. $\partial M=\emptyset$. Assume $h:M\to \R$ is a $\mu$-integrable function such that $\int_Mh\,d\mu=0$ and $h(\cdot)d(x_0, \cdot)\in L^1(M; \mu)$ for some $x_0\in M$. Then there exists a decomposition $M=\Tc\bigcupdot\Zc$,  a measure space $(Q, \Sigma_Q, \zeta)$ and a set of Lebesgue measures $\{\mu_q\}_{q\in Q}$ on $M$ such that: %$M=\Omega^{\phi_u}_0\coprod \Omega^{\phi_u}_1 \coprod \Omega^{\phi_u}_2$, such that

\begin{enumerate}
		\item $h=0$ $[\mu]$-a.e. on $\Zc$. 
		\item $\Tc=\bigcup_{q\in Q} X_q$ where 
		\begin{itemize}
			\item $X_q$ is the image of a positive-length unit-speed length-minimizing geodesic $\g_q:I_q\to M$, where $I_q\subset \R$          is an open interval.
			\item Up to a $[\mu]$ null-set this is a partition of $\Tc$.
		\end{itemize}
		\item For any Lebesgue measurable set $A\subset M$: 
		\begin{itemize}
			\item The map $q\mapsto \mu_{q}(A)$ is well defined $[\zeta]$ a.e. and is $[\zeta]$ measurable.
			\item $\mu(A\cap \Tc)=\int_{q\in Q}\mu_{q}(A)d\zeta(q)$.
		\end{itemize}
		\item For $[\zeta]$ a.e. $q\in Q$:
			\begin{enumerate}
				\item $supp(\mu_{q})\subset \overline{X_q}$. 
				\item\label{int_cond} $\int_{X_q}h(x)d\mu_{q}(x)=0$ and $\int_{X_q}d\mu_{q}(x)=1$.
				\item\label{RicNeedle} If $(M,\gfrak,\mu)$ is of class $CDD_b(K,N,D)$ where $N\in (-\infty,1)\cup [n,\infty]$, then	$\mu_{q}=\g_{q\sharp}\prnt{J_q\cdot m}$ where $J_q\cdot m\in \Fknd^{C^{\infty}}$. 
		   \end{enumerate}
	\end{enumerate}
\end{thm}
Here we denoted by $\g_{q\sharp}$ the pushforward defined by $\g_{q}$, so the condition $\mu_{q}=\g_{q\sharp}\prnt{J_q\cdot m}$ amounts to
$\int_{\R}\phi(\g_{q}(x))J_q(x)dm(x)=\int_M\phi(y)d\mu_{q}(y)$ for any $\phi\in C(M)$.

%$T:\R^n\to\R^n$ such that $T_{\sharp}d\mu_{+}=d\mu_{-}$ (also known as the set of rearrangements or transport maps), i.e $\int_{\R^n}\phi(T(x))d\mu_{+}(x)=\int_{\R^n}\phi(y)d\mu_{-}(y)$ for every $\phi\in C(\R^n)$.

Notice that by rescaling of the measure $\zeta$ we can assume any other convenient normalization of $\{\mu_q\}_{q\in Q}$ in $4.b$. 
We refer to the measures $\mu_q$ as '$CDD_b(K,N,D)$-{\bf needles}', and to the set $\{(X_q, \mu_q)\}_{q\in Q}$ as a needle decomposition of $(M,\gfrak,\mu)$ associated with the guiding function $h$. 
%Klartag shows that $\{X_q\}_{q\in Q}$ is a partition of $\Tc$ up to a $[\mu]$ null-set.

\subsection{Functional inequalities on geodesically convex domains}
\label{localization results}
Following Klartag \cite{Kla}, we now derive the main variational formulas for the estimates which we wish to study through the rest of this work. These formulas can also be found (in a possibly different form) in \cite{Kla} and \cite{CM1}. 
Recall \eqref{defn:PoincareIneq_intro} and \eqref{defn:LSIneq_intro} for the definitions of the  function spaces $\F_{Poi}(M,\mu)$ and $\F_{LS}(M,\mu)$, and to \eqref{dfn:Constants2} and \eqref{dfn:Constants3} for the definitions of $\Lambda_{Poi}^{(p)}(M,\gfrak,\mu)$ and $\Lambda_{LS}(M,\gfrak,\mu)$.

\begin{thm}[A $p$-\Poinc inequality]\label{thm:PoinInequality}[B.Klartag \cite{Kla}]
Assume $(M^n,\gfrak,\mu)$ is a \cwrm{}, which satisfies $CDD_b(K,N,D)$ with $K\in \R$, $N\in (-\infty,1)\cup [n,\infty]$ and $D\in (0,\infty]$. Then
\eql{\label{eqn:spectralGap} \Lambda_{Poi}^{(p)}(M,\gfrak,\mu)\geq \lamp_{K,N,D},}
where 
\eql{ \label{PoinConstant} \gls{lambda_p_knd}:=\inf_{0\neq \xi\in \Fknd^{C^{\infty}}}\inf \left\{ \frac{\int |f'(t)|^pd\xi(t)}{\int |f(t)|^pd\xi(t)}:\,\, f\in \F_{Poi}^{(p)}(\xi) \right\} \,.} 
%with $h_f(t):=f(t)|f(t)|^{p-2}$.
\end{thm} 
For $p=2$ we abbreviate and identify $\lam_{K,N,D}:=\lampc_{K,N,D}$. 

\begin{thm}[A Log-Sobolev inequality]\label{thm:LogSobolevInequality}
Assume $(M^n,\gfrak,\mu)$ is a \cwrm{}, which satisfies $CDD_b(K,N,D)$ with $K\in \R$, $N\in (-\infty,1)\cup [n,\infty]$ and $D\in (0,\infty]$. Then
\eql{\label{eqn:LSIneq} \Lambda_{LS}(M,\gfrak,\mu)\geq \rho_{K,N,D},}
where 
\eql{ \label{LogSobolevConstant} \gls{rho_knd}:=\inf_{0\neq \xi\in \Fknd^{C^{\infty}}}\inf\left\{ \frac{2\int f'(t)^2d\xi(t)}{\int f(t)^2\log\prnt{f(t)^2}d\xi(t)}:\,\, f\in \F_{LS}(\xi)\right\} \,.} 
%with $h_f(t):=(f(t)^2-1)$. 
\end{thm}
The two theorems are proved using the same approach, via Klartag's localization theorem. 
%The proofs are essentially straightforward applications of the localization theorem. By our definitions of w.r.m, the conditions of the theorem are satisfied; we condition that by apriori consideration of $M\setminus \partial M$ instead of $M$ whenever $\partial M\neq \emptyset$ ($\mu_{n-1}(\partial M)=0$). 
\begin{proof}[Proof of Theorem \ref{thm:PoinInequality}]
For the case $n=1$ the statement follows directly from the definition of $\lamp_{K,N,D}$. We proceed under the assumption that $n\geq 2$. 

Let $f\in \F^{(p)}_{Poi}(M,\mu)$. Define $$\tilde{h}_f(x):=f(x)|f(x)|^{p-2} \,.$$
Then by definition of $\F^{(p)}_{Poi}(M,\mu)$:
\begin{itemize}
	\item $\tilde{h}_f\in C_c(M)$, hence $\tilde{h}_f\in L^1(M; \mu)$ and moreover $\tilde{h}_f(\cdot)d(x_0, \cdot)\in L^1(M; \mu)$ for any $x_0\in M$.
	\item $\int_M \tilde{h}_f\,d\mu=0$.
\end{itemize}

Let us firstly assume $\partial M=\emptyset$. According to Theorem \ref{thm:localization}  $M=\Tc\bigcupdot \Zc$ and 
\begin{enumerate}
	\item $\Tc=\bigcup_{q\in Q} X_q$ where $\{X_q\}_{q\in Q}$ being the images of unit-speed length-minimizing geodesics $\g_q:I_q\to M$ (notice that this implies that $diam(I_q)\leq D$).  
	\item $\mu|_{\Tc}=\int_{q\in Q}\mu_q d\zeta(q)$ where $\mu_q=(\g_{q\sharp})J_q\cdot m$ is a $CDD_b(K,N,D)$-needle.
	\item $\tilde{h}_f=0$ for $[\mu]$-a.e. $x\in \Z$.
\end{enumerate} 
We set $f_q:=f(\g_q(t))$ and define $h_{f_q}(t):=\tilde{h}_f(\g_q(t))=f_q(t)|f_q(t)|^{p-2}$; evidently \\$\int h_{f_q}(t)J_q(t)dm(t)=0$, and by definition of $\lamp_{K,N,D}$ it follows that for all $q\in Q$ :
\[ \lamp_{K,N,D}\int|f_q(t)|^pJ_qdm(t) \leq \int |f'_q(t)|^pJ_qdm(t)\,. \]
Hence using the inequality $|f(\g_q(t))'|\stackrel{C.S.}{\leq} |\nabla f(\g_q(t))||\g_q'(t)|\stackrel{|\g_q'(t)|=1}{=}|\nabla f(\g_q(t))|$ we conclude for $p\in (1,\infty)$:
{\small 
\eql{\label{eqn:SpectralGapNeedle}
&\int_M |\nabla f|^p(x)d\mu(x)\geq \int_{\Tc} |\nabla f|^p(x)d\mu(x)=\int_{q\in Q} d\zeta(q) \int_{\Tc}|\nabla f|^p(x)d[\g_{q\sharp}(J_q\cdot m)]\\ \nonumber&=
\int_{q\in Q} d\zeta(q) \int_{I_q}|\nabla f|^p(\g(t))J_q(t)dm(t)\geq 
\int_{q\in Q} d\zeta(q) \int_{I_q}|f_q'(t)|^p J_q(t)dm(t) \\ \nonumber&
\geq \lamp_{K,N,D}\int_{q\in Q} d\zeta(q) \int_{I_q}|f_q(t)|^p J_q(t)dm(t)=
\lamp_{K,N,D}\int_{\Tc}|f(x)|^p d\mu(x)=\lamp_{K,N,D}\int_{M}|f(x)|^p d\mu(x)\,,
}}
where the last inequality is a consequence of $\tilde{h}_f=0$ for $[\mu]$ a.e. $x\in \Zc$. Therefore $\Lambda_{Poi}^{(p)}(M,g,\mu)\geq \lamp_{K,N,D}$. 

If $\partial M\neq \emptyset$ we get the same conclusions by applying the localization Theorem \ref{thm:localization} to $int(M)=M\setminus \partial M$, which by assumption is geodesically-convex; indeed, since $\partial M$ is a $[\mu]$ null-set, all integrals over $M$ in \eqref{eqn:SpectralGapNeedle} coincide with integrals over $int(M)$. 

\qedhere
\end{proof}

\begin{proof}[Proof of Theorem  \ref{thm:LogSobolevInequality}]
%The proof follows identically to the proof of \ref{thm:PoinInequality} under
Considering the proof for $\Lambda_{Poi}$, it is sufficient to prove the statement under the assumption that $n\geq 2$. 
Let $f\in \F_{LS}(M,\mu)$. Define $$\tilde{h}_f(x):=f^2(x)-1\,.$$
Then by definition of $\F_{LS}(M,\mu)$:
\begin{itemize}
	\item $\tilde{h}_f\in \Cinf_c(M)$, hence $\tilde{h}_f\in L^1(M; \mu)$ and moreover $\tilde{h}_f(\cdot)d(x_0, \cdot)\in L^1(M; \mu)$ for any $x_0\in M$.
	\item $\int_M \tilde{h}_f\,d\mu=0$.
\end{itemize}

The rest of the proof is identical to the proof of Theorem  \ref{thm:PoinInequality} subject to the replacement of $\lamp_{K,N,D}$ with $\rho_{K,N,D}$, $|\nab f|^p$ with $|\nab f|^2$, and $|f|^p$ with $f^2\log\prnt{f^2}$.

\end{proof}

 \section{The `synthetic' $CDD_b(K,N,D)$ condition on $\R$}

%In this section we formulate and prove equivalences of the $CDD(k,N,D)$ needle condition; the conclusions are trivial in view of the works of Sturm \cite{KTS1, KTS2, KTS3,KTSR} et al., who employed optimal transport methods on general metric-measure spaces. However for the sake of completeness, we provide 
%alternative proofs of these equivalences, which are rather rudimentary (following the methods of Borel \cite{Bo1, Bo2}) due to the fact that the underlying space is $\R$. This is an intermediate step towards the definition of the class $\MM_{(k,N,D)}$ which we have mentioned before. From this point on, unless explicitly specified otherwise, we will assume $N\in (-\infty, 0]\cup [1, \infty]$; this is the pertinent range for the validness of the most significant results.

In order to discuss about the `synthetic' definition mentioned in the title, we introduce additional definitions. 

\subsection{Distorted means and their properties}

Let $K\in \R$, $N$ and $\NN$ such that $\NN+1, N \in  (-\infty, 0]\cup [1, \infty)$. Following \cite{KTS2, KTS3}  we define the distortion coefficients $\sigma^{(t)}_{K,\NN}(\theta), \tau^{(t)}_{K,N}(\theta)$ for every $\theta\in \R_{+}$ and $t\in [0,1]$ by:
\eql{\label{dfn:SigmaTau} \gls{sigma_knd}:=\begin{cases}
%\infty & \mbox{ if }   \frac{K}{\NN}\theta^2\geq \pi^2  \\
%0 & \mbox{ if }   \frac{K}{\NN}\theta^2\geq \pi^2\text{ and }\NN<0  \\
 \infty & \mbox{ if } \frac{K}{\NN}\theta^2\geq\pi^2\,, \\
	\frac{\sin(t\theta\sqrt{\frac{K}{\NN}})}{\sin(\theta \sqrt{\frac{K}{\NN}})} & \mbox{ if } 0<\frac{K}{\NN}\theta^2<\pi^2\,, \\
	t & \mbox{ if } \frac{K}{\NN}\theta^2=0 \text{ or } \NN=0\,,\\
	\frac{\sinh(t\theta\sqrt{-\frac{K}{\NN}})}{\sinh(\theta \sqrt{-\frac{K}{\NN}})} & \mbox{ if } \frac{K}{\NN}\theta^2<0\,,
	\end{cases}
\qquad
\text{and   }\qquad  \gls{tau_knd}:=t^{\frac{1}{N}}\sigma_{K,N-1}^{(t)}(\theta)^{1-\frac{1}{N}}\,\,\,.
}
In addition for $t\in [0,1]$ we define the {\bf `distorted-means'}
$M_{K,\NN}^{(t)}[\cdot ](\cdot,\cdot),\, \tl{M}_{K,N}^{(t)}[\cdot ](\cdot,\cdot): \R_{+}\times \R_+^2\to \R_+\cup\{+\infty\}$ as follows: whenever $a\cdot b>0$ 
{\footnotesize 
\eql{ \gls{M_knd}=\begin{cases}
 \mbox{if $\frac{K}{\NN}d^2<\pi^2,\, \NN\neq 0$:} & \quad 
\begin{cases} 
\prnt{\sigma_{K,\NN}^{(1-t)}(d)a^{\frac{1}{\NN}}+\sigma_{K,\NN}^{(t)}(d)b^{\frac{1}{\NN}} }^{\NN}  &\quad\mbox{ if } \NN\in (-\infty,-1]\cup (0,\infty) \\ 
a^{1-t}b^te^{\frac{Kt(1-t)d^2}{2}} &\quad\mbox{ if } \NN=\infty\,, \\
\end{cases}\\
 \mbox{if $\frac{K}{\NN}d^2\geq\pi^2,\NN\neq 0$:} & \quad  \begin{cases} +\infty &\mbox{ if } K>0,\NN> 0\\ 0 &\mbox{ if } K<0,\NN\leq -1 \end{cases}\\
\mbox{if $\NN=0$} & \quad \max\prnt{a,b}\,.
\end{cases}
%\\
%\mbox{if $\frac{K}{\NN}d^2\geq\pi^2$:} & \qquad \begin{cases}  +\infty & \mbox{ if } $K>0$ \text{ and } $\NN>0$ \\  +\infty & \mbox{ if } $K>0$ \text{ and } $\NN<0$ \end{cases}
%\end{cases}
}
\eql{
 %\mbox{and} \quad
\gls{tlM_knd}:=\begin{cases} 
 \mbox{if $\frac{K}{N-1}d^2<\pi^2,\, N\neq 0$:} & \quad \begin{cases}
\prnt{\tau_{K,N}^{(1-t)}(d)a^{\frac{1}{N}}+\tau_{K,N}^{(t)}(d)b^{\frac{1}{N}} }^{N}
&\quad\mbox{ if } N\in (-\infty, 0)\cup [1,\infty) \\
a^{1-t}b^te^{\frac{Kt(1-t)d^2}{2}}  &\quad\mbox{ if } N=\infty\,,\\
%\max((1-t)(\sigma_{k,-1}(d)^{(1-t)})^{-1}a,t(\sigma_{k,-1}(d)^{(t)})^{-1}b)   &\quad\mbox{ if } N=0_+\, (\frac{1}{N}=+\infty)\\
\end{cases}\\
\mbox{if $\frac{K}{N-1}d^2\geq\pi^2,N\neq 0$:} & \quad  \begin{cases} +\infty &\mbox{ if } K>0,N\geq 1\\ 0 &\mbox{ if } K<0,N< 0 \end{cases}\\
\mbox{if $N=0$} & \quad \min\prnt{(1-t)(\sigma_{K,-1}^{(1-t)}(d))^{-1}a,t(\sigma_{K,-1}^{(t)}(d))^{-1}b} \,.
\end{cases}
 }  
 }
We set $M_{K,\NN}^{(t)}[d](a,b)=0$ (resp. $\tl{M}_{K,N}^{(t)}[d](a,b)=0$) by definition if $a\cdot b=0$ (for every value of $d$). 
Whenever $K=0$ these functions are independent of the parameter $d$; in this case we abbreviate $M_{0,\NN}^{(t)}(a,b)$ and $\tl{M}_{0,N}^{(t)}(a,b)$ .
\smallskip

\begin{remk} The marginal values (at $\NN,N=\infty$ or $\NN,N=0$) correspond to the limits of these functions when $\NN,N\to\infty$ or $\NN\to 0+$ and $N\to 0-$; the identity (\cite[Prop. 5.5]{BS}):
\eql{\label{eqn:sigma_taylor}\sigma_{K,\NN}^{(t)}(\theta)=t+\frac{K}{6}t(1-t^2)\prnt{\frac{\theta}{\sqrt{\NN}}}^2+O((\frac{\theta}{\sqrt{\NN}})^4)\,,\qquad \mbox{ for }\,\,\NN>0}
%and 
%\eql{\label{eqn:tau_taylor}
%\tau_{K,N}^{(t)}(\theta)=t\prnt{1+\frac{K}{6}(1-t^2)\prnt{\frac{\theta}{\sqrt{N-1}}}^2+O\prnt{\prnt{\frac{\theta}{\sqrt{N-1}}}^4}}^{\frac{N-1}{N}}\,,
%}
is useful for  the evaluation of these limits when $N,\NN\to \infty$.  In addition the cases $\frac{K}{\NN}d^2\geq\pi^2$ or $\NN=0$ follow from the definition of the functions $\sigma^{(t)}_{K,\NN}(\theta)$ (Definition  \ref{dfn:SigmaTau}).
\end{remk}
\begin{remk}\label{remk:lDelta}
Notice that for $N\neq 1$ the condition $\frac{K}{N-1}d^2\geq \pi^2$  can be satisfied only if $\delta>0$ and 
 $d\geq l_{\delta}$, where $\delta$ and $l_{\delta}$ are the functions of $K$ and $N$ which were defined in Definition  \ref{dfn:deltaSymbols}.
\end{remk} 
\begin{remk}[Relation to the class $\Fknd^{M}$]\label{rmk:FkndM_Diff_ID} Recall that for $N\neq 1$ the densities of $\xi\in \Fknd^M$ are positive solutions to the equation $-LogHess_{N-1}J= K$ (see  \eqref{eqn:CD(k,N):1}), which for  $N\neq \infty$ can be written as: 
$$(J^{\frac{1}{N-1}})''(x)+\frac{K}{N-1}J^{\frac{1}{N-1}}(x)=0\,,\qquad x_0<x<x_1\,.$$ 
If $x_0<x<x_1$ we may write $x=\convar{x_0}{x_1}$ for some $t\in (0,1)$, thus defining $d:=|x_1-x_0|$ we may rewrite the equation in the $t$ parametrization as  
$$(\tl{J}^{\frac{1}{N-1}})''(t)+\frac{Kd^2}{N-1}\tl{J}^{\frac{1}{N-1}}(t)=0\,,\qquad 0<t<1\,,$$
where $\tl{J}(t):=J(\convar{x_0}{x_1})$. One can verify directly that the solution $\tl{J}$ to the equation which satisfies $\tl{J}(0)=J(x_0)$ and  $\tl{J}(1)=J(x_0)$ is given by:
$$\tl{J}(t)=M_{K,N-1}^{(t)}[d](J(x_0),J(x_1))\,.$$
\end{remk}
%Mind that if $a\cdot b>0$ and $\frac{K}{N-1}d^2\geq \pi^2$ (i.e $\delta>0$ and $d\geq l_{\delta}$) then according to our conventions:
%\eq{
 %M_{K,N-1}^{(t)}[d](a,b)=\tl{M}_{K,N}^{(t)}[d](a,b)=\begin{cases}
%+\infty & \mbox{ if } N>1 \,\,(\,K>0)\\
 %0 & \mbox{ if } N\leq 0 \,\,(\,K<0)
%\end{cases}\,.
%}

%We remark that the functions $\tl{M}_{k,N}^{(t)}[\cdot ](\cdot,\cdot)$ arise naturally in the generalization of the Brunn-Minkowski inequality to w.r.m of class $CDD(k,N,D)$ (\cite{Vil} p.509). 

% For the first of these limits, the identity $\sigma{k,N}^{(t)}(\theta)=t+\frac{k}{6N}t(1-t^2)\theta^2+o(\theta^2)$ is useful. 

%Notice that $M_{0,N_1}[d_1](a_1,b_1)$ and $M_{k_1,0}[d_1](a_1,b_1)$ are independent of $d$; hence in these cases we may simply use the abbreviated notation $M_{0,N_1}(a_1,b_1)$ and $M_{k_1,0}(a_1,b_1)$. 
%We should also notice notice that whenever $k=0$ or $N=\infty$ we have the equality $M_{k,N}^{(t)}=\tilde{M}_{k,N}^{(t)}$.  

%In addition we define
%{\scriptsize 
%\eq{ &M_{k_1,k_2; N_1,N_2}^{(t)}[d_1,d_2](a,b)=\\&
%\begin{cases} \prnt{ \prnt{\sigma_{k_1,N_1}^{(1-t)}(d_1)^{N_1}\sigma_{k_2,N_2}^{(1-t)}(d_2)^{N_2}}^{\frac{1}{N_1+N_2}}a^{\frac{1}{N_1+N_2}} +\prnt{\sigma_{k_1,N_1}^{(t)}(d_1)^{N_1}\sigma_{k_2,N_2}^{(t)}(d_2)^{N_2}}^{\frac{1}{N_1+N_2}}b^{\frac{1}{N_1+N_2}}}^{N_1+N_2} & \mbox{ if }N_1,N_2\neq \infty\\
%a & N_1=\infty,\, N_2\neq \infty
%\end{cases}\,.
 %} }

\subsubsection{Important properties of the distorted means}

The following proposition is essential to the results presented in the next section. 
\begin{prop}[Elementary properties of the distorted means]\label{prop:properties} Assume $a_1,a_2,b_1,b_2>0$, $K\in \R$, $N\in (-\infty, 0]\cup (1,\infty]$, $t\in (0,1)$ and $d\in (0,l_{\delta})$. Then 
\begin{enumerate}
%	\item If $N_1\geq N_2$ then $M_{k, N_1}^{(t)}[d](a_1,b_1)\leq M_{k, N_2}^{(t)}[d](a_1,b_1)$. 
 \item \label{id:1} $M_{K, N-1}^{(t)}[d](a_1,b_1)M_{0, 1}^{(t)}(a_2,b_2)\geq  \tl{M}_{K,N}^{(t)}[d](a_1a_2,b_1b_2)$ with equality if and only if  
 
 $\prnt{\prnt{\frac{\sigma_{K,N-1}^{(1-t)}(d)}{1-t}}^{N}a_1^{\frac{N}{N-1}}, \prnt{\frac{\sigma_{K,N-1}^{(t)}(d)}{t}}^{N}b_1^{\frac{N}{N-1}}}$ is proportional to $(a_2^N,b_2^N)$.
 
 %$(a_1^{\frac{1}{N-1}}, b_1^{\frac{1}{N-1}})$ is proportional to $(a_2,b_2)$.
%\item \label{id:2} $M_{k_1,0; N_1-1,1}^{(t)}[d_1,\cdot ](a_1,b_1)=\tl{M}_{k_1,N_1}^{(t)}[d_1](a_1,b_1)$.
\item \label{id:2} For $K\geq 0$ (resp. $K\leq 0$) the function $\theta\mapsto \tl{M}^{(t)}_{K,N}[\theta](a_1,b_1)$ is non-decreasing (resp. non-increasing) on $[0, l_{\delta})$.
\item \label{id:3} The functions $r\mapsto \tl{M}^{(t)}_{K,N}[d](r,b_1)$ and $r\mapsto \tl{M}^{(t)}_{K,N}[d](a_1,r)$
are non-decreasing on $\R_+$. 
\item \label{id:4} \begin{enumerate}
  \item The function $\frac{1}{N}\mapsto M^{(t)}_{K,N-1}[d](a_1,b_1)$ is non-decreasing on $[-\infty,1)$.
	\item The function $\frac{1}{N}\mapsto \tl{M}^{(t)}_{K,N}[d](a_1,b_1)$ is non-decreasing on $[-\infty,1)$.
\end{enumerate}

\item \label{id:5} When $\delta=\frac{K}{N-1}>0$ and $d_0\in (0,+\infty)$:
\begin{enumerate}
	\item if $K>0, N\in(1,\infty)$: $\lim_{d\uparrow d_0}\tl{M}_{K,N}^{(t)}[d](a,b)=\tl{M}_{K,N}^{(t)}[d_0](a,b)$ ($\infty$ if $d_0\geq l_{\delta}$).
	\item if $K<0, N\in (-\infty,0]$: $\lim_{d\downarrow d_0}\tl{M}_{K,N}^{(t)}[d](a,b)=\tl{M}_{K,N}^{(t)}[d_0](a,b)$ ($0$ if $d_0\geq l_{\delta}$). 
\end{enumerate}

\end{enumerate}
\end{prop}

\begin{proof} 
\begin{enumerate}
\item  Let $p\in [-\infty,\infty]$. The classical weighted-means $\bM_p^{(t)}(a,b):\R_+^2\to\R_+$ are defined as follows: whenever $a,b>0$
\eql{\label{classicalWeightedMeans}
\bM_p^{(t)}(a,b)=\begin{cases} (\convar{a^p}{b^p})^{\frac{1}{p}}&\mbox{ if }\,\,p\in (-\infty, \infty]\setminus \{0\}\,,\\
																									 a^{1-t}b^t&\mbox{ if }\,\,p=0\,,\\
																									 \max\{a,b\}&\mbox{ if }\,\,p=+\infty\,,\\
																								  	\min\{a,b\}&\mbox{ if }\,\,p=-\infty\,,
									\end{cases}
}
and $\bM_p^{(t)}(a,b)=0$ by definition if $a\cdot b=0$.				

The theorem will follow as a direct consequence of the following lemma which we quote without proof (see \cite{Gar} p.20 Lemma 10.1 and also \cite{HaLi} p.24):
\begin{lem}\label{lem:Gardiner}  Assume $t\in (0,1)$ and $a_1,a_2,b_1,b_2,d\in \R_{+}$. If $p+q\geq 0$ 
$$\bM_p^{(t)}(a_1,b_1)\bM_q^{(t)}(a_2,b_2) \geq \bM_s^{(t)}(a_1a_2,b_1b_2) \,,\qquad \text{where}\qquad s=\begin{cases}\frac{pq}{p+q} & \mbox{ if }  |p|+|q|>0\\ 0 & \mbox{otherwise}\,.\end{cases} \,.$$
The inequality is strict unless $(a_1^{\frac{p}{s}},b_1^{\frac{p}{s}})$ and $(a_2^{\frac{q}{s}},b_2^{\frac{q}{s}})$ are proportional. 
\end{lem}
 
We will use this lemma in order to prove the assertion on the pertinent range $N\in (-\infty, 0]\cup (1,\infty]$. 
Define $p=\frac{1}{N-1}$ and $q=1$ (here we set $p=0$ if $N=\infty$). 
Notice that $p+q=\frac{N}{N-1}\geq 0$ since $N\in (-\infty, 0]\cup (1,\infty]$. With $s=\frac{1}{N}$ ($s=-\infty$ if $N=0$) we have by Lemma \ref{lem:Gardiner} the following inequality:
{\small 
\eq{ &M_{K,N-1}^{(t)}[d_1](a_1,b_1)M_{0,1}^{(t)}(a_2,b_2)\\&=\bM_{\frac{1}{N-1}}^{(t)}\prnt{((1-t)^{-1}\sigma_{K,N-1}^{(1-t)}(d_1))^{N-1}a_1, (t^{-1}\sigma_{K,N-1}^{(t)}(d_1))^{N-1}b_1}
\bM_{1}^{(t)}(a_2,b_2)\\&\geq 
\bM_{\frac{1}{N}}^{(t)}\prnt{((1-t)^{-1}\sigma_{K,N-1}^{(1-t)}(d_1))^{N-1}a_1a_2, (t^{-1}\sigma_{K,N-1}^{(t)}(d_1))^{N-1}b_1b_2}=\tl{M}_{K,N}^{(t)}[d_1](a_1a_2,b_1b_2) \,.}}
The equality case corresponds to 
$
\prnt{\prnt{\frac{\sigma_{K,N-1}^{(1-t)}(d_1)}{1-t}}^{N}a_1^{\frac{N}{N-1}}, \prnt{\frac{\sigma_{K,N-1}^{(t)}(d_1)}{t}}^{N}b_1^{\frac{N}{N-1}}}$ being proportional to $(a_2^N,b_2^N)$.

%When $N=\infty$ the statement follows directly from the AM-GM inequality.
%, and when $N=0$ it follows by quasi-concavity of the function $f(a, b)=\frac{a}{b}$ on $\R_+^{*2}$. 

%We introduce the functions 
%\[ s_{\delta}(x):=\begin{cases} \sin(\sqrt{\delta}t)/\sqrt{\delta} & \delta>0 \\
%x & \delta = 0 
%\\
%\sinh(\sqrt{-\delta}x)/\sqrt{-\delta} & \delta<0 \end{cases} \qquad
%c_{\delta}(x):=\begin{cases} \cos(\sqrt{\delta}x) & \delta>0 \\
%1 & \delta = 0 
%\\
%\cosh(\sqrt{-\delta}x) & \delta>0 \end{cases}
%\]

  \item[2.]	
	Throughout we assume $N<\infty$, since $M_{K,\infty}^{(t)}[d](a_1,b_1)=\lim_{N\to\infty}M_{K,N-1}^{(t)}[d](a_1,b_1)$, hence by continuity arguments, showing the statement for $N<\infty$ will imply it is valid for $\frac{1}{N}\in (-\infty,1)$.

	It was noted in Remark \ref{rmk:FkndM_Diff_ID} that the functions $[0,1]\ni t\mapsto M_{K,N-1}^{(t)}[d](a,b)$ 	can be identified as solutions $\tl{J}(t)$ to the equation:
	$$(\tl{J}^{\frac{1}{N-1}})''+\frac{Kd^2}{N-1}\tl{J}^{\frac{1}{N-1}}=0\qquad \tl{J}(0)=a,\quad \tl{J}(1)=b\,,$$ 
	which  can also be rephrased as 
	\eql{\label{tlJ:ODE} (\log \tl{J})''+\frac{1}{N-1}((\log \tl{J})')^2=-Kd^2\qquad \tl{J}(0)=a,\quad \tl{J}(1)=b\,. }

We will give a simple proof of (2) using \eqref{tlJ:ODE}. Assume $d_1>d_2$. We will also assume $d_1<l_{\delta}$, since otherwise the claim follows directly from the definition of the distorted means.
	Define $$\tl{J}_1(t)=M_{K,N-1}^{(t)}[d_1](a_1,b_1)\qquad \text{and}\qquad \tl{J}_2(t)=M_{K,N-1}^{(t)}[d_2](a_1,b_1)\,,$$ 
	and set $u_1(t):=\log\tl{J}_1(t)$ and $u_2(t):=\log\tl{J}_2(t)$.
	
	If $K>0$ then:
	\eq{ -Kd_1^2=u_1''+\frac{1}{N-1}(u_1')^2&\qquad u_1(0)=\log a_1,\quad u_1(1)=\log b_1\\
-Kd_1^2< -Kd_2^2=u_2''+\frac{1}{N-1}(u_2')^2&\qquad u_2(0)=\log a_1,\quad u_2(1)=\log b_1\,.}
Define $u(t):=u_2-u_1$, notice that it satisfies the equation
	\eql{\label{u:ODE} u''+\frac{1}{N-1}u'\prnt{u_2'+u_1'}> 0\qquad u(0)=u(1)=0\,.}
	Assume by contradiction that $u$ is positive at some point in $(0,1)$. Since $u(0)=u(1)$ there must be $t_*\in(0,1)$ which is a local maximum. However according to \eqref{u:ODE} since $t_*$ is a critical point: $u''(t_*)> 0$, hence there can be no such $t_*$, and we conclude that $u\leq 0$ on $[0,1]$, implying that $\tl{J}_1\geq \tl{J}_{2}$ on $[0,1]$. 
For $K<0$ it holds that $-Kd_1^2> -Kd_2^2$ and the implications are reversed, hence the same argument gives $\tl{J}_1\leq \tl{J}_{2}$ on $[0,1]$.

	%\eql{\label{eqn:CD(k,N):1}  -LogHess_{N-1}J\geq K\,,}
	%where
	%\[ LogHess_{N-1}J:=(\log J)''+\frac{1}{N-1}((\log J)')^2 \qquad \prnt{= (N-1)\frac{(J^{\frac{1}{N-1}})''}{J^{\frac{1}{N-1}}}  \text{ For } N\notin \{1, \infty\} } \]
	%

	\item[3,\,5.]  The statements follow directly from the definition of the function $\tl{M}_{K,N}^{(t)}[d](a,b)$. 

	\item[4.] Notice that $\frac{1}{N-1}=\frac{\frac{1}{N}}{1-\frac{1}{N}}$, hence for $\frac{1}{N}\in (-\infty, 1)\setminus\{0\}$: $\frac{1}{N_1-1}\geq \frac{1}{N_2-1}\Leftrightarrow 	\frac{1}{N_1}\geq \frac{1}{N_2}$. 
	\begin{enumerate}
	\item We prove the claim using arguments similar to the proof of (2).
	Define $$\tl{J}_1(t)=M_{K,N_1-1}^{(t)}[d](a_1,b_1)\qquad \text{and}\qquad \tl{J}_2(t)=M_{K,N_2-1}^{(t)}[d](a_1,b_1)\,.$$ 
	
	We will also assume $d<\min\{l_{\delta}(K,N_1),l_{\delta}(K,N_2)\}$, since otherwise, considering Remark  \ref{remk:obDiameter} (about the dependence of $l_{\delta}$ on the parameters) and the definition of the distortion means, the statement would  trivially hold. 
	
	Set $u_1(t):=\log\tl{J}_1(t)$ and $u_2(t):=\log\tl{J}_2(t)$.
		Assume $\frac{1}{N_1-1}>\frac{1}{N_2-1}$. Furthermore we will assume $N_1,N_2\neq \infty$ and with the same sign, considering that $M_{K,\infty}^{(t)}[d](a_1,b_1)=\lim_{N\to\infty}M_{K,N-1}^{(t)}[d](a_1,b_1)$, by continuity arguments it will follow that the statement holds for $\frac{1}{N}\in (-\infty,1)$. 
	
$u_1$ and $u_2$ satisfy the following ODEs on $[0,1]$:
\eql{\label{Equations_u_1} u_1''+\frac{1}{N_1-1}(u_1')^2=-Kd^2&\qquad u_1(0)=\log a_1,\quad u_1(1)=\log b_1\\ \nonumber
u_2''+\frac{1}{N_2-1}(u_2')^2=-Kd^2&\qquad u_2(0)=\log a_1,\quad u_2(1)=\log b_1\,.}
Define $u(t):=u_2-u_1$.  Assume by contradiction that $u(t)>0$ at some point of $(0,1)$. Hence there is a maximum point $t_*\in(0,1)$ at which $u(t_*)>0$.  By subtracting the equations we get the identities:
\eql{\label{Equations_u_2} 0&=u_2''-u_1''+\frac{1}{N_2-1}\prnt{(u_2')^2-(u_1')^2}-\prnt{\frac{1}{N_1-1}-\frac{1}{N_2-1}}(u_1')^2\\ \nonumber
0&=u_2''-u_1''+\frac{1}{N_1-1}\prnt{(u_2')^2-(u_1')^2}+\prnt{\frac{1}{N_2-1}-\frac{1}{N_1-1}}(u_2')^2\,.}
From the first equation we conclude that at $t_*$ (since $u_2'(t_*)^2-u_1'(t_*)^2=u'(t_*)(u_2'(t_*)+u_1'(t_*))=0$)
$$
u_2''(t_*)-u_1''(t_*)=\prnt{\frac{1}{N_1-1}-\frac{1}{N_2-1}}(u_1'(t_*))^2\geq 0\,,$$
where we used that $\frac{1}{N_1-1}> \frac{1}{N_2-1}$. 
Since $t_*$ is a local maximum by assumption, we conclude that  $u_2''(t_*)-u_1''(t_*)=0$, whence, considering equations \eqref{Equations_u_1} and \eqref{Equations_u_2}, $u_1'(t_*)=u_2'(t_*)=0$ and $u''_1(t_*)=u''_2(t_*)$ (=$-Kd^2$). Then if $\frac{1}{N_1-1}>\frac{1}{N_2-1}>0$ the functions $y_1:=u_1'$ and $y_2:=u_2'$ satisfy
\eq{y_1'&=F_1(t,y_1)  \qquad \text{where } F_1(t,w)=-\frac{1}{N_1-1}w^2-Kd^2
\\ y_2'&\geq F_1(t,y_2)\,, }
and if $\frac{1}{N_2-1}<\frac{1}{N_1-1}<0$ they satisfy
\eq{y_1'&\leq F_2(t,y_1)  \qquad \text{where } F_2(t,w)=-\frac{1}{N_2-1}w^2-Kd^2
\\ y_2'&= F_2(t,y_2)\,. }
By standard comparison of first order ODEs \cite[p.29]{BiRo} we conclude that $y_2\geq y_1$ on $(t_*,1]$, which amounts to $u'_2(t)\geq u'_1(t)$ on $(t_*,1]$. However since $u_2(t_*)>u_1(t_*)$, then $u_2(t)>u_1(t)$ on $(t_*,1]$, in contradiction to that $u_2(1)=u_1(1)$. 

		\item We will show that if $\frac{1}{N_2}\leq \frac{1}{N_1}$ then
$$ \tl{M}_{K, N_1}^{(t)}[d](a,b)\geq \tl{M}_{K, N_2}^{(t)}[d](a,b)\,.$$
\smallskip
%We refer the reader to \ref{dfn:deltaSymbols} for the definition of $\delta=\delta(K,N)$ and $l_{\delta}=l_{\delta}(K,N)$.
  
	Let us firstly assume that   $\max\{\frac{K}{N_1-1}d^2,\frac{K}{N_2-1}\}d^2<\pi^2$. Let $t\in (0,1)$, $a,b>0$,  and 
	 $\frac{1}{N_1}\geq \frac{1}{N_2}$ with $\frac{1}{N_1}\in (-\infty,1)$. 
	%assume that $\tl{M}_{K, N_1}^{(t)}[d](a,b)< \tl{M}_{K, N_2}^{(t)}[d](a,b)$ but $\frac{1}{N_2}< \frac{1}{N_1}$. 
	According to property \ref{id:1} of Proposition \ref{prop:properties} for $a_1,a_2,b_1,b_2\in \R_+$ it holds that
$$M_{K,N-1}^{(t)}[d](a_1,b_1)M_{0,1}^{(t)}(a_2,b_2)\geq  \tl{M}_{K,N}^{(t)}[d](a_1a_2,b_1b_2)\,,$$
 with equality if and only if
 \eql{\label{proportionality_condition} 
\prnt{\prnt{\frac{\sigma_{K,N-1}^{(1-t)}(d)}{1-t}}^{N}a_1^{\frac{N}{N-1}}, \prnt{\frac{\sigma_{K,N-1}^{(t)}(d)}{t}}^{N}b_1^{\frac{N}{N-1}}}\, \text{ is proportional to } \, (a_2^N,b_2^N)\,.}
 
 Given $a,b>0$ and fixed $t\in (0,1)$, we set:
\eq{ &a_1:=\Prnt{a\prnt{\frac{1-t}{\sigma_{K,N_1-1}^{(1-t)}(d)}  }  }  ^{\frac{N_1-1}{N_1}}\,,\qquad b_1:=\Prnt{a\prnt{\frac{t}{\sigma_{K,N_1-1}^{(t)}(d)}  }  }  ^{\frac{N_1-1}{N_1}}\,,\\&  a_2:=a_1^{\frac{1}{N_1-1}}\prnt{\frac{\sigma_{K,N_1-1}^{(1-t)}(d)}{1-t}}\,,\qquad 
b_2:=b_1^{\frac{1}{N_1-1}}\prnt{\frac{\sigma_{K,N_1-1}^{(t)}(d)}{t}} \,.}
 Notice that $a_1a_2=a$ and $b_1b_2=b$ and these numbers satisfy the proportionality condition \eqref{proportionality_condition} with $N=N_1$. Hence:
$$M_{K,N_1-1}^{(t)}[d](a_1,b_1)M_{0,1}^{(t)}(a_2,b_2)= \tl{M}_{K,N_1}^{(t)}[d](a,b)\,,$$
while 
$$M_{K,N_2-1}^{(t)}[d](a_1,b_1)M_{0,1}^{(t)}(a_2,b_2)> \tl{M}_{K,N_2}^{(t)}[d](a,b)\,.$$
Since $M_{K,N_1-1}^{(t)}[d](a_1,b_1)\geq M_{K,N_2-1}^{(t)}[d](a_1,b_1)$ whenever $\frac{1}{N_1}\geq \frac{1}{N_2}$ (as follows from part (a)) this implies that $\tl{M}_{K, N_1}^{(t)}[d](a,b)\geq  \tl{M}_{K, N_2}^{(t)}[d](a,b)$ as asserted.

We consider now the case   $\max\{\frac{K}{N_1-1},\frac{K}{N_2-1}\}d^2\geq \pi^2$, or equivalently (see Remark \ref{remk:lDelta}) $d\geq \min\{l_{\delta}(K,N_1),l_{\delta}(K,N_2)\}$. 
By Remark \ref{remk:obDiameter} for any fixed $K\in \R$  the function $\frac{1}{N}\to l_{\delta}(K,N)$ is non-increasing in $\frac{1}{N}$ if $K>0$, and  non-decreasing if $K<0$.
Since by assumption $\frac{1}{N_1}\geq \frac{1}{N_2}$: 
{\small 
\begin{itemize}
	\item If $K>0$ then $d\geq \min\{l_{\delta}(K,N_1),l_{\delta}(K,N_2)\}=l_{\delta}(K,N_1)$, whence $\tl{M}_{K, N_1}^{(t)}[d](a,b)=+\infty$.
	\item If $K<0$ then $d\geq\min\{l_{\delta}(K,N_1),l_{\delta}(K,N_2)\}=l_{\delta}(K,N_2)$, whence $\tl{M}_{K, N_2}^{(t)}[d](a,b)=0$.
\end{itemize} }
Thus in either case $\tl{M}_{K, N_1}^{(t)}[d](a,b)\geq \tl{M}_{K, N_2}^{(t)}[d](a,b)$.
	\end{enumerate}

\end{enumerate}
\end{proof}

\subsection{Equivalences of the $CDD_b(K,N,D)$ condition on $\R$ }

The following theorem can be considered as a `distorted' extension of the Pr\'{e}kopa-Leindler \cite{Pre1,Pre2, Lei}, Borell \cite{Bo1, Bo3} and Brascamp-Lieb \cite{BL} equivalences pertaining to $\frac{1}{N}$ (resp. $\frac{1}{N-1}$) concave measures (resp. densities) on $\R$, i.e. measures of class \red{$CD_b(0,N)$}. The reader is also referred to \cite{Bac} for a similar extension.  In the works of Sturm \cite{KTS1, KTS2, KTS3} and Lott-Villani \cite{LV}, an alternative, `synthetic', definition of the \red{$CD_b(K,N)$} condition was formulated. This `synthetic' definition requires no smoothness assumption, as it relies on concavity of certain entropy functional. The crucial observation is that while for weighted Riemannian manifolds this new definition turns out to be equivalent to Definition \ref{dfn:CDD} (which assumes a smooth structure), the synthetic definition applies also to general metric measure spaces. Similarly, for absolutely continuous measures on $\R$ with sufficiently smooth densities, the following theorem will show that the $CD_b(K,N)$ condition can be defined in several equivalent ways.  However one definition can be more general than another, if it applies to a larger class of measures. Under such circumstances we refer to the alternative definitions as `synthetic'. The next theorem encapsulates these equivalences. Proving it could have been omitted since most statements and their proofs can be found elsewhere (e.g \cite[Appendix A]{CaMil} or \cite{Bac,KTS3,KTSR}) in a wider generality; yet, since some of the equivalences were proved using the sophisticated optimal transport machinery which was developed by Sturm, Lott-Villani and others, it seems to be worth presenting a complete proof, based on very elementary arguments. 
\bigskip

%Throughout we denote by $\xi$ measures supported in $\R$. We denote by $\xi_J$ an absolutely continuous measure supported in $\R$ if $d\xi_J=Jdm$ for $J\in L^1_{loc}(\R)$. 

\begin{thm}\label{thm:Equivalence} Assume $K\in \R$ and $N\in (-\infty,0]\cup (1,\infty]$, and let $J\in L^1_{loc}(\R)$ be a non-negative function supported on an interval $I$  ($int(I)\neq \emptyset$). Define a measure $d\xi_J=J\,dm$. Then the following conditions are equivalent:
\begin{enumerate}
\item\label{property:III} For any two compact subsets $A_0,A_1\subset \R$ and $t\in [0,1]$
	\eql{\label{CDKN_3}
		\xi_J(A_t)\geq \tl{M}_{K, N}^{(t)}[\theta_{K}(A_0,A_1)](\xi_J(A_0),\xi_J(A_1))\,,
		}
		where 
		\eq{ \gls{theta_K}=\begin{cases}
			\inf_{x\in A_0, \, y\in A_1}|y-x| &\mbox{ if  } K\geq 0 \\
			\sup_{x\in A_0, \, y\in A_1}|y-x| &\mbox{ if  } K< 0\,.
			\end{cases}
		}
		\item\label{property:II} $\xi_{J}=\xi_{\tl{J}}$ where $\tl{J}\in L^1_{loc}(\R)$ and satisfies: 
	\eql{ \label{CDKN_2} \tl{J}(x_t)\geq M_{K, N-1}^{(t)}[|x_1-x_0|](\tl{J}(x_0),\tl{J}(x_1))\qquad \forall x_0,x_1\in \R,\quad \forall t\in [0,1]\,. }
\end{enumerate}
Moreover, if $J\in C^2(int(I))$, then any of the previous two conditions is also equivalent to the following condition:
		\\
	\smallskip   
\begin{enumerate}
\setcounter{enumi}{2}
	\item\label{property:I} The following differential inequality holds on $int(I)$: 
	\eql{\label{eqn:Diff:CD(k,N)} (-LogHess_{N-1}J:=)\qquad   &-(\log J)''-\frac{1}{N-1}((\log J)')^2\geq K\,,}
	with the usual interpretation when $N=\infty$.
	Equivalently for $N\neq \infty$:
	\eql{\label{eqn:Diff:CD(k,N)2}	(N-1)(J^{\frac{1}{N-1}})''\leq -K\,J^{\frac{1}{N-1}}\,. }
\end{enumerate}
\end{thm}
\bigskip

\begin{remk} Notice that by continuity of the addition operation $+$, compactness of $A_0\times A_1$ implies that $A_t$ is compact, in particular it is $[m]$ measurable. In addition one should notice that \ref{property:II}  (resp. \ref{property:III}) trivially holds if $J(x_0)J(x_1)=0$ (resp. $\xi_J(A_0)\xi_J(A_1)=0$) by definition of the distorted means. 
\end{remk}

\begin{remk}\label{rmk:ODEfacts}

The fundamental fact which underlies the equivalence of conditions 2 and 3 is that the functions $\bar{f}_0(t):=M_{K,N-1}^{(t)}[|x_1-x_0|](J(x_0),J(x_1))^{\frac{1}{N-1}}$ ($N\in (-\infty, 0]\cup (1,\infty)$) and $\bar{g}_0(t):=\log M_{K,\infty}^{(t)}[|x_1-x_0|](J(x_0),J(x_1))$ satsify the following ODEs on $(0,1)$: 
   \[ \bar{f}_0''(t)+\frac{K}{N-1}(x_1-x_0)^2 \tilde{f}_0(t)=0 \qquad \bar{f}_0(0)=J(x_0)^{\frac{1}{N-1}},\,\bar{f}_0(1)=J(x_1)^{\frac{1}{N-1}} \,,\]
	and 
	\[ \bar{g}_0''+ K(x_1-x_0)^2= 0 \qquad \bar{g}_0(0)=J(x_0),\,\bar{g}_0(1)=J(x_1) \,. \]
	These are fundamentally related to the model densities presented in Definition \ref{defn:Jknh}. This was briefly discussed in Remark \ref{rmk:FkndM_Diff_ID}, but for completeness we elaborate further about 
	
	these types of ODEs from a general standpoint.  
	
	%respectively . 
	%&\Leftrightarrow  J(x_t)\geq M_{K,N-1}^{(t)}[|x_1-x_0|](J(x_0),J(x_1)) \qquad \forall x_0,x_1\in I \text{ and }\forall t\in (0,1)\, }
	%(notice that by definition the direction $(1)\rightarrow (2)$ holds whenever $\delta>0$, $K<0$ and $|x_1-x_0|\geq l_{\delta}$ ). 
	%%If $\delta>0$ and $K<0$  then $(1)\Leftrightarrow (2)$ holds trivially, while if $K>0$ 
	%
	%while if $N=\infty$
%\eq{
   %&
	%(\log J)(x_t)\geq \convar{(\log J)(x_0)}{(\log J)(x_1)}+\frac{K t(1-t)}{2}(x_0-x_1)^2\qquad \forall x_0,x_1\in I \text{ and }\forall t\in (0,1)\,,}

Assume $x_0,x_1\in \R$, $\beta\in \R$ are s.t. $(x_1-x_0)^2\beta\leq\pi^2$, and let $f_0,g_0\in C^2([x_0,x_1],\R_+)$ be non-negative solutions to the equations $f_0''(x)+\beta f_0(x)=0$ and $g_0''(x)+\beta=0$. Assume $y_0,y_1\in [x_0,x_1]$, and write $y_t=\convar{y_0}{y_1}$. The $t$-parametrized solutions $\bar{f}_0(t):=f_0(y_t)$ and $\bar{g}_0(t):=g_0(y_t)$ solve the boundary value problems 
\eql{ \label{ModelEquationsODEs}
\bar{f}_0(t)''+(y_1-y_0)^2\beta \bar{f}_0(t)=0& \qquad \bar{f}_0(0)=f_0(y_0),\,\,\bar{f}_0(1)=f_0(y_1)\\ \nonumber
\bar{g}_0(t)''+(y_1-y_0)^2\beta =0& \qquad \bar{g}_0(0)=g_0(y_0),\,\,\bar{g}_0(1)=g_0(y_1)\,, }
and vice-versa, any of these boundary value problems determine solutions $\bar{f}_0(t)=\sigma^{(1-t)}_{\beta,1}(|y_1-y_0|)f(y_0) +\sigma^{(t)}_{\beta,1}(|y_1-y_0|)f(y_1)$ and $\bar{g}_0(t)=\convar{g(y_0)}{g(y_1)}+\frac{\beta t(1-t)}{2}(y_1-y_0)^2$; these after changing the parametrization from $t$ to $x$, correspond to the same solutions $f_0(x)$ and $g_0(x)$. 

\bigskip

Under the same assumptions, when one turns from equations to inequalities, if $f,g\in C^2([0,1], \R_{+})$ then the following equivalence holds: if $(x_1-x_0)^2\beta<\pi^2$ then for all $y_0,y_1\in [x_0,x_1]$:
	 \eq{&f''(t)+(y_1-y_0)^2\beta f(t)\begin{cases} \leq 0 & \\ \geq 0 & \end{cases} \text{ on } [0, 1] \quad \Leftrightarrow \quad f(y_t)\begin{cases} \geq \bar{f}_0(t) & \\ \leq \bar{f}_0(t) & \end{cases} \forall t\in [0,1]
	\\ & \quad \text{where} \quad \bar{f}_0(t)=\sigma^{(1-t)}_{\beta,1}(|y_0-y_1|)f(y_0)+\sigma_{\beta,1}^{(t)}(|y_0-y_1|)f(y_1)\,.}
	Similarly
	 \eq{&g''(x)+(y_1-y_0)^2\beta \begin{cases} \leq 0 & \\ \geq 0 & \end{cases} 0 \text{ on } (x_0, x_1)\quad \Leftrightarrow \quad g(y_t)\begin{cases} \geq \bar{g}_0(t) & \\ \leq \bar{g}_0(t) & \end{cases} \qquad \forall t\in [0,1]\\& \quad \text{where} \quad \bar{g}_0(t)=\convar{g(y_0)}{g(y_1)}+\frac{\beta t(1-t)}{2}(y_0-y_1)^2 \,.}
	For a proof of the $''\leq''\Leftrightarrow ''\geq ''$ inequalities the reader is referred to \cite[p. 409,398]{Vil}. The second implication $''\geq''\Leftrightarrow ''\leq ''$ can be proved by similar arguments; we refer the reader to \cite{Oht1}.
	
	Furthermore in the first equivalence ($''\leq''\Leftrightarrow ''\geq ''$) $f=0$ if $(x_1-x_0)^2\beta>\pi^2$ and $f(t)=c\sin(\pi t)$ (for some constant $c\geq 0$) if $(x_1-x_0)^2\beta=\pi^2$ (see \cite[p. 409]{Vil}). 
 The proofs of these inequalities rely on comparison with the solutions to \eqref{ModelEquationsODEs}. These elementary ODE facts can be useful in understanding the local geometry of densities which satisfy \eqref{eqn:Diff:CD(k,N)2} and will also be essential for the proof of Theorem \ref{thm:Equivalence} (\ref{property:II}$\Leftrightarrow$\ref{property:I}). 

\end{remk}

%\begin{prop}[Properties of the distorted means II]\label{prop:properties2} Assume $a_1,a_2,b_1,b_2$, $k\in \R$, $N\in (-\infty, 0]\cup [1,\infty]$, $t\in (0,1)$ and $d\in (0,l_{\delta})$. Then 
%\begin{enumerate}
%\item \label{id:4} The function $r\mapsto \tl{M}^{(t)}_{k,r}[d](a_1,b_1)^r$ is non-increasing on $[N,\infty)$ for $N\in [1,\infty)$; it is non-decreasing on $[N,0]$ for $N\in (-\infty, 0]$.
%\end{enumerate}
%\end{prop}
%

\bigskip
 %are solved respectively by
%\[M_{\beta,1}^{(t)}[|x_1-x_0|](f(x_0),f(x_1))=\sigma^{(1-t)}_{\beta,1}(|x_0-x_1|)f(x_0)+\sigma_{\beta,1}^{(t)}(|x_0-x_1|)f(x_1)\]
%and 
%\[M_{\beta,\infty}^{(t)}[|x_1-x_0|](f(x_0),f(x_1))=\convar{f(x_0)}{f(x_1)}+\frac{\beta t(1-t)}{2}(x_0-x_1)^2\]
%
%
%Given  , 

%\begin{remk}\label{Model_Smooth} It follows from the equivalence of $(1)$ and $(2)$ that if $J(x)\in C^2(I)$ satisfies $(2)$ such that \eqref{CDKN_2} is an equality, then $d\rho=Jdm \in CDD_M(k,N,D)$, i.e it is $C^{\infty}(I)$, since the solutions of \eqref{eqn:Diff:CD(k,N)} stated as an equality, are smooth; explicit expressions for these solution will be given 
%in the sequel.
%\end{remk}
%\begin{remk}\label{remk:equivalences}
%From the proof it will follow the equivalence of $(2)$ and $(3)$ even without the continuity assumption; more precisely, $(2)$ holds $[m]$ a.e on its support if and only if $(3)$ holds on its support. This observation motivates the synthetic definition which will be presented in \ref{subsec:CDDKN}. 
%\end{remk}
%\begin{remk} If a function $J$ satisfies $(2)$ on its support, then $J$ must be continuous; continuity is a a local property, and within infinitesimally small neighborhoods we can consider $J^{\frac{1}{N-1}}$ as (almost) concave (if $N-1>0$) or convex (if $N-1<0$). 
%\end{remk}

%The theorem will follow as a direct consequence of the following lemma (c.f C.Borell \cite{Bo1, Bo2}):
%############
A useful and important consequence is the following lemma:
\begin{lem}\label{lem:CondImplic} Under the same assumptions of Theorem \ref{thm:Equivalence}, if $I:=supp(\xi_J)$ then
%, let  be an interval.
any of the three conditions  \eqref{CDKN_3}, \eqref{CDKN_2} or \eqref{eqn:Diff:CD(k,N)} (henceforth referred to as conditions \ref{property:III},\ref{property:II} and \ref{property:I} respectively) of Theorem  \ref{thm:Equivalence}  implies independently  that when $\delta>0$ and $K>0$ then
 $$diam(I)\leq\frac{\pi}{\sqrt{\delta}}=l_{\delta}\,.$$ 
Moreover if $\{x_0',x_1'\}=\partial I$, then from condition \ref{property:II} it follows that $\tl{J}(x_0')=\tl{J}(x_1')=0$, and from condition 
 \ref{property:I} it follows that  $J$ can be extended to $J\in C(I)$ s.t. $J(x_0')=J(x_1')=0$.
\end{lem}
\begin{proof} We discuss separately how this follows from   the conditions of Theorem \ref{thm:Equivalence}:
\begin{itemize}
	\item {\bf Condition \ref{property:III}}: If we had  $diam(I)>l_{\delta}$, then we can pick two compact sets $I_0,I_1\subset I$ s.t. $\xi_J(I_0)\xi_J(I_1)>0$ and 
	$\theta_K(I_0,I_1)>l_{\delta}$, then by assumption
	\[\xi_J(\convar{I_0}{I_1})\geq \tl{M}_{K, N}^{(t)}[\theta_{K}(I_0,I_1)](\xi_J(I_0),\xi_J(I_1))=\infty\,,\]
		and this contradicts the assumption that $J\in L^1_{loc}(\R)$.
		
		\item {\bf Condition \ref{property:II}}: Since $\xi_J=\xi_{\tl{J}}$ we justify the claim for $\tl{J}$.   If $diam(I)>l_{\delta}$ then there are points 
		 $x_0,x_1\in I$ s.t. $|x_1-x_0|\geq l_{\delta}$ and $\tl{J}(x_0)\tl{J}(x_1)>0$,  whence $\tl{J}(x_t)\geq M_{K,N-1}^{(t)}[|x_1-x_0|](\tl{J}(x_0),\tl{J}(x_1))=\infty$, which contradicts our assumption that $\tl{J}\in  L^1_{loc}(\R)$. We conclude that $diam(I)\leq l_{\delta}$. The same argument shows that if $diam(I)=|x_1'-x_0'|= l_{\delta}$ then $\tl{J}(x_0')\tl{J}(x_1')=0$. However we can further conclude that $\tl{J}(x_0')=\tl{J}(x_1')=0$. Indeed, if, say, $\tl{J}(x_1')>0$, let $x\in (x_0',x_1')$, then $x=\convar{x_0'}{x_1'}$ for some $t\in (0,1)$. 
		 Let $x_0^{(n)}\to x_0'$ be a sequence, and assume w.l.o.g. $x\in (x_0^{(n)},x_1')$ for all $n\in\Nbb$. Let $t_n\in (0,1)$ be defined so that $x=(1-t_n)x_0^{(n)}+t_nx_1'$. Then by assumption
		 {\small 
		 $$\tl{J}(x)=\tl{J}((1-t_n)x_0^{(n)}+t_nx_1')\geq M_{K,N-1}^{(t_n)}[|x_1'-x_0^{(n)}|](\tl{J}(x_0^{(n)}),\tl{J}(x_1'))\geq \prnt{\sigma_{K,N-1}^{(t_n)}(|x_1'-x_0^{(n)}|)}^{N-1}\tl{J}(x_1') \,.$$}
		 Since $t_n\to t$, and $|x_1'-x_0^{(n)}|\to \frac{\pi}{\sqrt{\delta}}$ and $N>1$ by definition $\sigma_{K,N-1}^{(t_n)}(|x_1'-x_0^{(n)}|)^{N-1}\to\infty$, hence the RHS diverges; since $x$ was an arbitrary point this contradicts $\tl{J}\in L^1_{loc}(I)$. 

	\item {\bf Condition \ref{property:I}}: Assume by contradiction that
	 $diam(I)>l_{\delta}=\frac{\pi}{\sqrt{\delta}}$. 

	The function $F(x):=J^{\frac{1}{N-1}}(x)$ satisfies the differential inequality
	$$F''+\delta \,F\leq 0\,,\quad \text{ on } int(I)\,.$$
	We can then pick $x_0,x_1\in int(I)$  s.t. $x_1-x_0>l_{\delta}$ and $F\in C^2([x_0,x_1])$. Any $x\in(x_0,x_1)$ can be expressed as $x_t=\convar{x_0}{x_1}$ with $t\in (0,1)$; then in this parametrization  $f(t):=F(x_t)$ satisfies the inequality
	$$f''+(x_1-x_0)^2\delta \,f\leq 0\,,\qquad \text{ on } (0,1)\,.$$
	Since $(x_1-x_0)^2\delta>\pi^2$, according to  Remark \ref{rmk:ODEfacts}, this implies that $f\equiv 0$ on $[0,1]$, or equivalently that $F\equiv 0$ on $[x_0,x_1]$; this contradicts the assumption $x_0,x_1\in int(I)$ (which is an interval by assumption). \smallskip
	
	We consider now the case $diam(I)=|x_1'-x_0'|=\frac{\pi}{\sqrt{\delta}}$.
	For $J\in C(I)\cap C^2(int(I))$ which extends $J$ we have $J(x_0')=J(x_1')=0$  due to Remark \ref{rmk:ODEfacts}. In order to define such a continuous extension, we notice that by the above $F(x):=J^{\frac{1}{N-1}}(x)$ satisfies
	$F''\leq -\delta \,F\leq 0\,\text{ on } int(I)\,,$
	hence $F$ is concave on $int(I)$. 
	Any concave function $F:int(I)\to \R$ can be extended to a continuous concave function $F:I\to \R\cup\{-\infty\}$; taking into account that $F$ is non-negative on $int(I)$ and hence on $I$, it follows that the same holds for $J=F^{N-1}$. Therefore $J$ can be continuously extended to the entire $I$ concluding the proof. 
\end{itemize}	
\end{proof}

We will now begin to prove several important results which will eventually lead to the proof of Theorem \ref{thm:Equivalence}. 
\bigskip

The following lemma can be considered as a 1-dimensional distorted version of Borell's lemma \cite{Bo1, Bo2} (see also \cite{Gar}):
\begin{lem}\label{lem:Borell} Let $K\in \R$, $N\in (-\infty, 0]\cup (1,\infty]$, and assume $0\leq f\in L^1(\R)$. Define $d\xi_f = fdm$ and assume $\xi_f$ is supported on an interval $I$. Then the following two statements are equivalent
\begin{enumerate}
	\item $\xi_f=\xi_{\tl{f}}$ where $\tl{f}\in L^1_{loc}(\R)$ satisfies
	$$\tl{f}(x_t)a_t\geq \tl{M}_{K,N}^{(t)}[|x_1-x_0|](\tl{f}(x_0)a_0,\tl{f}(x_1)a_1)\qquad \forall x_0,x_1\in \R,\,\forall t\in [0,1],\, \forall a_0,a_1\in \R_+\,.$$
	\item
	$\xi_f(A_t)\geq \tl{M}_{K,N}^{(t)}[\theta_{K}(A_0, A_1)] \prnt{\xi_f(A_0), \xi_f(A_1) } $ for all compact sets $A_0,A_1\subset \R$ and every  $t\in [0,1]$.
\end{enumerate}
\end{lem}

\begin{proof}[Proof]
$(1\Rightarrow 2)$ Assume $A_0, A_1\subset \R$ are two compact sets; we will assume $int(A_0),int(A_1)\neq \emptyset$ because otherwise $(2)$ trivially holds. 
Assume first that $\theta_K(A_0,A_1)< l_{\delta}$. Fix $t\in (0,1)$. For $i\in \{0,1,t\}$ we define the normalized densities $\tl{f}_i:=\xi_{\tl{f}}(A_i)^{-1}\tl{f} \,1_{A_i}$. Then it is enough to show
{\small 
\eq{\tl{f}_t(x_t)a_t&\geq \tl{N}_{K,N}^{(t)}[|x_1-x_0|](\tl{f}_0(x_0)a_0,\tl{f}_1(x_1)a_1)\,,\quad \forall\,\,x_0,x_1\in\R,\,
	\text{ and } \forall a_0,a_1> 0 \\&\Rightarrow 1\geq \tl{N}_{K,N}^{(t)}[\theta_{K}(A_0,A_1)](1, 1)\,,
} }
where
\[ \tl{N}_{K,N}^{(t)}[r](\xi_0, \xi_1):=\xi_{\tl{f}}(A_t)^{-1} \tl{M}_{K,N}^{(t)}[r]\prnt{\xi_{\tl{f}}(A_0)\xi_0,\xi_{\tl{f}}(A_1)\xi_1  }\,.\]

For $i\in \{0,1\}$ we define $u_i:[0,1]\to A_i$ by $u_i(s):=\inf\{w\in A_i:\, F_i(w)=s\}$ where $F_i(w):=\int_{-\infty}^w\tl{f}_i(x)dx$. The functions $u_i(s)$ are strictly increasing (possibly discontinuous)  hence differentiable a.e. Thus 
for a.e. $s\in [0,1]$ it holds that $(\tl{f}_i\circ u_i)(s)u_i'(s)=1$.  Set $w(s)=\convar{u_0(s)}{u_1(s)}$. 
Then by hypothesis for a.e. $s\in [0,1]$:
\eql{\label{ineq:f_tw_s} \tl{f}_t(w(s))w'(s)&\geq \tl{N}_{K,N}^{(t)}[|u_1(s)-u_0(s)|](\tl{f}_0(u_0(s))u_0'(s),\tl{f}_1(u_1(s))u_1'(s))\\ \nonumber &=\tl{N}_{K,N}^{(t)}[|u_1(s)-u_0(s)|](1,1)\,.}

Being a composition of the two monotonically increasing functions $F_t$ and $w$, the function $G=(F_t\circ w)(s)$ is monotonically increasing in $s$ and differentiable a.e., with derivative $G'(s)=\tl{f}_t(w(s))w'(s)$; moreover (see \cite[p.96]{Roy}) $G'$ is a measurable function which satisfies the inequality:
$$ \prnt{\int_{0}^{1}\tl{f}_t(w(s))w'(s)ds=\,}\quad \int_{0}^{1}G'(s)ds\leq G(1)-G(0)\quad \prnt{ \,=\int_{w(0)}^{w(1)}\tl{f}_t(y)dy}\,.$$

Then by integrating over the inequality \eqref{ineq:f_tw_s} we conclude that
\eq{1&=\int \tl{f}_t(y)dy\geq \int_{w(0)}^{w(1)} \tl{f}_t(y)dy\geq \int_0^1\tl{f}_t(w(s))w'(s)ds\geq \int_0^1 \tl{N}_{K,N}^{(t)}[|u_1(s)-u_0(s)|](1,1)ds
\\&\geq \tl{N}_{K,N}^{(t)}[\theta_K(A_0, A_1)](1,1)\,.}

We now verify the statement when $\theta_K(A_0,A_1)\geq l_{\delta}$ (which is possible only when $\delta>0$). If $K<0$ there is nothing to prove since $\tl{M}_{K,N}^{(t)}[\theta_K(A_0,A_1)](\cdot\,,\, \cdot)$ vanishes by definition; if however $K>0$ and if $\xi_{\tl{f}}(A_0)\xi_{\tl{f}}(A_1)>0$ 
then, since $\theta_K(A_0,A_1)=d(A_0,A_1)$, there are points $x_0, x_1\in \{\tl{f}>0\}$ s.t. $|x_1-x_0|> l_{\delta}$. 
By hypothesis for every $t\in [0,1]$ and $a_0=a_1=1$:
\eql{\label{eqn:infinity_delta_pos}\tl{f}(x_t)\geq\tl{M}_{K,N}^{(t)}[|x_1-x_0|](\tl{f}(x_0),\tl{f}(x_1))=+\infty\,,} 
contradicting that $f\in L^1(\R)$. 
Therefore when $K>0$ for any compact sets $A_0,A_1\subset \R$ s.t. $\theta_K(A_0,A_1)\geq l_{\delta}$ it holds that $\xi_{\tl{f}}(A_0)\xi_{\tl{f}}(A_1)=0$ and then $(2)$ holds by definition.
\smallskip
%there are points $x_0,x_1\in supp(\xi_f)$ such that $|x_1-x_0|>l_{\delta}$; but then by condition $(1)$: $f(x_t)\geq\tl{M}_{K,N}^{(t)}[|x_1-x_0|](f(x_0)a_0,f(x_1)a_1)=\infty$, in contradiction to the continuity of $f$.    
%

$(2\Rightarrow 1)$ The claim will follow from the following Lemma:
\begin{lem}\label{lem:Borell} Let $I$ be a closed interval and assume $0\leq f\in L^1_{loc}(I)$. Assume $I'\subset int(I)$ is s.t. $m(I\setminus I')=0$. Assume $f$ satisfies the inequality 
\eql{\label{LebesguePointsCond} f(x_t)a_t\geq \tl{M}_{K,N}^{(t)}(f(x_0)a_0,f(x_1)a_1)\qquad \forall a_0,a_1\geq 0 \,\text{ and }\forall x_0,x_1\in I', \,\,\forall t\in[0,1] \mbox{ s.t. } x_t\in I' \,.}
Then there is a function $\tl{f}\in L^1_{loc}(I)$ s.t.
\begin{itemize}
	\item $\xi_f=\xi_{\tl{f}}$\,.
	\item $\tl{f}$ satisfies the inequality 
	\eql{\label{LebesguePointsCond2} \tl{f}(x_t)a_t\geq \tl{M}_{K,N}^{(t)}(\tl{f}(x_0)a_0,\tl{f}(x_1)a_1)\qquad \forall a_0,a_1\geq 0 \,\text{ and }\forall x_0,x_1\in I, \,\,t\in[0,1]  \,.}
\end{itemize} 
\end{lem}
\begin{proof}[Proof of Lemma \ref{lem:Borell}]
%Throughout the proof we use the convention that points with a " '  ", e.g $y'$, belong to $I'$.

Define $\tl{f}$ as follows:
$$ \tl{f}(x):=\begin{cases} f(x) &\mbox{if } x\in I'\\
												 \liminf_{I'\ni y'\to x}f(y') &\mbox{otherwise } \,. \end{cases} $$
Notice that $\tl{f}$ is well defined on $I$, as $I'$ is dense in $I$, and $\tl{f}\in L^1_{loc}(I)$.  

Let $x_0, x_1\in I$ s.t. $\tl{f}(x_0)\tl{f}(x_1)>0$ (if $\tl{f}(x_0)\tl{f}(x_1)=0$ then 
$\tl{M}_{K,N}^{(t)}(\tl{f}(x_0)a_0,\tl{f}(x_0)a_1)= 0$ and there is nothing to prove) and let $a_0,a_1>0$.
Define $I_0:=[x_0,x_1]$ and set $I_0':=I_0\cap I'$. 
Let us firstly assume that $|x_1-x_0|<l_{\delta}$.
\smallskip

Given fixed $t\in (0,1)$ we define $x:=x_t=\convar{x_0}{x_1}\in int(I_0)$. 
%If $x_0,x_1,x\in I'$, then there is nothing to prove, since \eqref{LebesguePointsCond2} amounts to  \eqref{LebesguePointsCond}. 
%Therefore we may assume $(x_0,x_1)\in (I_0\times  I_0)\setminus (I_0'\times I_0')$. 

Denote by $\phi:[0,1]\times I_0'\times I_0'\to I_0$ the function
$\phi(s,y^{\prime}_0,y^{\prime}_1)=\convars{y^{\prime}_0}{y^{\prime}_1}$. 
%We also define $\phi_{x_0}:=\phi(t,x_0, y_1)$, and $\phi_{x_0}:=\phi(t,y_0, x_1)$. 
Given $x\in int(I_0)$ define
$$E_{x_0,x,x_1}:=\begin{cases}\phi^{-1}(x) &\mbox{(i) if } (x_0,x_1)
\in (I_0\setminus I_0')\times (I_0\setminus I_0')\\
\phi^{-1}(x)\cap \prnt{[0,1]\times\{x_0\}\times I_0'} &\mbox{(ii) if } (x_0,x_1)\in  I_0' \times (I_0\setminus I_0') \\
\phi^{-1}(x)\cap \prnt{[0,1]\times I_0'\times\{x_1\}} &\mbox{(iii) if } (x_0,x_1)\in (I_0\setminus I_0')\times I_0'\\
\phi^{-1}(x)\cap \prnt{[0,1]\times\{x_0\}\times \{x_1\}}
&\mbox{(iv) if } (x_0,x_1)\in  I_0' \times I_0'\,.
\end{cases}$$

By definition of $E_{x_0,x,x_1}$ for every  $(s,y^{\prime}_{0},y^{\prime}_{1})\in E_{x_0,x,x_1}$: $\convars{y^{\prime}_{0}}{y^{\prime}_{1}}=x$. 

Let us firstly assume that $x\in I_0'$. Then in case (iv) inequality \eqref{LebesguePointsCond2} follows directly from \eqref{LebesguePointsCond} and the definition of $\tl{f}$. It remains to verify \eqref{LebesguePointsCond2} for cases (i)-(iii). By definition of $E_x$:
\eql{\label{eq:E_x_seq}
\tl{f}(x)=f(x)=f((1-t_n)y^{\prime(n)}_{0}+t_ny^{\prime(n)}_{1}) \qquad \text{for all } (t_n,y^{\prime(n)}_{0},y^{\prime(n)}_{1})\in E_x\,.}
Therefore, considering continuity of $\tl{M}_{K,N}^{(t)}[d](a,b)$ in the variables on the domain $t\in [0,1]$, $d\in (0,l_{\delta})$ and $a,b>0$, on which it is also monotone in $a,b$ separately (see Proposition \ref{prop:properties} property \ref{id:3}): {\small 
\eq{\tl{f}(x)
&\stackrel{\substack{\eqref{eq:E_x_seq}\\\eqref{LebesguePointsCond}}}{\geq} \limsup_{E_x\ni (t_n,y^{\prime(n)}_{0},y^{\prime(n)}_{1})\to (t, x_0, x_1)}\tl{M}_{K,N}^{(t_n)}[|y_1^{\prime(n)}-y_0^{\prime (n)}|](f(y_0^{\prime(n)})\frac{a_0}{a_{t_n}},f(y_1^{\prime (n)})\frac{a_1}{a_{t_n}})\\&\geq 
\begin{cases} 
\tl{M}_{K,N}^{(t)}[|x_1-x_0|](\frac{a_0}{a_{t}}\liminf_{I'\ni y_0^{\prime(n)}\to x_0}f(y_0^{\prime(n)}),\frac{a_1}{a_{t}}\liminf_{I'\ni y_1^{\prime(n)}\to x_1}f(y_1^{\prime (n)}))&\mbox{in case (i)}\\
 \tl{M}_{K,N}^{(t)}[|x_1-x_0|](\frac{a_0}{a_{t}}f(x_0),\frac{a_1}{a_{t}}\liminf_{I'\ni y_1^{\prime(n)}\to x_1}f(y_1^{\prime (n)}))&\mbox{in case (ii)}\\
 \tl{M}_{K,N}^{(t)}[|x_1-x_0|](\frac{a_0}{a_{t}}\liminf_{I'\ni y_0^{\prime(n)}\to x_0}f(y_0^{\prime(n)}),\frac{a_1}{a_{t}}f(x_1))&\mbox{in case (iii)}
%\\
%=\tl{M}_{K,N}^{(t)}[|x_1-x_0|](\frac{a_0}{a_{t}}f(x_0),\frac{a_1}{a_{t}}f(x_1))&\mbox{in case (iv)}
\end{cases}
\\&=
\tl{M}_{K,N}^{(t)}[|x_1-x_0|](\tl{f}(x_0)\frac{a_0}{a_{t}},\tl{f}(x_1)\frac{a_1}{a_{t}})\,.
}}

Assume now $x\in I_0\setminus I_0'$, then there is a sequence $I'\ni z'_k\stackrel{k\to\infty}{\to} x$ s.t. $\tl{f}(x)=\liminf_{I'\ni z'\to x} f(z)=\lim_{k\to\infty}\tl{f}(z'_k)$; w.l.o.g. we assume $(z'_k)_{k\in \Nbb}\subset int(I_0)\cap I'$. According to what we showed, since $z'_k=(1-t_k)x_0+t_k x_1$  for some $t_k\in [0,1]$, at each $z'_k$ the following inequality is satisfied:
\eq{
\tl{f}(z'_k)&=f(z'_k)\geq \tl{M}_{K,N}^{(t_k)}[|x_1-x_0|](\tl{f}(x_0)\frac{a_0}{a_{t_k}},\tl{f}(x_1)\frac{a_1}{a_{t_k}})\,,
}
hence 
{\small 
\eq{\tl{f}(x)=\lim_{k\to\infty}f(z_k)
\geq \lim_{k\to\infty}\tl{M}_{K,N}^{(t_k)}[|x_1-x_0|](\tl{f}(x_0)\frac{a_0}{a_{t_k}},\tl{f}(x_1)\frac{a_1}{a_{t_k}})=\tl{M}_{K,N}^{(t)}[|x_1-x_0|](\tl{f}(x_0)\frac{a_0}{a_{t}},\tl{f}(x_1)\frac{a_1}{a_{t}})\,.
}}
This completes the proof of  \eqref{LebesguePointsCond2} for $|x_1-x_0|<l_{\delta}$. 
\bigskip

Let us consider now the case $|x_1-x_0|\geq l_{\delta}$; this is only possible if $\delta>0$ (because otherwise $l_{\delta}=\infty$). Then under the assumption that $\delta>0$:
\begin{itemize}
	\item  If $K<0$ then $\tl{M}_{K,N}^{(t)}[d](\cdot\,,\,\cdot)\equiv 0$ for $d\geq l_{\delta}$ and the stated inequality clearly holds. 
	\item If  $K>0$ then  $\tl{f}(x_0)\tl{f}(x_1)=0$, since for a.e. $t\in (0,1)$ it holds that $x_t=\convar{x_0}{x_1}\in int(I_0)\cap I'$, by what we previously showed it follows that:
	$$f(x_t)\geq \tl{M}_{K,N}^{(t)}[|x_1-x_0|](\tl{f}(x_0)\frac{a_0}{a_{t}},\tl{f}(x_1)\frac{a_1}{a_{t}})=\infty\qquad \mbox{for a.e. } t\in (0,1)\,,$$
	which contradicts $f\in L_{loc}^{1}(I)$. 
\end{itemize}

\bigskip
We conclude that $\tl{f}$ satisfies the inequality on $I$ and by definition $\xi_{\tl{f}}=\xi_f$. 
\end{proof}

\bigskip
%, and $\tl{f}=f$ by definition on $I'$, whence $\xi_f=\xi_{\tl{f}}$. 

	%$\lim_{n\to\infty}|y^{\prime(n)}_{1}-y^{\prime(n)}_{0}\geq l_{\delta}$ 
	%whence for almost every $t\in (0,1)$: 
	%f(x_t)\geq \tl{M}_{K,N}^{(t)}(f(x_0),f(y^{\prime(n)}_{1}))\qquad \forall a_0,a_1\geq 0 

We proceed with the proof of the direction $(2\Rightarrow 1)$. Define $I'$ to be the set of Lebesgue points of $f$ in $I:=\R$. For any $x'\in I'$ it holds that $\lim_{\epsilon\to 0}\frac{\xi_f((x'-\frac{\epsilon}{2},x'+\frac{\epsilon}{2}))}{m((x'-\frac{\epsilon}{2},x'+\frac{\epsilon}{2}))}=f(x')$. 
We will show that $f(x)$ satisfies
\eql{ \label{LebPtsClaim} f(x_t)a_t\geq \tl{M}_{K,N}^{(t)}(f(x_0)a_0,f(x_1)a_1)\qquad \forall a_0,a_1\geq 0 \,\text{ and }\forall x_0,x_1\in I', \,\,\forall t\in[0,1] \mbox{ s.t. } x_t\in I' \,.}
Assume $x_0,x_1\in I'$  are such that $f(x_0)f(x_1)>0$. Let $t\in[0,1]$ be s.t. $x_t\in I'$. For $i\in \{0,1\}$ and $\epsilon>0$ we set $A_i^{\epsilon}:=x_i+(-\half\epsilon a_i, \half\epsilon a_i)$ where $a_i>0$ are fixed. 

Let us assume first that $|x_1-x_0|<l_{\delta}$; then we can assume that for some $\epsilon_0>0$ it holds that $\theta_{K}(A^{\epsilon}_0, A^{\epsilon}_1)<l_{\delta}$ for every $0<\epsilon<\epsilon_0$.

Then by hypothesis
\[ \xi_f(A_t^{\epsilon})\geq \tl{M}_{K,N}^{(t)}[\theta_{K}(A^{\epsilon}_0, A^{\epsilon}_1)]\prnt{\xi_f(A^{\epsilon}_0), \xi_f(A_1^{\epsilon})}\,. \]
By assumption $f(x_0)f(x_1)>0$ hence we may further assume that $\xi_f(A^{\epsilon}_0)\xi_f(A^{\epsilon}_1)>0$ for every $0<\epsilon<\epsilon_0$. We may thus divide the inequality by $m(A^{\epsilon}_t)$:
\[ \frac{\xi_f(A_t^{\epsilon})}{m(A^{\epsilon}_t)}\geq \tl{M}_{K,N}^{(t)}[\theta_{K}(A^{\epsilon}_0, A^{\epsilon}_1)]\prnt{\frac{\xi(A^{\epsilon}_0)}{m(A^{\epsilon}_t)},\frac{\xi_f(A^{\epsilon}_1)}{m(A^{\epsilon}_t)}}\,. \]
%If $|x_1-x_0|< l_{\delta}$, we may assume $\epsilon'$ sufficiently small so that $\theta_{K}(A^{\epsilon}_0, A^{\epsilon}_1)<l_{\delta}$ for all $\epsilon<\epsilon'$; 

Then by Lebesgue differentiation theorem, by taking $\epsilon\to 0$ it follows that 
\eql{\label{LebPtsClaim2} f(x_t)\geq \tl{M}_{K,N}^{(t)}[|x_0-x_1|]\prnt{\frac{a_0}{a_t}f(x_0),\frac{a_1}{a_t}f(x_1)}  \,. }
%\qquad \forall x_0,x_1\in I', t\in[0,1] \mbox{ s.t. } x_t\in I'
Assume now $|x_1-x_0|\geq l_{\delta}$. This is only possible when $\delta>0$. If $K<0$  the $RHS$ is $0$ and the inequality trivially holds at $x_t$; if $K>0$ then since as $\epsilon\to 0$ it holds that $\lim_{\epsilon\to 0}\theta_{K}(A^{\epsilon}_0, A^{\epsilon}_1)\geq l_{\delta}$ while $\frac{\xi_f(A^{\epsilon}_i)}{m(A^{\epsilon}_t)}\to f(x_i)>0$, the $RHS$ approaches $\infty$, implying that $f(x_t)=\infty$. Since for a.e. $t\in [0,1]$ the point $x_t\in I'$, this contradicts $f\in L^1_{loc}(\R)$. Therefore whenever $|x_1-x_0|\geq l_{\delta}$ then $f(x_0)f(x_1)=0$ and  inequality \eqref{LebPtsClaim2} is still valid. 
\bigskip 
%the condition \eqref{LebesguePointsCond}
We can thus conclude that inequality \eqref{LebPtsClaim} is satisfied. 
By Lemma \ref{lem:Borell} it holds that $\xi_f=\xi_{\tl{f}}$ where $\tl{f}$ satisfies: 
\eq{
\tl{f}(x_t)a_t\geq \tl{M}_{K,N}^{(t)}(\tl{f}(x_0)a_0,\tl{f}(x_1)a_1)\qquad \forall a_0,a_1\geq 0 \,\text{ and }\forall x_0,x_1\in I, \,\,t\in[0,1]  \,.}

%
%
%
%, contradicting $f\in L^1_{loc}(\R)$, unless $\xi(A^{\epsilon}_0)\xi(A^{\epsilon}_1)=0$ for all $0<\epsilon<\epsilon_0$, whence $f(x_0)f(x_1)=0$, implying that $diam(supp(f))\leq l_{\delta}$. 
%
%According to Proposition \ref{prop:Borell} $\xi_f=\xi_{\tl{f}}$, where $\tl{f}$ satisfies \eqref{eqToVerify}
%for every $x_0,x_1\in I$, $t\in [0,1]$ and $a_0,a_1>0$.

\end{proof}
\bigskip

\subsubsection{Proof of the equivalences Theorem \ref{thm:Equivalence}}
%We proceed to the proof of \ref{thm:Equivalence}.
\begin{proof}
 %\ref{rmk:CondImplic}
%In view of Lemma \ref{lem:CondImplic}  we may assume throughout that whenever $\delta>0$ and $K>0$ then $diam(I)<l_{\delta}$.

	%We proved  the statement under the assumption that  $\theta_K(A_0,A_1)< l_{\delta}$, since as we showed in the proof of \ref{lem:Borell}, if $\theta_K(A_0,A_1)\geq l_{\delta}$ then when $K>0$ this implies that $\xi_f(A_0)\xi_f(A_1)=0$, while if $K<0$ then $\tl{M}_{K,N}^{(t)}[\theta_{K}(A_0, A_1)]\equiv 0$, hence the statement holds in either case.
%

%Sufficient to prove the statement under the assumption that $Diam(I)\leq l_{\delta}$ when $\delta>0$ and $K>0$; indeed, if $|x_1-x_0|\geq l_{\delta}$
%(which is possible only when $\delta>0$), then for $N\in(-\infty,0]$ (i.e $K<0$) it holds that $\tl{M}_{K,N}^{(t)}[|x_1-x_0|](\cdot, \cdot)=0$ by definition, 
%and for $K>0$,

$(2 \Rightarrow 1)$  Assume $\xi_J(A_0)\xi_J(A_1)>0$ (otherwise there is nothing to prove).  
Throughout we will assume $\theta_{K}(A_0,A_1)<l_{\delta}$, since $l_{\delta}=\infty$ when $\delta<0$, and for $\delta>0$:  if $K<0$ then $\tl{M}_{K,N}^{(t)}[\theta_{K}(A_0,A_1)](\cdot, \cdot)\equiv 0$ whenever $\theta_{K}(A_0,A_1)\geq l_{\delta}$, while for $K>0$ by Lemma \ref{lem:CondImplic} $diam(supp(\xi_J))\leq l_{\delta}$, and considering that $\xi_J$ is a.c. and $\xi_J(A_0)\xi_J(A_1)>0$, we conclude that $\theta_{K}(A_0,A_1)=d(A_0,A_1)<l_{\delta}$. 

%Sufficient to prove the claim under the assumption that $\theta_{K}(A_0,A_1)<l_{\delta}$;  while if $K>0$, as we previously mentioned, if $\theta_{K}(A_0,A_1)=d(A_0,A_1)\geq l_{\delta}$ then $\xi_J(A_0)\xi_J(A_1)=0$, implying that $\tl{M}_{K,N}^{(t)}[\theta_{K}(A_0,A_1)](\xi_J(A_0), \xi_J(A_1))= 0$ and there is again nothing to prove. 
\bigskip 

From Proposition \ref{prop:properties} we conclude that for every $x_0, x_1\in I$, $a_0,a_1>0$ and $t\in [0,1]$:
\eq{
\tl{J}(x_t)a_t&\geq M_{K,N-1}^{(t)}[|x_1-x_0|](\tl{J}(x_0),\tl{J}(x_1))
M_{0,1}^{(t)}(a_0,a_1)
\geq \\& \tl{M}_{K,N}^{(t)}[|x_1-x_0|](\tl{J}(x_0)a_0, \tl{J}(x_1)a_1)\,. }

By Lemma \ref{lem:Borell} it follows that for any $t\in [0,1]$ and $A_0,A_1\subset \R$ compact the following inequality holds:
\[ \xi_J(A_t)\geq \tilde{M}_{K,N}^{(t)}[\theta_{K}(A_0,A_1)](\xi_J(A_0),\xi_J(A_1))\,. \]

$(1 \Rightarrow 2)$  
 By Lemma \ref{lem:Borell}  there exists $\tl{J}$ s.t. $\xi_{J}=\xi_{\tl{J}}$ where 
\eql{  \tl{J}(x_t)a_t\geq \tilde{M}_{K,N}^{(t)}[|x_1-x_0|](\tl{J}(x_0)a_0, \tl{J}(x_1)a_1)\qquad \forall x_0,x_1\in I,\,t\in [0,1]\,\,\mbox{and}\,\,\,a_0,a_1\geq 0\,\,  \,. \label{BorelIneq}    }

The statement trivially holds if $|x_1-x_0|\geq l_{\delta}$, since then when $\delta>0$ and $K<0$ it implies that $M_{K,N-1}^{(t)}[|x_1-x_0|](\cdot\,, \,\cdot)=0$, while when $\delta>0$ and $K>0$ we know  (see Lemma \ref{lem:CondImplic}) that $\tl{J}(x_0)\tl{J}(x_1)=0$ whence $M_{K,N-1}^{(t)}[|x_1-x_0|](\tl{J}(x_0),\tl{J}(x_1))=0$.
\bigskip

Therefore we assume throughout that $|x_1-x_0|< l_{\delta}$. 
Fix $t\in (0,1)$. If $N\in (-\infty,0]\cup (1,\infty)$ we pick $a_0:=\tl{J}(x_0)^{\frac{1}{N-1}}\frac{\sigma_{K,N-1}^{(1-t)}(|x_1-x_0|)}{1-t}$ and  $a_1:=\tl{J}(x_1)^{\frac{1}{N-1}}\frac{\sigma_{K,N-1}^{(t)}(|x_1-x_0|)}{t}$ then by \eqref{BorelIneq} it follows that

{\tiny
\eq{ \tl{J}(x_t)&\geq \frac{\prnt{\tau^{(1-t)}_{K,N}(|x_1-x_0|)\tl{J}(x_0)^{\frac{1}{N}+\frac{1}{N(N-1)}}\prnt{\frac{\sigma_{K,N-1}^{(1-t)}(|x_1-x_0|)}{1-t}}^{\frac{1}{N}} +
\tau^{(t)}_{K,N}(|x_1-x_0|)\tl{J}(x_1)^{\frac{1}{N}+\frac{1}{N(N-1)}}\prnt{\frac{\sigma_{K,N-1}^{(t)}(|x_1-x_0|)}{t}}^{\frac{1}{N}}  }^{N}}{\convar{\tl{J}(x_0)^{\frac{1}{N-1}}\frac{\sigma_{K,N-1}^{(1-t)}(|x_1-x_0|)}{1-t}}{\tl{J}(x_1)^{\frac{1}{N-1}}\frac{\sigma_{K,N-1}^{(t)}(|x_1-x_0|)}{t}}}\\&= M_{K,N-1}^{(t)}[|x_1-x_0|]\prnt{\tl{J}(x_0),\tl{J}(x_1)}\,, }    }
where we used the identity $\tau^{(t)}_{K,N}(\theta):=t^{\frac{1}{N}}\sigma_{K,N-1}^{(t)}(\theta)^{1-\frac{1}{N}}$. If $N=\infty$ we set $a_0=a_1=1$, and we conclude from \eqref{BorelIneq} that $\tl{J}(x_t)\geq \tl{M}_{K,\infty}^{(t)}[|x_1-x_0|]\prnt{\tl{J}(x_0),\tl{J}(x_1)}=M_{K,\infty}^{(t)}[|x_1-x_0|]\prnt{\tl{J}(x_0),\tl{J}(x_1)}$. 

%Lastly if $N=1$, we assume w.l.o.g that $\tl{J}(x_0)\geq \tl{J}(x_1)$, and we set $a_0^{\epsilon}=1+\epsilon$ and $a_1^{\epsilon}=\epsilon$, with $\epsilon$ some positive number. By assumption $\tl{J}(x_t)\geq \frac{1}{\convar{a_0^{\epsilon}}{a_1^{\epsilon}}}\tl{M}_{K,1}^{(t)}[|x_1-x_0|]\prnt{\tl{J}(x_0)a_0^{\epsilon},\tl{J}(x_1)a_1^{\epsilon}}$, hence by letting $\epsilon\to 0$ we conclude that $\tl{J}(x_t)\geq \tl{J}(x_0)=M_{K,0}^{(t)}[|x_1-x_0|]\prnt{\tl{J}(x_0), \tl{J}(x_1)}$. 
% Since $t$ is arbitrary the foregoing proves that \eqref{CDKN_2}.

($2\Leftrightarrow 3$) If $ N\in (-\infty, 0]\cup (1,\infty)$ then according to Remark \ref{rmk:ODEfacts} the following equivalence holds 
\eq{
   &(N-1)(J^{\frac{1}{N-1}})''+KJ^{\frac{1}{N-1}}\leq 0 \text{ on } I\\&\Leftrightarrow  J(x_t)\geq M_{K,N-1}^{(t)}[|x_1-x_0|](J(x_0),J(x_1)) \qquad \forall x_0,x_1\in I \text{ and }\forall t\in (0,1)\, }
	(notice that by definition the direction $(3)\rightarrow (2)$ holds whenever $\delta>0$, $K<0$ and $|x_1-x_0|\geq l_{\delta}$ ). 
	%If $\delta>0$ and $K<0$  then $(1)\Leftrightarrow (2)$ holds trivially, while if $K>0$ 
	
	while if $N=\infty$
	{\small
\eq{
   &(\log J)''+ K\leq 0 \text{ on } I\\&\Leftrightarrow
	(\log J)(x_t)\geq \convar{(\log J)(x_0)}{(\log J)(x_1)}+\frac{K t(1-t)}{2}(x_0-x_1)^2\qquad \forall x_0,x_1\in I \text{ and }\forall t\in (0,1)\,,} }
	whence 
	\eq{
	J(x_t)\geq J^{1-t}(x_0)J^{t}(x_1)e^{\frac{K t(1-t)}{2}(x_0-x_1)^2}=M_{K,\infty}^{(t)}[|x_1-x_0|](J(x_0),J(x_1))\,.
}

\bigskip 

%\end{enumerate}
\end{proof}
\bigskip

We conclude our discussion about the equivalences with a useful proposition. It shows that for a measure $\xi_J=J\cdot m$ ($J\in L^1_{loc}(\R)$) which satisfies  either of the hypotheses (1) or (2) of Theorem \ref{thm:Equivalence}, we can further assume that $J$ is locally-Lipschitz on $int(supp(\xi_J))$.  

\begin{prop}\label{prop:contRepresentative} Let $K\in \R$ and $N\in (-\infty,0]\cup (1,\infty]$. Assume that  $0\leq f\in L^1_{loc}(\R)$ and set $I:=\isupp(f)$. If 
%	For all $x_0,x_1\in \R$ such that $|x_1-x_0|<l_{\delta}$ and 
	\eql{\label{assumed:ineq}  f(x_t)\geq M_{K, N-1}^{(t)}[|x_1-x_0|](f(x_0),f(x_1)) \qquad \forall x_0,x_1\in int(I),\, \forall t\in[0,1]\,.}
	Then
	\begin{enumerate}
		\item $f$ is locally bounded on $int(I)$ away from $0$ and $\infty$. 
		\item $f$ is locally Lipschitz continuous on $int(I)$.
	\end{enumerate}
\end{prop}
\begin{proof} 
\begin{enumerate}
	\item Fix a point $x\in int(I)$, we will show that $f$ is bounded from below and from above by positive constants in some neighborhood of $x$. Fix two points $x_0,x_1\in int(I)$ s.t.
	\begin{itemize}
		\item $x_0<x<x_1$, 
		\item $f(x_0),f(x_1)\in (0,\infty)$.
	\end{itemize}
	Since for every $t\in [0,1]$ it holds that $f(x_t)\geq M_{K,N-1}^{(t)}[|x_1-x_0|](f(x_0),f(x_1))>0$ we conclude that $f(x)>0$ on $[x_0,x_1]$. Furthermore, since $M_{K,N-1}(a,b)\geq M_{K,-1}(a,b)$ for every $N\in (-\infty,0]\cup (1,\infty]$ and $a,b>0$ (see Proposition  \ref{prop:properties} (\ref{id:4}a)), it is sufficient to prove the statement under the assumption that $N=0$ and $0<|x_0-x_1|<l_{\delta_0}$ where 
	$$l_{\delta_0}:=\begin{cases}  \frac{\pi}{\sqrt{-K}} &\mbox{if } K<0\\ +\infty &\mbox{if } K\geq 0\,.\end{cases}$$

To see that it is actually bounded away from zero by a positive constant notice that for
 %any point $z\in [x_0,x_1]$ there is $t\in [0,1]$ s.t. $z=\convar{x_0}{x_1}$ hence
any $t\in [0,1]$:
\eq{&\frac{1}{f(x_t)}\leq \sigma_{K,-1}^{(1-t)}(|x_0-x_1|)f(x_0)^{-1}+\sigma_{K,-1}^{(t)}(|x_0-x_1|)f(x_1)^{-1}\\&\leq \max\cprnt{\sigma_{K,-1}^{(1-t)}(|x_0-x_1|)f(x_0)^{-1},\sigma_{K,-1}^{(t)}(|x_0-x_1|)f(x_1)^{-1} }\\&
\leq \max_{t'\in[0,1]}\sigma_{K,-1}^{(t')}(|x_0-x_1|)\max\cprnt{f(x_0)^{-1},f(x_1)^{-1}}:=m_{K,x_0,x_1}<\infty\,. }
	
We will now show that $f(x)$ is bounded from above in some neighborhood of $x$. Assume by contradiction that this is not the case. Define $r=\min(|x_1-x|,|x-x_0|)$ and let $(r_n)_{n\in \Nbb}$ be a sequence such that $0<r_{n+1}<r_n<r$ and $r_n\to 0$ as $n\to \infty$, and for any $n\in \Nbb$ let $y_n\in B_{x}(r_n)$ such that $f(y_n)>n$. 

Notice that for some $t_n\in [0,1]$ we may write $x=(1-t_n)y_n+t_nx_0$ or $x=(1-t_n)y_n+t_nx_1$ depending on whether $y_n$ is to the right or left of $x$ respectively.  As $n\to \infty$ (which implies $y_n\to x$) $t_n\to 0$. For simplicity of notation we will assume $y_n>x$. 
 Considering that $\sigma_{K,-1}^{(t')}(\epsilon)\to 0$ as $t'\to 0$ and $\epsilon\to 0$, we may pass to a sub-sequence $\{n_k\}_{k\in \Nbb}$ s.t. for every $k\in \Nbb$:
\begin{enumerate}
	\item $f(y_{n_k})\geq n_k$.
	\item $\sigma_{K,-1}^{(t_{n_k})}(|x_0-y_{n_k}|)\leq \frac{1}{k}$.
\end{enumerate}
Then
\eq{& f(x)^{-1}\leq \sigma_{K,-1}^{(1-t_{n_k})}(|x_0-y_{n_k}|)f(y_{n_k})^{-1}+\sigma_{K,-1}^{(t_{n_k})}(|x_0-y_{n_k}|)f(x_0)^{-1}\\ &\leq 
 \sigma_{K,-1}^{(1-t_{n_k})}(|x_0-y_{n_k}|)\frac{1}{n_k}+\frac{1}{k}\cdot f(x_0)^{-1}\stackrel{k\to\infty}{\longrightarrow } 0\,.
 }
Therefore $f(x)=\infty$. 
In such a case if $z_t=\conv{x_0}{x}\in [x_0,x]$ ($t\in (0,1)$) then:
\[ f(z_t)^{-1}\leq \sigma_{K,-1}^{(t)}(|x_0-x|)f(x_0)^{-1}+0\Rightarrow f(z_t)\geq \sigma_{K,-1}^{(t)}(|x_0-x|)^{-1}f(x_0)\,.\]
Therefore,
 $$\int_{x_0}^{x_1}f(z)dz\geq \int_{0}^{1}\sigma_{K,-1}^{(t)}(|x_0-x|)^{-1}f(x_0)(x_1-x_0)dt\to \infty\,,$$
	since in a neighborhood of $0$ it holds that $\sigma_{K,-1}^{(t)}(|x_0-x|)^{-1}\sim \frac{1}{t}$. 
This contradicts our assumption that $J\in L^1_{loc}(\R)$.

\item

%Let $x_0\in Int(I)$.  Take $x\in B_{x_0}(r)$ and set $x_r:=x-r\cdot sgn(x-x_0)=x+\alpha (x_0-x)$, where $\alpha:=\frac{r}{|x-x_0|}$. 
%Notice that $x_0=\frac{1}{\alpha+1}\bar{x}+\frac{\alpha}{\alpha+1}x$
As before it is sufficient to prove the statement for the case $N=0$. 
Let $x\in int(I)$. As we showed we may assume $0<M_{K,x,r}^{-1}\leq f(y)\leq M_{K,x,r}<\infty$ for every $y$ in the open ball $B_{x}(r)$ for some $0<r<l_{\delta}$. Let $x_1,x_2\in B_{x}(\frac{r}{2})$ and set $\bar{x}_1=x_2+s(x_2-x_1)$ where $s=\frac{r}{2|x_2-x_1|}$ (notice that $\bar{x}_1\in B_{x}(r)$). Then
$x_2=\frac{1}{s+1}\bar{x}_1+\frac{s}{s+1}x_1$. 

%---------------------------------
Consider the identity (see \eqref{eqn:sigma_taylor}): 
$$
\sigma^{(t)}_{K,-1}(\theta)=\frac{\si_{-K}(t\theta)}{\si_{-K}(\theta)}=t\Prnt{1-\frac{K}{3!}\theta^2(1-t^2)+O(\theta^4) }\,.$$
Using Taylor's theorem with remainder, for some $\theta_1,\theta_2\in (0,|\bar{x}_1-x_1|)\subset (0,2r)$ holds the identity:
{\small 
\eq{&f(x_2)^{-1}-f(x_1)^{-1}\leq \sigma_{K,-1}^{(\frac{1}{s+1})}(|\bar{x}_1-x_1|)f(\bar{x}_1)^{-1}+\prnt{\sigma_{K,-1}^{(\frac{s}{s+1})}(|\bar{x}_1-x_1|)-1}f(x_1)^{-1}\\&=
\prnt{\frac{1}{s+1}}\Prnt{1-\frac{K}{6}\prnt{1-\prnt{\frac{1}{s+1}}^2}\theta_1^2}f(\bar{x}_1)^{-1}+
\prnt{\frac{s}{s+1}}\Prnt{1-\frac{K}{6}\prnt{1-\prnt{\frac{s}{s+1}}^2}\theta_2^2}f(x_1)^{-1} -f(x_1)^{-1}\,.}
}
Hence
{\small
\eq{&f(x_2)^{-1}-f(x_1)^{-1}\\&\leq \prnt{\frac{1}{s+1}}f(\bar{x}_1)^{-1}+\prnt{\prnt{\frac{s}{s+1}}-1}f(x_1)^{-1}\\&-\frac{K}{6(s+1)}\cprnt{\prnt{1-\prnt{\frac{1}{s+1}}^2}\theta_1^2f(\bar{x}_1)^{-1}+s\prnt{1-\prnt{\frac{s}{s+1}}^2}\theta_2^2f(x_1)^{-1} }\\&\leq 
\prnt{\frac{1}{s+1}}\Abs{f(\bar{x}_1)^{-1}-f(x_1)^{-1}}+\frac{|K|}{6(s+1)}\Abs{\prnt{1-\prnt{\frac{1}{s+1}}^2}\theta_1^2f(\bar{x}_1)^{-1}+s\prnt{1-\prnt{\frac{s}{s+1}}^2}\theta_2^2f(x_1)^{-1} }\\&\leq 
\prnt{\frac{1}{s+1}}\prnt{f(\bar{x}_1)^{-1}+f(x_1)^{-1}}+\frac{|K|(2r)^2}{6(s+1)}\max(f(\bar{x}_1)^{-1},f(x_1)^{-1})\Prnt{\prnt{\frac{s^2+2s}{(s+1)^2}}+\prnt{\frac{2s^2+s}{(s+1)^2}}} \\&\leq 
\frac{1}{s+1}\Prnt{2M_{K,x,r}+\frac{|K|4r^2}{6}M_{K,x,r}\cdot 3}\stackrel{s=\frac{r}{2|x_2-x_1|}}{\leq }C_{K,x,r}|x_2-x_1|
%
%+\frac{|K|r^2}{6(s+1)}2M_{r,x}\\&\stackrel{s=\frac{r}{2|x_2-x_1|}}{\leq}\prnt{4m_{K,x,r}+\frac{|K|r^2}{6}2m_{K,x,r}}\prnt{\frac{|x_2-x_1|}{r}}
}      }
for some constant $C_{K,x,r}>0$, 
where we used the estimate $\Prnt{\prnt{\frac{s^2+2s}{(s+1)^2}}+\prnt{\frac{2s^2+s}{(s+1)^2}}}\leq 3$ for every $s>0$. 

Repeating the argument, with the roles of $x_2$ and $x_1$ exchanged, we conclude that $f^{-1}(y)$ is Lipschitz continuous in a neighborhood $B_{x}(\frac{r}{2})$ of $x$ for some $r>0$. Considering that 
$$|f(x_2)^{-1}-f(x_1)^{-1}|=\frac{|f(x_1)-f(x_2)|}{f(x_1)f(x_2)}\geq \frac{|f(x_1)-f(x_2)|}{M_{K,x,r}^{2}}\,,$$ we conclude that $f(y)$ is Lipschitz on $B_{x}(\frac{r}{2})$. 

\end{enumerate}

\end{proof}
 
\section{Measures satisfying `synthetic' $CDD_b(K,N,D)$ conditions and their properties}

%The eqivalence of (3) to (2) and (1) in \ref{thm:Equivalence} required a smoothness assumption of the density $J$. 
%The equivalences in ,  were proved under the assumption that the density $J$ is twice continuously differentiable. However, one should observe the following: The differentiability assumption was not essential for the equivalence of $(2)\Leftrightarrow (3)$ in \ref{thm:Equivalence}; sufficient to assume $J\in L^1_{loc}(\R)$

  For applications it is beneficial to work with classes of measures which are closed under weak convergence. The class $\Fknd^{\Cinf}$ was defined using the differential condition \ref{eqn:Diff:CD(k,N)} appearing in $(3)$ of Theorem \ref{thm:Equivalence}, which is not preserved under weak convergence (the limit might not have a $\Cinf$ nor $C^2$ density); a class of measures which satisfies condition $(2)$ (which is more general than $(3)$) is also not weakly closed (since it can't accommodate singular measures). Fortunately condition $(1)$, which is the most general, can incorporate singular measures as well. We will accept it as the `synthetic' curvature-dimension condition; it will define a new class, which can be considered as an extension of $\Fknd^{\Cinf}$. Worthy of note, this new class is closed under weak convergence; this property makes it more natural for applications of convex optimization methods.  We will now make a digression into defining this new extended class, and prove some of its essential properties.

%Merkle 4.3 measures with finite support are dense in M^+

%p.188 Bauer  Measures and integraion, conv. Radon
\subsection{The classes $\Fknd(I)$ and $\Fknd^{M}(I)$}\label{dfnClasses}

We denote by $\M^{\pm}$ the set of all Radon signed  measures on $\R$, and by $\M$ the subset of all non-negative Radon measures on $\R$. By the Riesz representation theorem we can identify $\M$ as the space of positive bounded linear functionals on $C_c(\R)$. Since for any $\xi_1, \xi_2\in \M$ and any two numbers $\alpha_1,\alpha_2\geq 0$ it holds that $\alpha_1\xi_1+\alpha_2\xi_2\in \M$, it is a convex cone. We also define
$$ \gls{finite_Radon}=\{ \xi\in \M:\, \xi(\R)<\infty \} \qquad \text{and}\qquad \gls{probabilities}=\{ \xi\in \M:\, \xi(\R)=1 \}\,.$$
When we consider sets of measures supported in a subset $I\subset \R$ we denote by $\M(I),\M_b(I),\P(I)$ the respective sets of measures.
%p.188
%\textcolor{red}{do we need to mention finite additivity for it to be dual of $M^b$?}

We now present the classes of measures which will replace the former class $\Fknd^{\Cinf}
(I)$. 
\begin{defn}[The class $\gls{M_cl_knd}$] \label{dfn:CDKNSynt} Let $K\in\R$, $N\in(-\infty,0]\cup (1,\infty]$,  $D\in(0,+\infty]$, and let $I\subset \R$ be a closed interval. We set $\Fknd(I)$ to be the  cone of measures $\xi\in  \M_b(\R)$ such that 
\begin{enumerate}
	\item $supp(\xi)\subset I$.
	\item For any two compact sets $A_0,A_1\subset \R$ and every $t\in [0,1]$:
	\eql{\label{eqn:CD(k,N)}
		\xi(A_t)\geq \tl{M}_{K, N}^{(t)}[\theta_{k}(A_0,A_1)](\xi(A_0),\xi(A_1))\,,
		}
				where 
		\eq{ \theta_{K}(A_0,A_1)=\begin{cases}
			\inf_{x\in A_0,\, \, y\in A_1}|x-y| &\mbox{ if  } K\geq 0 \\
			\sup_{x\in A_0,\, \, y\in A_1}|x-y| &\mbox{ if  } K< 0\,.
			\end{cases} }
	\item $diam(supp(\xi))\leq D $.
\end{enumerate}
We refer to the class $\Fkn(I)$ as the members of the class $\Fkninf(I)$.
We denote the subclass of probability measures in $\Fknd(I)$ (resp. $\Fkn(I)$) by $\gls{Pknd}$ (resp. $\Pkn(I)$).
%If $f\in C_c^{\infty}$, then we define $\M_{(k,N,D),f}:=\M_{(k,N,D)}(supp(f))\cap \{\xi: f(\xi)=0\}$.
\end{defn}
\bigskip
\begin{remk} If $\xi\in \Fknd(I)$ then subject to the proviso that $D< l_{\delta}$ if   $\delta>0$ and $K<0$, it holds that $diam(supp(\xi))\leq l_{\delta}$. Indeed, $l_{\delta}=\infty$ if $\delta\leq 0$, and if $\delta>0$ and $K>0$ then according to Lemma \ref{lem:CondImplic}  $diam(supp(\xi))\leq \frac{\pi}{\sqrt{\delta}}=l_{\delta}$ (a constraint which is expected in view of the generalized Bonnet-Myers theorem \cite{KTS3}, which states that if \blue{a CWRM } $(M,\gfrak,\mu)$, \blue{with $\mu$ locally finite,  satisfies $CD(K,N)$} with $K>0, N\geq n$,  then  $diam(M)\leq \frac{\pi}{\sqrt{\delta}}$).

%
%
 %, if $\delta>0$ and $K>0$; otherwise, when $\delta\leq 0$ or $N=1$ then $l_{\delta}=\infty$ by definition.
\end{remk}
\bigskip
An important sub-class of measures is the set of `model-space measures'. We define it in analogy to   Definition \ref{dfn:CDKNSynt} of the class $\Fknd(I)$. 
\begin{defn}[The model class $\gls{M_cl_M_knd}$]\label{dfn:ModelMeasures}  Let $K\in\R$, $N\in(-\infty,0]\cup (1,\infty]$,  $D\in(0,+\infty]$, and let $I\subset \R$ be a closed interval.  We define the class $\Fknd^M(I)$ to be the set of a.c. measures $0\nequiv \xi=J\cdot m\in  M_b(\R)$ s.t. 
\begin{enumerate}
	\item $I_J:=supp(J)\subset I$.
  \item $J$ satisfies: for all $x_0,x_1\in int(I_J)$ and every $t\in [0,1]$:
	\eql{\label{ModelEqnDefn}  J(x_t)= M_{K, N-1}^{(t)}[|x_1-x_0|](J(x_0),J(x_1))\,.  }
	\item $diam(I_{J})\leq D$.
	\end{enumerate}
\end{defn} 
 Evidently from the equivalences of Theorem \ref{thm:Equivalence} it follows that  $\Fknd^M(I)\subset \Fknd(I)$.
\begin{remk}\label{rmk:Model_equivalent_def}
According to this definition $J\in \Cinf(int(I_J))$, and according to Remark \ref{rmk:ODEfacts} $J$ satisfies the differential equation $-(\log J)''-\frac{1}{N-1}((\log J)')^2= K$ on $int(I_J)$; the solutions $J$ to this ODE are indeed  smooth on $int(I_J)$, and can be identified as restrictions of scaling and translations of the functions $J_{K,N,\hfrak}$ (for some $\hfrak\in \R$) which were defined in \eqref{defn:Jknh}. In view of this, for $N\in(-\infty,0]\cup (1,\infty]$ there is no ambiguity between this definition and Definition   \ref{defn:CD_Model_First}. Yet, one should note that the ODE condition in Definition \ref{defn:CD_Model_First} applies to  $N\in(-\infty,\infty]\setminus \{1\}$.
\end{remk}

The following theorem is a natural analogue of the statement that for $N\in(-\infty,0]\cup (1,\infty]$ it holds that \blue{$CDD(K,N_1,D)\Rightarrow CDD(K,N_2,D)$} whenever $\frac{1}{N_1}\geq \frac{1}{N_2}$.
\begin{thm}\label{CDKN_1CDKN_2_mono} Let $I\subset \R$ be a fixed closed interval, and let $K\in\R$, $N_1,N_2\in(-\infty,0]\cup (1,\infty]$ and $D\in (0,\infty]$. If $\frac{1}{N_1}\geq \frac{1}{N_2}$ then $\Fkndg{K}{N_1}{D}(I)\subset \Fkndg{K}{N_2}{D}(I)$.
\end{thm}

\begin{proof}
 The statement is a straightforward consequence of Proposition \ref{prop:properties} (\ref{id:4}b).
\end{proof}

In general $\xi\in \Fknd(\R)$ admits a Lebesgue decomposition $\xi=\xi_{ac}+\xi_s$ into absolute continuous and singular parts. Since $\xi$ is a finite Radon measure, $d\xi_{ac}=J\,dm$ for some function $J\in L^1(\R)$. In \cite{Bo1} Borell showed that there are restrictions on the Lebesgue decomposition of `$\frac{1}{N}$-concave' (i.e. $\FF_{0,N}(\R)$) measures. Following his arguments we show that the same restrictions apply also to general $\Fknd(\R)$ measures. 
\noindent \begin{thm}\label{thm:NoSingularContinuous} Assume $\xi\in \Fknd(\R)$ where $K\in\R$,  $N\in (-\infty,0]\cup (1,\infty]$ and $D\in (0,\infty]$. Then, subject to the proviso $D< l_{\delta}$ if $\delta>0$ and $K<0$, it holds that: 
\begin{enumerate}
	\item $supp(\xi)$ is connected.
	\item Either
	\begin{enumerate}
		\item $supp(\xi_s)$ contains a single point and $\xi_{ac}\equiv 0$. 
		\item $supp(\xi_s)=\emptyset$, i.e. $d\xi=Jdm$ where $J\in L^1(\R)$. 
	\end{enumerate} 
	\end{enumerate}
\end{thm}
\begin{remk} The proviso `$D< l_{\delta}$ if $\delta>0$ and $K<0$' is necessary for the validity of this simple characterization of 
$supp(\xi)$ under any Curvature-Dimension conditions  within the pertinent range of the parameters $K$ and $N$. The necessity of this proviso will also emerge later on, when we discuss about the identification of the model-space measures as extreme points of a subset of $\Pknd(\R)$. 
Recall that the model-space measures have densities which are up to translations and rescalings are just restrictions of the function $J_{K,N,\hfrak}(x)$. When $\delta>0$ and $K<0$, i.e. $K<0$ and $N\in (-\infty,0]$, then $J_{K,N,\hfrak}(x)$ is of the form  $\cos(\sqrt{\delta}x)^{N-1}$; since $N-1\leq -1$ the function $J_{K,N,\hfrak}(x)$ is not  integrable around  the points $-\frac{\pi}{2\sqrt{\delta}}=-\frac{l_{\delta}}{2}$ and $\frac{\pi}{2\sqrt{\delta}}=\frac{l_{\delta}}{2}$; hence if $\xi_0\in \Fknd^M(\R)$  is supported on an interval $I_0$ then $diam(I_0)<l_{\delta}$. However, without this proviso a member $\xi$ of $\Fknd(\R)$ ($\delta>0$ and $K<0$) might have the following undesirable properties:
\begin{itemize}
	\item $diam(supp(\xi))> l_{\delta}$. 
	
	For example, assume $N_1\in (-\infty,0]$. Let $d\xi_0:=e^{\frac{|K|x^2}{2}}1_{I_0}(x)\,dm\in \FkndgM{K}{\infty}{\infty}(\R)$ be a measure supported on an interval $I_0$. Measures from  $\FkndgM{K}{\infty}{\infty}(\R)$ can be supported on intervals of arbitrarily large finite diameter, so we may assume $diam(I_0)>l_{\delta_1}:=\sqrt{\frac{N_1-1}{K}}\pi$.
	By Theorem 	\ref{CDKN_1CDKN_2_mono} $\xi_0\in \Fkndg{K}{N_1}{\infty}(\R)$
    (since $\frac{1}{N_1}<0$).  
	\item $supp(\xi)$ might not even be connected; 
it might contain several (or even infinitely many)  components $I_k$ ($k=1,2...$), s.t. $diam(I_k)<l_{\delta}$. 

For example, assume $\xi_1,\xi_2\in \Fknd^M(\R)$ are supported on intervals $I_1,I_2$ s.t.
\begin{itemize}
    \item $\max(diam(I_1),diam(I_2))<l_{\delta}$
    \item $d(I_1,I_2)>l_{\delta}$.
\end{itemize}
Define a measure $\xi\in \M_b(\R)$  by $\xi_1+\xi_2$. This measure has a density $J$ which verifies condition \ref{property:II} of the equivalences of  Theorem \ref{thm:Equivalence}, which is equivalent to condition \ref{property:III} of the Theorem, therefore  
$\xi\in \Fknd(\R)$. 
\end{itemize}

For the proof of the extreme points characterization Theorem \ref{thm:ExtremePoints}, which is one of the main goal of this chapter, we want to exclude these possibilities. 
\end{remk}
Before proving Theorem \ref{thm:NoSingularContinuous} we mention several facts about differentiation of measures; for the proof of these statements the reader is referred to \cite[p.143]{Rud1}.
 \begin{thm}\label{lem:measDeriv} Assume $d\xi=Jdm+\xi_s$ is the Lebesgue decomposition of a measure $\xi\in\M$ w.r.t. $m$ ($J\in L^1_{loc}(\R)$). We define the derivative of $\xi$ with respect to $m$ at $x$ (when it exists) 
$$ (D\xi)(x)=\lim_{r\to 0} \frac{\xi(I(x; r))}{m(I(x; r))}\,.$$
Then
\begin{enumerate}
	\item $(D\xi)(x)=J(x)$ a.e. $[m]$. 
	\item $\xi\bot m$ if and only if $(D\xi)(x)=0$ for $[m]$-a.e. $x$. 
	\item If $\xi_s\neq 0$ then $(D\xi_s)(x)=\infty$ for $[\xi_s]$-a.e. $x$. 
\end{enumerate}
\end{thm}

\begin{proof}[Proof of Theorem \ref{thm:NoSingularContinuous}]

Let $\xi=\xi_{ac}+\xi_{s}$ be the Lebesgue decomposition of $\xi$; here $\xi_{ac}=Jdm$ where $0\leq J \in L^1(\R)$ and $\xi_s\bot m$.
By Lemma \ref{lem:CondImplic} we have $diam(supp(\xi))\leq l_{\delta}$ subject to our proviso. We verify the statements:

%Notice that since $\Fkng{K}{N}\subset \Fkng{K}{0}$, it is enough to prove the statement for $\xi\in \Fkng{K}{0}$ . We firstly show that the $CD(K,0)$ condition implies connectedness of $supp(\xi)$.
\begin{enumerate}
	\item Assume $supp(\xi)$ contains two distinct points $x_0$ and $x_1$. Let $t\in (0,1)$ and set $x_t=\convar{x_0}{x_1}$. Then for every $r\in (0,\min\{\half|x_1-x_0|, \half(l_{\delta}-|x_1-x_0|))$ holds the inequality:
%, \half\prnt{l_{\delta}-|x_1-x_0|}\})$. Then for all $0<r<r_0$:
\eq{&\xi(I(x_t; r)\geq \tl{M}_{K, N}^{(t)}[\theta_{K}(I(x_0; r),I(x_1; r)](\xi(I(x_0; r),\xi(I(x_1; r))\,,
}

and the RHS is positive for every such $r$, hence $x_t\in supp(\xi)$. Since $t$ was arbitrary we can conclude that $supp(\xi)$ must be path-connected. 
%min\left\{(1-t)\prnt{\sigma_{K,-1}^{(1-t)}(d_r)}^{-1}\xi(I(x_0; r)),t%\prnt{\sigma_{K,-1}^{(t)}(d_r)}^{-1}\xi(I(x_1; r))\right\}>0\,,

%where $d_r:=\theta_K(I(x_0; r), I(x_1; r))>0$ (here we used property \ref{id:2} of \ref{prop:properties}). 
\item  
Assume $supp(\xi_s)$ contains two distinct points $x_0$ and $x_1$. Let us firstly assume that $|x_1-x_0|<l_{\delta}$. By definition for any $r>0$ it holds that $\xi_s(I(x_0; r))>0$ and $\xi_s(I(x_1; r))>0$. It follows from Theorem \ref{lem:measDeriv} that there are points $y_0\in I(x_0; r)$ and $y_1\in I(x_1; r)$  such that $(D\xi_s )(y_0)$ and $(D\xi_s)(y_1)$ both exist and equal $\infty$. By choosing $r$ sufficiently small we may further assume that $|y_1-y_0|<l_{\delta}$. 
Any point $y\in (y_0,y_1)$ can be expressed as $y=y_t=\convar{y_0}{y_1}\in (y_0,y_1)$, for some $t\in (0,1)$. We will show that $(D\xi_s)(y_t)$ exists and is infinite. 

Assume $r\in (0,\min\{\half|y_1-y_0|, \half\prnt{l_{\delta}-|y_1-y_0|}\})$; evidently $I(y_t; r)=\convar{I(y_0; r)}{I(y_1; r)}$. Since $\xi\in \Fkn(\R)$: 
%Since $\Fkn\subset \Fkng{K}{0}$ (property \ref{id:4} in \ref{prop:properties}) it follows that
%\label{eqn:RadonNikodym}
%\eq{\lim_{i\to\infty}\frac{\xi(I(y_t; r_i))}{m(I(y_t; r_i))}&\geq \lim_{i\to\infty}
%\min\left\{(1-t)(\sigma_{K,-1}^{(1-t)}(d_i))^{-1}\frac{\xi(I(y_0; r_i))}{m(I(y_0; r_i))},t(\sigma_{K,-1}^{(t)}(d_i))^{-1}\frac{\xi(I(y_1; r_i))}{m(I(y_1; r_i))}\right\}\\ \nonumber&\geq 
%\min\left\{(1-t)\lim_{i\to\infty}(\sigma_{K,-1}^{(1-t)}(d_i))^{-1}\frac{\xi_s(I(y_0; r_i))}{m(I(y_0; r_i))},t\lim_{i\to\infty}(\sigma_{K,-1}^{(t)}(d_i))^{-1}\frac{\xi_s(I(y_1; r_i))}{m(I(y_1; r_i))}\right\}
%=\infty\,,
%}
{\footnotesize
\eq{
&\frac{\xi(I(y_t; r))}{m(I(y_t; r))}\geq \tl{M}_{K, N}^{(t)}[\theta_{K}(I(x_0; r),I(x_1; r)]\prnt{\frac{\xi(I(y_0; r))}{m(I(y_0; r))},\frac{\xi(I(y_1; r))}{m(I(y_1; r))})}\\&\geq \tl{M}_{K, N}^{(t)}[\theta_{K}(I(x_0; r),I(x_1; r)]\prnt{\frac{\xi_s(I(y_0; r))}{m(I(y_0; r))},\frac{\xi_s(I(y_1; r))}{m(I(y_1; r))})}\to
\tl{M}_{K, N}^{(t)}[|y_1-y_0|]\prnt{(D\xi)(y_0),(D\xi)(y_1)}=\infty\,.
} 
}

Hence the only possibility of having more than one point in $\xi_s$ is when $\delta>0,K>0$ and  $|x_1-x_0|=l_{\delta}$; however 
in this case $\xi_s(\{x_0\})\xi_s(\{x_1\})$ must be zero, since otherwise (for the compact sets $A_0:=\{x_0\}, A_1:=\{x_1\}$) it holds that:
\[  \xi(\{x_t\})\geq \tl{M}_{k, N}^{(t)}[l_{\delta}]\prnt{\xi_s(\{x_0\}),\xi_s(\{x_1\})}\equiv \infty\,,\qquad \forall t\in (0,1)\,. \]
contradicting our assumption $\xi\in L_{loc}^1(\R)$. We conclude that $supp(\xi_s)$ can contain no more than a single point for every $\xi\in \Fknd(\R)$. 

We will now prove that if $supp(\xi_s)\neq \emptyset$ then it must hold that $supp(\xi_{ac})=\emptyset$. Indeed, if $supp(\xi_s)=\{x_0\}$ and $supp(\xi_{ac})=I\neq \emptyset$ notice that for $t\in (0,1)$
$$\xi(\convar{\{x_0\}}{I})\geq \tl{M}_{K, N}^{(t)}[\theta_K(\{x_0\},I)](\xi(\{x_0\}, \xi(I)) $$
As $t\to 0$ the LHS approaches $0$, while the RHS (considering the definitions of $\sigma^{(t)}_{K,N-1}(\theta)$ and $\tau^{(t)}_{K,N}(\theta)$ in \eqref{dfn:SigmaTau}) approaches $\xi_{s}(\{x_0\})>0$. This is a contradiction, hence if $\xi$ is not the zero-measure, either $supp(\xi_{ac})\neq \emptyset$ or $supp(\xi_{s})\neq \emptyset$, but not both.    
\end{enumerate}
\end{proof}

%The proof of the proposition follows almost verbatim from the proof for $s$-concave measures which were studied by C.Borell \cite{Bo3}. 
\subsection{Topological properties}

Recall that a sequence $(\xi_n)_{n\in \Nbb}\subset \M_b$  converges to $\xi$  in the weak topology (resp. weak-* topology) if for all $f\in C_b(\R)$  (resp. for all $f\in C_c(\R)$) $$\lim_{n\to\infty}\int fd\xi_n=\int fd\xi \,.$$
We denote these modes of convergence by $\xi_n\stackrel{w}{\longrightarrow}\xi$ and $\xi_n\stackrel{w*}{\longrightarrow}\xi$ respectively. 
\smallskip

\red{Note that the above definition of weak topology}, sometimes referred to as the `narrow topology', should not be confused with the weak topology of functional analysis, and the former is weaker than the latter. However, the above weak* topology is precisely the weak* topology of functional analysis, since by the Riesz Theorem $\M^{\pm}$ is the functional analytic dual of $C_c(\R)$.

\bigskip
 Evidently when $I$ is compact $C_c(I)=C_b(I)$ and there is no distinction between $w$ and $w*$  convergence. We remind the reader about the Portmanteau theorem which gives equivalent characterizations of weak convergence of probability measures: 

\begin{thm}[Portmanteau Theorem \cite{Kle}]\label{thm:Portmenteau} Assume $\xi$ and $(\xi_n)_{n\in \Nbb}$ are measures in $\P$. Then the following assertions are equivalent :
\begin{enumerate}
\item $\xi_n\stackrel{w}{\longrightarrow} \xi$.
\item $\limsup_{n\to\infty} \xi_n(F)\leq \xi(F)$ for every closed set $F\subset \R$.
\item $\liminf_{n\to\infty} \xi_n(F)\geq \xi(F)$ for every open set $G\subset \R$.
\item $\lim_{n\to\infty}\xi_n(Q)=\xi(Q)$ for every Borel set $Q\subset \R$ such that $\xi(\partial Q)=0$. 
\end{enumerate}
\end{thm}
\bigskip
%https://www.ma.utexas.edu/users/gordanz/notes/weak.pdf

We now prove that the class $\Pknd(I)$ is closed under weak convergence, i.e. closed in the weak topology. The reader should keep in mind that for later applications (specifically, the Banach-Alaoglu theorem) we will need to assume that the class is closed under weak* convergence, that is w.r.t. to the coarser weak* topology. However, for these applications it will turn out that the underlying interval $I$ is compact; as a result the weak and weak* topologies coincide, and therefore any result where weak convergence is involved, can be equivalently rephrased for weak* convergence.

\begin{thm} \label{thm:CDKNclosed} Let $I$ be a fixed interval. For $K\in \R$, $N\in (-\infty,0]\cup (1,\infty]$, and $D\in (0,\infty]$, s.t. $D< l_{\delta}$ if $\delta>0$ and $K<0$, the class $\Pknd(I)$ is weakly closed.
\end{thm}

\begin{proof}
 If $\xi_n \stackrel{w}{\longrightarrow} \xi$ where $\xi_n\in\Pknd(I)$ for $n\in \Nbb$, then by testing the function $1\in C_b(\R)$ we conclude that  $\xi\in \P$. We verify it satisfies the conditions in Definition \ref{dfn:CDKNSynt}. The arguments are  analogous to the arguments used in \cite{Bo3, BGVV} to show  that the class of $\frac{1}{N}$-concave measures is weakly closed.

\begin{myitemize}
	\item[Condition 1]: By the Portmanteau equivalences $0=\lim_{n\to\infty}\xi_n(I^c)=\xi(I^c)$,
	
	therefore $supp(\xi)\subset I$.

	\item[Condition 2]: \red{The case $K=0$ is proved in \cite{Bo3, BGVV}, therefore throughout we assume $K\neq 0$. }
	%Firstly notice that for any two compact sets $A^{(0)},A^{(1)}\subst \R$ it holds that $\lim_{n\to\infty}\theta^{\xi_n}_k(\hat{A}^{(0)},\hat{A}^{(1))=\theta^{\xi_n}_k(A^{(0)},A^{(1)})$. Indeed if 
	%
	Let $A^{(0)},A^{(1)}\subset \R$ be two compact sets such that $diam(A^{(0)}\cup A^{(1)})\leq l_{\delta}$ and $\xi(A^{(0)})\xi(A^{(1)})>0$. For $t\in (0,1)$ we set $A^{(t)}=\convar{A^{(0)}}{A^{(1)}}$. 	Given $r>0$, we denote the closed and open $r$-extensions of $A^{(i)}$ by $\hat{A}^{(i)}_r$ and $\check{A}^{(i)}_r$ respectively (the latter can be identified as the interior of the former). 	
	
	By the Portmanteau equivalences and \red{Proposition} \ref{prop:properties}:
\eql{\label{ineq:limMeasures} 
\xi(\hat{A}^{(t)}_{\frac{1}{m}})&\geq \limsup_{n\to\infty} \xi_{n}(\hat{A}^{(t)}_{\frac{1}{m}}  )\\ \nonumber&\geq \limsup_{n\to\infty}\tl{M}_{K,N}^{(t)}[\theta_K(\hat{A}^{(0)}_{\frac{1}{m}},\hat{A}^{(1)}_{\frac{1}{m}})](\xi_{n}(\hat{A}^{(0)}_{\frac{1}{m}}),\xi_{n}(\hat{A}^{(1)}_{\frac{1}{m}}))\\ \nonumber&\geq \liminf_{n\to\infty} \tl{M}_{K,N}^{(t)}[\theta_K(\hat{A}^{(0)}_{\frac{1}{m}},\hat{A}^{(1)}_{\frac{1}{m}})](\xi_{n}(\hat{A}^{(0)}_{\frac{1}{m}}),\xi_{n}(\hat{A}^{(1)}_{\frac{1}{m}}))\\ \nonumber&\geq \liminf_{n\to\infty} \tl{M}_{K,N}^{(t)}[\theta_K(\hat{A}^{(0)}_{\frac{1}{m}},\hat{A}^{(1)}_{\frac{1}{m}})](\xi_{n}(\check{A}^{(0)}_{\frac{1}{m}}),\xi_{n}(\check{A}^{(1)}_{\frac{1}{m}}))\\ \nonumber&\geq 
 \tl{M}_{K,N}^{(t)}[\theta_K(\hat{A}^{(0)}_{\frac{1}{m}},\hat{A}^{(1)}_{\frac{1}{m}})](\liminf_{n\to\infty}\xi_{n}(\check{A}^{(0)}_{\frac{1}{m}}),\liminf_{n\to\infty}\xi_{n}(\check{A}^{(1)}_{\frac{1}{m}}))\\ \nonumber&\geq 
\tl{M}_{K,N}^{(t)}[\theta_K(\hat{A}^{(0)}_{\frac{1}{m}},\hat{A}^{(1)}_{\frac{1}{m}})](\xi(\check{A}^{(0)}_{\frac{1}{m}}),\xi(\check{A}^{(1)}_{\frac{1}{m}}))\\
&\geq\tl{M}_{K,N}^{(t)}[\theta_K(\hat{A}^{(0)}_{\frac{1}{m}},\hat{A}^{(1)}_{\frac{1}{m}})](\xi(\A^{(0)}),\xi(A^{(1)}))\,.
}
Notice that 
\begin{enumerate}
	\item $\theta_K(\hat{A}^{(0)}_{\frac{1}{m}},\hat{A}^{(1)}_{\frac{1}{m}})\leq \theta_K(A^{(0)},A^{(1)})$ if $K>0$.
	\item $\theta_K(\hat{A}^{(0)}_{\frac{1}{m}},\hat{A}^{(1)}_{\frac{1}{m}})\geq \theta_K(A^{(0)},A^{(1)})$ if $K<0$.
	\item  $\lim_{m\to\infty} \theta_K(\hat{A}^{(0)}_{\frac{1}{m}},\hat{A}^{(1)}_{\frac{1}{m}})=\theta_K(A^{(0)},A^{(1)})$.
\end{enumerate}
Recall that if $a\cdot b>0$ and $d_0>0$, then by identity \ref{id:5} of Proposition \ref{prop:properties}, if $K>0$ then 

$\lim_{d\uparrow d_0}\tl{M}_{K,N}^{(t)}[d](a,b)=\tl{M}_{K,N}^{(t)}[d_0](a,b)$ ($+\infty$ if $d=l_{\delta}$), and if $K<0$ then $\lim_{d\downarrow d_0+}\tl{M}_{K,N}^{(t)}[d](a,b)=\tl{M}_{K,N}^{(t)}[d_0](a,b)$ ($0$ if $d=l_{\delta}$).  Since $A^{(t)}=\bigcap_{m=1}^{\infty}\hat{A}^{(t)}_{\frac{1}{m}}$   by measure continuity from above it follows from inequality \eqref{ineq:limMeasures} and these observations  that 
$$\xi(A^{(t)})=\lim_{m\to\infty} \xi(\hat{A}^{(t)}_{\frac{1}{m}}) \geq 
\tl{M}_{K,N}^{(t)}[\theta_K(A^{(0)},A^{(1)})](\xi(A^{(0)}),\xi(A^{(1)}))\,.$$
	\item[Condition 3]: Since $\xi\in \Pkn(\R)$ by \red{Theorem} \ref{thm:NoSingularContinuous} it must be supported on either an interval or a point. In the latter case there is nothing to prove. Assume $supp(\xi)$ contains at least two points; let $x_0, x_1\in supp(\xi)$ s.t. $x_0<x_1$.
	For any $0<\epsilon<\half(x_1-x_0)$ and $i\in\{0,1\}$ it holds that $\xi(I(x_0; \epsilon)),\xi(I(x_1; \epsilon))>0$; by the Portmanteau's equivalences $\lim_{n\to\infty}\xi_n((I(x_i; \epsilon))=\xi(I(x_i; \epsilon))>0$.  Therefore there is $n_\epsilon\in \Nbb$ such that for all $n\geq n_{\epsilon}$ it holds that $\xi_n(I(x_0; \epsilon))\xi_n(I(x_1; \epsilon))>0$; since by \red{Theorem} \ref{thm:NoSingularContinuous} the measures $\xi_n$ and $\xi$ have connected supports, this means that for all $n\geq n_{\epsilon}$: $[x_0+\epsilon,x_1-\epsilon]\subset supp(\xi_n)$, whence $(x_1-x_0)\leq diam(supp(\xi_n))+2\epsilon$. Since this can be shown for any $\epsilon>0$ we conclude that $x_1-x_0\leq D$. But this applies also to  any two such points $x_0,x_1\in supp(\xi)$, hence $diam(supp(\xi))\leq D$.

\end{myitemize}

\end{proof}

%If $x_0, x_1\in supp(\xi)$ such that $|x_1-x_0|=l_{\delta}$ then this condition prevents $x_0,x_1$ being both atoms of the measure, while if $d\xi=Jdm$ this condition, in view of the definition of the distortion coefficients 
%\eqref{dfn:SigmaTau} (where we identify $0\cdot \infty$ as $0$) enforces $J$ to vanish on $x_0$ or $x_1$ (cf. remark \ref{rmk:ODEfacts}, specifically the inequality case). 

%Theorem \ref{thm:NoSingularContinuous} motivates the following definitions
The theorem motivates the introduction of the following definitions:

\begin{defn}[The sub-classes  $\gls{Mknd_ac}$,$\gls{Mknd_Ck}$,$\gls{Mknd_s}$]\label{defn:newFkndSubclasses} Let $I\subset \R$ be a closed interval. We denote by
\begin{enumerate}
	 \item $\Fknd^{ac}(I)$:  The class of absolutely-continuous  measures $\xi=J\cdot m\in \Fknd(I)$. Throughout we will always assume that the density $J$ refers to the representative which is continuous on $int(supp(\xi))$ (as justified by Proposition  \ref{prop:contRepresentative}). 

	\item $\Fknd^{C^{k}}(I)$ ($k\in\N\cup\{\infty\}$):  The class of measures $\xi=J\cdot m\in \Fknd^{ac}(I)$  s.t. $J\in C^k(int(supp(\xi)))$.
 \item $\Fknd^s(I)$: The class of delta-measures   $\xi=\delta_{x_0}\in \Fknd(I)$  where $x_0\in I$. 
\end{enumerate}

\end{defn}
\begin{remk}\label{CDKN_hierarchy}
According to Theorem \ref{thm:NoSingularContinuous}, subject to the proviso that $D<l_{\delta}$ when $\delta>0$ and $K<0$, every measure $\xi\in \Fknd(I)$ belongs to one, and only one, of the classes $\Fknd^{ac}(I)$ and $\Fknd^{s}(I)$. 
Clearly we have the following hierarchy
$$\Fknd^{C^{2}}(I)\subset \Fknd^{ac}(I)\subset \Fknd(I)\,.$$
However, while all members of $\Fknd(I)$ satisfy condition \ref{property:III} of Theorem \ref{thm:Equivalence} by definition, the theorem shows that the subset of measures $\xi=J\cdot m\in \Fknd^{ac}(I)$  satisfy also condition \ref{property:II} of the theorem (and in view of Proposition  \ref{prop:contRepresentative}, as a subset of $\M_b(\R)$ we can identify $\Fknd^{ac}(I)$ with $\Fknd^{C^0}(I)$), while the members of  $\Fknd^{C^2}(I)\subset\Fknd^{ac}(I)$ satisfy in addition condition \ref{property:I}. 
\end{remk}

We identify \eqref{eqn:CD(k,N):1} from Definition \ref{dfn:Z_KND} as the  fundamental property satisfied by the members of $\Fknd^{C^{\infty}}(I)$. This class is fundamental for the needle decomposition applications which were presented in \red{Theorem} \ref{localization results}. Due to Theorem \ref{thm:Equivalence}  the members of this class satisfy also condition \ref{property:II} of Theorem \ref{thm:Equivalence}, hence the former  definition of $\Fknd^{\Cinf}(I)$ in \red{Definition} \ref{dfn:Z_KND} causes  no notational ambiguity with the present definition of the subclass $\Fknd^{C^{k}}(I)$ (Definition  \ref{defn:newFkndSubclasses}) when $k=\infty$. 
\bigskip

Condition \ref{property:I} of Theorem \ref{thm:Equivalence} corresponded to $CD_b(K,N)$ conditions in dimension 1; in view of Theorem \ref{thm:Equivalence} it is reasonable to 
refer to the two other conditions \ref{property:III} and \ref{property:II} of Theorem \ref{thm:Equivalence} (being more general than condition \ref{property:I}) as  '{\bf synthetic}' $CD_b(K,N)$ (resp. $CDD_b(K,N,D)$) conditions \red{on $\R$}.

%We will refer to condition \eqref{eqn:CD(k,N)} as well as \eqref{SyntCDKN} as 'synthetic' $CD(K,N)$ conditions.
%
%
%In addition the class $\Fknd^{C^{\infty}}$, which corresponded to the classical $CDD(K,N,D)$ conditions \ref{property:III} and \ref{property:III} of  matches the class we have previously defined in \ref{dfn:Z_KND}, and therefore there is no notational ambiguity. 

%
%
%
%Throughout the rest of this work we will always assume $J$ to be the continuous-representative. 

%According to . The class $\Fkn(I)$ whose members satisfy \eqref{eqn:CD(k,N)} contains the subclass of absolutely-continuous measures, denoted by $\Fkn^{ac}$, whose densities $J\in L^1(\R)$ satisfy :
%\eql{\label{SyntCDKN} J(x_t)\geq M_{K, N}^{(t)}[|x_1-x_0|](J(x_0),J(x_1))\qquad \forall t\in [0,1]\,, }
%for $[m]$ a.e $x_0,x_1\in \R$ such that $|x_1-x_0|<l_{\delta}$.
%
%
%If $I$ is an interval s.t $Diam(I)\leq l_{\delta}$, 
%the class $\Fkn^{ac}(I)$ contains the subclass $\Fkn^{\Cinf}(I)$ of non-negative measures with densities $J$ such that 
%$(i)$ J is supported inside $I$, $(ii)$ $J$ is smooth on its support and satisfies the classical $CD(K,N)$ condition given by the differential inequality: $-(\log J)''-\frac{1}{N-1}((\log J)')^2\geq K$. 

\bigskip

\subsection{Abstract formulation of the problem}\label{subsec:abstract_form}

Let $\M_b(I)$ denote the space of bounded non-negative Radon measures on $\R$  supported inside a closed interval $I$, and let $\P(I)\subset \M_b(I)$ be the subset of probability measures (when $I=\R$ we will sometimes abbreviate and write $\M_b$ and $\P$ respectively). Given a function $u\in C(\R)$ we define $u^*$ to be the linear functional on $\M(I)$ defined by $u^*(\xi)=\int ud\xi$; this is a well defined bounded functional whenever either $u$ is compactly supported or $I$ is compact. 

Given $A\subset \M_b(I)$ and two such functionals  $u^*, v^*$ which are non-negative on $A$, we define their associated quotient $\Phi_{u^*,v^*}:A\to \R_+\cup \{+\infty\}$ as the non-negative function:
\eq{\gls{Phi_Ray}:=\begin{cases} \frac{u^*(\xi)}{v^*(\xi)} &\mbox{ if } v^*(\xi)\neq 0 \\ 
																			  +\infty &\mbox{ if } v^*(\xi)= 0\,\,. \end{cases}}
 
%$f\in C_c^{\infty}$ we denote by $f^*$ the linear functional on $\M$ defined by $f^*(\xi)= \int fd\xi$. Assume $u,v\in C_c^{\infty}(\R, \R_+)$,

By application of the localization theorem we have shown (Theorems \ref{thm:PoinInequality} and \ref{thm:LogSobolevInequality}) that lower bounds for $\Lambda_{Poi}$ and $\Lambda_{LS}$ are given as solutions $\lamp_{K,N,D}$ and $\rho_{K,N,D}$ (equations \eqref{PoinConstant} and \eqref{LogSobolevConstant}) of an optimization problem of the following general form:
%\eql{\label{const:CC1} \CC:=\inf_{\rho\in \Z_{(k,N,D)}}\inf_{\stackrel{f\in C_c^{\infty}(\R)}{f\bot_{\rho} 1}}  \Ry{f}{\rho} =\inf_{\rho\in \Z_{(k,N,D)}}Poin(\rho)\,, }
\eql{\label{const_alpha_def} \gls{alpha_knd}:=\inf_{0\not\equiv\xi\in \Fknd(\R)}\quad \inf_{\substack{f\in \F_*(\xi)}}  \Phi_{u_f^*,v_f^*}(\xi)\,, }
where $u_f, v_f$ and the function space $\F_{*}$ ($*$=`$Poi$', `$Poi^{(p)}$' or `$LS$') are defined according to the following assignments:

{\small 
\begin{align}\label{dfn:specifications} \qquad\\ \nonumber
  &u_f(x):=f'(x)^{p},\, & v_f(x):&=|f(x)|^p,\,  & h_f(x):&=f(x)|f(x)|^{p-2} &\quad\mbox{   in problem } \eqref{PoinConstant}\,,	\\ \nonumber
	%  \\ \nonumbe
	\qquad\\ \nonumber
  &u_f(x):=2f'(x)^{2},\, & v_f(x):&=f(x)^2\log f(x)^2,\, & h_f(x):&=f(x)^2-1 &\quad\mbox{   in problem} \eqref{LogSobolevConstant}\,,
	%:= & h_f:= & c=1\,.
\end{align}

\[ \gls{F_fun}=\{  0\not\equiv f\in \Cinf(\R):\, \, h_f\in C_c(\R) \text{ and } h_f^*(\xi)=0\}\,. \]
}
\smallskip

%where \phi_{[-\frac{D}{2},\frac{D}{2}]}, $\phi_{[-\frac{D}{2},\frac{D}{2}]}\in \Cinf(\R)$ is a compactly supported extension of the index function $1_{[-\frac{D}{2},\frac{D}{2}]}$.

%If for any $f\in \Cinf_c(\R)$ the following condition holds:
%$$\prnt{supp(u_f)\cup supp(v_f)}\subset conv(supp(\bar{h}_f)) \qquad\mbox{(as in  \eqref{dfn:specifications})}\,,$$
 %then $conv\prnt{ supp(u_f)\cup supp(v_f)\cup supp(\bar{h}_f) }=conv(supp(\bar{h}_f))$. 
\begin{remk}\label{remk:uvh_prop} The following crucial properties are satisfied:
\begin{itemize}
	\item $u_f, v_f, h_f\in C_c(\R)$ by definition of our function space.
	\item We can identify 
	$$\F_*(\xi)=\{f\in \F_*^a: \,\, h_f^*(\xi)=0\,\}\,,$$
	where $\F^a_*$ is an auxiliary function space  (independent of $\xi$) defined by
\[ \gls{F_fun_aux}:=\{  0\not\equiv f\in \Cinf(\R):\, \, h_f\in C_c(\R)\}\,. \]

	%Considering the definition of the function spaces $\F_*(\xi)$, we can identify them with:
%\[ \{\text{const}\not\equiv  f\in \Cinf(\R):\, \, h_f\in C_c(\R) \text{ and } h_f^*(\xi)=0\}\,. \]
	\item $supp(u_f)\subset supp(v_f)$.
	\item $supp(v_f)\subset supp(h_f)$. 
	\item If $h_f^*(\xi)=0$ then $v_f^*(\xi)\geq 0$. Indeed, in the log-Sobolev problem, although $v_f(x)$ is not a non-negative function,  by Jensen's inequality applied to $\bar{\xi}:=\frac{1}{\xi(1)}\xi$ (considering that $h_f^*(\bar{\xi})=0$): $$v_f^*(\bar{\xi})=\bar{\xi}(f^2\log f^2)\geq \bar{\xi}(f^2)\log\prnt{\bar{\xi}(f^2)}\stackrel{h_f^*(\xi)=0}{=} 0\,,$$
	hence $v_f^*(\xi)\geq 0$. 
	%
	%apply Jensen's inequality to the normalized measure $\frac{1}{\xi(1)}\xi$)
	%. Justification for that is necessary only for the log-Sobolev problem, if $\xi\in \Fknd(\R)$ is s.t $h_f^*(\xi)=0$, set , then since $\bar{\xi}(f^2)=1$ \,.$$ 
\end{itemize}
\end{remk}
\bigskip

%We define an auxiliary function space $\F^a_*$ defined as follows:
%\[ \F^a_*:=\{\text{const}\not\equiv  f\in \Cinf(\R):\, \, h_f\in C_c(\R)\}\,. \]
 %Notice that $\F_*(\xi)=\F_*^a\cap \{h_f^*(\xi)=0\}$.

By Remark \ref{remk:uvh_prop} it follows that $$supp(u_f)\cup supp(v_f)\subset conv(supp(h_f))\,,$$ hence 
by exchanging the minimization order we may equivalently reformulate the problem \eqref{const_alpha_def} as follows:

\eql{\label{const_alpha_def_equiv} \CC=\inf_{f\in \F^a_*}\,\,\inf_{\substack{\xi\in \Fknd\prnt{conv(supp(h_f))}\\ \text{s.t. }h_f^*(\xi)=0}}\Phi_{u_f^*,v_f^*}(\xi)\,.}
By 0-homogeneity   this is equivalent to 
\eql{\label{const_alpha_def_equiv2} \CC=\inf_{f\in \F^a_*}\,\,\inf_{\substack{\xi\in \Pknd\prnt{conv(supp(h_f))}\\ \text{s.t. }h_f^*(\xi)=0}}\Phi_{u_f^*,v_f^*}(\xi)\,.}

Now in this equivalent form we have two consecutive minimization problems, the inner is on measures and the outer on functions.
Our first important result is a reduction of the class of measures in the inner problem from $\Fknd$ to the class $\Fknd^M$. We approach the problem via functional analytic methods, in particular identification of extreme points and application of the Krein-Milman theorem (or an equivalence of the theorem). 
%We will conclude that rather optimizing over the class $\Fknd(conv\prnt{supp(f)})$, it is sufficient to consider the sub-class of model densities $\Fknd^M(conv\prnt{supp(f)})$.
 Then, after exchanging the order of minimization again we face the following simpler problem: 
\eql{\label{Find_alpha}\text{ Find : } \quad \inf_{\xi\in \Fknd^M(\R)}  \inf_{f\in \F_*(\xi)}\Phi_{u_f^*,v_f^*}(\xi)\,. } 

A solution to this problem, by a finer refinement over the model class, will be approached by other methods, which will depend on the specific functional inequality that we wish to establish.

 \section{Extreme points - classification and applications}
\subsection{The role of extreme points} \label{subsec:role_extreme}

% We assume the topology on $X$ is the $\tau_{\F_w}$ topology induced by a set of linear functional $\F_w$ which separates points of $X$ (i.e $\forall x_1\neq x_2\in X$ there is $f\in \F_w$ s.t $f(x_1)\neq f(x_2)$). The basis for this topology is formed by sets of the form $B_{x_0, A_{\F_w}, \epsilon}:=\{x\in X:\, |f(x)-f(x_0)|<\epsilon,\, \forall f\in A_{w}\}$, where $x\in X$, $A_w $ is a finite subset of $\F_w$, and $\epsilon>0$.  It is known that $(X,\tau_{\F_w})$ is a locally convex vector space, whose topology is Hausdorff.
%One important example is $X$ with the 'weak' topology induced by $\F_w=X^*$, where $X^*$ is the dual space of a normed space $X$; another important example is    $X=L^*$, for some linear space $L$, with the 'weak*' topology induced by $\F_w=L$.

When we consider optimization problems over subsets $A\subset X$, convexity of $A$ is an extra piece of information which can significantly  facilitate the solution due to many important results which rely on this property. In particular, the problem is greatly simplified once we know what are the extreme points of the convex set.

To this end we firstly recall the definition of extreme points; we provide a definition for general sets which are not necessarily convex. 
\begin{defn}[Extreme Points] We say that a point $x$ in a set $A$ of a real vector space is an extreme point of $A$, denoted by $x\in \gls{Ext_poi}$ if $x$ does not belong to the relative interior of any non-degenerate line segment $[y_0,y_1]\subset A$, where 
$$ [y_0,y_1]=\{\convar{y_0}{y_1}:\,\, t\in [0,1]\}\,.$$
Equivalently $x\in \Ext{A}$ if and only if the following implication holds: \\
$$\mbox{If $[y_0,y_1]\subset A$ and $x=\half y_0+\half y_1$ then $x=y_0=y_1$\,. }$$
 %for some $t\in (0,1)$
%under the same assumptions
%$x\in \Ext{A}$ iff whenever  for $y_0,y_1\in A$ then $x=y_0=y_1$. 
\end{defn}

\bigskip

We will need the following known facts \cite[p.74]{Geo} (see also \cite{AK}): 
%locally convex implies separates points
\begin{thm}\label{thm:KM} Assume $A$ is a non-empty compact convex subset of a locally convex Hausdorff t.v.s. $X$. Denote by $\Ext{A}$ the set of extreme points of $A$. Then:
\begin{enumerate}
	\item(Krein-Milman) $A=\overline{conv}(\Ext{A})$.
	\item(Milman) If $B\subset A$, then the following conditions are equivalent:
	\begin{itemize}
		\item $\overline{conv}(B)=A$.
		\item $\Ext{A}\subset \bar{B}$. 	
	\end{itemize}
	%Hitchiker's in analysisp.175
	\item(Bauer)\label{Bauer}   $\min\{\phi(x):\,x\in \Ext{A}\}=\min\{\phi(x):\,x\in A\}$ for any $\phi\in X^*$.
\end{enumerate}
\end{thm}

The following proposition can be considered as an extension of \ref{Bauer} above to non-linear functions on $\M_b$ of the form $\Phi_{u^*,v^*}(\xi)$ .
\begin{prop}\label{thm:infExtreme} Assume $A\neq \emptyset$ is a $w*$-compact convex set in $\M_b$, and that $u^*(\xi),v^*(\xi)\geq 0$ for all $\xi\in A$. Define
$\lambda:=\inf_{A} \Phi_{u^*,v^*}(\xi)$ and $\lambda_{\mathfrak{E}}:=\inf_{\Ext{A}} \Phi_{u^*,v^*}(\xi) $. Then 
\[ \lambda=\lambda_{\mathfrak{E}}\,. \]
\end{prop}

\begin{proof}
Clearly $\lambda \leq \lambda_{\mathfrak{E}}$, therefore it is enough to prove $\lambda_{\mathfrak{E}}\leq \lambda$. If $\lambda =\infty$ there is nothing to prove. If $\lambda_{\mathfrak{E}}=\infty$ then by Bauer's theorem (see Theorem \ref{thm:KM}) $0=\max_{\xi\in \Ext{A}}v^*(\xi)=\max_{\xi\in A}v^*(\xi)$, whence $\lambda=\infty$. We will thus assume $\lambda_{\mathfrak{E}}<\infty$ . Assume by contradiction that $\lambda<\lambda_{\mathfrak{E}}  $. We define $\phi_{\lambda_{\mathfrak{E}}}(\xi)$ to be the continuous linear functional defined by
\[ \phi_{\lambda_{\mathfrak{E}}}(\xi)=u^*(\xi)-\lambda_{\mathfrak{E}}v^*(\xi) \,.\]
By definition of $\lambda$, for any $\epsilon>0$ there exists $\xi_{\epsilon}\in A$ such that $\Phi_{u^*,v^*}(\xi_{\epsilon})<\lambda +\epsilon$.  For sufficiently small $\epsilon$ we can assume $\Phi_{u^*,v^*}(\xi_{\epsilon})<\lambda_{\mathfrak{E}}$ whence $\min_{\xi\in A}\phi_{\lambda_{\mathfrak{E}}}(\xi)\leq \phi_{\lambda_{\mathfrak{E}}}(\xi_{\epsilon})<0$. By Bauer's theorem $\min_{A}\phi_{\lambda_{\mathfrak{E}}}(\xi)=\min_{\Ext{A}}\phi_{\lambda_{\mathfrak{E}}}(\xi)$, hence there exists $\xi_*\in \Ext{A}$ such that $u^*(\xi_*)-\lambda_{\mathfrak{E}} v^*(\xi_*)<0$; since $u^*$  is non negative on $A$ and $\lambda_{\mathfrak{E}}> 0$ (since by assumption $\lambda_{\mathfrak{E}}>\lambda \geq 0$), it follows that $v^*(\xi_*)>0$. This implies that $\Phi_{u^*,v^*}(\xi_*)<\lambda_{\mathfrak{E}}$, which is a contradiction to the definition of $\lambda_{\mathfrak{E}}$. 
\qedhere
\end{proof}

\begin{prop}\label{prop:infExtreme} Let $A\neq \emptyset$ be a $w*$-compact convex set in $\M_b$, and let $B\subset A$ be a closed subset of $A$ such that $A=\overline{conv}B$. Then 
\begin{enumerate}
	\item $\Ext{A}\subset \Ext{B}$.
	\item $\inf_{B} \Phi_{u^*,v^*}(\xi)=\inf_{\Ext{B}} \Phi_{u^*,v^*}(\xi)=\inf_{\Ext{A}} \Phi_{u^*,v^*}(\xi)$.
\end{enumerate}
\end{prop}   

\begin{proof}
 By Milman's theorem (2 of Theorem \ref{thm:KM}) $\Ext{A}\subset \bar{B}=B$, therefore $\Ext{A}\cap B=\Ext{A}$. Since $B\subset A$ it holds that $\Ext{A}\cap B \subset \Ext{B}$. We can thus conclude that $\Ext{A}\subset \Ext{B}$.
	The second statement is a straightforward consequence of the following inequalities
\eql{\label{ext_pts_ineq} &\inf_{A} \Phi_{u^*,v^*}(\xi)\stackrel{Prop. \ref{thm:infExtreme}}{=}
\inf_{\Ext{A}} \Phi_{u^*,v^*}(\xi)
\stackrel{\Ext{A}\subset \Ext{B}}{\geq} \inf_{\Ext{B}} \Phi_{u^*,v^*}(\xi)\\ \nonumber&\geq 
\inf_{B} \Phi_{u^*,v^*}(\xi)\geq \inf_{A} \Phi_{u^*,v^*}(\xi)\,.}
\qedhere
\end{proof}

 \subsubsection{A digression into the theory of cones}

We make a brief digression into the general theory of cones. The goal of this digression is to present another interpretation for our results; yet, we stress that we will not use it for the derivation of the main results, but rather consider it merely as an alternative form of stating results. 
We refer the reader to \cite{Phe,Bar, Bou} for further details about the theory of cones and their extreme rays.  The cone viewpoint is motivated by the simple fact that the function $\Phi_{u^*,v^*}(\xi)$ is $0$-homogeneous in $\xi$, hence its value is the same for all members of the class $\{ s \xi:\,s>0\}$. 

Assume $X$ is a locally convex Hausdorff t.v.s. A set $K\subset X$ is called a cone if $0 \in K$ and $\lambda x\in K$ for every $\lambda\geq 0$ and every $x\in K$. $K$ is a convex cone if in addition for any two points $x,y\in K$, and any two numbers $\alpha,\beta\geq 0$, it holds that $z=\alpha x+\beta y\in K$. We will only consider proper cones, i.e. $K\cap (-K)=\{0\}$. 

Given points $x_i\in X$ and $\alpha_i\in \R_+$, where $i=1,...,m$, the point $x=\sum_{i=1}^m\alpha_ix_i$ is called a conic combination of the points $x_1,..,x_m$. The set $co(A)$ of all conic combinations of points from a set $A\subset X$ is called the conical-hull of the set $A$. It can be identified as the smallest cone which contains $A$. A ray is defined as the conic-hull of a single non-zero point, i.e. it is a set of the form $\gls{conic_hull}=\{\lambda x\,:\, \lambda\geq 0\}$ where $x\in K\setminus \{0\}$.

%
%We can define on $X$ an equivalence relation $\sim$ as follows : $x\sim y$ if there is $\lambda>0$ such that $x=\lambda y$. We denote by $X/ \sim$ the quotient space, and by $\pi:X\to X/ \sim$ the quotient mapping ; it is equipped with the strongest topology such that the quotient mapping $\pi$ is continuous. 
%\begin{defn} A subset $A$ of $X$ is said to be $\pi$-convex if $\pi(A)$ is of the form $\pi(conv(A_1))$ for some subset $A_1$ of $X$. 
%\end{defn}
%
%\begin{defn} A subset $A$ of $X$ is said to be $\pi$-compact if the image $\pi(A)$ under the quotient map $\pi$ is compact in the quotient space $X/ \sim$ . 
%\end{defn}
%\noindent This terminology is picked from \ref{thm:ConesClosed}, where the following theorem is also proved.
%\begin{thm}[\cite{Kom}]\label{thm:ConesClosed} If $A$ is a $\pi$-convex subset of $X$, such that the set $X\setminus \{0\}$ is $\pi$-compact, then the conical hull of $X$ is closed. 
%\end{thm}

We denote by $\text{Rays}(K)$ the set of rays of $K$. A ray $R$ of $K$ is said to be an extreme ray of $K$ if whenever $u\in R$ and $u=\half(x+y)$, where $x,y\in K$, then $x,y\in R$. We define $\gls{Ext_ray}$ to be the set of extreme rays of $K$. 

\begin{defn}[Bases]
A set $B\subset K$ is called a base of $K$ if $0\notin B$ and for every point $u\in K$, $u\neq 0$, there is a unique representation $u=\lambda v$, with $v\in B$ and $\lambda>0$. 
\end{defn}
\noindent We refer the reader to \cite{Bar} (lemma 2.10 on p.116) for the proof of the following theorem :
\begin{thm}\label{thm:ClosedCone}  If $A\subset X$ is a compact convex set such that $0\notin A$, then $K=co(A)$ is a closed convex cone.
\end{thm}
\begin{defn}[\cite{Phe,Bou}] A non-empty set $C$ of a closed convex cone $K$ is called a cap of $K$ provided $C$ is compact, convex and $K\setminus C$ is convex. 
\end{defn}
We relate this to our previous discussion. We will be interested in a cone $K\subset \M_b(I)$ defined by $K=co(A)$, where $A=\overline{conv(B_1)}$ for a subset of probability measures $B_1\subset \P$, and the closure is w.r.t. the $w*$-topology. 
We take $A$ as a base, and the sets $C_n=[0,n]\cdot A$ $\,(n\in \Nbb)$ are caps (since $K\setminus C$ is convex due to that $K$ is defined as $co(A)$, and for $\xi\in K\setminus C_n$ the total mass is greater than n). Evidently $K=\bigcup_{n\in \Nbb}C_n$; in such a situation the following cone version of the Krein-Milman theorem holds \cite{Phe}:
\begin{thm} Suppose that $K$ is a closed convex cone of $X$, and $K$ is the union of its caps, then $K$ is the closed convex hull of $\cup \Exr{K}$.
\end{thm}
It is natural to expect that extreme rays will be closely related to the extreme points of the generating set. Indeed
such a relation is given by the following theorem \cite{Phe}:
\begin{prop}\label{prop:exRaysPoint}
 Let $K$ be a convex cone with a base $A$, and let $u\in K\setminus \{0\}$. Then $u$ spans an extreme ray of $K$ if and only if $u=\lambda v$ where $\lambda>0$ and $v$ is an extreme point of $A$. Thus $\Ext{A}=A\cap \prnt{\cup \Exr{K}}$. 
\end{prop}
\begin{remk}\label{con:app}
The cone we have previously defined was constructed with a base $A=\overline{conv(B_1)}$. In Proposition  \ref{prop:infExtreme} we showed $\Ext{A}\subset \Ext{B_1}$, but then by Proposition  \ref{prop:exRaysPoint} we conclude that $\Ext{A}=A\cap \prnt{\cup \Exr{K}} \subset \Ext{B_1}$; A is a base hence $A\cap \prnt{\cup \Exr{K}}$ contains at most one point from each ray; we can thus conclude that the points $\Ext{B_1}$ belong to distinct rays. We can conclude by the cone version of the Krein-Milman theorem that $K=\bigcup_n C_n=\overline{co}(\cup \Exr{K})=\overline{co}(\Ext{A})\subset \overline{co}(\Ext{B_1})\subset K$, whence $K=\overline{co}(\Ext{B_1})$. 
\end{remk}

\
 
\subsection{Classification of extreme points}\label{subsec:Classification}

The definitions given in this sub-section are motivated by the identity \eqref{const_alpha_def_equiv2} which we previously obtained for $\CC$. 
\bigskip

Assume $h\in C_c(\R)$ and define 
$$\gls{Interval_h}:= conv(supp(h))\,.$$
%Being the convex-hull of a compact set, $I_{\bar{h}}$ is a compact convex set. Given a fixed $c\in \R$ we also define $h\in \Cinf(I_{\bar{h}})$ by 
%$$h(x):=\bar{h}|_{I_{\bar{h}}}(x)-c\,.$$ 
 Recall that the set $\Pknd(I_{h})$ stands for the set of probability measures $\xi\in \Pknd(\R)$ which are supported inside $I_{h}$, and $\Pknd^M(I_{h})\subset \Pknd(I_{h})$ stands for the subset of model-space probability measures as defined in Definition \ref{dfn:ModelMeasures}. According to Theorem \ref{thm:NoSingularContinuous} (subject to the proviso that $D<l_{\delta}$ when $\delta>0$ and $K<0$)  the class $\Pknd(I_{h})$ admits a partition into 
\begin{itemize}
	\item $\Pknd^{ac}(I_h)$ the absolutely continuous measures $\xi\in \Pknd(I_h)$.
	\item $\Pknd^{s}(I_h)$ the singular measures $\xi\in \Pknd(I_h)$, which must be of the form $\delta_{x_0}$ for some $x_0\in I_h$.
\end{itemize}
We may thus write $\Pknd(I_h)=\Pknd^{ac}(I_h)\coprod\Pknd^{s}(I_h)$.

\begin{defn}[The classes $\Pkndfs(I_h)$ and $\Pkndffs(I_h)$]\label{defn:Pkndf_class} We define 
\begin{itemize}
	\item $\gls{Pknd_h}:=\cprnt{ \xi\in \Pknds(I_h):\,\,h^*(\xi)=0}$;
	\item $\gls{hatPknd_h}:=\cprnt{ \xi\in \Pkndfs(I_h):\,\,\int_{-\infty}^x hd\xi\neq 0 \,\,\forall x\in int(supp(\xi))}$;
\end{itemize}
for $*=$ `ac', `s',`M', or `\,\,\,\,' (the latter stands for no superscript).  
\end{defn}

\begin{prop}\label{prop:PkndCpct} $\Pkndf(I_{h})$ and $\overline{conv}(\Pkndf(I_{h}))$  are $w*$-compact. \end{prop}
Here the closure is taken in the weak (or equivalently, as $I_h$ is compact, weak*) topology.

\begin{proof}[Proof of Proposition \ref{prop:PkndCpct}] In Theorem  \ref{thm:CDKNclosed} we showed that the class $\Pknd(I_{h})$ is $w$-closed. Since $h\in C_b(I_{h})$ the condition $h^*(\xi)=0$ is closed under weak convergence of measures in $\Pknd(I_{h})$. Therefore the class $\Pkndf(I_{h})$ is also $w$-closed.
Since $I_{h}$ is compact it holds that $$C(I_{h})=C_c(I_{h})=C_b(I_{h})\,,$$ hence the $w$ and $w*$ convergence modes define the same topology (and there is no ambiguity in using the notation $\overline{A}$ for the closure of a set $A$); in particular this implies that $\Pkndf(I_{h})\subset\P(I_{h})$ and $\overline{conv}\Pkndf(I_{h})\subset \P(I_{h})$ are both $w$, and hence $w*$, closed. 

It is known that the set $\P(I_{h})$ is $w*$ compact \cite{Zim} (considering that $I_h$ is compact). For completeness of the argument we justify that. 
Indeed by the Riesz representation theorem we can identify  $\P(I_{h})$  with $\{\xi\in C(I_{h})^*: \,\,\xi(f)\geq 0\,\,\forall 0\leq f\in C(I_{h}) \text{ and } \xi(1)=1\}$;  for each $f\in C(I_{h})$ the sets $A_f:=\{\xi:\,\xi(f)\geq 0\}$ and $A_1:=\{\xi:\, \xi(1)=1\}$ are $w*$ closed (by testing the condition for $w*$  convergence w.r.t. functions $\tl{f},\tl{1}\in C_c(\R)$ which extend the functions $f$ and $1|_{I_h}$). Hence 
\[ \P(I_{h})=\bigcap_{0\leq f\in C(I_{h})}A_f\cap \{\xi:\, \xi(1)=1\}\]
is a $w*$-closed subset of the unit ball of $C(I_{h})^*$. 
It now follows by the Banach-Alaoglu theorem that $\P(I_{h})$ is $w^*$ compact, and hence $w$ compact. Since $\Pkndf(I_{h})$ and $\overline{conv}\Pkndf(I_{h}) $ are $w*$-closed subsets of $\P(I_{h})$, it follows that they are both $w*$-compact.
\end{proof}

 The main result of this chapter is the following theorem, which characterizes  the extreme points of the set $\Pkndf(I_{h})$. Notice that in general it is not a convex set. 

%We have seen that $\Pknd(I_h)$ is $w$ closed; the additional condition $\xi(h)=0$ is closed under $w$ convergence, whence $\Pkndf(I_h)$ is $w$ closed as well. Since $C^*(I_h)=C_b^*(I_h)$ we can identify the $w$-topology with the $w*$-topology and conclude that $\Pknd(I_h)$ (being a $w$ closed subset of the $w$ compact set $\P$) is $w$-compact. 
%%The weak* topology on $\M_b$ is defined by the condition $\xi_n\stackrel{n\to\infty}{\longrightarrow}\xi\Leftrightarrow \int fd\xi_n\to \int fd\xi$ for any $f\in \C_c(\R)$.
%\end{remk}

%Notice that in our problems either $h$ compactly supported or $I_h$ is compact, hence $\Pkndf$ remains a weakly closed subset of $\Pkndf(I_*)$. 
%
%The main result of this chapter is the following theorem:
\bigskip
\begin{thm}[Extreme Points Characterization]\label{thm:ExtremePoints} Assume $h\in C(I_{h})$. Let $K\in\R$, $N\in (-\infty,0]\cup(1,\infty]$, and $D\in (0,\infty]$ subject to the proviso that $D<l_{\delta}$ if $\delta>0$ and $K<0$. Define $A:=\Pkndf(I_{h})$ and $B:=\Pkndff^M(I_{h})\coprod \Pkndf^s(I_{h})$.  The following correspondence holds between $\Ext{A}$ and $B$: 
\begin{enumerate}
	\item When $N\in (-\infty,0]\cup[2,\infty]$,  $\Ext{A}\subset B$.
	\item When $N\in (-\infty,0]\cup(1,\infty]$,  $B\subset\Ext{A}$.
\end{enumerate}
\end{thm}

\begin{remk} This shows that for $N\in (-\infty,0]\cup[2,\infty]$ it holds that $\Ext{A}=B$; we suspect that this statement is valid for $N\in (-\infty,0]\cup(1,\infty]$. However, the current proof of the theorem relies on the arguments of Fradelizi and Gu\'{e}don \cite{FG}. 
As such, it shares the same limitations, which hinder us from showing that. For the main applications this is not a problem, since by using the localization theorem we apriori assumed that the manifold dimension is at least 2; thus if $N$ is positive it is already bounded below by 2. 
\end{remk} 

%\Chk
%In view of proposition \ref{prop:exRaysPoint} we can straightforwardly conclude the following corollary from the previous extreme points characterization:
%\begin{cor}[Extreme Rays Characterization]\label{thm:ExtremeRays} Assume $h\in C_c^{\infty}(\R)$, $K\in\R$, $N\in (-\infty,0]\cup [2,\infty]$ and $D\in (0,\infty]$ subject to the proviso that $D<l_{\delta}$ if $\delta>0$ and $K<0$. Then a ray $R$ is an extreme ray of $\Fkndf$ if and only if $R=co(x)$, where $x\in \Fkndff^M\coprod \Fkndf^{sin}$ . 
%\end{cor}
%\Chk
%\bigskip 

Before we prove Theorem \ref{thm:ExtremePoints}, we show how the reduction of finding $\CC$ to problem \eqref{Find_alpha} follows from Theorem \ref{thm:ExtremePoints}. The following is the main application of interest in this chapter.
 %or \ref{thm:ExtremeRays}.

%---------------------------------
\begin{cor}\label{cor:CC} Let $K\in \R$, $N\in (-\infty, 0]\cup [2,\infty]$ and $D\in (0,\infty]$. Subject to the proviso: $D<l_{\delta}$ if $\delta>0$ and $K<0$, 
the constant $\CC$ defined in \eqref{const_alpha_def} is given by
\eq{\CC=\inf_{\xi\in \Fknd^M(\R)}\inf_{f\in \F_*(\xi)}\Phi_{u_f^*,v_f^*}(\xi)\,. }

%if for all $f\in \F_*^a$ it holds that $f\mapsto h_f\in C_c(\R)$ and $f\mapsto v_f\in C_c(\R)$ are s.t
%\begin{enumerate}
	%\item[(a)] $v_f^*(\xi)\geq 0$ for any $f\in \Cinf_c(\R)$ and $\xi\in\Fkndf(I_{h_f})$, and
	%\item[(b)] $v_f^*(\xi)=0$ if $\xi\in \Fkndf^s(I_{h_f})$.
%\end{enumerate}
%and if in addition
%\begin{enumerate}
%\item[(c)] The functions space $\F_*(\xi)$ can be identified with $\F^a_*\cap \{h_f^*(\xi)=0\}$.
%\end{enumerate}
%Then  

%Let $K\in \R, N\in (-\infty, 0]\cup [2,\infty]$ and $D\in (0,\infty]$ subject to the proviso that $D<l_{\delta}$ if $\delta>0$ and $K<0$. Let $u_{(\cdot)}, v_{(\cdot)}, h_f_{(\cdot)}: \Cinf_c(\R)\to \Cinf_c(\R)$ and $h_{(\cdot)}:=h_f_{(\cdot)}-c$, where $c\in \R$ is some fixed constant, be s.t for any $f\in \Cinf_c(\R)$:
%
%%f\mapsto h_f$ and $f\mapsto v_f(\xi)$ are 
%\begin{itemize}
  %\item $supp(u_f)\cup supp(v_f)\subset conv(supp(h_f))$. 
	%\item $u_f^*(\xi),v_f^*(\xi)\geq 0$  for all $\xi\in \Pkndhf(I_{h_f})$.
	%\item $\xi\in \Pkndhf^s\,\, \text{(a delta measure)}\Rightarrow v^*_f(\xi)=0$.
%\end{itemize} 

 %$$\CC= \inf_{\rho\in \Z^M_{(k,N,D)}}Poin(\rho)\,.$$
%\qquad \text{where } \qquad Poin(\rho)\triangleq \inf_{\stackrel{f\in C_c^{\infty}(\R)}{f\bot_{\rho} 1}}\Ry{f}{\rho}\,.$$ 
\end{cor}
\begin{proof} 
Recall that according to \eqref{const_alpha_def_equiv2}:
\eq{\CC=\inf_{f\in \F^a_*}\,\,\inf_{\xi\in \Pkndhf(I_h)}\Phi_{u_f^*,v_f^*}(\xi)\,.}
Notice that the following hold:
\begin{enumerate}
	\item By Proposition \ref{prop:PkndCpct} the sets $\Pkndhf(I_{h_f})$ and $\overline{conv}\Pkndhf(I_{h_f})$ are $w*$-compact.
	\item\label{step:2} By Proposition \ref{prop:infExtreme} (applied with $B=\Pkndhf(I_{h_f})$):
	$$\inf_{\Pkndhf(I_{h_f})} \Phi_{u^*,v^*}(\xi)=\inf_{\Ext{\Pkndhf(I_{h_f})}} \Phi_{u^*,v^*}(\xi)\,.$$
	\item\label{step:3} From Theorem \ref{thm:ExtremePoints} it follows that $$\Ext{\Pkndhf(I_{h_f})}=\Pkndhff^M(I_{h_f})\coprod \Pkndhf^s(I_{h_f})\,.$$
  \item \label{step:4} By definition $\Phi_{u_f^*,v_f^*}(\xi)=\infty$ whenever $v_f^*(\xi)=0$, therefore  by \ref{step:2} and \ref{step:3}
\eql{\label{eql:Model} \inf_{\xi\in \Ext{\Pkndhf(I_{h_f})}}\Phi_{u_f^*,v_f^*}(\xi) = \inf_{\xi\in \Pkndhff^M(I_{h_f})}\Phi_{u_f^*,v_f^*}(\xi)
\,. }
\end{enumerate}
%, . However if $\xi\in\Pkndhf^s(I_{h_f})$ then by assumption  whence . Therefore in view of theorem \ref{thm:ExtremePoints} 

In view of these observations we consider the following series of inequalities
{\small 
\eq{&\inf_{\xi\in \Pkndhf^M(I_{h_f})}\Phi_{u_f^*,v_f^*}{(\xi)}\stackrel{\Pkndhf^M(I_{h_f})\subset \Pkndhf(I_{h_f})}{\geq}\inf_{\xi\in \Pkndhf(I_{h_f})}\Phi_{u_f^*,v_f^*}{(\xi)}\\
&\stackrel{by\,\,\ref{step:2}}{=}\inf_{\xi\in \Ext{\Pkndhf(I_{h_f})}}\Phi_{u_f^*,v_f^*}{(\xi)}\nonumber
\stackrel{by\,\,\ref{step:4}}{=}\inf_{ \xi\in\Pkndhff^M(I_{h_f})}\Phi_{u_f^*,v_f^*}{(\xi)}\\& \stackrel{\Pkndhff^M(I_{h_f})\subset \Pkndhf^M(I_{h_f})}{\geq} \inf_{\xi\in \Pkndhf^M(I_{h_f})}\Phi_{u_f^*,v_f^*}{(\xi)}\,. } }
%In view of \ref{prop:exRaysPoint} we can identify the points of $\Pkndff^M$ as points lying on extreme rays  

%\stackrel{\Pkndhf^M(I_{h_f})\subset \Pkndhf}{\geq}
We conclude that 
$$\inf_{\xi\in \Pkndhf(I_{h_f})}\Phi_{u_f^*,v_f^*}{(\xi)}=\inf_{\xi\in \Pkndhff^M(I_{h_f})}\Phi_{u_f^*,v_f^*}{(\xi)}=\inf_{\xi\in \Pkndhf^M(I_{h_f})}\Phi_{u_f^*,v_f^*}{(\xi)}\,.$$ 

Now the expression \eqref{const_alpha_def} reduces to
%\eqref{const:CC1} 
\eql{\label{CC_expression}\CC&=\inf_{f\in \F^a_*}\,\inf_{\xi\in \Pkndhf(I_{h_f})} \Phi_{u_f^*,v_f^*}(\xi)=
\inf_{f\in \F^a_*}\, \inf_{\xi\in \Pkndhf^M(I_{h_f})}\Phi_{u_f^*,v_f^*}{(\xi)}\\ \nonumber&=\inf_{f\in \F^a_*}\, \inf_{\substack{\xi\in \Pknd^M(I_{h_f})\\ \nonumber h_f^*(\xi)=0}}\Phi_{u_f^*,v_f^*}{(\xi)}=
 \inf_{\xi\in \Pknd^M(\R)}\,\inf_{\substack{f\in \F^a_*\\ \nonumber h_f^*(\xi)=0}}\Phi_{u_f^*,v_f^*}{(\xi)}
\\ \nonumber&
=\inf_{\xi\in \Pknd^M(\R)}\,\inf_{f\in \F_*(\xi)}\Phi_{u_f^*,v_f^*}{(\xi)}=\inf_{\xi\in \Fknd^M(\R)}\,\inf_{f\in \F_*(\xi)}\Phi_{u_f^*,v_f^*}{(\xi)} \,.
}
\end{proof}

\subsection{Proving Theorem \ref{thm:ExtremePoints}}
We precede the proof with two lemmas.  Throughout we assume $D<l_{\delta}$ if $\delta>0$ and $K<0$. 
We remind the reader that whenever $\xi=J\cdot m\in \Fknd^{ac}(\R)$, then by Proposition \ref{prop:contRepresentative} we may assume
 $J$ is a continuous representative of $\xi$ (i.e. $J$ is continuous on $int(supp(\xi))$) which satisfies the following condition:  
%for all $x_0,x_1\in int(supp(\xi))$ and $t\in (0,1)$
	%	For all $x_0,x_1\in \R$ such that $|x_1-x_0|<l_{\delta}$ and 
	\eql{ \label{AgainSyntCDKN} J(x_t)\geq M_{K, N-1}^{(t)}[|x_1-x_0|](J(x_0),J(x_1))\,, \qquad \forall x_0,x_1\in \R,\quad t\in [0,1]\,. }
	Recall (Definition \ref{defn:vee}) that given a function $f(x)$, we define \blue{$f_{\vee}:=\max\{f, 0\}$}.

%, an assumption which is justified by Lemma \ref{lem:contRepresentative}.  
%\begin{cor}[Extreme Rays Characterization]\label{thm:ExtremeRays} Assume $f\in C_c^{\infty}(\R)$, $k\in\R$, $N\in (-\infty,0]\cup [2,\infty]$ and $0<D_{\delta}$, then a ray $R$ is an extreme ray of $\M_{(k,N,D);f}$ if and only if $R=co(x)$, where $x\in \Pkndff^M\bigcupdot \Pkndf^s$ . 
%\end{cor}

\begin{lem}\label{lem:FG1} Assume $\nu_1=J\cdot m\in \Fknd^{ac}(\R)$ and $\nu_0=J_0\cdot m\in \Fknd^M(\R)$ are both supported on an interval $I$. If $N\in (-\infty,0]\cup [2,\infty]$ then
\blue{$\nu:=(J-J_0)_{\vee}\cdot m\in \Fknd(\R)$}. 
\end{lem}
\begin{remk}\label{rmk:J_1_minus_J_0}
Notice that the set $I:=\{(J_1-J_0)> 0\}$ must be an interval, since if $x_0<x_1$ are two points in $I$, then 
$$J_1(x_t)\geq M_{K,N-1}[|x_1-x_0|](J_1(x_0),J_1(x_1)) > M_{K,N-1}[|x_1-x_0|](J_0(x_0),J_0(x_1))=J_0(x_t)\,,$$
hence $x_t\in I$ as well. 
\end{remk}
\begin{proof}[Proof of Lemma \ref{lem:FG1}]
Except for the case $\delta>0$ and $K>0$ it clearly holds that $|x_1-x_0|<l_{\delta}$; however when $\delta>0$ and $K>0$, whenever $|x_1-x_0|\geq l_{\delta}$, according to Lemma \ref{lem:CondImplic}, it holds that $J(x_0)J(x_1)=0$ and $J_0(x_0)J_0(x_1)=0$, and the statement clearly holds. Thus throughout we assume $|x_1-x_0|<l_{\delta}$. 

Recall the reverse Minkowski inequality \cite[p. 31]{HaLi} : if $0\neq p\leq 1$ and $a_0,b_0,a_1,b_1>0$ then 
%\[ \prnt{ t_0(a_0+b_0)^p+t_1(a_1+b_1)^p}^{\frac{1}{p}}\geq 
\[  \prnt{(a_0+a_1)^{p}+(b_0+b_1)^{p}}^{\frac{1}{p}}\geq \prnt{ a_0^{p}+ b_0^{p}}^{\frac{1}{p}}+\prnt{ a_1^{p}+ b_1^{p}}^{\frac{1}{p}}\,.\]
\red{Let us firstly assume that} $N\neq \infty$, then since $\frac{1}{N-1}\leq 1$ this inequality implies that 
 for all $x_0,x_1\in \{(J-J_0)>0\}$ :
\eq{(J-J_0)(x_t)&\geq  M^{(t)}_{K,N-1}[|x_1-x_0|](J(x_0),J(x_1))-M^{(t)}_{K,N-1}[|x_1-x_0|](J_0(x_0),J_0(x_1))\\&
\geq  M^{(t)}_{K,N-1}[|x_1-x_0|]((J-J_0)(x_0),(J-J_0)(x_1))\,. }

Since this holds for all $N\geq 2$, this conclusion extends to the case $N=\infty$ by considering the limit $N\to \infty$.  
\end{proof}

\begin{lem}\label{lem:FG2} Assume $I_1=[x_0,z],I_2=[z, x_1]$ (or $I_2=[x_0,z]$ and $I_1=[z, x_1]$) are two adjacent intervals which share a single point $\{z\}$, and let $I=I_1\cup I_2$. Let $\nu_1=J_1\cdot m \in \Fknd(I)$ and $\nu_2=J_2\cdot m\in \Fknd^M(I_2)$ where $K\in \R$ and $N\in (-\infty,0]\cup (1,\infty]$. Assume $supp(\nu_1)=I$ and $supp(\nu_2)=I_2$.  If
\begin{enumerate}
	\item $J_1>0$ on $\mathring{I}$, and 
	\item $0<J_2\leq J_1$ on $\mathring{I}_2$, and
	\item at the common point $\{z\}=I_1\cap I_2$:  $J_2(z)=J_1(z)$\,.
\end{enumerate}
Then the measure $\nu:=\tilde{J}\cdot m$, where $\tilde{J}|_{I_1}=J_1$ and $\tilde{J}|_{I_2}=J_2$, belongs to the class $\Fknd(I)$. 
\end{lem}
\begin{proof} We prove the lemma for the case $I_2=[z,x_1]$, where the case $I_2=[x_0,z]$ follows mutatis-mutandis. 
Let $x\in \mathring{I}$. We fix $x_0', x_1' \in \mathring{I}$ s.t. $x_0'<z,x<x_1'$. For some $t\in (0,1)$ it holds that $x=x'_t=\convar{x'_0}{x'_1}$. If $x\leq z$ then since $\nu_1\in\Fknd(I)$: 
\eq{\tilde{J}(x)=\tilde{J}(x'_t)&=J_1(x_t')\geq M_{K, N-1}^{(t)}[|x_1'-x_0'|](J_1(x_0'),J_1(x_1'))\\&\geq M_{K, N-1}^{(t)}[|x_1'-x_0'|](J_1(x_0'),J_2(x_1'))=M_{K, N-1}^{(t)}[|x_1'-x_0'|](\tilde{J}(x_0'),\tilde{J}(x_1'))\,,}
where the second inequality is due to \ref{id:3} of Proposition \ref{prop:properties}. 

If $x=x_t'>z$ then for some $r\in (0,1)$ it holds that $z=x_r'=(1-r)x_0'+r x_1' $, and there is a unique $s_t\in (0,1)$ such that $x_t'=(1-s_t)z+s_t x_1'$. We can thus conclude again from property \ref{id:3} of Proposition \ref{prop:properties}:
{\footnotesize
\eq{
\tilde{J}(x_t')&=J_2((1-s_t)z+s_t x_1')\\&\geq M_{K, N-1}^{(s_t)}[|x_1'-z|](J_2(z),J_2(x_1'))\qquad \prnt{\text{since  } \nu_2\in\Fknd(I_2)}\\&
=M_{K, N-1}^{(s_t)}[|x_1'-z|](J_1(z),J_2(x_1'))\qquad \prnt{\text{since  } J_1(z)=J_2(z)}\\ &\geq M_{K, N-1}^{(s_t)}[|x_1'-z|]\prnt{M_{K, N-1}^{(r)}[|x_1'-x_0'|](J_1(x_0'),J_1(x_1')),J_2(x_1')}\qquad \prnt{\text{since  } \nu_1\in\Fknd(I)}
\\&\geq M_{K, N-1}^{(s_t)}[|x_1'-z|]\prnt{M_{K, N-1}^{(r)}[|x_1'-x_0'|](J_1(x_0'),J_2(x_1')),J_2(x_1')}\qquad \prnt{\text{property \ref{id:3} of Proposition \ref{prop:properties}}}\,.} }

Define
\eq{ F_1(t)&:=\prnt{M_{K, N-1}^{(t)}[|x_1'-x_0'|](J_1(x_0'),J_2(x_1'))}^{\frac{1}{N-1}}\,,\\
     F_2(s)&:=\prnt{ M_{K, N-1}^{(s)}[|x_1'-z|]\prnt{F_1(r)^{N-1},J_2(x_1')}}^{\frac{1}{N-1}}\,. }

		 Notice that $F_1$ and $F_2$ are respectively solutions to the following boundary value problems: 
		 {\footnotesize
\eq{ &F_1''(t)+\frac{K}{N-1}|x_1'-x_0'|^2F_1(t)=0 \qquad F_1(r)=\prnt{M_{K, N-1}^{(r)}[|x_1'-x_0'|](J_1(x_0'),J_2(x_1'))}^{\frac{1}{N-1}}\,,\quad F_1(1)=J_2(x_1')^{\frac{1}{N-1}}\,,\\
 &F_2''(s)+\frac{K}{N-1}|x_1'-z|^2F_2(s)=0 \qquad F_2(0)=F_1(r)\,\quad F_2(1)=J_2(x_1')^{\frac{1}{N-1}} \,,} }
on $(r,1)$ and $(0,1)$ respectively. 
Hence (see also Remark \ref{rmk:ODEfacts}) $F_2$ is a reparametrization of $F_1$ (they represent the same function $F(x)$ which satisfies on $(z,x_1')$ the ODE $F''+\frac{K}{N-1}F_1=0$); thus for all $s_t\in (0,1)$ such that $(1-s_t)z+s_tx_1'=\convar{x_0'}{x_1'}$ it holds that 
\[M_{K, N-1}^{(s_t)}[|x_1'-z|]\prnt{M_{K, N-1}^{(r)}[|x_1'-x_0'|](J_1(x_0'),J_2(x_1')),J_2(x_1')}=
M_{K, N-1}^{(t)}[|x_1'-x_0'|](J_1(x_0'),J_2(x_1'))\,.\]
We can thus conclude that 
\[ \tilde{J}(x_t')\geq M_{K, N-1}^{(t)}[|x_1'-x_0'|](J_1(x_0'),J_2(x_1'))=M_{K, N-1}^{(t)}[|x_1'-x_0'|](\tilde{J}(x_0'),\tilde{J}(x_1'))\,. \]
Since $x_0', x_1'$ and $x$ are arbitrary, the inequality in \eqref{AgainSyntCDKN}  is satisfied on $\mathring{I}$. By Theorem \ref{thm:Equivalence} it follows that $\nu\in \Fknd(I)$. 
The proof of the case $N=\infty$ follows mutatis-mutandis; considering Remark \ref{rmk:ODEfacts}, one only needs to modify the ODEs satisfied by $F_1$ and $F_2$. 
\qedhere
\end{proof}

\begin{proof}[Proof of Theorem \ref{thm:ExtremePoints}] 

Recall from Theorem \ref{lem:measDeriv} that $$A=\Pkndf(I_{h})=\Pkndf^{ac}(I_{h})\coprod \Pkndf^{s}(I_{h})\,,$$ where $\Pkndf^{s}(I_{h})$ consists exclusively of delta measures supported on a single point $\{x_0\}\in I_h$. Clearly any singular measure $\nu\in \Ext{\Pkndf(I_{h})}$ is in $\Pkndf^{s}(I_{h})$, and any delta measure $\nu\in \Pkndf^{s}(I_{h})$ must be in $\Ext{\Pkndf(I_{h})}$. Hence throughout we consider only a.c. measures. 
\smallskip

Assume $\nu\in \Pkndf^{ac}(I_{h})$ where $\deriv{\nu}{m}=J\in L^1(\R; m)$ and $J$ is continuous and strictly positive on $int(supp(\nu))$ (both are justified by Proposition  \ref{prop:contRepresentative}). 
\begin{enumerate}
	\item
	Assume $\nu=Jdm\in \Ext{\Pkndf(I_{h})}$ is supported on $[a,b]\subset I_h$ ($J>0$ on $(a,b)$). If for some $x\in (a,b)$ it holds that 
$\nu|_{[a, x]}(h)= 0$, then since 
$$\nu|_{[a,x]}(h)=\nu|_{[x,b]}(h)=\nu(h)=0\,,$$
$\nu$ admits the following non-trivial convex decomposition: 
{\small 
\[ \nu=\convar{\nu_1}{\nu_2}\qquad \text{where } \qquad \nu_1=\frac{1}{\nu([a,x])}\nu|_{[a,x]}, \quad \nu_2=\frac{1}{\nu([x,b])}\nu|_{[x,b]} \,\,\text{ and } t=\nu([x,b])\,\, \]}
with distinct $\nu_1, \nu_2\in \Pkndf(I_{h})$; this is clearly inconsistent with the assumption $\nu$ being an extreme point \red{of} $\Pkndf(I_{h})$. Hence $\nu\in \Pkndff^{ac}(I_{h})$ and we may assume w.l.o.g. that $\nu|_{[a,x]}(h)> 0$ for any $x\in (a,b)$. In \eqref{eqn:JhkN} we defined the densities $J_{K,N,\hfrak}(x)$ by
 $$J_{K, N, \hfrak}(x)= 
\begin{cases}   
(\co_{\delta}(x)+\frac{\hfrak}{N-1} \si_{\delta}(x))_+^{N-1} &\mbox{ if } N \in (-\infty,0]\cup (1,\infty) \\
\exp(\hfrak x-\frac{K}{2}x^2) &\mbox{ if } N=\infty\,,
\end{cases}$$ 
where $\hfrak$ is some parameter which accounts for the derivative at $x=0$ of $J_{K,N,\hfrak}$. Let $x'\in (a,b)$ be some fixed point.
We define a family of densities
$$ J^{(\hfrak)}(x):=\frac{J(x')}{2} J_{K,N,\hfrak}(x-x')\,, $$
and respectively a family of measures:
%We set $c_{\infty}:=\lim_{h\to\infty}g(h)$ and $c_{-\infty}:=\lim_{h\to-\infty}g(h)$. We define 
\blue{
$$d\sigma_{\hfrak}:=\min\{J,J^{(\hfrak)}\}dm \qquad \text{and} \qquad d\varsigma_{\hfrak}:=(J-J^{(\hfrak)})_{\vee}dm\,.$$}
    
Clearly for any $\hfrak\in \R$ it holds that $\nu=\varsigma_{\hfrak}+\sigma_{\hfrak}$. Since $J^{(\hfrak)}(x')=\frac{J(x')}{2}$ in a sufficiently small neighborhood of $x'$ it holds that $J^{(\hfrak)}(x)< J(x)$, therefore this decomposition is non-trivial. We claim that $\sigma_{\hfrak},\,\varsigma_{\hfrak}\in \Fknd$. Clearly $supp(\sigma_{\hfrak})$ and $supp(\varsigma_{\hfrak})$ are subsets of $supp(\nu)$, hence the diameter condition of Definition \ref{dfn:CDKNSynt} is satisfied. 
\smallskip
If $x_0, x_1\in \R$ and $|x_1-x_0|<l_{\delta}$ then  \eql{\label{eqn:MinIneq}
&\min\{J,J^{(\hfrak)}\}(x_t)\\ \nonumber&\geq \min\cprnt{M_{K,N-1}^{(t)}[|x_1-x_0|](J(x_0),J(x_1)),M_{K,N-1}^{(t)}[|x_1-x_0|](J^{(\hfrak)}(x_0),J^{(\hfrak)}(x_1))}\\ \nonumber& \geq M_{K,N-1}^{(t)}[|x_1-x_0|](\min\{J,J^{(\hfrak)}(x_0)\}, \min\{J,J^{(\hfrak)}(x_1)\})\,, }
by property \ref{id:3} of Proposition \ref{prop:properties}. Under our assumptions the case $|x_1-x_0|\geq l_{\delta}$ is possible only if $\delta>0$ and $K>0$; in this case by Lemma \ref{lem:CondImplic} it follows $J(x_0)J(x_1)=0$ and $J^{(\hfrak)}(x_0)J^{(\hfrak)}(x_1)=0$ and \eqref{eqn:MinIneq} is satisfied. Therefore $\sigma_{\hfrak}\in \Fknd(I_h)$. 
\smallskip

We claim that $\varsigma_{\hfrak}\in \Fknd(I_h)$ as well; indeed, if $J^{(\hfrak)}>0$ on the whole of $(a,b)$ then $\varsigma_{\hfrak}\in \Fknd(I_h)$ due to Lemma \ref{lem:FG1}; \blue{otherwise
$\cprnt{J^{(\hfrak)}>0}\cap (a,b)$
is an interval, since $\isupp \prnt{J^{(\hfrak)}}=x'+\isupp(J_{K,N,\hfrak})$.}
% {\small $$\prnt{\,\zfrak_{-}(T_{x'}[J^{(\hfrak)}])+x',\,\,\zfrak_{+}(T_{x'}[J^{(\hfrak)}])+x'}\cap (a,b) \quad\text{is an interval,\, where } 
% T_{x'}[J^{(\hfrak)}](x):=J^{(\hfrak)}(x+x')\,.$$}
 Then 
$\varsigma_{\hfrak}\in \Fknd(I_h)$ due to Lemma \ref{lem:FG2}: in  view of Remark \ref{rmk:J_1_minus_J_0} there will be one interval $(a',b')\subset (a,b)$ where the function \blue{$\tilde{J}:=(J-J^{(\hfrak)})_{\vee}$} satisfies:  $\tilde{J}|_{(a,a')}=J$, $\tilde{J}|_{(b',b)}=J$ and $\tilde{J}|_{(a',b')}=(J-J^{(\hfrak)})_{+}$, and $J^{(\hfrak)}>0$ on $(a',b')$ (mind that we don't exclude the case $a'=a$ or $b'=b$). 
\bigskip

%We also notice that for fixed $h\in \R$ one of the following holds: either $(J-J_0)(x)>0$ on $(z,b)$ or else $(J-J_0)(x)>0$ on $(a,z)$. we consider only the first case, as the second is justified by similar arguments mutatis mutandis. Define $J_1:=J 1_{[a,z]}$ and $J_2:=(J-J_0)_{+} 1_{[z,b]}$. Clearly $\nu_1:=J_1 \cdot m\in \Fknd$, and according to \ref{lem:FG1}, $\nu_2:=J_2 \cdot m\in \Fknd$ as well. It follows from lemma \ref{lem:FG2} and the equivalences \ref{thm:Equivalence}
%that $\varsigma_h=\nu_1+\nu_2\in \Fknd$. By \ref{id:3} of \ref{prop:properties} regarding monotonicity of $M^{(t)}_{k,N-1}[d](a,b)$ in $a$ and $b$ we conclude that $\sigma_h\in \Fknd$ as well. 

We will now show that for some $\hfrak_0\in \R$ it holds that $\sigma_{\hfrak_0},\varsigma_{\hfrak_0}\in \Fkndf(I_h)$.
By assumption $h\in C_c(\R)$, hence by Lebesgue dominated convergence the map $\hfrak\mapsto \int hd\sigma_{\hfrak}$ is continuous.
By definition $diam(\isupp(J_{\hfrak}))=l_{\delta}$, while $diam(\isupp(J))\leq l_{\delta}$ (considering that $\nu\in \Pkndf(I_{h})$ and Lemma \ref{lem:CondImplic}); considering that
$$ \lim_{\hfrak\to +\infty}J_{\hfrak}=+\infty\qquad \lim_{\hfrak\to -\infty}J_{\hfrak}=0 \qquad \forall x\in (x',x'+l_{\delta})\,$$
and
$$ \lim_{\hfrak\to +\infty}J_{\hfrak}=0\qquad \lim_{\hfrak\to -\infty}J_{\hfrak}=+\infty \qquad \forall x\in (x'-l_{\delta},x')\,,$$
we conclude that for all $x\in (a,b)$ 
$$\lim_{\hfrak\to +\infty}\min\{J,J_{\hfrak}\}(x)=J(x)1_{x>x'}+\frac{J(x')}{2}\delta_{x=x'}\,,$$
and
$$\lim_{\hfrak\to -\infty}\min\{J,J_{\hfrak}\}(x)=J(x)1_{x<x'}+\frac{J(x')}{2}\delta_{x=x'}\,$$
(notice that this applies also to the case $N<0$).  

%Notice that when $\delta>0$ the functions $J_h$ are positive on an interval of diameter $l_{\delta}\geq diam(supp(J))$, hence either $[a,z]\subset supp(J_h)$ or $[z,b]\subset supp(J_h)$. 
%Furthermore, if $J(x)$ and $J_h(x)$ intersect at two points $y_0, y_1$ then for any $y_t=(1-t)y_0+ty_2\in (y_1,y_2)$ it holds that $J(y_t)\leq M_{k, N-1}^{(t)}[|x_1-x_0|](J_1(x_0),J_2(x_1))=J_h(y_t)$ thus we can assume that for any fixed $h\in\R$ one of the following holds: 
%\begin{enumerate}
	%\item $J_h\leq J$ on $[a,b]$.
	%\item There is $z'>z$ such that $J_0\geq J$ on $[z',b]$.
	%\item There is $z'<z$ such that $J_0\geq J$ on $[a,z']$ .
%\end{enumerate}
Therefore the following holds :
\[ \lim_{\hfrak\to+\infty}\sigma_{\hfrak}(h)=\int_{x'}^{b}hd\nu<0
\qquad \text{and} \]
\[ \lim_{\hfrak\to-\infty}\sigma_{\hfrak}(h)=\int_{a}^{x'}hd\nu>0\,. \]
By the intermediate value theorem for some $\hfrak_0\in \R$ it holds that
$\int hd\sigma_{\hfrak_0}=0$. By assumption $\int hd\nu=0$ hence $\int hd\varsigma_{\hfrak_0}=0$ as well. 
We may thus write $\nu=\sigma_{\hfrak_0}(1)\prnt{\frac{1}{\sigma_{\hfrak_0}(1)}}\sigma_{\hfrak_0}+\varsigma_{\hfrak_0}(1)\prnt{\frac{1}{\varsigma_{\hfrak_0}(1)}}\varsigma_{\hfrak_0}$.
 However, by assumption $\nu$ is an extreme point, whence $\nu=\frac{1}{\sigma_{\hfrak_0}(1)}\sigma_{\hfrak_0}=\frac{1}{\varsigma_{\hfrak_0}(1)}\varsigma_{\hfrak_0}$. Recalling that $\nu \in \Pkndff^{ac}(I_h)$, these equalities imply that $\nu\in \Pkndff^M(I_h)$. \\

\item

  Assume $\nu=J_0\cdot m\in \Pkndff^{M}(I_h)$ is supported on $I_0:=[a,b]\subset I_h$. w.l.o.g. $\int_a^xhd\nu>0$ for all $x\in (a,b)$. Being a model density $J_0\in C(int(I_0))$ and it satisfies condition \eqref{AgainSyntCDKN} as equality on $I_0$.    
	If $\nu\notin \Ext{\Pkndf(I_h)}$, then it admits a non-trivial decomposition 
	\eql{\label{ExtDecomp} \nu=\half \xi_1+\half \xi_2 \qquad \text{where} \qquad \xi_1,\xi_2\in \Pkndf^{ac}(I_h)\setminus \{\nu\}\,. }

For $i\in\{1,2\}$ the following hold:
\begin{enumerate}
  \item[(a)] $\xi_i$ is supported on an interval $I_i:=[a_i,b_i]\subset I_0$. We define
	$$\tl{I}_i:=int(I_0)\cap I_i\,.$$
	\item[(b)] $I_0=I_1\cup I_2$ and $int(I_1)\cap int(I_2)\neq \emptyset$; indeed, if $int(I_1)\cap int(I_2)= \emptyset$ for $(a,b)\ni z=I_1\cap I_2$ we get $\int_a^zhd\nu=0$ which contradicts our assumption.  
	\item[(c)] $\frac{d\xi_i}{dm}$ exists and equals $J_i$, where
	\begin{itemize}
		\item $J_i$ is continuous and strictly positive on $int(I_i)$.
	  \item $J_i$ satisfies condition \eqref{AgainSyntCDKN} on $I_i$.
	\end{itemize}
	Then we observe that 
	$J_0=\half J_1+ \half J_2$ on $I_0\setminus S_{bd}$, where $S_{bd}:=\partial I_1\cup \partial I_2$ (notice that $\partial I_0\subset S_{bd}$).
	
	\begin{itemize}
		\item We can moreover assume $J_i\in C(\tl{I}_i)$. Indeed $J_i\in C(int(I_i))$ and at the boundary points $z_2\in int(I_0)\cap \partial I_1$ and $z_1\in int(I_0)\cap \partial I_2$ we can define $J_1(z_2)=2J_0(z_2)- J_2(z_2)$ and $J_2(z_1)=2J_0(z_1)- J_1(z_1)$. 
	By (b) $int(I_1)\cap int(I_2)\neq \emptyset$, hence this gives us a continuous extensions of $J_i$ from $int(I_i)$ to $\tl{I}_i$, considering that $2J_0- J_1$ and $2J_0- J_2$ are continuous in some one sided neighborhoods of $z_1$ and $z_2$, and coincide with $J_2$ and $J_1$ respectively on  these neighborhoods. 
	\end{itemize}

	\item[(d)] In view of (c) we get the identity:
	\[ J_0(x)=\half J_1(x)+ \half J_2(x) \qquad \text{ on } int(I_0)\,.\]
\end{enumerate}

Subject to these assumptions we conclude the following:
\begin{enumerate}
%\item[(e)]  $  
\item[(e)] For $i\in\{1,2\}$ we may define on $\tl{I}_i$ the functions: 
$$\eta_i:=\half\frac{J_i(x)}{J_0(x)}\,.$$
Notice that by (d) it holds that $\eta_i(x)\in [0,1]$. 
On $\tl{I}_i$ ($i\in\{1,2\})$:
\begin{enumerate}
	\item The functions $\eta_i$ are continuous.
	\item The functions $\eta_i$ are quasi-concave (this follows from the relations \eqref{AgainSyntCDKN} and \eqref{ModelEqnDefn} and the inequality $\prnt{\frac{a^{\g}+b^{\g}}{c^{\g}+d^{\g}}}^{\frac{1}{\g}}\geq \min\prnt{\frac{a}{c},\frac{b}{d}}$ whenever $a,b,c,d>0$ and $\g\in \R\setminus \{0\}$.
\end{enumerate}
\item[(f)] On $int(I_0)$
 $$ 1= \eta_1(x)+ \eta_2(x)\,.$$
\end{enumerate}

%By $(g)$ the functions $\eta_i$ are bounded by 1 on $I\setminus S_{bd}$, we can thus define continuous extensions $\eta_i:I\to [0,1]$ in the following sense:
%\begin{itemize}
	%\item $\eta_i|_{I\setminus \partial I_i}=\tilde{\eta}|_{I\setminus \partial I_i}$. 
	%\item $\eta_i|_{I_i}$ is continuous on $I_i$.
%\end{itemize}

 %Then $\eta_1$ and $\eta_2$  are bounded and satisfy
%\begin{enumerate}
%\item[(g)] $1= \eta_1(x)+ \eta_2(x)$ for any $x\in I$.
%\end{enumerate}

	%\item $J_i$ is continuous and satisfies the condition \eqref{AgainSyntCDKN} on $\mathring{I}_i$. 	
%\end{enumerate}
%Furthermore we will assume $\nu=J_0\cdot m$ where
%\begin{enumerate}
	%\item $J_0$ is continuous and satisfies the condition \eqref{AgainSyntCDKN} on $\mathring{I}$. 

	%\begin{itemize}
	%Considering that we may now define $J=J_1+J_2$. Then 
	%\begin{enumerate}
		%\item $J$ is continuous on $I\setminus S_{\partial}$ where $S_{\partial}:=(\partial I_1\cup \partial I_2$. 
	%\end{enumerate}

	%\end{itemize}

\end{enumerate}

We will show that $\eta_1$ and $\eta_2$ must be constant; this will follow as a consequence of the following observations:

%The possibility $I_1\cap I_2=\emptyset $ must be excluded since this would imply $\nu|_{[a,x]}(f)=0
\begin{itemize}
	\item We saw in (b) that $\tl{I}_1\cap \tl{I}_2$ is a (non-empty) interval.
	%Since $\xi_1(h)=\xi_2(h)=0$, and $\int_a^x hd\nu>0$ for all $x\in (a,b)$, necessarily 
	
	%(moreover, since $\half J_1+\half J_2=J_0>0$ on $(a,b)$ there can be no point $z\in (a,b)$ where $J_1(z)J_2(z)=0$); hence $I_1\cup I_2 = I$. W.l.o.g we may assume $a\in I_1$ and $b\in I_2$. 
	%
	\item By property (f) $\eta_1(x)=1$ (resp. $\eta_2(x)=1$) on $\tl{I}_1\setminus I_2$ (resp on $\tl{I}_2\setminus I_1$).
	\item By $(e)$ $\eta_1$ and $\eta_2$ (being quasi-concave on $\tl{I}_1\cap \tl{I}_2$) are either monotone or first non-decreasing and then non-increasing on  $\tl{I}_1\cap \tl{I}_2$. However then by property (f) it follows that they both must be  monotone on $\tl{I}_1\cap \tl{I}_2$. By the foregoing we conclude they are monotone on the whole of $\tl{I}_1$ and $\tl{I}_2$ respectively. 
\item This implies that $I_i\cap \partial I_0\neq \emptyset$; for example if $I_1\cap \partial I_0=\emptyset$ then $\eta_2=1$ on $\tl{I}_2\setminus I_1$, while $1=\eta_1+\eta_2$ in $\tl{I_1}\cap \tl{I_2}$; however, unless $\eta_1=0$ on $\tl{I}_1\cap \tl{I}_2$, this is inconsistent for $\eta_2$ which is monotone. Therefore w.l.o.g. we may assume $a\in I_1$.

	Now there are two possibilities:
\end{itemize}
	\begin{myitemize2}
	\item[$I_1\setminus I_2=\emptyset$ :] Then $I_1\subset I_2$ and $a\in I_2$. If $I_1\subsetneqq I_2$ then by property $(f)$ we know that 
	 $\eta_2=1$ on $I_2\setminus I_1$ and it is therefore non-decreasing on $I_2$. 	 Let $F(x):=\int_a^{x}hd\nu$ (notice that $F(a)=F(b)=0$ and $F(x)>0$ on $(a,b)$). Consider the identity
	$$\xi_2(h)=0=\int_{a}^{b}\eta_2dF=F\eta_2|_{a}^{b}-\int_{a}^{b}Fd\eta_2=0-\int_{a}^{b}Fd\eta_2\,.$$
	Since $\eta_2$ is non-decreasing  on $I_2$ this is only possible if $\eta_2= 1$ on $\tl{I}_2=int(I_0)$. But then by property $(f)$ we conclude $\eta_1\equiv 0$ on $I_1$, contradicting the assumption $\xi_1\in \Pkndf(I_h)$. 	If $I_1= I_2=I$ by the foregoing $\eta_2$ is either non-decreasing or non-increasing; however the last identity excludes any possibility that $\eta_2$ is non-constant. Hence $\eta_2=const$, but then by property $(f)$ $\eta_1=const$ as well. Since $J_1,J_2$ and $J_0$ are probability densities we conclude that $\eta_1\equiv 1$ and $\eta_2\equiv 1$ on $int(I)$, which amounts to $\nu=\xi_1=\xi_2$. This clearly contradicts the assumption $\xi_1,\xi_2\in \Pkndf^{ac}(I_h)\setminus \{\nu\}$.
 	\item[$I_1\setminus I_2\neq \emptyset$ :] This implies that $\eta_1$ must be non-increasing on $\tl{I}_1$; indeed, $\eta_1$ and $\eta_2$ assume values in $[0,1]$; 	since $\eta_1=1$ on $\tl{I}_1\setminus I_2$, and monotone on $\tl{I}_1$ it must be non-increasing. We conclude that in this case $\eta_1$ ought to be the constant $1$ on $\tl{I}_1$; indeed, by the identity
	\eq{\xi_1(h)=0=\int_{a}^{b_{1}}\eta_1dF=F\eta_1|_{a}^{b_{1}}-\int_{a}^{b_1}Fd\eta_1=F(b_1)\eta_1(b_1)-\int_{a}^{b_1}Fd\eta_1 \,,}
	where $F(b_1)\eta_1(b_1)\geq 0$. Since $\eta_1$ is non-increasing and $F(x)>0$ on $(a,b_1)$ this identity can hold only if $\eta_1=const$. 
	 Since on $\tl{I}_1\setminus I_2$ (which is non-empty by assumption) it holds that $\eta_1=1$, we conclude that $\eta_1\equiv 1$ on $\tl{I}_1$. 
	By property $(f)$ it then follows that $\tl{I}_1\cap \tl{I}_2=\emptyset$; as we observed before, this situation contradicts the assumption $\nu\in \Pkndff(I_h)$.
	\end{myitemize2}

\qedhere
\end{proof}

\chapter{The general setting}\label{chp:GeneralSetting}
For $K\in\R$, $N\in (-\infty,0]\cup[2,\infty]$ and $D\in (0,\infty]$, subject to the proviso that $D<l_{\delta}$ if $\delta>0$ and $K<0$,  Corollary \ref{cor:CC} gave the following simplified characterization of the constant $\CC$ associated with the respective problem:
\eql{\label{alpha_Characterization}\CC= \inf_{\xi\in \Fknd^M(\R)}\Lambda_{*}(\xi) \qquad \text{where  } \quad \Lambda_*(\xi):=\inf_{f\in\F_{*}(\xi)}\Phi_{u_f^*,v_f^*}(\xi)\,.}
The restriction to the class $\Fknd^M(\R)$ simplifies the optimization problem significantly, yet a finer optimization over $\Fknd^M(\R)$ is required. 

\bigskip
Recall that  the expression $\inf_{f\in\F_{*}(\xi)}\Phi_{u_f^*,v_f^*}(\xi)$ for $\Lambda_*(\xi)$ is 0-homogeneous in $\xi$. In addition, whenever $f\in \F_*(\xi)$ and $r\in \R$ then also $T_{r}[f](x)\in \F_*(\xi)$, where $T_{r}[f](x):=f(x+r)$, and so we conclude that $\Lambda_*(\xi)$ is invariant under scaling and translations of $\xi$.

 %(in the case of $\Lambda_{LS}(\xi)$ this is possible if we set $v_f^*(\xi)=\int f^2\log (f^2)d\xi$ and we consider the function space $\tlF_{LS}(\xi):=\{ f\in \Cinf(M)\quad \text{s.t}\quad f^2=c+g \quad \text{with}\quad c>0, g\in \F_{Poi}(\xi)\}$). 

\bigskip

  Up to a translation and scaling, each $\xi\in \Fknd^M(\R)$ is given by $d\xi=J\,dm$ with $J(x):=J_{K,N,\hfrak}(x)1_{[a,b]}(x)\in L^1(\R)$, where
	\begin{enumerate}
	  \item $J_{K,N,\hfrak}$ are the densities which were introduced in \eqref{eqn:JhkN}\,,
		\item $a,b\in \R\cup \{-\infty,\infty\}$ are s.t.
		\begin{itemize}
		\item $0<b-a\leq \min(D, l_{\delta})$, and
		\item $[a,b]\subset \isupp(J_{K,N,\hfrak})$.
		\end{itemize}
	\end{enumerate}

%\begin{itemize}
%
	%\item $0<d\leq D_{\delta}$ where 
	%$$D_{\delta}=\min\{D,l_{\delta}\}\,\qquad \mbox{as defined in \ref{dfn:deltaSymbols}}\,.$$
	%\item $\hfrak\in \R$ is such that $[-d/2+r,d/2+r]\subset \left[\zeta_{-}(J_{K,N,\hfrak}),\zeta_{+}(J_{K,N,\hfrak})\right]$.
%\end{itemize} 

\blue{Note that} we may explicitly write   $J_{K,N,\hfrak}(x)$ on $\isupp(J_{K,N,\hfrak})$ as
\eql{\label{DensityExpressions_0} 
&J_{K, N, \hfrak}(x)=(\co_{\delta}(x)+\frac{\hfrak}{N-1} \si_{\delta}(x))_+^{N-1} =g_{s_{K,N,\hfrak}}\cdot Y_{K,N,\hfrak}(x+s_{K,N,\hfrak})}
where
%=
%
%Y_{K,N}(x) =
\eql{\label{DensityExpressions_1}
\gls{Yknh}
=\begin{cases}
         \cos(\sqrt{\delta} x)_{+}^{N-1} & \mbox{ (a) \,\,if } \delta >0, N<\infty\,,\\
         x^{N-1}_+ & \red{\mbox{ (b1)\,if } \delta =0, N<\infty\text{ and } \hfrak\neq 0}\,, \\
         1 & \red{\mbox{ (b2)\,if } \delta =0, N<\infty\text{ and } \hfrak= 0}\,, \\
 \cosh(\sqrt{-\delta} x)^{N-1}\,\,\,\,  &\mbox{ (c1) if } \delta <0, N<\infty \text{ and } |\frac{\hfrak}{N-1}|<1\,,\\
 \sinh(\sqrt{-\delta} x)_+^{N-1}  &\mbox{ (c2) if } \delta <0, N<\infty \text{ and } |\frac{\hfrak}{N-1}|> 1\,, \\
 \red{ e^{\pm \sqrt{-\delta}(N-1)x} }& \red{\mbox{ (c3) if } \delta <0, N<\infty \text{ and } |\frac{\hfrak}{N-1}|= 1 }\,,\\
  \red{ e^{-\frac{Kx^2}{2}} }& \red{\mbox{ (d1) if } N=\infty \text{ and } K\neq 0 }\,,\\
   \red{ e^{\hfrak x} }& \red{\mbox{ (d2) if } N=\infty \text{ and } K=0 }\,,
\end{cases}
}
%  & \quad & \quad
and
    % {\footnotesize
    % \eql{\label{DensityExpressions}  g_{s_{K,N,\hfrak}}= & \quad & \quad  & \quad \\ \nonumber
    % \quad & &\quad\text{with}\quad   \,\,&\qquad\mbox{  \,\,\, if }& \\ \nonumber
    %     \quad &     &\quad\text{with}\quad \qquad\qquad\qquad \,\,\,\,\,&\qquad\mbox{ \,\, if }&  \delta =0, N<\infty\text{ and } \hfrak\neq 0\,,\\ \nonumber
        
    %         \quad &    1 &\qquad\qquad \,\,\,\,\,&\qquad\mbox{\,\, if }&  \delta =0, N<\infty\text{ and } \hfrak= 0\,,\\ \nonumber
    % \quad & &\quad\text{with}\quad  &\qquad\mbox{ \,\, if }& \delta <0, N<\infty \text{ and } |\frac{\hfrak}{N-1}|<1\,, \\ \nonumber
    % \quad &  &\quad\text{with}\quad   &\qquad\mbox{ \,\, if }& \delta <0, N<\infty \text{ and } |\frac{\hfrak}{N-1}|\geq 1\,,\\
    % \nonumber
    % \quad &  &\quad\quad   &\qquad\mbox{ \,\, if }& \red{\delta <0, N<\infty \text{ and } |\frac{\hfrak}{N-1}|= 1}\,,\\
    % \nonumber
    % \quad & \red{} &\quad\quad\text{with}\quad  \qquad\qquad\qquad \,\,\,\,\,   &\qquad\mbox{ \,\, if }&  \red{N=\infty \text{ and } K\neq 0}\,,\\
    % \nonumber
    % \quad & \red{1} &\quad\quad   &\qquad\mbox{ \,\, if }& \red{N=\infty \text{ and } K=0}\,.
    %  } }

\eql{\label{DensityExpressions}}
\begin{table}[H]
\centering
\blue{
\begin{tabular}{ c|c|c  }
 \hline
 $s_{K,N,\hfrak}$ & $g_{s_{K,N,\hfrak}}$ &  In case\\
 \hline
$-\tan^{-1}\prnt{\frac{\hfrak}{(N-1)\sqrt{\delta}}}$ & $\frac{1}{\cos(s_{K,N,\hfrak})_{+}^{N-1}}$  &    (a)\, \\
 $\frac{N-1}{\hfrak}$ & $\frac{1}{s_{K,N,\hfrak}^{N-1}}$ &    (b1)\\
 $0$ &$1$ &  (b2)\\
 $\tanh^{-1}\prnt{\frac{\hfrak}{(N-1)\sqrt{-\delta}}}$ & $\frac{1}{\cosh(s_{K,N,\hfrak})^{N-1}}$ &   (c1)\\
 $\coth^{-1}\prnt{\frac{\hfrak}{(N-1)\sqrt{-\delta}}}$ & $\frac{1}{\sinh(s_{K,N,\hfrak})^{N-1}}$ &  (c2)\\
$0$ & $1$& (c3)\\
$-\frac{\hfrak}{K}$ & $ e^{\frac{\hfrak^2}{2K}}$ &  (d1)\\
 $0$ & $1$ & (d2)\\
\hline
\end{tabular}
\caption{$s_{K,N,\hfrak}$ and $g_{s_{K,N,\hfrak}}$  in each of the cases}\label{table:2}
}
\end{table}

\red{In cases $(a)$ and $(d1)$} we may abbreviate and write $Y_{K,N}(x)$ instead of  $Y_{K,N,\hfrak}(x)$, since in these cases $Y_{K,N,\hfrak}(x)$ is independent of $\hfrak$.

\section{The pertinent parameters}\label{sec:per_parameters}
\blue{Now let} $\xi=J_{K,N,\hfrak}1_{[a,b]}\cdot m\in \Fknd^M(\R)$ \blue{as above, and assume in addition it has a compact support 
(recall $[a,b]\subset \isupp\prnt{J_{K,N,\hfrak}}$ and $J_{K,N,\hfrak}1_{[a,b]}\in L^1(\R)$)}. The identities \eqref{DensityExpressions_0} - \eqref{DensityExpressions} give us a dictionary for a translation from $\hfrak$ to  $s_{K,N,\hfrak}$. A multiplicative factor $g_{s_{K,N,\hfrak}}$ is involved with this translation, however since $
\Lambda_{*}(\xi)$ is 0-homogeneous in $\xi$, this factor is immaterial for the outcomes of $\Lambda_{*}(\xi)$.  \blue{Considering invariance also under translations, if  $d=b-a<\infty$ and denoting $r_{a,b}=\half(a+b)$},
\red{ we may identify}: 
 \red{
 {\footnotesize
 \eq{
 \Lambda_*(J_{K,N,\hfrak}(x)1_{[a,b]}(x) \cdot m)=\Lambda_*(Y_{K,N,\hfrak}(x+s_{K,N,\hfrak})1_{[a,b]}(x) \cdot m)=\Lambda_*(Y_{K,N,\hfrak}(x+ r_{a,b}+s_{K,N,\hfrak})1_{[-\frac{d}{2},\frac{d}{2}]}(x) \cdot m)\,.
  %\eq{&\Lambda_*(J_{K,N,\hfrak}(x)1_{[a,b]}(x) \cdot m)\\&=\begin{cases}\Lambda_*(Y_{K,N,0}(x+s_{K,N,\hfrak})1_{[a,b]}(x) \cdot m)\\  &  \\
  %\Lambda_*(Y_{K,N,1}(x+s_{K,N,\hfrak})1_{[a,b]}(x)\cdot m)\end{cases}=\,\,
  %\begin{cases}\Lambda_*(Y_{K,N,0}(x+r_{a,b}+s_{K,N,\hfrak})1_{[-\frac{d}{2},\frac{d}{2}]}(x) \cdot m)& \substack{\mbox{ if }\delta>0\\ \mbox{or } \delta<0 \mbox{ and } |\frac{\hfrak}{N-1}|<1}\\
 %  & \\
 % \Lambda_*(Y_{K,N,1}(x+r_{a,b}+s_{K,N,\hfrak})1_{[-\frac{d}{2},\frac{d}{2}]}(x) \cdot m)& \substack{\mbox{ if }\delta= 0\\ \mbox{or } \delta<0 \mbox{ and } |\frac{\hfrak}{N-1}|\geq 1}\,.\end{cases}
  %}
}
  }
  }

  \begin{remk} \red{Immediate observations:}
  
 \begin{enumerate}
     \item Note that the above representation implies in particular that 
     $$\Lambda_*(J_{K,N,\hfrak}(x)1_{[a,b]}(x) \cdot m) = \Lambda_*(J_{K,N,\tl{\hfrak}}(x)1_{[-\frac{d}{2},\frac{d}{2}]}(x) \cdot m)$$
     for a different $\tl{\hfrak}\in \R$ s.t. \blue{$[-\frac{d}{2},\frac{d}{2}]\subset  \isupp\prnt{J_{K,N,\tl{\hfrak}}}$  and $J_{K,N,\tl{\hfrak}}1_{[-\frac{d}{2},\frac{d}{2}]}\in L^1(\R)$}. Hence w.l.o.g. we may always assume  that $[a,b]= [-\frac{d}{2},\frac{d}{2}]$.
     
\item In addition, we can exhaust the values of $\Lambda_*(J_{K,N,\hfrak}(x)1_{[-\frac{d}{2},\frac{d}{2}]}(x) \cdot m)$ by considering \blue{$\Lambda_{*}(Y_{K,N,\hfrak}(x+s)1_{[-d/2,d/2]}(x) \cdot m)$} for all translations $s\in \R$ s.t. \\
\blue{$[s-\frac{d}{2}, s+\frac{d}{2}]\subset\isupp\prnt{Y_{K,N,\hfrak}}$ and 
$Y_{K,N,\hfrak}(x+s)1_{[-\frac{d}{2},\frac{d}{2}]}(x)\in L^1(\R)$}
as follows:
\begin{itemize}
    \item  In case (a), with a single fixed density $Y_{K,N}(x)=\cos(\sqrt{\delta}x)_+^{N-1}$; 
    \item In case (b), considering the fixed density $Y_{K,N,\infty}(x):=x_+^{N-1}$, corresponding to the case (b1), as well as the density $Y_{K,N,0}\equiv 1$ with $s=0$, corresponding to the exceptional case (b2); 
    \item In case (c), with the two fixed densities  $Y_{K,N,0}(x)=\cosh(\sqrt{-\delta}x)^{N-1}$ corresponding to the case (c1), and $Y_{K,N,\infty}(x):=\sinh(\sqrt{-\delta}x)_+^{N-1}$ corresponding to the case (c2), as well as the density $Y_{K,N, N-1}=e^{(N-1)\sqrt{-\delta}x}$ with $s=0$, corresponding to the exceptional case (c3) (there is no need to consider the case $\hfrak=-(N-1)$ as $\Lambda_{*}(\xi)$ is also invariant \blue{w.r.t.}  reflections);
    \item In case (d1), with a single fixed density $Y_{K,\infty}(x)=e^{-\frac{Kx^2}{2}}$\pink{;}
    \item In case (d2), $Y_{0,\infty,\hfrak}(x)$  does depend on $\hfrak$, and we do not interpret the variations over $\hfrak$ as translations of a fixed density. 
   
\end{itemize}
\item 
\red{We may thus summarize:} with the exception of case (d2), for the exhaustion of the outcomes of $\Lambda_{*}(\xi)$ for $\xi\in \Fknd^M(\R)$ with bounded support, we may consider either 
\begin{enumerate}
	\item a parameter space $(\hfrak,d)$ and the outcomes of the map:\\
	\pinka{$(\hfrak,d)\mapsto \Lambda_{*}(J_{K,N,\hfrak}(x)1_{[-\frac{d}{2},\frac{d}{2}]}(x) \cdot m)$} (assuming \blue{$ [-\frac{d}{2}, \frac{d}{2}]\subset \isupp\prnt{J_{K,N,\hfrak}}$ \\and 
$J_{K,N,\hfrak}1_{[-\frac{d}{2},\frac{d}{2}]}\in L^1(\R)$}), or
 \item a parameter space $(s,d)$ and the outcomes of a map of the form $(s,d)\mapsto \Lambda_*(Y_{K,N,\hfrak}(x+s)1_{[-\frac{d}{2},\frac{d}{2}]}(x)m)$ where $\hfrak$ takes one, two or three values, depending on whether we consider case (a),(b),(c) or (d1) as described above (assuming  \blue{$ [s-\frac{d}{2}, s+\frac{d}{2}]\subset \isupp\prnt{Y_{K,N,\hfrak}}$ and 
$Y_{K,N,\hfrak}(x+s)1_{[-\frac{d}{2},\frac{d}{2}]}(x)\in L^1(\R)$}). 
\end{enumerate}
 \end{enumerate}
\end{remk}

\begin{remk}
We do not restrict ourselves to a single set of parameters. For certain purposes, especially conciseness of definitions and  theorem statements, the parameter $\hfrak$ is preferable.
However, since it is easier to interpret the parameter $s$, for the proofs of certain statements, as well as for reasons of clarity, the parameter $s$ is sometimes preferable. 
For different purposes we will use a different parameter while keeping in mind \eqref{DensityExpressions}, which constitutes a dictionary for translating from one parameter to another. 
\end{remk}

\bigskip

\section{The regularity domain}
The observations of the previous section motivate the following definitions.
\begin{defn}[$\gls{D_knd}$,\,$\gls{xi_knd_rc}$]\label{defn:Dreg}
 Assume $K\in \R$, $N\in(-\infty,0]\cup (1,\infty]$ and $D\in (0,\infty]$. We define the `parametric domain of regularity':
 \blue{
$$\D^{reg}_{(K,N,D)}:=\cprnt{ (\hfrak, d)\in \R\times (0,D]: \,\,d<\infty \,\,\,\text{and}\,\,\, [-\frac{d}{2}, \frac{d}{2}]\subset int\prnt{
\isupp\prnt{ J_{K,N,\hfrak}  }}} \,.$$
%\prnt{ \zfrak_{-}(J_{K,N,\hfrak}),\zfrak_{+}(J_{K,N,\hfrak})} \}\,. 
}
We denote by $\pi_{\hfrak},\pi_d:\D^{reg}_{(K,N,D)}\to \R$ the natural projections defined by:
\eq{ &\pi_{\hfrak}(\hfrak_0, d_0)=\hfrak_0 \qquad \text{and}\qquad
\pi_{d}(\hfrak_0, d_0)=d_0 \,.}
%so that
%\eq{ &\pi_{\hfrak}\D^{reg}_{(K,N,D)}=\{d:\,(\hfrak_0, d)\in \D^{reg}_{(K,N,D)}\}\\&
%\pi_{d}\D^{reg}_{(K,N,D)}=\{d\in (0,D]:\exists\hfra\in\R: (\hfrak, d_0)\in \D^{reg}_{(K,N,D)}\} }
For $r\in \R$ and $c>0$ we define $$\xi_{(K,N),r,c}: \D^{reg}_{(K,N,D)}\to \Fknd^M(\R)$$  by 
$$\xi_{(K,N),r,c}(\hfrak, d)(x):=cJ_{K,N,\hfrak}(x+r)1_{[-\frac{d}{2}, \frac{d}{2}]}(x+r) \cdot m\,.$$
\end{defn}
\begin{defn}[$\gls{M_knd_reg}$,  $\overline{\Fkndreg}$]\label{def:Mreg}
We define
$$\Fkndreg:= \bigcup_{\substack{r\in\R\\c\in\R_+^*}}\xi_{(K,N),r,c}\prnt{\D^{reg}_{(K,N,D)}}\,,$$ 
where $\R_+^*:=\{c\in \R:\, c>0\}$. 
This is the subset of $\Fknd^M(\R)$ of compactly supported measures whose densities do not vanish or become infinite at the end-points of their support. We will refer to it as the `regularity domain'. In addition we define $\overline{\Fkndreg}$ to be the closure of $\Fkndreg$ with respect to the weak topology. 
\end{defn}
\bigskip

\begin{remk} We use the term `regularity' since, as will be discussed in the next chapter, calculation of $\Lambda_{Poi}(\xi)$ for measures $\xi\in \Fkndreg$ corresponds to a calculation of an eigenvalue of a `regular Sturm-Liouville problem', a term which will be defined there. 
\end{remk}
\begin{remk}\label{SetsEqualityLambda} 
By invariance to scalings and translations we may identify
$$\Lambda_{*}\prnt{\Fkndreg}=\Lambda_{*}\prnt{\xi_{(K,N)}\prnt{\D^{reg}_{(K,N,D)}}}\,,$$
where 
$$\gls{xi_knd_01}:=\xi_{(K,N),0,1}\,.$$
\end{remk}
\bigskip
%\begin{defn}[ $\partial\Fkndreg$ and $\overline{\Fkndreg}$]\label{defn:clDreg}
%We define  $\partial\Fkndreg:=\Fknd^M(\R)\setminus \Fkndreg$. We also define 
%$\overline{\Fkndreg}:=\Fknd^M(\R)$.
%\end{defn}

Our goal is to solve the optimization problem:
{\large \eql{\label{optimizationProblem_Model}
 \text{Find: } \inf_{\xi\in \Fknd^M(\R)} \Lambda_{*}(\xi)\,. }
}

For technical reasons which will be clarified in the next chapter, it is preferable to study the optimization problem over the regularity domain
\eql{\label{optimizationProblem_Model2}
 \text{Find: } \inf_{\xi\in \Fkndreg} \Lambda_{*}(\xi)\,. 
}
In general it is not justified to switch from problem  \eqref{optimizationProblem_Model} to problem \eqref{optimizationProblem_Model2}, since $\Fkndreg\subsetneq \Fknd^M(\R)$. The following two examples demonstrate which situations might occur:
\begin{exmp}\label{bdry_type_I}
Assume $N\leq 0$ and $K=-(N-1)$. Consider a measure $\xi\in \Fknd^M(\R)$ defined by: $d\xi(x)=\cosh(x)^{N-1}1_{[0,\infty)}(x)dm$.  It is supported on an interval of infinite diameter, hence $\xi\notin \Fkndreg$.
%Notice that $\xi\in\Fknd^M(\R)$, but since $\xi$ is supported on an unbounded interval, the SL problem associated with $\Lambda_{Poi}(\xi)$ is not regular.
\end{exmp}
\begin{exmp}\label{bdry_type_II}
Assume $N>2$ and $K=N-1$. Consider a measure $\xi\in \Fknd^M(\R)$ defined by: $d\xi(x)=\cos(x)^{N-1}1_{[0,\frac{\pi}{2}]}(x)dm$. Notice that $\cos(x)^{N-1}$ vanishes at $\frac{\pi}{2}$ and hence $\xi\notin \Fkndreg$.
\end{exmp}

It turns out that although $\Fkndreg\subsetneq \Fknd^M(\R)$,
the `difference' is not substantial as asserted by the following theorem:
\begin{thm}\label{thm:identify_weak_lim} $\Fknd^M(\R)\subset \overline{\Fkndreg}$. 
\end{thm}
 The assertion is equivalent to the following statement:  for every $\xi_0\in\overline{\Fkndreg}$ there is a sequence $(\xi_n)_{n\in\Nbb}\subset \Fkndreg$ s.t. $\xi_n\xrightarrow{w}  \xi_0$.  
\begin{proof}
Every $\xi=J\cdot m\in\Fknd^M(\R)$ is supported on an interval $I_{\xi}\subset \R$, where $J>0$ on $int(I_{\xi}):=(a,b)$. Let $a_n\searrow a$ and $b_n\nearrow b$ be two sequences of points, so that  $[a_n,b_n]\subset (a,b)$ for every $n\in\Nbb$. Accordingly we define a sequence of measures $\xi_n:=J1_{[a_n,b_n]}\cdot m=\xi|_{[a_n,b_n]}$. Then for every $f\in C_b(\R)$ it holds that $\int fd\xi_n=\int f1_{[a_n,b_n]}d\xi \xrightarrow{n\to\infty} \int fd\xi$; indeed, considering that $\xi(\R)<\infty$ and $f$ is bounded:
\eq{&\Abs{\int f1_{[a_n,b_n]}d\xi- \int fd\xi}=\Abs{\int f1_{\R\setminus [a_n,b_n]}d\xi}\leq \int \Abs{f}1_{\R\setminus [a_n,b_n]}d\xi\\&\leq ||f||_{\infty}\int 1_{\R\setminus [a_n,b_n]}d\xi\to 0\,,}
by monotonic convergence.
\end{proof}

\begin{prop}\label{prop:usc_min}
If $\M_b\ni \xi\mapsto \Lambda_{*}(\xi)$ is upper semi-continuous with respect to the weak topology then 
$$\inf_{\xi\in \Fkndreg}\Lambda_{*}(\xi)=\inf_{\xi\in \Fknd^M(\R)}\,\Lambda_{*}(\xi)\,.$$
\end{prop}
\begin{proof}
According to Theorem \ref{thm:identify_weak_lim}:
$$\Fkndreg\subset \Fknd^M(\R)\subset \overline{\Fkndreg}\,,$$
hence
$$\inf_{\xi\in \Fkndreg}\,\Lambda_{*}(\xi)\geq \inf_{\xi\in \Fknd^M(\R)}\,\Lambda_{*}(\xi) \geq \inf_{\xi \in \overline{\Fkndreg}}\,\Lambda_{*}(\xi)\,.$$
However if $\xi\mapsto \Lambda_{*}(\xi)$ is \red{u.s.c.} with respect to the weak topology then $$\inf_{\xi\in \Fkndreg}\,\Lambda_{*}(\xi)= \inf_{\xi\in \overline{\Fkndreg}}\,\Lambda_{*}(\xi)\,,$$
and the asserted equality follows.
\end{proof}

\section{A general 1-dimensional diameter monotonicity lemma}

Assume $\bar{\xi}\in\P(\R)$ (a probability measure)  which is compactly supported on an interval $I$. 
For many problems which we might consider, we may identify $\Lambda_*(\bar{\xi})=\inf_{f\in\F_{*}(\bar{\xi})}\Phi_{u_f^*,v_f^*}(\bar{\xi})$ with 
\[ \inf_{f\in\F_{*}^a}\Phi_{u_f^*,v^{(g)*}_f}(\bar{\xi})\qquad \text{where}\quad v^{(g)*}_f(\bar{\xi}):=\inf_{r\in \mathcal{I}} \int g(f(x),r)d\bar{\xi}\,,\]
where for the definition of $v^{(g)*}_f(\bar{\xi})$:
\begin{itemize}
    \item $\F_{*}^a$ is the auxiliary function space which was  previously defined in Remark \ref{remk:uvh_prop} as
\[ \F^a_*:=\{0\not\equiv f\in \Cinf(\R):\, \, h_f\in C_c(\R)\}\,, \]
\item $\mathcal{I}$ is some subset of $\R$\,,
\item $\F\subset C^1(\R)$ is a given    function space\,, and
\item $g$ is some fixed non-negative function on $\R\times \mathcal{I}$. 
\end{itemize}

%where $\F^a_*:=\{\text{const}\not\equiv  f\in \Cinf(\R):\, \, h_f\in C_c(\R)\}$ is the auxiliary function space we have previously introduced. 
\begin{exmp}\label{exmp:LS} By the so-called Holley-Stroock Lemma \cite{HoSt}, in the log-Sobolev problem we may identify (see also \cite[p.240]{BGL}):
	$$Ent_{\bar{\xi}}(f^2)=v^{(g)*}_f(\bar{\xi})=\inf_{r>0}\int g(f(x),r)d\bar{\xi}(x)\,,$$ 
	where $g(y,r)=\phi(y^2)-\phi(r)-\phi'(r)(y^2-r)$ with $\phi(r)=r\log r$. Note that by convexity of $\phi$ it follows that $g(x,r)\geq 0$. Then with $u_f=f^{\prime 2}$, $v_f=f^2\log f^2$ and $h_f=f^2-1$ (from Subsection \ref{subsec:abstract_form}  \eqref{dfn:specifications}), we have

 \eq{\Lambda_{LS}(\xi)&=\inf_{f\in\F_{LS}(\bar{\xi})}\Phi_{u_f^*,v_f^*}(\bar{\xi})=\inf_{f\in\F_{LS}^a}\frac{\int f'(x)^2d\bar{\xi}(x)}{Ent_{\bar{\xi}}(f^2)}\\&=
\inf_{f\in\F_{LS}^a}\frac{\int f'(x)^2d\bar{\xi}(x)}{\inf_{r>0}\int g(f(x),r)d\bar{\xi}(x)}=\inf_{f\in\F_{LS}^a}\Phi_{u_f^*,v^{(g)*}_f}(\bar{\xi})\,.}

 %
 %\Phi_{u_f^*,v_f^*}(\bar{\xi})
\end{exmp}
\begin{exmp}\label{exmp:p_Poinc}
 In the p-\Poinc problem define	$g(y,r):=|y+r|^p$. 
  A minimum $r_*$ of $\int g(x,r)d\bar{\xi}(x)$ exists and it is  global, since $r\mapsto \prnt{\int |f(x)+r|^p d\bar{\xi}(x)}^{\frac{1}{p}}$ is coercive and strictly convex for $p\in (1,\infty)$. The following identity holds  at $r_*$:
$$\int |f(x)+r_*|^{p-1}sgn(f(x)+r_*) d\bar{\xi}(x)=\int |f(x)+r_*|^{p-2}(f(x)+r_*) d\bar{\xi}(x)=0\,.$$

 Then with $u_f=f^{\prime p}$, $v_f=f^p$ and $h_f=|f|^{p-2}f$ (from Subsection \ref{subsec:abstract_form} \eqref{dfn:specifications}) we have
\eq{		\Lambda_{Poi}^{(p)}(\bar{\xi})&=\inf_{f\in\F_{Poi}^{(p)}(\bar{\xi})}\Phi_{u_f^*,v_f^*}(\bar{\xi})=\inf_{f\in \F_{Poi}^{(p)}(\bar{\xi})}  \frac{\int |f'(x)|^p d\bar{\xi}(x)}{\int |f(x)|^pd\bar{\xi}(x)}=
\inf_{f\in \F_{Poi}^{a\, (p)}} \left\{ \frac{\int |f'(x)|^p d\bar{\xi}(x)}{\inf_{r\in\R}\int |f(x)+r|^pd\bar{\xi}(x)} \right\}\\&=
\inf_{f\in \F_{Poi}^{a\, (p)}}  \frac{\int |f'(x)|^p d\bar{\xi}(x)}{\inf_{r\in\R}\int g(f(x),r)d\bar{\xi}(x)}=\inf_{f\in\F_{Poi}^{a\, (p)}}\Phi_{u_f^*,v^{(g)*}_f}(\bar{\xi}) \,.  }
 
Notice that the third equality is justified by that $f'$ is invariant under translations $f\mapsto f+r$, and that $\bar{\xi}$ is compactly supported, so by replacing $f\in\F_{Poi}^{a\, (p)}$ with $f+r_*$, we may w.l.o.g. assume that $r_*=0$ is the critical point of $r\mapsto \int |f(x)+r|^p d\bar{\xi}(x)$.

\end{exmp}

These examples motivate the following general result, which gives sufficient conditions for the validity of the inequality $\Lambda_*(\xi_1)\geq \Lambda_*(\xi_0)$ whenever $\xi_1$ is a restriction of $\xi_0$. Throughout given $0\neq \xi\in\M_b$, we denote by $\bar{\xi}$ the probability measure $\frac{\xi}{\xi(1)}$ associated with the measure $\xi$. 

\begin{lem}\label{D_monotonicity} 
Assume for every a.c. measure $0\neq \xi\in \M_b$ we may express  $\Lambda_*(\bar{\xi})$ as 
{\small 
\eq{ \Lambda_*(\bar{\xi})=\inf_{f\in \F}\Phi_{u_f^*,v^{(g)*}_f}(\bar{\xi})  \qquad\text{with}\quad u^*_f(\bar{\xi}):=\int |f'(x)|^p d\bar{\xi}(x)\quad\quad v^{(g)*}_f(\bar{\xi}):=\inf_{r\in \mathcal{I}} \int g(f(x),r)d\bar{\xi}(x)\,,}}
where $p\in (1,\infty)$ and  $\F\subset C^1(\R)$ is a given function space. 

Given a measure $0\neq \xi_0=J\cdot m\in \M_b$, supported on a compact interval $I_0$, and a measure $\xi_1$ which is a restriction of $\xi_0$ to an interval $I_1\subset I_0$, then if

\begin{enumerate}
	\item $J\in L^{\infty}(I_0)$, and 
	\item\label{cond:norm_cont} $\F_*^a\ni f\mapsto v^{(g)*}_f(\bar{\xi}_0)$ is  continuous w.r.t. convergence in the norm $||\cdot ||_{L^{\infty}(I_0)}$,
\end{enumerate}
then $\Lambda_*(\xi_1)\geq \Lambda_*(\xi_0)$.
 \end{lem}

\begin{remk}\label{remk:D_monotonicity} Notice that given $\xi_0$ as above, condition \ref{cond:norm_cont} is satisfied for the functions $f\mapsto v^{(g)*}_f(\xi_0)$ associated with $\Lambda_{Poi}^{(p)}$ ($p\in (1,\infty)$) and $\Lambda_{LS}$. Indeed, 
\begin{itemize}
    \item  In the p-\Poinc problem problem, considering Example \ref{exmp:p_Poinc}, we need to show continuity of $f\mapsto v^{(g)*}_{f}(\bar{\xi}_0):=\inf_{r\in\mathcal{I}}\int g(f(x),r)d\bar{\xi}_0$ w.r.t. convergence in $||\cdot||_{L^{\infty}(I_0)}$, where  $g(f(x),r)=|f(x)+r|^p$.  Notice that if $f_n\to f$ uniformly on $I_0$, then by Minkowski's inequality for every fixed $r\in \mathcal{I}$
    {\small 
    $$\Abs{\prnt{\int g(f_n(x),r)d\xi_0}^{\frac{1}{p}}-\prnt{\int g(f(x),r)d\xi_0}^{\frac{1}{p}}}=\Abs{\,||f_n+r||_{L^p(\xi_0)}-||f+r||_{L^p(\xi_0)}\,}\leq ||f-f_n||_{L^p(\xi_0)}\to 0\,.$$}
    This implies that $\int g(f_n(x),r)d\xi_0\xrightarrow{n\to\infty} \int g(f(x),r)d\xi_0$ uniformly on $\mathcal{I}$, in particular  
    $$ v^{(g)*}_{f_n}(\bar{\xi}_0):=\inf_{r\in\mathcal{I}}\int g(f_n(x),r)d\bar{\xi}_0\xrightarrow{n\to\infty} \inf_{r\in\mathcal{I}}\int g(f(x),r)d\bar{\xi}_0=v^{(g)*}_{f}(\bar{\xi}_0)\,. $$
 \item In the LS problem, considering Example \ref{exmp:LS}, we need to show that the map $f\mapsto v^{(g)*}_{f}(\bar{\xi}_0)=Ent_{\bar{\xi}_0}(f^2)$ is continuous w.r.t. convergence in  $||\cdot ||_{L^{\infty}( I_0)}$. Assume $f_n\xrightarrow{n\to\infty} f$ uniformly;  the conclusion is due to the following observations:
  \begin{itemize}
     \item $f\in \F_*^a\subset C(I_0)$ hence $f$ is bounded, and since $f_n\xrightarrow{n\to\infty} f$ uniformly we may assume w.l.o.g. that the functions $f_n$ are bounded; 
     \item  $Ent_{\bar{\xi}_0}(f^2)=\int \phi_1(f)d\bar{\xi}_0-\int\phi_2(f)d\bar{\xi}_0\log\int\phi_2(f)d\bar{\xi}_0$, where $\phi_{1}(s):=s^2\log s^2$ and $\phi_2(s):=s^2$; these functions are continuous on any compact interval, and hence uniformly continuous;
     \item we conclude $\phi_i(f_n), \phi_i(f)$ are bounded, and that $\phi_i(f_n)\to \phi_i(f)$ uniformly on $I_0$.
     \item therefore by Lebesgue dominated convergence  $\int\phi_i(f_n)d\bar{\xi}_0\to \int \phi_i(f)d\bar{\xi}_0$, implying that $Ent_{\bar{\xi}_0}(f_n^2)\xrightarrow{n\to\infty} Ent_{\bar{\xi}_0}(f^2)$.
 \end{itemize}
 In fact the asserted continuity holds w.r.t. convergence in $||\cdot ||_{L^{p'}(\bar{\xi}_0)}$ for any $p'>2$, but we do not justify it here.

 %, the function $f\mapsto  is continuous w.r.t  $||\cdo
 %Indeed, we can write  . 
 %We notice that if $\F_*^a\ni f_n\stackrel{n\to\infty}{\longrightarrow} f$ uniformly

 %Since the functions , they are uniformly continuous on $I_0$; hence if $f_n\to f$ uniformly on $I_0$, given  by Lebesgue dominated convergence, using $\phi_i(f)$ we conclude that $\int \phi_i(f_n)\to \int \phi(f)$ 
 
 %  for every $p'>2$ (refer to \cite[prop. 4.1, 4.4]{BobG} for the relation between the Entropy and the Orlicz norm associated with the function $N(x)=x^2\log(1+x^2)$), hence it is continuous w.r.t $||\cdot ||_{L^{\infty}(I_0)}$.
\end{itemize}

\end{remk}

Before we prove Lemma \ref{D_monotonicity}  we briefly discuss about continuous extensions of functions by constants and their mollifications.

\begin{defn}[The continuous extension by constants $\tilde{f}$]\label{defn:cont_ext} Assume $f\in C^1(\R)$ and let $I=[a,b]\subset \R$ be an interval. We define $\tl{f}$ to be the continuous extension of $f$ defined by:
$$ \tl{f}(x):=\begin{cases} f(x) & \mbox{ if } x\in I\\
														f(a)& \mbox{ if } x< a\\
													  f(b)& \mbox{ if } x> b \,.
							\end{cases} $$
\end{defn}

Define $$g(x)=\begin{cases} f'(x) &\mbox{ if } x\in (a,b)\\0 &\mbox{ otherwise }\,. \end{cases}$$
Notice that $g\in L^1(\R)$, and $\tl{f}\in AC(\R)$ since $\tl{f}(x)=\int_a^x g(y)dy+\tl{f}(a)$ and $\tl{f}'=g$ a.e. 				
We remind the reader some basic facts. A function $\tl{f}\in L^1_{loc}(\R)$ is weakly differentiable if and only if $\tl{f}\in AC_{loc}(\R)$.					
Since $\tl{f}\in AC(\R)$ we conclude $\tl{f}(x)$ is weakly differentiable on $\R$; for the moment we refer to this weak derivative by $v\in L^1_{loc}(\R)$. By definition for any $\phi\in C_c^{\infty}(\R)$ it holds that $\int\phi'(y) \tl{f}(y)dy=-\int\phi(y) v(y)dy$. 
On any sub-interval $(x_0,x_1)\subset \R$ where $\tl{f}$ is (strongly) differentiable it holds that $\int (v-g)\phi=0$ (by definition of $v$); since $\phi\in C_c^{\infty}(\R)$ is arbitrary we conclude $v=g$ as functions in $L^1_{loc}(\R)$.
Since there is no fear of ambiguity from this point on we refer to the weak derivative of  $\tl{f}$ by $\tl{f}'$. 
 
\bigskip
Denote by  $\eta$  the compactly supported mollifier
$$  \eta(x)=\begin{cases} c\exp(\frac{1}{|x|^2-1}) & \mbox{if  } |x|<1\\ 0 & \mbox{otherwise}\,, \end{cases} $$
where $c$ is a normalization constant s.t. $\int \eta=1$. For each $s>0$ we define the smoothing of $\tilde{f}$: 
\eql{\label{eqn:mollified} \tl{f}_s(x):=\frac{1}{s}\int \eta(\frac{x-y}{s})\tl{f}(y)dm(y)\,.}
It is known \cite{Eva} that $\tl{f}_s\in\Cinf(\R)$. 
\begin{lem}\label{lem:MollifierConvergence} Assume $d\xi=Jdm$ is an a.c. measure. Then for any $\tl{f}\in AC(\R)$,  $\tl{f}_s\xrightarrow{s\to 0} \tl{f}$ uniformly on compact sets, and $\tl{f}_s'\xrightarrow{s\to 0} \tl{f}'$
 in $L^p(K; \xi)$ for all $p\in [1,\infty)$ and compact sets $K\subset \R$ s.t. $J\in L^{\infty}(K)$. 
%s.t $J\in L^{\infty}_{loc}(int(supp(\xi))$. Then $\tl{f}_s\to \tl{f}$ and $\tl{f}'_s\to \tl{f}'$  in $L^p_{loc}(\R; \xi)$ for any $p\in [1,\infty)$. 
\end{lem}
\begin{proof}
It is known \cite{Eva} that  $\tl{f}_s\to f$ $[m]$ a.e., and uniformly on compact sets, as $s\to 0$.
Furthermore $\tl{f}'_s\to \tl{f}'$  in $L_{loc}^p(\R; m)$ for every $p\in [1,\infty)$. 
Therefore for every $K\Subset \R$ s.t. $J\in L^{\infty}(K)$:
\eq{ 
||\tl{f}'_s- \tl{f}'||_{L^p(K; \xi)}\leq ||J||_{L^{\infty}(K)} ||\tl{f}'_s- \tl{f}'||_{L^p(K; m)} \stackrel{s\to 0}{\longrightarrow} 0 \,.}
\end{proof}

%such that $\xi(A)\leq M_K \cdot m(A)$ for any Lebesgue measurable set $A\subset K$. Hence for any non-negative simple function $f$: $\int_K f d\xi\leq M_K\int_K fdm$, therefore the same applies to general non-negative measurable function. This implies that 

%We can thus conclude that 
% $\tl{f}_s\to \tl{f}$ and $\tl{f}'_s\to \tl{f}'$  in $L_{loc}^p(K; \xi)$ for any $p\in [1,\infty)$. 

We now turn to the proof of Lemma \ref{D_monotonicity}.
\begin{proof} [Proof of Lemma  \ref{D_monotonicity}]
%Let $\epsilon>0$ be given, and assume $0\neq f_1\in \F$ is such that 
%$\Phi_{u_{f_1}^*,v^{(g)*}_{f_1}}(\xi_1)<\Lambda(\xi_1)+\epsilon$ .
By definition of $\Lambda_*(\bar{\xi}_1)$ for any $\epsilon>0$ there exists $f_{\epsilon}\in \F$ s.t. $\Phi_{u_{f_{
\epsilon}}^*,v^{(g)*}_{f_{
\epsilon}}}(\bar{\xi}_1)<\Lambda(\bar{\xi}_1)+\epsilon$ (in particular $v^{(g)*}_{f_{
\epsilon}}(\xi_1)>0$). Denote by $\tilde{f}_{\epsilon}$ the continuous extension of $f_{\epsilon}$ to $I_0$ by constants. Clearly $\tilde{f}_{\epsilon}'(x)=0$ for any $x\in I_0\setminus I_1$; in general it is not a smooth function. We define $\tilde{f}_{\epsilon,s}\in \Cinf(\R)$ to be the smoothing of $\tilde{f}_{\epsilon}$ as defined in \eqref{eqn:mollified}. $\tilde{f}_{\epsilon,s}$ is not compactly supported, however since $I_0$ is compact, and we only consider integration inside $I_0$ we may assume throughout that it is compactly supported (by multiplying it with a smooth bump function which equals 1 on $I_0$). 

By Lemma \ref{lem:MollifierConvergence}, considering that $I_0$ is compact, for any $p\in [1,\infty)$:
 $$||\tl{f}_{\epsilon,s}- \tl{f}_{\epsilon}||_{L^p(I_0; \xi_0)}\to 0,\quad  ||\tl{f}'_{\epsilon,s}- \tl{f}'_{\epsilon}||_{L^p(I_1; \xi_0)}\to 0\quad \text{and} \quad  ||\tl{f}'_{\epsilon,s}- \tl{f}'_{\epsilon}||_{L^p(I_0\setminus I_1; \xi_0)}\to 0 \quad \text{ as } s\to 0\,,$$
 and 
 $$||\tl{f}_{\epsilon,s}- \tl{f}_{\epsilon}||_{L^\infty(I_0)}\to 0\quad \text{ as } s\to 0\,.$$
Due to assumption \ref{cond:norm_cont}  we conclude that as $s\to 0$
 \eql{\label{moll_converg} \inf_{r\in \mathcal{I}}\prnt{\int_{I_0}g(\tilde{f}_{\epsilon,s}(x),r)d\bar{\xi}_0(x)}=v^{(g)*}_{\tilde{f}_{\epsilon,s}}(\bar{\xi})\to v^{(g)*}_{\tilde{f}_{\epsilon}}(\bar{\xi})=
\inf_{r\in \mathcal{I}}\prnt{\int_{I_0}g(\tilde{f}_{\epsilon}(x),r)d\bar{\xi}_0(x)}\,.}
 
We conclude that
{\small
\eq{ &\Lambda(\bar{\xi}_1)+\epsilon>\frac{\int_{I_1}f_{\epsilon}'(x)^pd\bar{\xi}_1(x)}{\inf_{r\in \mathcal{I}}\int_{I_1}g(f_{\epsilon}(x),r)d\bar{\xi}_1(x)} \geq 
\frac{\int_{I_1}f_{\epsilon}'(x)^pd\bar{\xi}_1(x)}{\inf_{r\in \mathcal{I}}\int_{I_1}g(f_{\epsilon}(x),r)d\bar{\xi}_1(x)+\inf_{r\in \mathcal{I}}\int_{I_0\setminus I_1}g(\tilde{f}_{\epsilon}(x),r)d\bar{\xi}_1(x)}\\ & \geq 
\frac{\int_{I_1}f_{\epsilon}'(x)^pd\bar{\xi}_1(x)}{\inf_{r\in \mathcal{I}}\prnt{\int_{I_1}g(f_{\epsilon}(x),r)d\bar{\xi}_1(x)+\int_{I_0\setminus I_1}g(\tilde{f}_{\epsilon}(x),r)d\bar{\xi}_1(x)}}
=\frac{\int_{I_1}\tilde{f}_{\epsilon}'(x)^pd\bar{\xi}_0(x)+\int_{I_0\setminus I_1}\tilde{f}_{\epsilon}'(x)^2d\bar{\xi}_0(x)}{\inf_{r\in \mathcal{I}}(\int_{I_0}g(\tilde{f}_{\epsilon}(x),r)d\bar{\xi}_0(x)}\\&
\stackrel{\text{by } \eqref{moll_converg}}{=} \frac{\int_{I_0}\tilde{f}_{\epsilon,s}'(x)^pd\bar{\xi}_0(x)+o(1)}{\inf_{r\in \mathcal{I}}(\int_{I_0}g(\tilde{f}_{\epsilon,s}(x),r)d\bar{\xi}_0(x)+o(1)}=\frac{\int_{I_0}\tilde{f}_{\epsilon,s}'(x)^pd\bar{\xi}_0(x)}{\inf_{r\in \mathcal{I}}(\int_{I_0}g(\tilde{f}_{\epsilon,s}(x),r)d\bar{\xi}_0(x)}+o(1)\geq \Lambda(\bar{\xi}_0)+o(1)
\,.
}
}

For $s$ sufficiently small it holds that $RHS \geq \Lambda(\bar{\xi}_0)-\epsilon$. Therefore we conclude that for any $\epsilon>0$ :
$$\Lambda(\bar{\xi}_1)\geq \Lambda(\bar{\xi}_0)-2\epsilon\,.$$
Since $\epsilon>0$ can be chosen to be arbitrarily small, we conclude that $\Lambda(\bar{\xi}_1)\geq \Lambda(\bar{\xi}_0)$. 
\end{proof}

\newpage
 \chapter{Functional Inequalities: Explicit Lower Bounds\\ \bigskip The \Poinc Inequality} 
\label{chp:Poinc}

In this chapter we solve the optimization problem associated with $\Lambda_{Poi}(M,\gfrak,\mu)$. The general tools have already been developed, and it remains to solve the simpler optimization problem over the class $\Fknd^M(\R)$.

\section{Refinement of the optimization problem}

The abstract formulation introduced in Subsection \ref{subsec:abstract_form} was implemented in order to derive general results, in particular the reduction to the model class, which amounts to solving the general problem \eqref{alpha_Characterization}. In this chapter we solve the  particular optimization problem defined in Theorem \ref{thm:PoinInequality} for sharp lower bounds for the \Poinc constant.
\bigskip

Considering the conventions of Subsection \ref{subsec:abstract_form}, with the functionals\\ $u_f^*(\xi):=\int f'(t)^2d\xi,\,v_f^*(\xi):=\int f(t)^2d\xi$ and $h_f^*:=\int fd\xi$, we can identify 

\eql{\label{prel_defn_Poincare}  \Phi_{u_f^*,v_f^*}(\xi)&\longrightarrow \text{The classical Rayleigh quotient }\gls{Ray_quo}:=\begin{cases}\frac{\int f'(x)^{2}d\xi(x)}{\int f(x)^2d\xi(x)}& \mbox{ if } \int f^2d\xi>0\\ \nonumber
+\infty & \mbox{ otherwise}\,.   \end{cases}\\ \nonumber
			\Lambda_*(\xi) &\longrightarrow \text{The \Poinc constant }\Lambda_{Poi}(\xi) \quad(\text{defined in } \eqref{dfn:Constants1})\,. \nonumber\\
      \CC &\longrightarrow \text{The \Poinc constant  lower bound }\,\, \lam_{K,N,D} \,. } 
 Therefore we literally seek after the `worst' (i.e. minimal) \Poinc constant we can obtain amongst all $\xi\in \Fknd^M(\R)$. Set $p(x):=\deriv{\xi}{m}(x)$; if $p(x)$ is smooth and strictly positive on $[a,b]$, then the Euler-Lagrange equation associated with $\Lambda_{Poi}(\xi)$ (defined in \eqref{dfn:Constants1}) is the Sturm-Liouville problem (SLP) with Neumann boundary conditions:
$$  (p(x)f'(x))'=-\lambda p(x)f(x)\qquad f'(a)=f'(b)=0\,.$$
This is a particular example of a  regular SLP, whose theory is rather simple and has been thoroughly studied. Below we make a brief excursion in order to provide the general definition of the space of regular SLPs with Neumann boundary conditions and state important relevant results. The reader should be aware that mere consideration of such SLPs is not sufficient for a complete solution of the optimization problem \eqref{optimizationProblem_Model}, however it is an essential preliminary step before we consider the exceptional situations (e.g when $p(x)$ vanishes at one of the endpoints or $(a,b)$ is an unbounded interval).

\subsubsection{A brief overview of the space of regular Sturm-Liouville boundary value problems  with Neumann boundary conditions}
%We briefly discuss the general theory of regular Sturm-Liouville boundary value problems (SL-BVPs) with Neumann B.C following \cite{Zet}. This discussion is more general than 
Assume $E=(a',b')$ where $-\infty\leq a'< b'\leq \infty$. Following \cite{Zet} we define the metric space of `regular Neumann SL-BVPs' as the set of tuples
$$\Omega_E=\{(a,b,p):\,\, a'<a<b<b',\,\,\, \frac{1}{p},p\in L_{loc}^1(E)\cap C^{1}(E; \R_+^*)\}\,.$$
 We say that the points $\omega\in \Omega_E$ correspond to `regular Neumann SLP':
\[ (p(x)f(x)')'=-\lambda p(x)f(x) \qquad \mbox{ on }  E\,, \]
 together with the boundary conditions $f'(a)=f'(b)=0$; notice that $f$ is a solution defined on $E$, while $[a,b]\subset E$ (as in \cite{KoZet}). For the comparison of problems defined on different intervals, or with different weight functions $p$, we equip $\Omega_E$ with a metric \cite{Zet}:
\eql{\label{SL_BVP_Metric} &d\prnt{ (a_1,b_1,p_1), (a_2,b_2,p_2)}=\\&|a_2-a_1|+|b_2-b_1|+\int_{a'}^{b'}|\frac{1}{p_2}1_{[a_2,b_2]}-\frac{1}{p_1}1_{[a_1,b_1]}| dm+\int_{a'}^{b'}|p_2 1_{[a_2,b_2]}-p_1 1_{[a_1,b_1]}| dm \,.}
This definition is motivated by that the spectrum of the problem determined by $\omega=(a,b,p)$,  does not depend on the values of $p$ on $E\setminus [a,b]$. 
 The metric $d$ provides a precise definition for the closeness of such SLPs.
%The equation should be understood in the weak sense when we consider general $p,\frac{1}{p}\in L^1_{loc}(E)$ functions. However throughout our discussion we will only consider smooth 
%
%only problems of $\Omega'_E\subset \Omega_E$, 
%where 
%$\Omega'_E=\{(a,b,\frac{1}{p}\in \Omega_E:\,\,\frac{1}{p} \text{ is smooth and finite on } (a,b)\}$. 
%We can then interpret the equation in the strong sense. 

We recall the following known result regarding the eigenvalues and eigenfunctions of these regular Neumann SLPs (see \cite[p.84-87]{Zet} and \cite[p.133]{PiRu}):
\begin{thm}\label{SLResults1} There are infinitely many eigenvalues $\{\lambda_k\}_{k\in \N}$  associated with a problem $(a,b,p)=\omega\in\Omega_E$; they are all real, non-negative and simple; they form an unbounded discrete strictly monotone sequence: $$0=\lambda_0< \lambda_1<...\,.$$
Furthermore the $k'th$ eigenfunction $u_k$ (which can be assumed to be real) has exactly $k$ zeros in $(a,b)$. 
\end{thm}

The following theorem can be found in \cite[p.55-56]{Zet}  (where it is stated in greater generality); for the last part regarding continuous differentiability we refer the reader to \cite[Th. 4.2]{KoZet}.
\begin{thm}\label{SLResults2} Assume $\lambda(\omega_0)$ is a (simple) eigenvalue associated with $\omega_0=(a_0,b_0,p_0)\in \Omega_E$, and let $u(\cdot,\omega_0)$ denote an eigenfunction of $\lambda(\omega_0)$. Then there is a neighborhood $W\subset \Omega_E$ of $\omega_0$ such that $\lambda(\omega)$ is simple for every $\omega=(a,b,p)\in W$, and there exist normalized eigenfunctions $u(\cdot, \omega)$ of $\lambda(\omega)$ for $\omega\in W$ (i.e. $\int_a^b|u(x,\omega)^2|p(x)dx=1$) such that 
$$ u(\cdot,\omega)\to u(\cdot,\omega_0),\qquad (pu')(\cdot,\omega)\to  (pu')(\cdot,\omega_0)\,,\qquad as\,\, \omega\to \omega_0\,\, \text{ in } \Omega_E\,,$$
where both convergences are uniform on any compact sub-interval $K$ of $E$. 
Furthermore the map $\lambda:W\to \R$ is continuously differentiable w.r.t. the parameters $a$ and $b$, and Fr\'{e}chet differentiable w.r.t. $p$. 
\end{thm}

\bigskip

\subsection{Extraction of the minimal $\Lambda_{Poi}(\xi)$ from  $\Fknd^M(\R)$}
\label{ExtractionMinimal}

Following \eqref{optimizationProblem_Model} our goal is to solve the problem  of finding $\inf_{\xi\in \Fknd^M(\R)}\Lambda_{Poi}(\xi) $. 
As a preliminary step we characterize the solution to the simpler optimization problem:
$$\inf_{\xi\in \Fkndreg}\Lambda_{Poi}(\xi)\,,  $$
where $\Fkndreg$ is the set defined in the previous chapter (Definition \ref{def:Mreg}).

Recall Definition \ref{defn:Dreg}     for the definition of $\xi_{(K,N)}:=\xi_{(K,N),0,1}: \D^{reg}_{(K,N,D)}\to \Fkndreg$.
We define $\lam:\D^{reg}_{(K,N,D)}\to \R_+$   by 
\eql{\label{defn:Lambda_h_d}\lam(\hfrak, d):=\Lambda_{Poi}(\xi_{(K,N)}(\hfrak, d))\,.}
From Remark \ref{SetsEqualityLambda} it follows that 
\eql{\label{SimplerOptimizDreg}\inf_{(\hfrak, d)\in \D^{reg}_{(K,N,D)}}\lam(\hfrak, d)=\inf_{\xi\in \Fkndreg}\Lambda_{Poi}(\xi)\,.}

The reader is referred to Section \ref{sec:per_parameters} to recall   the correspondence between the parameters $\hfrak$ and $s$, and the possible representations of $\Lambda_{Poi}(\xi)$ for $\xi\in \Fkndreg$ in terms of the parameters.  

The following theorem characterizes the dependence of $\lam(\hfrak, d)$ on the parameters $\hfrak$ and $d$ for $(\hfrak, d)\in \D^{reg}_{(K,N,D)}$. As it turns out, $\lam(\hfrak, d)$ depends monotonically on $d$ and on $|\hfrak|$. This characterization amounts to a solution to the problem \eqref{SimplerOptimizDreg}. In the next section we will provide a solution to the original optimization problem \eqref{optimizationProblem_Model}, after we prove additional results regarding the upper semi-continuity of $\xi\mapsto \Lambda_{Poi}(\xi)$.

%, by $\xi_{(K,N),r}(\hfrak, d)=J_{K,N,\hfrak}1_{[-\frac{d}{2}, \frac{d}{2}]} \cdot m$. 

%
%
%By invariance under scaling and translations, i.e
%$$\Lambda_{Poi}(J(x)1_{[a,b]}(x)dm)=\Lambda_{Poi}(cJ(x+r)1_{[a+r,b+r]}(x)dm)\qquad \forall r\in \R,\, c>0\,,$$ we may identify
%

%$\D^{reg}_{(K,N,D)}$,\,$\xi_{(K,N)}$,\,$\FF^{reg}_{(K,N,D)}$

%As we explained before our problem amounts to finding the infimum of $\lam(\hfrak, d)$ in $\overline{\D^{reg}_{(K,N,D)}}$. A partial answer, 
%by considering minimization only over $\D^{reg}_{(K,N,D)}$, is encapsulated in the following theorem. 
%Later on we will also handle the missing points $(\hfrak, d)\in \overline{\D^{reg}_{(K,N,D)}}\setminus \D^{reg}_{(K,N,D)}$.

\begin{thm} \label{thm:SpectralMonotonicity} Let $K\in \R$, $N\in (-\infty, 0]\cup (1,\infty]$ and $D\in (0,\infty)$. 
\begin{enumerate}
   \item For any fixed $\hfrak_0\in \pi_{\hfrak}\D^{reg}_{(K,N,D)}$, the function $d\mapsto \lam(\hfrak_0, d)\,\,\,\,(=\Lambda_{Poi}(\xi_{(K,N)}(\hfrak_0, d)))$ on $\pi_{\hfrak}^{-1}(\hfrak_0)$ monotonically decreases as $d$ increases.
    \item
        For any fixed $d_0\in \pi_{d}\D^{reg}_{(K,N,D)}$ the function $\hfrak\mapsto \lam(\hfrak, d_0)\,\,\,\,(=\Lambda_{Poi}(\xi_{(K,N)}(\hfrak, d_0)) )$ on  $\pi_{d}^{-1}(d_0)$ 
        \begin{enumerate}
	    \item monotonically increases as $|\hfrak|$ increases if $N\in (-\infty,-1)\cup (1,\infty]$
	    \item monotonically decreases as $|\hfrak|$ increases if $N\in (-1,0]$.
	    \item is independent of $\hfrak$ if $N=-1$.
        \end{enumerate}
\end{enumerate}

\end{thm}
\begin{remk} Here we identified $\pi_{d}^{-1}(d_0)$ (resp. $\pi_{\hfrak}^{-1}(\hfrak_0)$) with the set of points \\$\pi_{\hfrak}\prnt{\pi_{d}^{-1}(d_0)}=\{\hfrak: \, (\hfrak, d_0)\in D^{reg}_{(K,N,D)}\}$ (resp. $\pi_{d}\prnt{\pi_{\hfrak}^{-1}(\hfrak_0)}=\{d: \, (\hfrak_0, d)\in D^{reg}_{(K,N,D)}\}$).
\end{remk}
\begin{remk}\label{remk:indep_s} It is \blue{quite} remarkable that from $2(c)$ (in view of \eqref{DensityExpressions_0} -  \eqref{DensityExpressions} regarding the translation from the parameter $\hfrak$ to translations $s$) it follows, for example, that $\Lambda_{Poi}(\cosh^{-2}(x)1_{[s-\frac{d}{2},s+\frac{d}{2}]}\cdot m)$ is independent of $s$.
\end{remk}

\begin{remk}
In Definition \ref{defn:Dreg} $\D^{reg}_{(K,N,D)}$ was defined for  $N\in (-\infty, 0]\cup (1,\infty]$.

As we noted in Remark \ref{rmk:Model_equivalent_def}, for $N$ in this range Definitions \ref{dfn:ModelMeasures} and \ref{defn:CD_Model_First} of the class $\Fknd^M(\R)$ are equivalent; however, in Remark \ref{rmk:Model_equivalent_def} we also noted that Definition  \ref{defn:CD_Model_First} 
is more general, being meaningful for 
$N\in(-\infty,\infty]\setminus \{1\}$. Accordingly it is meaningful to define
$\D^{reg}_{(K,N,D)}$ for $N\in(-\infty,\infty]\setminus \{1\}$. 
 In this sense, the proof below shows that actually on the whole range of $N\in (-1,1)$ the function $\hfrak\mapsto \lam(\hfrak, d_0)$ monotonically decreases as $|\hfrak|$ increases. 
\end{remk}

The proof of these monotonicity statements relies on the spectral theory of regular SL-BVPs, in particular Theorems \ref{SLResults1} and \ref{SLResults2}. For any $(\hfrak,d)\in \D^{reg}_{(K,N,D)}$ it holds that $J(x):=\deriv{\xi_{(K,N)}(\hfrak,d)}{m}$ is smooth and strictly positive on $[-\frac{d}{2},\frac{d}{2}]$. We may identify $\lam(\hfrak, d)=\Lambda_{Poi}(\xi_{(K,N)}(\hfrak, d))$ with the first non-zero eigenvalue of the regular SLP:
$$ (J(x)f(x)')'=-\lambda J(x) f(x) : \qquad f'(-\frac{d}{2})=f'(\frac{d}{2})=0\,. $$
Therefore as long as we consider $(\hfrak,d)\in \D^{reg}_{(K,N,D)}$ we can freely study the dependence of $\lam(\hfrak, d)$ on $\hfrak$ and $d$ via properties of regular SL-BVPs. 
%is regular  and we can identify its first non-zero eigenvalue with $\lam(\hfrak, d)$.

\begin{proof}[Proof of (1) ]
The statement is a straightforward consequence of classical results from the theory of regular Sturm-Liouville problems (SLPs), in particular explicit identities expressing the derivatives of $\lambda$ with respect to the endpoints; the identities can be found in \cite{Zet} and \cite{DHe} (and also in the proof of \red{Theorem} \ref{cont:SL_Eval} in the appendix section). However for a complete and concise argument one can just apply Lemma \ref{D_monotonicity} with the function $g(y,r):=(y+r)^2$ using the characterization of the variance as $Var_{\xi}(f)=\inf_{r\in \R} \int(f(x)+r)^2d\xi(x)$ whenever $\xi$ is normalized to be a probability measure.
The justification for this function and for that it verifies the conditions of the lemma, follows as a consequence of Example \ref{exmp:p_Poinc} and Remark \ref{remk:D_monotonicity}.
\end{proof}

Statement $(2)$ will be proved by an extension of a method of Kr{\"o}ger \cite{Kro}. 
We study the dependence of $\lam(\hfrak, d)$ on $\hfrak$ via the principle of diameter comparison; in principle this essentially means that given $\hfrak_1, \hfrak_2\in \R$ which are infinitesimally close, rather than comparing $\lam(\hfrak_1, d)$ and $\lam(\hfrak_2, d)$ directly, we compare $d$ and $d'$ for which  $\lam(\hfrak_1, d)=\lam(\hfrak_2, d')$; if $d\leq d'$ (resp. $d\geq d'$)  then by (1) we conclude that $\lam(\hfrak_1, d)\leq\lam(\hfrak_2, d)$ (resp. $\lam(\hfrak_1, d)\geq\lam(\hfrak_2, d)$). 
We remark that our implementation of the principle will be slightly different, nevertheless  what we have just described is in essence what this principle stands for. 

\begin{proof}[Proof of (2) ]
%
%We will assume $N\neq 1$, since this case is trivial . Furthermore the conclusions applying to cases $(d1)$ and $(d2)$ can be inferred from the conclusions which apply to cases $(a),(b),(c1)$, by the continuous dependence of the eigenvalues of SL BVP on the problem parameters, hence we may restrict only to $N<\infty$.

%For $N\neq \infty$  it suffices to inspect the cases $\delta\in \{-1,0,+1\}$. 
We begin the proof under the assumption that $N\neq \infty$. The case $N=\infty$ will follow by a continuity argument. Without loss of generality (by scaling) we assume $\delta\in \{-1,0,+1\}$.

Let $(\hfrak, d)\in \D^{reg}_{(K,N,D)}$; we will study the sign of $\partial_{\hfrak}\lam(\hfrak,d)$. To this end we resort to \eqref{DensityExpressions_0} -  \eqref{DensityExpressions}, where we get the dictionary to translating variations over $h$ as translations; \red{the reader is advised to recall the definitions of cases (a),(b1-2),(c1-3) in 
\eqref{DensityExpressions_0} (due to the assumption $N\neq \infty$ we do not consider the cases (d1-2), which correspond to $N=\infty$). }

Specifically we have a representation  $J_{K,N,\hfrak}(x)=g_sY(x+s)$ where $s=s(K,N,\hfrak)$ and \red{$Y=Y_{K,N,\hfrak}$}. \red{Recall that in cases (b) and (c), $Y$ is determined as one of two or three fixed densities, depending on the value of $\hfrak$. Note that there is no need to treat the exceptional densities (b2) and (c3), as these correspond to a single value or two values of exceptional $\hfrak$'s; according to Theorem \ref{SLResults2} $\lambda(\hfrak,d)$ is continuous in $\hfrak$, therefore discarding these finite number of $\hfrak$'s is immaterial for establishing the monotoniciy of $\lambda(\hfrak, d)$ in $\hfrak$. }
\bigskip

Define $Y_s:=Y(x+s)$.  The monotonicity statements can be interpreted as monotonicity of the map $s\mapsto \Lambda_{Poi}(Y_s(x)1_{[-\frac{d}{2}, \frac{d}{2}]} dm)$ on the set of $s$ determined by the condition \blue{$[-\frac{d}{2}, \frac{d}{2}]\subset int\prnt{\isupp\prnt{ Y_s  }}$} \red{(in case (a) this means that $[s-\frac{d}{2},s+\frac{d}{2}]\subset (-\frac{\pi}{2\sqrt{\delta}},\frac{\pi}{2\sqrt{\delta}})$, in cases (b1) and (c2) this means that $[s-\frac{d}{2},s+\frac{d}{2}]\subset (0,\infty)$, and in case $(c1)$ this condition poses no restriction}). 

Given $s$ we can identify $\Lambda_{Poi}(Y_s(x)1_{[-\frac{d}{2}, \frac{d}{2}]}(x)\cdot m)$ as the first eigenvalue $\lambda_{1,s}$ of the SLP:

\eql{  \label{SL1} (Y(x)\psi'(x))'=-\lambda_{1,s} Y(x) \psi(x) \qquad \qquad \psi'(-\frac{d}{2}+s)=\psi'(\frac{d}{2}+s)=0 \,.}
 For all $s'$ in a sufficiently small right neighborhood of $s$ there is a SLP

\eql{ \label{SL1a} (Y(x)\bar{\psi}'(x))'=-\lambda_{1,s} Y(x) \bar{\psi}(x) \qquad \qquad \bar{\psi}'(-\frac{d}{2}+s')=\bar{\psi}'(\frac{d}{2}+s'+\epsilon_{s'-s})=0\,, }

and $\epsilon_{s'-s}$ is $o(1)$ as $|s'-s|\to 0$ (see Theorem \ref{cont:SL_Eval} in the appendix to this section). The conclusions of statement $(2)$ of the theorem will be inferred from diameter comparison; specifically if  $\epsilon_{s'-s}>0$, then $\Lambda_{Poi}(Y_{s'}(x)1_{[-\frac{d}{2}, \frac{d}{2}]}(x) \cdot m)\geq \Lambda_{Poi}(Y_s(x)1_{[-\frac{d}{2}, \frac{d}{2}]}(x) \cdot m)$ as a consequence of statement $(1)$. 

%In terms of $ \psi_s(x):=\psi(x-s)$ and $ \bar{\psi}_{s'}(x):=\bar{\psi}(x-s')$ we can write the previous SL problems as follows 
%\eq{ (Y_0(x) \psi_s'(x))'=-\lambda_{1,s} Y_0(x) \psi_s(x)\,,\qquad \psi_s'(x_0(s))=\psi_s'(x_0(s)+\Delta_s)=0\qquad &(SL_s)\quad \mbox{and}\\
%(Y_0(x) \bar{\psi}_{s'}'(x))'=-\lambda_{1,s} Y_0(x) \bar{\psi}_{s'}(x)\,,\qquad \bar{\psi}_{s'}'(x_0(s'))=\bar{\psi}_{s'}'(x_0(s')+\Delta_{s'})=0\qquad &(SL_{s'})}
%where $x_0(s):=x_0+s$. 

To this end we use the (Liouville) transformations (similar to \cite{Kro,PaWe})
\eql{\label{defn:Psi_s}
\Psi_s(x):=Y_0(x)^{\half}\psi'(x)\qquad \text{and} \qquad \Psi_{s'}(x):=Y_0(x)^{\half}\bar{\psi}'(x)\,,} 
which turn the previous equations into 
\eql{\label{ODE1} \Psi_s''(x)+H_0(x)\Psi_s(x)=0\qquad\qquad \text{with \,\, BC } \quad \Psi_s(-\frac{d}{2}+s)=\Psi_s(\frac{d}{2}+s)=0 \,,}
and
\eql{ \label{ODE2} \Psi_{s'}''(x)+H_0(x)\Psi_{s'}(x)=0\qquad\qquad \text{with \,\, BC } \quad \Psi_{s'}(-\frac{d}{2}+s')=\Psi_{s'}(\frac{d}{2}+s'+\epsilon_{s'-s} )=0 \,,}
where
\eql{\label{exp:H0}
&H_0(x):=
\lambda_{1,s}+\half \prnt{ (\log (Y_0)''(x)-\frac{1}{2}\prnt{\frac{Y_0'(x)}{Y_0(x)}}^2 }\,. }
One can verify that $H_0$ is given explicitly by
\eql{\label{exp:H0_2} H_0(x)=
\begin{cases}
 \lambda_{1,s}-\frac{\delta}{2} (N-1)- \frac{1}{4}(N^2-1)\tan_{\delta}(x)^2 &\mbox{ in cases } (a),(c1)\\
\lambda_{1,s}-\frac{\delta}{2} (N-1)-\frac{1}{4}(N^2-1)\cot_{\delta}(x)^2   &\red{\mbox{ in cases } (b1),(c2)}\,,
\end{cases}
 }
where 
$$\tan_{\delta}(x):=\begin{cases}
\frac{\si_{\delta}(x)}{\co_{\delta}(x)}& \mbox{ if } \delta\neq 0\\
0 &\mbox{otherwise} \end{cases}
\qquad \text{and}\qquad \cot_{\delta}(x):=\begin{cases}\tan_{\delta}(x)^{-1}& \mbox{ if } \delta\neq 0\\1 &\mbox{otherwise} \end{cases}
\,,$$ 
($\si_{\delta}(x)$ and $\co_{\delta}(x)$ are the functions defined in \eqref{SinCos}). 

Notice that \eqref{exp:H0_2} implies that if $Y_0(x)$ is a \red{$CD_b(K,-1)$} model density then $\lambda_{1,s}$ is independent of $s$ (cf. Remark \ref{remk:indep_s}). We will therefore further assume that $N\neq -1$ below. 
%\\ \mu_s-\half k -\frac{1}{4}k^2x^2   &\mbox{ in case } (d1)

By a change of coordinates  $\bar{\Psi}_s(x):=\Psi_{s'}(x+s'-s)$ equation \eqref{ODE2} turns into
\eql{ \label{ODE2b}\bar{\Psi}_s''(x)+H_{s'-s}(x)\bar{\Psi}_s(x)=0\,,\qquad \bar{\Psi}_s(-\frac{d}{2}+s)=\bar{\Psi}_s(\frac{d}{2}+s+\epsilon_{s'-s})=0\,,} 
with $H_{s'-s}(x)=H_0(x+(s'-s))$. 
%Observe that \eqref{exp:H0} implies that if $Y_0(x)$ is a $CD(K,-1)$ model density then $\lambda_{1,s}$ is independent of $s$. 

This motivates consideration of the following assumptions which will facilitate the proof:
\begin{enumerate}
	\item Without loss of generality we assume $s>0$; indeed, since $H(x)$ is even the analysis is completely symmetric. 
	\item In cases $(a)$ and $(c1)$ we assume $0< s<s'$ if we are given that $N\in (-1,0]$, and that $0< s'<s$ if $N\in (-\infty, -1)\cup (1,\infty)$ (equivalently $-sgn(N^2-1)=sgn(s'-s)$). \\
	\red{In cases $(b1)$ and $(c2)$} we assume $0< s'<s$ if we are given that $N\in (-1,0]$, and that $0< s<s'$ if $N\in (-\infty, -1)\cup (1,\infty)$ (equivalently $sgn(N^2-1)=sgn(s'-s)$). \\
	Notice that under these assumptions $sgn\prnt{H_0'|_{(0, l_{\delta}/2)}(x)}=sgn(s'-s)$ on $(0,l_{\delta}/2)$.
	  \item $s$ and $s'$ are sufficiently close, so that $\int_{-\frac{d}{2}+s}^{\frac{d}{2}+s}\prnt{H_{s'-s}(x)-H_0(x)}\Psi_{s}^2(x)dx$ has the same sign as $\int_{-\frac{d}{2}+s}^{\frac{d}{2}+s}(s'-s)H_0'(x)\Psi_{s}^2(x)dx$. Indeed, $H_t(x)$  has continuous partial derivatives; setting  $g(t):=\int_{-\frac{d}{2}+s}^{\frac{d}{2}+s} H_{t}(x)\Psi^2_s(x)dx$, then according to Leibniz' rule $g(t)$ is differentiable at $0$, and we may thus write $g(t)=g(0)+g'(0)t+o(t)$ where $g'(0)=\int_{-\frac{d}{2}+s}^{\frac{d}{2}+s} H'_{0}(x)\Psi^2_s(x)dx\neq 0$ since $H'_{0}(x)\neq 0$ on $(-\frac{d}{2}+s,\frac{d}{2}+s)$ for $N^2\neq 1$ (see \eqref{exp:H0_2}), in particular it maintains the same sign on the interval; therefore for $t$ sufficiently small it holds that $g(t)-g(0)$ has the same sign as $g'(0)t$. 

		%(in particular $0\in (x_0(s), x_0(s)+\Delta_s)\Leftrightarrow 0\in (x_0(s), x_0(s)+\Delta_{s'})$). 
	
\end{enumerate}
%The proof essentially relies on comparison of the distance between the zeroes of the 2nd order ODEs \eqref{ODE1} and \eqref{ODE2}. Since $(SL_s)$ and $(SL_{s'})$ both correspond to the same eigenvalue $\lam_s$, proving $\Delta_{s'}< \Delta_{s}$ will imply that $\lam_{s'}< \lam_s$ (in view of part $(1)$ of the theorem).\\

%Furthermore it will be , which in particular implies that $x_0(x)+\Delta_{s'}>-x_0(s)$.\\
%Set $s_m:=\min\{s,s'\}$ and $s_M=\max\{s,s'\}$.\\
%$x_L(s,s')=\min\{x_0(s), x_0(s')\}$ and $x_R(s,s')=\max\{x_0(s)\delta_, x_0(s')\}$
 
Recall that by assumption $\Psi_{s}(-\frac{d}{2}+s)=\Psi_{s}(\frac{d}{2}+s)=0$. 

Notice that since $\lambda_{1,s}$ corresponds to the first eigenvalue of a Neumann SLP, the eigenfunctions $\psi$ and $\bar{\psi}$ in \eqref{SL1} and \eqref{SL1a} are monotonic and we may thus assume w.l.o.g. that the function $\Psi_s$ is positive inside $(-\frac{d}{2}+s, \frac{d}{2}+s)$ and that $\bar{\Psi}_s$ is positive inside $(-\frac{d}{2}+s, \frac{d}{2}+s+\epsilon_{s'-s})$.

Assume by contradiction that $\bar{\Psi}_s(x)$ has no zeros in $(-\frac{d}{2}+s, \frac{d}{2}+s)$ (or equivalently that $\epsilon_{s'-s}>0$).  According to Picone's identity \cite{Pic}:
$$\deriv{}{x}\prnt{\Psi_{s}\Psi_{s}'-\Psi_{s}^2\frac{\bar{\Psi}_s'}{\bar{\Psi}_s}}=  \prnt{H_{s'-s}-H_0}\Psi_{s}^2+\prnt{\Psi_{s}'-\Psi_{s}\frac{\bar{\Psi}_s'}{\bar{\Psi}_s}}^2  \,.$$

We integrate over the left and right terms from $-\frac{d}{2}+s$ to $\frac{d}{2}+s$, considering that $\Psi_s(-\frac{d}{2}+s)=\Psi_s(\frac{d}{2}+s)=0$, we conclude that
\eql{\label{eqn:integralPicone} 0=  
 \int_{-\frac{d}{2}+s}^{\frac{d}{2}+s}    \prnt{H_{s'-s}(x)-H_0(x)}\Psi_{s}^2(x)dx+\int_{-\frac{d}{2}+s}^{\frac{d}{2}+s} \prnt{\Psi_{s}'(x)-\Psi_{s}(x)\frac{\bar{\Psi}_s'(x)}{\bar{\Psi}_s(x)}}^2 dx\,.}
The rightmost term of this identity is non-negative. Since by assumption  $s$ and $s'$ are sufficiently close (in the sense of assumption (3)\,), a contradiction will follow if we prove that 
\eql{\label{SturmPicone2} 
(s'-s) \int_{-\frac{d}{2}+s}^{\frac{d}{2}+s} H_0'(x)\Psi_s^2(x)dx>0\,. }
To establish \eqref{SturmPicone2}  note that one of the following possibilities must hold (recall that $s>0$ by assumption): 
\begin{enumerate}
	\item $(-\frac{d}{2}+s, \frac{d}{2}+s)\subset (0,l_{\delta}/2)$, or
	\item $0\in (-\frac{d}{2}+s, \frac{d}{2}+s)$ . 
\end{enumerate}

%We define $\sigma \Psi_s(x)=\Psi_s(-x)$. 
\noindent {\em Possibility 1}
\smallskip\\
On $(0,l_{\delta}/2)$ evidently $sgn(H_0'(x))$ equals $-sgn(N^2-1)$ in cases $(a)$ and $(c1)$ and  to $sgn(N^2-1)$ \red{in cases $(b1)$ and $(c2)$}. According to our assumptions this coincides with $sgn(s'-s)$ in all four cases. 
Hence $(s'-s)H_0'(x)>0$ on $(0,l_{\delta}/2)$, and it follows that \eqref{SturmPicone2} holds.
\bigskip\\
\noindent {\em Possibility 2}
\smallskip\\
This possibility applies only to cases $(a)$ and $(c1)$.  Recall that $H_0'$ has a constant sign on $(0,\frac{l_{\delta}}{2})$. Since $H_0(x)$ is even, $\sigma \Psi_s(x):=\Psi_s(-x)$ is also a solution to ODE \eqref{ODE1} which satisfies $\sigma \Psi_s(0)=\Psi_s(0)$ and $\sigma \Psi_s(\frac{d}{2}-s)=0$. Unless $\Psi_s$ and $\sigma\Psi_s$ coincide they have no additional intersection point inside $(0,\frac{d}{2}+s)$; indeed, assume by contradiction that there is a second intersection point $x_1\in (0,\frac{d}{2}+s)$, then the function $g_{\Psi_s}(x):= \Psi_s(x)-\sigma \Psi_s(x)$ satisfies $$g_{\Psi_s}''(x)+H_0(x)g_{\Psi_s}(x)=0\qquad g_{\Psi_s}(0)=g_{\Psi_s}(x_1)=0\,,$$
then by Sturm's separation  theorem (e.g \cite[p.314]{BiRo}), since $\Psi_s(x)$ and $g_{\Psi_s}(x)$ are not proportional on $[0,x_1]$ (considering that $g_{\Psi_s}(0)=0$),  $\Psi_s(x)$ must have a zero inside $(0,x_1)$. However $\Psi_s(x):=Y_0(x)^{\half}\psi'(x)$ according to our definition in \eqref{defn:Psi_s}, hence it must be non-zero on $(-\frac{d}{2}+s, \frac{d}{2}+s)$, and we arrive a contradiction.
By assumption $s>0$ therefore $|\Psi_s(x)|\geq |\sigma \Psi_s(x)|$ on $[0, \frac{d}{2}-s]$ (see figure \ref{fig:Psi_s}).
\begin{figure}[H]\label{fig:Psi_s}
	\centering
	\caption{Illustration of the functions $|\Psi_s(x)|$ and $|\sigma\Psi_s(x)|$.}
		\includegraphics[width=14cm, height=2.5cm]{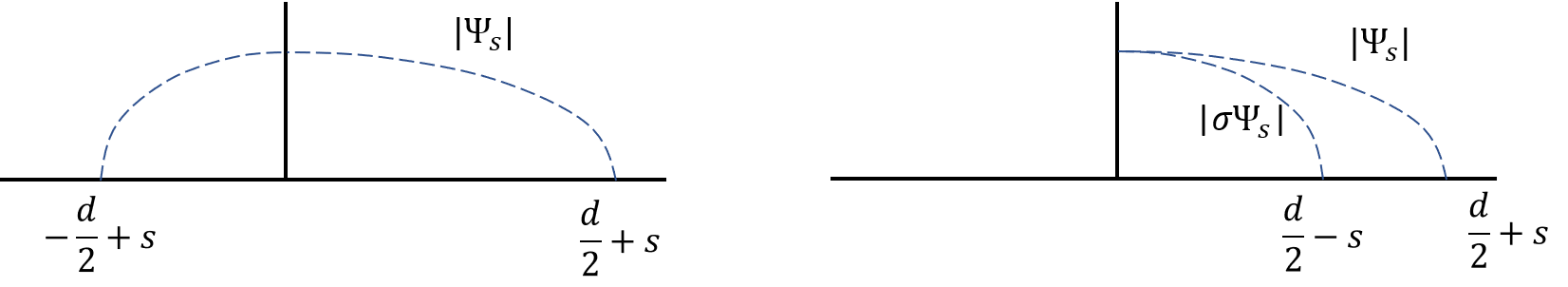}
	\label{fig:Psi_s}
\end{figure}

Considering that $H_0'(x)=-H_0'(-x)$, we conclude the following inequality
\eql{ \label{eqn:signatureIntegral} (\star):&=\int_{-\frac{d}{2}+s}^{\frac{d}{2}+s}H_0'(x)\Psi_s^2(x)dx
\begin{cases} 
>\int_{-\frac{d}{2}+s}^{\frac{d}{2}-s}H_0'(x)\Psi_s^2(x)dx \\
<\int_{-\frac{d}{2}+s}^{\frac{d}{2}-s}H_0'(x)\Psi_s^2(x)dx 
\end{cases}\\&=\int_{0}^{\frac{d}{2}-s}H_0'(x)(\Psi_s^2(x)-(\sigma \Psi_s)^2(x))dx
\begin{cases} 
\geq 0 &\mbox{if} \,\, H_0'>0 \quad\text{on} \quad (0,\frac{l_{\delta}}{2})\\ \nonumber
 \leq 0 &\mbox{if} \,\, H_0'<0 \quad\text{on} \quad (0,\frac{l_{\delta}}{2})\,.
%\int_{-(x_0(s)+\Delta_{s'})}^{x_0(s)+\Delta_{s'}}H_0'\Psi_s^2 &\mbox{if} \,\, H_0'<0
\end{cases}
  }
	
According to our assumptions $sgn\prnt{H_0'|_{(0, \frac{d}{2}-s)}(x)}=sgn(s'-s)$, thus \eqref{eqn:signatureIntegral} states that $sgn(\star)=sgn(s'-s)$. This implies that \eqref{SturmPicone2} holds, which translates into a contradiction. 
\bigskip

Thus $\bar{\Psi}_s(x)$ must have a zero inside $(-\frac{d}{2}+s, \frac{d}{2}+s)$,  whence $\epsilon_{s'-s}<0$; this implies that $\lambda_{1,s'}\leq \lambda_{1,s}$. Recall that the sign of $s'-s$ was assumed to be $-sgn(N^2-1)$ in cases $(a)$ and $(c1)$, and $sgn(N^2-1)$ \red{in cases $(b1)$ and $(c2)$}. The conclusion $\lambda_{1,s'}\leq \lambda_{1,s}$ holds for $s'$ arbitrarily close to $s$ satisfying these premises, thus given $s=s_0>0$ we can conclude that $\deriv{}{s}|_{s=s_0} \lambda_{1,s}$ is non-positive in cases $(a)$ and $(c1)$ when $N^2-1<0$ and non-negative when $N^2-1>0$; similarly $\deriv{}{s}|_{s=s_0} \lambda_{1,s}$ is non-positive \red{in cases $(b1)$ and $(c2)$} when $N^2-1>0$ and non-negative when $N^2-1<0$. Translating from the parameter $s$ to $\hfrak$ (following \eqref{DensityExpressions_0} -  \eqref{DensityExpressions}) we conclude the claims of the theorem in the case $N\in (-\infty,0]\cup (1,\infty)$. 

\bigskip

Lastly we consider the case $N=\infty$. Suppose $(\hfrak_1,d),(\hfrak_2,d)\in \D^{reg}_{(K,\infty,D)}$, then there exists $N_0>1$ such that for all $N\geq N_0$: $(\hfrak_1,d),(\hfrak_2,d)\in \D^{reg}_{(K,N,D)}$.
We claim that: 
\begin{prop}\label{prop:case:infty} If for all $N\geq N_0$ it holds that $\Lambda_{Poi}(\xi_{(K,N)}(\hfrak_1, d))\leq \Lambda_{Poi}(\xi_{(K,N)}(\hfrak_2, d))$ then $\Lambda_{Poi}(\xi_{(K,\infty)}(\hfrak_1, d))\leq \Lambda_{Poi}(\xi_{(K,\infty)}(\hfrak_2, d))$ as well. \end{prop}
\begin{proof}
Denote by $\omega_{(\hfrak, d),N}$ and $\omega_{(\hfrak, d),\infty}$ the points in $\Omega_E$ which correspond to $\Lambda_{Poi}(\xi_{(K,N)}(\hfrak, d))$ and $\Lambda_{Poi}(\xi_{(K,\infty)}(\hfrak, d))$. If $\omega_{(\hfrak, d),N}\stackrel{N\to\infty}{\longrightarrow} \omega_{(\hfrak, d),\infty}$ in the SL-BVP metric \eqref{SL_BVP_Metric} then  it follows from Theorem \ref{SLResults2} that
{\small 
$$ \Lambda_{Poi}(\xi_{(K,\infty)}(\hfrak_1, d))=\lim_{N\to\infty}\Lambda_{Poi}(\xi_{(K,N)}(\hfrak_1, d))\leq \lim_{N\to\infty}\Lambda_{Poi}(\xi_{(K,N)}(\hfrak_2, d))= \Lambda_{Poi}(\xi_{(K,\infty)}(\hfrak_2, d))\,.$$}

%In order to justify convergence in the SL-BVP metric we recall \ref{rmk:FkndM_Diff_ID} for the %identification of the functions $J_{k,N,\hfrak}$, in a propert parametrization, as the distorted means 
%{prop:properties} {id:4}

In order to justify convergence in the SL-BVP metric we recall Definition \ref{defn:Jknh} and the following Remark \ref{remk:ModelDensitiesProp}, according to which for $N\in (1,\infty]$ the functions  
$J_{K,N,\hfrak}$ (resp. $J_{K,\infty,\hfrak}$) satisfy the following differential inequality (resp. equality) on  $int(\isupp(J_{K,N,\hfrak}))$ (resp. $int(\isupp(J_{K,\infty,\hfrak}))=\R$):
{\footnotesize 
	\eq{&-K=(\log J_{K,N,\hfrak})''+\frac{1}{N-1}((\log J_{K,N,\hfrak})')^2&\geq \quad (\log J_{K,N,\hfrak})''&\,,\qquad  J_{K,N,\hfrak}(0)=1\,,\,\, J_{K,N,\hfrak}'(0)=\hfrak\,,\\
	& &-K\quad = \quad (\log J_{K,\infty,\hfrak})''&\,,\qquad J_{K,\infty,\hfrak}(0)=1\,,\,\, J_{K,\infty,\hfrak}'(0)=\hfrak \,.}}
	Notice that the functions $\tl{J}_{K,N,\hfrak}(x):=J_{K,N,\hfrak}(-x)$  and $\tl{J}_{K,\infty,\hfrak}:=J_{K,\infty,\hfrak}(-x)$ satisfy the same differential equality/inequality with $\hfrak$ replaced with $-\hfrak$. We can thus conclude by standard comparison of 2nd order initial value problems that $J_{K,N,\hfrak}\leq J_{K,\infty,\hfrak}$ on $\R$. 

	%and $J_{K,\infty,\hfrak}$ satisfies the ODE
%	\[ LogHess_{\infty}J:=(\log J)''=-K\qquad J(0)=1\,,\,\, J'(0)=\hfrak\,.\]
Hence by Lebesgue dominated convergence, considering that for every $N>1$ it holds that $0\leq J_{K,\infty,\hfrak}-J_{K,N,\hfrak}\leq J_{K,\infty,\hfrak}$,  we conclude that 
$$\lim_{N\to\infty}\int_{-\frac{d}{2}}^{\frac{d}{2}}\Abs{J_{K,\infty,\hfrak}(x)-J_{K,N,\hfrak}(x)}dx=\lim_{N\to\infty}\int_{-\frac{d}{2}}^{\frac{d}{2}}\prnt{J_{K,\infty,\hfrak}(x)-J_{K,N,\hfrak}(x)}dx=0\,.$$
It remains to show that 
$\lim_{N\to\infty}\int_{-\frac{d}{2}}^{\frac{d}{2}}\Abs{J_{K,N,\hfrak}^{-1}(x)-J_{K,\infty,\hfrak}^{-1}(x)}dx=0\,.$
Notice that for $N_d>1$ sufficiently large it holds that $[-\frac{d}{2},\frac{d}{2}]\subset int(\isupp(J_{K,N,\hfrak}))$ for all $N\geq N_d$. In addition, an ODE comparison argument similar to the above shows that $(1,\infty]\ni N\mapsto J_{K,N,\hfrak}(x)$ is monotonic increasing, therefore for $N\geq N_d$: $$0\leq J_{K,N,\hfrak}(x)^{-1}-J_{K,\infty,\hfrak}(x)^{-1}\leq J_{K,N_d,\hfrak}(x)^{-1} \qquad \quad \forall x\in [-\frac{d}{2},\frac{d}{2}]\,. $$
Then by Lebesgue dominated convergence it follows that
$$
\lim_{N\to\infty}\int_{-\frac{d}{2}}^{\frac{d}{2}}\Abs{J_{K,\infty,\hfrak}^{-1}(x)-J_{K,N,\hfrak}^{-1}(x)}dx=\lim_{N\to\infty}\int_{-\frac{d}{2}}^{\frac{d}{2}}\prnt{J_{K,N,\hfrak}^{-1}(x)-J_{K,\infty,\hfrak}^{-1}(x)}dx=0\,.$$
We can therefore conclude that $\lim_{N\to\infty}d(\omega_{(\hfrak, d),N},\omega_{(\hfrak, d),\infty})=0$. 
\end{proof}

This completes the proof of (2) of Theorem \ref{thm:SpectralMonotonicity}.

\end{proof}

%$\co_{\delta}(x)=1-\frac{\delta}{2}x^2+o(\delta}$ $\si_{\delta}$
 %$J_{K,N,\hfrak}(x)\to J_{K,\infty,\hfrak}(x)$ pointwise monotonically for any fixed $x\in [-\frac{d}{2},\frac{d}{2}]$ (assuming $N\geq N_0$), therefore  (by Dini's theorem) $J_{K,N,\hfrak}\to J_{K,\infty,\hfrak}$ uniformly on $[-\frac{d}{2},\frac{d}{2}]$, therefore for sufficiently large $N$ we may assume $|J_{K,N,\hfrak}-J_{K,\infty,\hfrak}|\leq M$ on $[-\frac{d}{2},\frac{d}{2}]$, therefore by Lebesgue dominated convergence we conclude that
%$$\int_{-\frac{d}{2}}^{\frac{d}{2}}|J_{K,N,\hfrak}-J_{K,\infty,\hfrak}|dx\to 0\,,$$
%which amounts to the convergence $\omega_{d,N}\to \omega_{d,\infty}$ in the SL-BVP metric. 
%

\section{Sharp Estimates}

In this section we prove the main result of this chapter - Theorem \ref{Cor:results}.  For notational convenience throughout whenever $\xi=J1_I\cdot m$ is a measure supported on an interval $I$, where $J$ is positive and continuous on $int(I)$,  we write $\Lambda_{Poi}(J, I)$ instead of $\Lambda_{Poi}(J1_I\cdot m)$. 

\begin{thm}\label{Cor:results} Assume $(M^n, \gfrak, d\mu=e^{-V}d\mu_{\gfrak})$ is a \cwrm{} of class $CDD_b(K,N,D)$, where $K\in \R$, $N\in (-\infty, 0]\cup [\max(n,2),\infty]$ and $0<D\leq \infty$. 
\begin{enumerate}
	\item Then for any $f\in C_c^{\infty}(M)$ such that $\mu(f)=0$ the following \Poinc inequality holds 
\eql{\label{eqn:Poincareineq} \lamM_{K,N,D}\int_Mf^2d\mu\leq \int_M|\nabla f|^2d\mu\,, }
where $\gls{lam_star}=\begin{cases} \inf_{ \xi\in \Fknd}\Lambda_{Poi}(\xi) &\mbox{ if } K<0,\, N\leq 0\,\,\text{and } D\geq l_{\delta} \\
\inf_{ \xi\in \Fknd^M}\Lambda_{Poi}(\xi) &  \mbox{ otherwise }\,.
\end{cases}$
%($\Fknd^M$ the model class defined in \ref{dfn:ModelMeasures}). 
\item The lower bound $\lamM_{K,N,D}$ is sharp (i.e. best possible).
\item $\lamM_{K,N,D}$ is explicitly given by the following:
{\small 
\eql{\label{cor:results}
\begin{cases}
 1)\,\,  N\in [\max(n,2),\infty) & \\
	a\triangleright\begin{cases} 
	\Lambda_{Poi}(\cos(\sqrt{\delta}x)^{N-1}, [-D/2, D/2] ) & \quad\mbox{if  } D<\frac{\pi}{\sqrt{\delta}}\\
	\frac{KN}{N-1} & \quad\mbox{otherwise} 
	\end{cases}
	& \mbox{  if  } K>0\\
	b\triangleright\begin{cases} 
	\Lambda_{Poi}(\cosh(\sqrt{-\delta}x)^{N-1}, [-D/2, D/2] ) & \mbox{if  } D<\infty\\
	0 & \mbox{otherwise} 
	\end{cases}
	& \mbox{  if  } K<0\\
	c\triangleright\begin{cases}
	\prnt{\frac{\pi}{D}}^2 & \quad\qquad\qquad\qquad\qquad\qquad\qquad\,\mbox{if  } D<\infty\\
	0 & \quad\qquad\qquad\qquad\qquad\qquad\qquad\,\mbox{otherwise} 
	\end{cases}
	& \mbox{  if  } K=0\\
%----------------------------------------------
	2)\,\, N=\infty & \\	
    a\triangleright\begin{cases} 
	\Lambda_{Poi}(e^{-\frac{Kx^2}{2}}, [-D/2, D/2] ) & \quad\mbox{if  } D<\infty\\
	K & \quad\mbox{otherwise} 
	\end{cases}
	& \mbox{  if  } K>0\\
	b\triangleright\begin{cases} 
	\Lambda_{Poi}(e^{\frac{|K|x^2}{2}}, [-D/2, D/2] ) & \quad\mbox{if  } D<\infty\\
	0 & \quad\mbox{otherwise} 
	\end{cases}
	& \mbox{  if  } K<0\\
	c\triangleright\begin{cases}
	\prnt{\frac{\pi}{D}}^2 & \quad\qquad\qquad\qquad\quad\quad\,\,\,\,\mbox{if  } D<\infty\\
	0 & \quad\qquad\qquad\qquad\quad\quad\,\,\,\,\mbox{otherwise} 
	\end{cases}
	& \mbox{  if  } K=0\\	
%----------------------------------------------
 3)\,\, N\in (-\infty,-1] & \\
    a\triangleright\begin{cases} 
	\Lambda_{Poi}(\cos(\sqrt{\delta}x)^{N-1}, [-D/2, D/2] ) & \quad\text{if  }D<\frac{\pi}{\sqrt{\delta}}\\
	0 & \quad\mbox{otherwise  } 
	\end{cases}
	& \mbox{  if  } K<0\\
	b\triangleright\begin{cases} 
	\Lambda_{Poi}(\cosh(\sqrt{-\delta}x)^{N-1}, [-D/2, D/2] ) & \mbox{if  } D<\infty\\
	\frac{KN}{N-1} & \mbox{otherwise}
	\end{cases}
	& \mbox{  if  } K>0\\
	c\triangleright\begin{cases}
	\prnt{\frac{\pi}{D}}^2 & \quad\qquad\qquad\qquad\qquad\qquad\qquad\mbox{if  } D<\infty\\
	0 & \quad\qquad\qquad\qquad\qquad\qquad\qquad\mbox{otherwise} 
	\end{cases}
	& \mbox{  if  } K=0\\
%----------------------------------------------
 4)\,\, N\in [-1,0] & \\
    a\triangleright\begin{cases} 
    	\lim_{\epsilon\to 0+} \Lambda_{Poi}(\sin(\sqrt{\delta}x)^{N-1}, [\epsilon, \epsilon+D] ) & \qquad\,\,\,\mbox{if  } D<\frac{\pi}{\sqrt{\delta}}\\ 0 & \qquad\,\,\,\mbox{otherwise  } 
	\end{cases}
	& \mbox{  if  } K<0\\
	b\triangleright\begin{cases} 
	\lim_{\epsilon\to 0+} \Lambda_{Poi}(\sinh(\sqrt{-\delta}x)^{N-1}, [\epsilon, \epsilon+D] )>0 & \mbox{if  } D<\infty\\
	\lim_{\epsilon\to 0+} \Lambda_{Poi}(\sinh(\sqrt{-\delta}x)^{N-1}, [\epsilon, \infty) )>0 & \mbox{otherwise  } 
	\end{cases}
	& \mbox{  if  } K>0\\
	c\triangleright\begin{cases}
	\lim_{\epsilon\to 0+} \Lambda_{Poi}(x^{N-1}, [\epsilon, \epsilon+D] ) & \quad\qquad\qquad\quad\mbox{if  } D<\infty\\
	0 & \quad\qquad\qquad\quad\mbox{otherwise  } \\
	\end{cases}
	& \mbox{  if  } K=0\,.\\
\end{cases}
}
}

\end{enumerate}
\end{thm}

\subsubsection{Remarks about Theorem \ref{Cor:results}}
\begin{enumerate}
	\item {\bf Results which preceded the theorem:} Cases 1 and 2 were proved before for $D<\infty$ by Kr\"{o}ger \cite{Kro} (when $N=n$) and Bakry-Qian \cite{BaQi} (when $1\neq N\geq n$). Their results unified the previous estimates of Lichnerowicz \cite{Lic} (for $K>0,N=n$ and $D\geq \frac{\pi}{\sqrt{\delta}}$) and Zhong-Yang \cite{ZhY} (for $K=0$, $N=n$ and $D<\infty$). The estimates in 1 and 2 are characterized as the \Poinc constant $\Lambda_{Poi}(\xi)$ of a measure $\xi$ supported on a symmetric interval $I=[-\min(D,\frac{l_{\delta}}{2}),\min(D,\frac{l_{\delta}}{2})]$ with density $J(x):=\co(\sqrt{\delta}x)^{N-1}$ (`symmetric profile'). The theorem shows that this characterization of the sharp lower bound is maintained when we consider $N\in (-\infty,-1]$. However, \pink{quite} surprisingly the range $N\in (-1,0]$ is anomalous; the characterization of the lower bound turns out to be of completely different nature (`non-symmetric profile'). Hints for this new phenomena could be found in the work \cite{Mil5}  of A. Kolesnikov and E. Milman. In this work (and also in \cite{Oht1} when $\partial M=\emptyset$)  the authors showed that the Lichnerowicz bound $\Lambda(M^n,\gfrak,\mu)\geq \frac{KN}{N-1}$ holds whenever $\delta>0$, $N\in (-\infty,0)\cup [n,\infty]$ (assuming $N>1$) and $D\geq l_{\delta}$; however, this estimate was no longer sharp in the range $N\in (-1,0]$. The estimates of case (4) fill the gap regarding the understanding of this lack of sharpness.  
	\item {\bf For positive $N$ the theorem is valid only for $\pmb{N\geq 2}$:}  The origin of this constraint is within the proof of Theorem \ref{thm:ExtremePoints} regarding the extreme points of the class $\Pkndf(I_h)$ defined in \ref{defn:Pkndf_class}. Although the estimates of case (1) are known be valid for $N>1$ due to the work of Bakry-Qian \cite{BaQi}, there is an inherent difficulty in extending our method to $N\in (1,2)$ which we did not manage to overcome. 
	\item {\bf Positivity of the estimates in case (4b):} As we mentioned in the previous remark, according to \cite{Mil5} and \cite{Oht1}, when  $K>0$, $N\in (-\infty,0)$ and $D=\infty$ holds the estimate $\Lambda_{Poi}(M^n,\gfrak,\mu)\geq \frac{KN}{N-1}>0$ (although it is not necessarily sharp). Being positive for $D=\infty$ it is clearly positive for $D<\infty$ 
	%(for example Lemma \ref{D_monotonicity} and Theorem \ref{thm:specCont})
	, hence the estimates of case (4b) are strictly positive. This is also supported by the following estimate, which is valid for $CD_b(K,N)$ condition with $K>0$ and $N\in (-\infty,1)$ (\cite[Thm 6.1]{Mil3}):
	$$ \Lambda_{Poi}\geq \frac{K}{4(1-N)\prnt{\int_0^{\infty}\cosh^{N-1}(t)dt}^2}>0\,.$$

	\item {\bf The case (2b):} The expression in case (2b), i.e. $\Lambda_{Poi}(e^{\frac{|K|x^2}{2}}, [-D/2, D/2] )$, will be evaluated explicitly up to universal numerical constants in Chapter \ref{chp:LS}, where we discuss about lower bounds for the log-Sobolev constant. In Theorem \ref{equiv:PoiLs_xi_K} we show that 
	$\Lambda_{Poi}(e^{\frac{|K|x^2}{2}}, [-D/2, D/2] )\eqsim \Lambda_{LS}(e^{\frac{|K|x^2}{2}}, [-D/2, D/2] )\,,$
	and in Proposition \ref{prop:Xi_0} we show that
	{\small 
	$$\Lambda_{LS}(e^{\frac{|K|x^2}{2}}, [-D/2, D/2] )\eqsim \max\{ \sqrt{|K|}, \frac{1}{D}\}\frac{|K|D}{e^{|K|\frac{D^2}{8}}-1} \eqsim \begin{cases}
|K|^{\frac{3}{2}}De^{-\frac{|K| D^2}{8}}& \qquad \sqrt{|K|}D>1\\ \frac{1}{D^2}& \qquad \sqrt{|K|}D\leq 1\,.
\end{cases}$$ }

	\item {\bf Distinct expressions in the case $\pmb{N=-1}$:} Notice that when $N=-1$ we actually have 2 seemingly distinct expressions for $\Lambda_{Poi}(M,g,\mu)$, however due to the theorem they must be equal. 
\end{enumerate}

\subsection{Preliminaries to the proof of main Theorem \ref{Cor:results}}\label{subsec:prelim_results}

The results stated in the theorem correspond to solutions of the main optimization problem \ref{optimizationProblem_Model} over the model class :
\eq{
 \text{Find: } \inf_{\Fknd^M(\R)} \Lambda_{Poi}(\xi)\,. }

From Theorem \ref{thm:SpectralMonotonicity} we have a characterization of the minimizing sequences of the problem \ref{optimizationProblem_Model2}:
\eq{
 \text{Find: } \inf_{\Fkndreg} \Lambda_{Poi}(\xi)\,. }

We defined $\overline{\Fkndreg}$ (see Definition \ref{def:Mreg}) as the closure of $\Fkndreg$ w.r.t. the weak topology.  
According to Theorem \ref{thm:identify_weak_lim} $\Fknd^M(\R)\subset \overline{\Fkndreg}$. In Examples \ref{bdry_type_I} and \ref{bdry_type_II} we saw what sort of weak limits $\xi\in \Fknd^M(\R)\setminus \Fkndreg$ we may possibly have.  
However, in view of Proposition  \ref{prop:usc_min}, if we show that $\xi\mapsto \Lambda_{Poi}(\xi)$ is upper semi continuous on $\overline{\Fkndreg}$, then the two problems \ref{optimizationProblem_Model} and \ref{optimizationProblem_Model2} have the same solution. It turns out that this is indeed the case, as asserted by the following theorem.

%------------------------------------------------------------------------------------------
\begin{thm}\label{thm:PoiUSC}
The map $\xi\mapsto \Lambda_{Poi}(\xi)$ on $\M_b(\R)$ is weakly upper semi-continuous (u.s.c.), i.e.
$$ \xi_n\xrightarrow{w} \xi_0 \quad\Rightarrow \quad\Lambda(\xi_0)\geq \limsup_{n\to\infty}\Lambda_{Poi}(\xi_n)\,. $$
%\stackrel{w}{n \rightarrow\infty}
\end{thm}
\bigskip
We precede the proof of Theorem \ref{thm:PoiUSC} with several lemmas. 
Recall \eqref{prel_defn_Poincare} for the definition of $Ray[f](\xi)$ given $f\in \Cinf_c(\R)$.

For certain weak limits we can improve the statement on upper semi-continuity, into continuity. 
We will show that  whenever $\xi\in \Fknd^M(\R)$ is supported on an interval $I$ (possibly of infinite diameter), we can identify $\Lambda_{Poi}(\xi)$ with limits of sequences $\{\Lambda_{Poi}(\xi|_{I_n})\}$, where $(I_n)_{n\in \N}$ is a nested sequence of compact intervals $I_n\subset \mathring{I}$ which exhausts $\mathring{I}$; for each $n$ it holds that $0<\deriv{\xi}{m}|_{I_n}<\infty$, whence $\Lambda_{Poi}(\xi|_{I_n})$ can be identified as an eigenvalue of a regular SLP. To this end we give a precise definition for  exhausting sequences.
\begin{defn}
We say that $(I_n)_{n\in \N}$ is a monotonic exhausting sequence of intervals (\mesi{}) of an interval $I_{\infty}\subset \R$ with respect to a measure $\xi$, if  
\begin{enumerate}
		\item $\xi(I_0)>0$.
		\item $\xi(I_{\infty}\setminus \bigcup_n I_n)=0$.
    \item For each $n\in \N$: $I_n$ is a compact interval and $I_n \subset \mathring{I}_{n+1}\subset I_{\infty}$.
\end{enumerate}
%The case described in Example \ref{bdry_type_I} is treated by the following theorem. 
\end{defn}
 \begin{thm} \label{thm:specCont}
 Let $\xi=J\cdot m$ be an absolutely continuous probability measure supported on an interval $I_{\infty}\subset \R$, s.t. $J\in L^{\infty}_{loc}(int(I_{\infty}))$. Assume $(I_n)_{n\in \N}$ is a MESI of $I_{\infty}$ with respect to $\xi$, then 
\[ \lim_{n\to \infty}\Lambda_{Poi}(\xi_n)=\Lambda_{Poi}(\xi)\,,\]
where $\xi_{n}=\frac{1}{\xi(I_n)}\xi|_{I_n}$.
\end{thm}

We precede the proof of Theorems \ref{thm:PoiUSC} and \ref{thm:specCont} with several useful lemmas. 

Throughout we denote by $\F_b$ the function space:
$$\F_b:=\cprnt{f:\,\, f,\, f'\in \Cinf_b(\R) }\,.$$
\begin{lem} \label{lem:lemma1}
Assume $\xi$ is an a.c. probability measure on $\R$, s.t. $\deriv{\xi}{m}\in L^{\infty}_{loc}(int(supp(\xi)))$. Let $I=[a,b]\subset int(supp(\xi))$, and assume $f\in C_c^{\infty}(\R)$ is such that $\xi|_{I}(f)=0$ and $\xi|_I(f^2)>0$. Then for any $\epsilon>0$ there exists $g\in \F_b$ such that
\begin{enumerate}
	\item $\xi(g)=0$.
	\item $Ray[g](\xi)\leq Ray[f](\xi|_{I})+\epsilon \,.$
\end{enumerate}  
\end{lem}

\begin{lem} \label{lem:lemma2} Assume $\xi\in\M_b(\R)$, and let $g\in \F_b$ be s.t. $\xi(g)=0$. Then for any $\epsilon>0$ there exists $h\in \F_{Poi}(\xi)$ such that 
\[ \Ry{h}{\xi}\leq \Ry{g}{\xi}+\epsilon \,.\]
\end{lem}

The proofs of these two lemmas are rather technical, and therefore will be deferred to the appendix section of this chapter. 

The following lemma is essential specifically for the proof of Theorem \ref{thm:PoiUSC}. 
\begin{lem}\label{lem:RayRayb}
For all $\xi\in M_b(\R)$:
$$ \prnt{\Lambda_{Poi}(\xi)=\,} \inf_{f\in \F_{Poi}(\xi)} Ray[f](\xi)\quad =\quad\inf_{f\in \F_b }\Ryb{f}{\xi}$$
where 
$$\Ryb{f}{\xi}=\begin{cases}
\frac{\int f'(x)^2d\xi(x)}{Var_{b\,\xi}(f)}&\mbox{ if } Var_{b\,\xi}(f)>0\\
+\infty & \mbox{ otherwise }\,, \end{cases}$$
with $$Var_{b\,\xi}(f):=\begin{cases}\xi(f^2)-\frac{\xi(f)^2}{\xi(1)}& \mbox{if } \xi\neq 0\\
0 & \mbox{otherwise}\,.
\end{cases}$$
\end{lem}

\begin{proof}[Proof of Lemma \ref{lem:RayRayb}]\quad \\
$\pmb{LHS\geq RHS:\,}$     Since $\F_{Poi}(\xi)\subset \F_b$, and for $f\in \F_{Poi}(\xi)$ it holds that 
$\Ryb{f}{\xi}=\Ry{f}{\xi}$ (since $\xi(f)=0\Rightarrow Var_{b\,\xi}(f)=\int f^2d\xi$), we conclude 
    $$\inf_{f\in \F_{Poi}(\xi)} \Ry{f}{\xi}\geq \inf_{f\in \F_b }\Ryb{f}{\xi}\,.$$

   $\pmb{LHS\leq RHS:\,}$ Set 
   $l:=\inf_{f\in \F_b }\Ryb{f}{\xi}$. Given $\epsilon>0$ there is a function $f_b\in\F_b$ s.t. $\Ryb{f_b}{\xi}<l+\frac{\epsilon}{2}$. 
   Set $g_b:=f_b-\frac{\xi(f_b)}{\xi(1)}$, then $g_b\in \F_b$, $\xi(g_b)=0$ and
   $\Ry{g_b}{\xi}=\Ryb{f_b}{\xi}<l+\frac{\epsilon}{2}$. From Lemma
  \ref{lem:lemma2} we conclude that there is $h_0\in \F_{Poi}(\xi)$ s.t. $\Ry{h_0}{\xi}\leq Ray[g_b](\xi)+\frac{\epsilon}{2}<l+\epsilon$. Since we can take $\epsilon>0$ arbitrarily small, it follows that
   $$\inf_{f\in \F_{Poi}(\xi) }Ray[f](\xi)\leq \inf_{f\in \F_b }Ray_b[f](\xi)\,.$$

%$\bm{RHS\geq LHS:\,}$      Given $g\in \F_b$ and $\epsilon>0$ we construct $h\in \F_{Poi}(\x)$ s.t $Ray[h](\xi)\leq Ray_b[g](\xi)+\epsilon\,$. 

\end{proof}
\bigskip

We can now prove  Theorems \ref{thm:PoiUSC} and \ref{thm:specCont}.

\begin{proof}[Proof of Theorem \ref{thm:PoiUSC}] 
If $\xi_0=0$ (the 0-measure) then $\Lambda_{Poi}(\xi_0)=\infty$ and there is nothing to prove. Assume henceforth that $\xi_0\neq 0$. According to Lemma \ref{lem:RayRayb}:
$$\Lambda_{Poi}(\xi)=\inf_{f\in\F_{Poi}(\xi)} \Ry{f}{\xi}=\inf_{f\in\F_b} \Ryb{f}{\xi}\,.$$
In order to prove the Theorem it is sufficient to show that $\M_b(\R)\setminus\{0\}\ni \xi\mapsto\Ryb{f}{\xi}$ is \red{u.s.c.} for each individual $f\in\F_b$, since the pointwise infimum in $\F_b$ over the class of u.s.c. functions $\{G_f\}_{f\in\F_b}$, where $G_f(\xi):=\Ryb{f}{\xi}$, is a u.s.c. function $G(\xi):=\inf_{f\in \F_b}G_f(\xi)$ as well.

For a fixed $f_0\in \F_b$ notice that $\xi\mapsto \xi(1)$, $\xi\mapsto \xi(f_0)$, $\xi\mapsto \xi(f_0^2)$ and $\xi\mapsto \xi(f_0^{\prime 2})$ 
are weakly continuous on $\M_b(\R)$ (since $1,f,f^2,f'^2\in C_b(\R)$), thus  
  $\xi\mapsto Var_{b\,\xi}(f_0)$ is weakly continuous on $\M_b(\R)\setminus\{0\}$; in particular we can conclude that  the conditions $Var_{b\,\xi}(f_0)=0$ is weakly closed on $\M_b(\R)$.

It may still happen that $\xi_0\in \M_b(\R)\setminus\{0\}$ is s.t. $Var_{b\,\xi_0}(f_0)=0$ but in any weak neighborhood of it there is $\xi$ s.t. $Var_{b\,\xi}(f_0)>0$; in this case the following inequality trivially holds 
$$\limsup_{\xi\to\xi_0}G_{f_0}(\xi)=\limsup_{\xi\to\xi_0}\Ryb{f_0}{\xi}\leq +\infty\equiv \Ryb{f_0}{\xi_0}=G_{f_0}(\xi_0)\,. $$
This shows that $G_f(\xi)$ is u.s.c. for every $f\in \F_b$, hence $G(\xi)$ is u.s.c. as well.

\end{proof}

\begin{proof}[Proof of Theorem \ref{thm:specCont}] The upper semi continuity $$\Lambda_{Poi}(\xi)\geq \limsup_{n\to \infty}\Lambda_{Poi}(\xi_n)\,,$$
follows from Theorem \ref{thm:PoiUSC}. We now prove the direction $$\Lambda_{Poi}(\xi)\leq \liminf_{n\to \infty}\Lambda_{Poi}(\xi_n)\,.$$ 
Define $l:=\liminf_{n\to\infty} \Lambda_{Poi}(\xi_n)$, and assume $(\xi_{n_k})_{k\in \Nbb}$ is s.t. $\lim_{k\to\infty}\Lambda_{Poi}(\xi_{n_k})=l$. Then for any $\epsilon>0$ we have $k\in \Nbb$ such that 
$\Lambda_{Poi}(\xi_{n_k})<l+\frac{\epsilon}{2}$, and respectively a function $f_k\in C_c^{\infty}(\R)$ such that $\xi_{n_k}(f_k)=0$ and $\Ry{f_k}{\xi_{n_k}}<\Lambda_{Poi}(\xi_{n_k})+\frac{\epsilon}{2}$,  whence $ \Ry{f_k}{\xi_{n_k}}<l+\epsilon$. By Lemmas \ref{lem:lemma1} and \ref{lem:lemma2}, there exists a function $h_k\in C_c^{\infty}(\R)$ such that $\Ry{h_k}{\xi}\leq \Ry{f_k}{\xi_{n_k}}+\epsilon $. This implies that $$\Lambda_{Poi}(\xi)\leq \liminf_{n\to \infty}\Lambda_{Poi}(\xi_n)\,.$$ 
\qedhere

\end{proof}

The following two propositions are essential to the proof of the main theorem; their proofs are also deferred to the appendix section. 
\begin{prop}\label{prop:RayZero} Define $I:=(-\frac{d}{2}, \frac{d}{2})$ where  $d\in (0,\infty]$. Assume $J$ is a positive even continuous function on $I$. Set $\xi=J1_I\cdot m$, and assume $\xi(I) =\infty$. If $(I_n)_{n\in \N}$ is a MESI of $I$ w.r.t. $\xi$ then $\lim_{n\to\infty} \Lambda_{Poi}(\xi|_{I_n})=0$ .

%
%Let $(E_n)_{n\in \N}$ is a monotonic exhausting sequence of $I$ with respect to $m$, and $\int_{E_n} Jdm<\infty$ for each $n\in \N$,   while $\int_IJdm=\infty$.  is such that  be an open interval, and assume $J\in L^1_{loc}(I)$ is a non-negative function on $I$ which is symmetric with respect to , and such that $\int_I J(x)dm=\infty$. Then
%$$\lim_{n\to \infty}Poin(J\, 1_{(a+\frac{1}{n}, b-\frac{1}{n})}\cdot m)=0\,.$$. 
\end{prop}

\begin{prop}\label{prop:RayZero2} Let $(I_n)_{n\in \N}$ be a MESI of $[0,\infty)$ w.r.t. $\xi$. 
Assume $N\in (-1, 0]$ and set $\xi_n(x)=x^{N-1}1_{I_n}\cdot m$. Then $\lim_{n\to\infty}\Lambda_{Poi}(\xi_n)=0$.
\end{prop}

\subsection{Proof of the main theorem}

\begin{proof}[Proof of Theorem \ref{Cor:results}]

\begin{enumerate}

\item The statement follows straightforwardly from our previous results. Recall
$$\lam_{K,N,D}:=\inf_{0\neq \xi\in \Fknd^{C^{\infty}}(\R)}\,\,\Lambda_{Poi}(\xi)\,.$$
\begin{enumerate}
\item
By Theorem \ref{thm:PoinInequality} we got the estimate  $$\Lambda_{Poi}(M^n,\gfrak,\mu)\geq \lam_{K,N,D}\,,$$ for any \cwrm{} which satisfies $CDD_b(K,N,D)$ with $K\in\R$, $N\in (-\infty,1)\cup [n,\infty]$ and $D\in (0,\infty]$.
%\begin{itemize}
	%\item If $n=1$ then $M$ (being connected) is topologically an interval or a circle. On the circle $([a,b]/\sim,\xi)$ we impose periodic B.C therefore the \Poinc constant on an interval $\Lambda_{Poi}\geq ([a,b]/\sim,\xi)\Lambda_{Poi}\geq([a,b],\xi)$. Therefore by definition $\lam_{K,N,D}$ holds the inequality $\Lambda(M^n,|\cdot|,\mu)\geq \lam_{K,N,D}$.
	%\item If $n\geq 2$ then by Theorem \ref{thm:PoinInequality} (which is a consequence of the localization Theorem), for $N\in (-\infty,0]\cup [n,\infty]$ s.t $n\geq 2$:
%$$\Lambda_{Poi}(M,\gfrak,\mu)\geq \lam_{K,N,D}\,.$$ 
%\end{itemize}

\item By Corollary \ref{cor:CC} for $N\in (-\infty,0]\cup [2,\infty]$, subject to the proviso that $D<l_{\delta}=\frac{\pi}{\sqrt{\delta}}$ if $K<0$ and $N\leq 0$,  the solution to the optimization problem which defines $\lam_{K,N,D}$ (which is defined as the bound $\CC$ associated with the \Poinc inequality problem) is equal to the solution of the simpler optimization problem
\eq{\lam_{K,N,D}=\inf_{\xi\in \Fknd^M(\R)}\,\Lambda_{Poi}(\xi)\,. } 

\end{enumerate}
Hence we can summarize that for $N\in (-\infty,0]\cup [\max(n,2),\infty]$ and $K\in \R$ it holds that  
$$\Lambda_{Poi}(M^n,\gfrak,\mu)\geq \lamM_{K,N,D}=\begin{cases} \inf_{ \xi\in \Fknd(\R)}\Lambda_{Poi}(\xi) &\mbox{ if } K<0,\, N\leq 0 \,\,\text{and}\,\, D\geq l_{\delta} \\
\inf_{ \xi\in \Fknd^M(\R)}\Lambda_{Poi}(\xi) &  \mbox{ otherwise }\,.
\end{cases}$$
%^^^^^^^^^^^^^^^^^^^^^^

	\item Our estimate for the \Poinc inequality is sharp since the infimum is realized by intervals in $\R$ equipped with the corresponding measures $\xi\in \Fknd(\R)$ (1-dimensional manifolds) where $N\in (-\infty,0]\cup [2,\infty]$. 
We remark that it is also possible to establish the sharpness on a \cwrm{} of arbitrary topological dimension $n$ for $N\geq 2$; see the construction in \cite{Mil2}, which presumably may be extended to $N\leq 0$ as well, but we did not verify the details.	

\item According to Theorem \ref{thm:PoiUSC} the measure map $\M_b\ni \xi\mapsto \Lambda_{Poi}(\xi)$ is weakly upper semi-continuous, hence it follows from Proposition \ref{prop:usc_min} that
$$\inf_{\xi\in \Fknd^M(\R)}\,\Lambda_{Poi}(\xi)= \inf_{\xi\in \Fkndreg}\,\Lambda_{Poi}(\xi)\,.$$
%
%
%By  we concluded that 
%
 %therefore 
 
We can thus conclude that 
\eql{\label{identity:lam_M_reg}\lamM_{K,N,D}=\begin{cases}\inf_{\xi\in \Fknd(\R)}\,\Lambda_{Poi}(\xi)&\mbox{if } K<0,\, N\leq 0,\,\text{ and } D\geq l_{\delta}\\
\inf_{\xi\in \Fkndreg}\,\Lambda_{Poi}(\xi)&\mbox{otherwise}\,. \end{cases}
}
Recall that the set of values $\Lambda_{Poi}(\Fkndreg)$ is identical to the set $\lam(\D^{reg}_{(K,N,D)})$. The monotonicity properties of $\lam(\hfrak,d)$ on $(\hfrak,d)\in \D^{reg}_{(K,N,D)}$ was described in Theorem \ref{thm:SpectralMonotonicity}. Therefore we have a full 
 characterization of the minimizing sequences of $\lam(\hfrak,d)$ on $\D^{reg}_{(K,N,D)}$, or equivalently of minimizing sequences of $\Lambda_{Poi}(\xi)$ for $\xi\in \Fkndreg$.

%For $(\hfrak, d)\in \D^{reg}_{(K,N,D)}$ the dependence of $\lam(\hfrak,d)$ on the parameters $\hfrak$ and $d$ was described in Theorem \ref{thm:SpectralMonotonicity}. If $(\hfrak_1,d_1)\neq (\hfrak_2,d_2)$ are two distinct points in $\D^{reg}_{(K,N,D)}$ we may compare $\lam(\hfrak_1,d_1)$ and $\lam(\hfrak_2,d_2)$ using Theorem  \ref{thm:SpectralMonotonicity}.

These observations give as the following expressions for $\lam_{K,N,D}$ subject to the proviso that $D<l_{\delta}$ if $K<0$ and $N\leq 0$:
{\small 
\eq{&\lam_{K,N,D}=\lamM_{K,N,D}\quad=
&\begin{cases}
 1)\,\,  N\in [\max(n,2),\infty) & \\
	\triangleright\begin{cases} 
	\Lambda_{Poi}(\co_{\delta}(x)^{N-1}, [-D/2, D/2] ) & \quad\mbox{if  } D<l_{\delta}\\
	\lim_{d\to l_{\delta}}\Lambda_{Poi}(\co_{\delta}(x)^{N-1}, [-d/2, d/2] )  & \quad\mbox{if  }  D\geq l_{\delta}
	\end{cases}\\
%----------------------------------------------
	2)\,\, N=\infty & \\	
    \triangleright\begin{cases} 
	\Lambda_{Poi}(e^{-\frac{Kx^2}{2}}, [-D/2, D/2] ) & \quad\mbox{if  } D<\infty\\
	\lim_{d\to\infty}\Lambda_{Poi}(e^{-\frac{Kx^2}{2}}, [-d/2, d/2] ) & \quad\mbox{if  }  D=\infty 
	\end{cases}\\
%----------------------------------------------
 3)\,\, N\in (-\infty,-1] & \\
    \triangleright\begin{cases} 
	\Lambda_{Poi}(\co_{\delta}(x)^{N-1}, [-D/2, D/2] ) & \quad\text{if  }D<l_{\delta}\\
	\lim_{d\to \infty}\Lambda_{Poi}(\co_{\delta}(x)^{N-1}, [-d/2, d/2] ) & \quad\mbox{if  } D= \infty \quad (\mbox{and }K\geq 0)
	\end{cases}\\
%----------------------------------------------
 4)\,\, N\in [-1,0] & \\
    \triangleright\begin{cases} 
    	\lim_{\epsilon\to 0+} \Lambda_{Poi}(\si_{\delta}(x)^{N-1}, [\epsilon, \epsilon+D] ) & \qquad\,\,\,\mbox{if  } D<l_{\delta}\\ 
			 \lim_{d\to\infty} \Lambda_{Poi}(\si_{\delta}(x)^{N-1}, [\frac{1}{d}, d] )&\qquad\,\,\,\mbox{if  } D=\infty \quad (\mbox{and }K\geq 0)\,.
	\end{cases}
\end{cases}
}
}

Notice that the estimate $\frac{\pi^2}{D^2}$ in (1c),(2c) and (3c) of the theorem when $K=0$ (and hence $\delta=0$) follows as  explicit solution of the corresponding SL-BVP on $[-\frac{D}{2},\frac{D}{2}]$ with density $p(x)=\co_{0}^{N-1}(x)\equiv 1$ in (1)-(3) above; this is the known Zhong-Yang \cite{ZhY} lower bound.

In order to complete the proof\pinka{,} all that remains is to identify some of these limits as the expressions stated in the theorem for $D\geq l_{\delta}$. For the validity of some of the results which were stated in Subsection \ref{subsec:prelim_results}, one should keep in mind that the measures $\xi=J\cdot m\in \Fknd^M(\R)$ have densities $J\in L^{\infty}_{loc}(int(supp(\xi)))$.
\begin{myitemize}
	\item[(1a)] From Theorem \ref{thm:specCont} it follows that 
	$$\lim_{d\to l_{\delta}}\Lambda_{Poi}(\cos(\sqrt{\delta}x)^{N-1}, [-\frac{d}{2}, \,\frac{d}{2}] )=\Lambda_{Poi}(\cos(\sqrt{\delta}x)^{N-1}, [-\frac{\pi}{2\sqrt{\delta}},\, \frac{\pi}{2\sqrt{\delta}}] )\,.$$ 
	$\lamM_{K,N,D}$, being a sharp lower bound, must equal the sharp Lichnerowicz bound $\frac{KN}{N-1}$ for \red{$CD_b(K,N)$} conditions with $K>0$, $N>1$ when there is no Diameter condition. 
	\item[(2a),(3b)] We conclude from Theorem \ref{thm:specCont}  that in case (2a)
	$$\lim_{d\to\infty}\Lambda_{Poi}(e^{-\frac{Kx^2}{2}}, [-d/2, d/2] )=\Lambda_{Poi}(e^{-\frac{Kx^2}{2}}, (-\infty, \infty))\,.$$ 
	The latter expression is equal to $K$, the \Poinc constant of the Gaussian  (see \cite{BGL} for a proof). In case (3b):
	 $$\lim_{d\to\infty}\Lambda_{Poi}(\cosh(\sqrt{\delta}x)^{N-1},[-d/2, d/2] )=\Lambda_{Poi}(\cosh(\sqrt{\delta}x)^{N-1}, (-\infty, \infty))\,,$$ 
	which must equal $\frac{KN}{N-1}$, since this is the sharp bound proved in \cite{Mil5} (see also \cite{Oht1} when $\partial M=\emptyset$).
		\item[{\Large $\substack{(1b),(1c),(2b),(2c)\\(4a),(3a),(3c)}$}]
	
	In all these cases but cases (3a) and (4a) (in which the estimate\\ $\Lambda_{Poi}(M,\gfrak,\mu)\geq \lam_{K,N,D}$ was valid subject to the proviso $D<l_{\delta}$ if $K<0$ and $N\leq 0$),  we can straightforwardly conclude from  Propositions \ref{prop:RayZero} and \ref{prop:RayZero2} that
	the limits to which $\lamM_{K,N,D}$ equals, are all 0. 
%	\\
%	 $\inf_{(\hfrak,d)\in \D^{reg}_{(K,N,l_{\delta})}}\lam(\hfrak,d)=0$, i.e 0	is the solution to the optimization problem.
	However, it turns out that we can conclude by a similar argument that also in cases $(3a)$ and $(4a)$, when $D\geq l_{\delta}=\frac{\pi}{\sqrt{\delta}}$, the sharp lower bound must also be 0. 
	Indeed, consider a sequence of numbers  $\epsilon_n>0$, and $0<d_n<l_{\delta}$ s.t. 
	\begin{itemize}
  	\item $\epsilon_n\to 0$\,,
		\item $d_n\to l_{\delta}$\,,
		\item $[\epsilon_n, \epsilon_n+d_n]\subset (0,l_{\delta})$\,.
	\end{itemize}
	According to Proposition \ref{prop:RayZero}
	$$\Lambda_{Poi}(\cos(\sqrt{\delta}x)^{N-1}, [-\frac{d_n}{2},\frac{d_n}{2} ] )\stackrel{n\to\infty}{\longrightarrow} 0\,,$$
	hence in case $(3a)$ $\lamM_{K,N,D}=0$. Moreover, according to the monotonicity  Theorem \ref{thm:SpectralMonotonicity}, for the measures $d\xi_n:=\sin(\sqrt{\delta}x)^{N-1}1_{[\epsilon_n, \epsilon_n+d_n]}dm$
	 for every $n\in\Nbb$ it holds that :
	\eq{&\Lambda_{Poi}(\sin(\sqrt{\delta}x)^{N-1}, [\epsilon_n, \epsilon_n+d_n] )\\& \leq \Lambda_{Poi}(\sin(\sqrt{\delta}x)^{N-1}, [\frac{l_{\delta}}{2}-\frac{d_n}{2},\frac{l_{\delta}}{2}+\frac{d_n}{2} ] )=\Lambda_{Poi}(\cos(\sqrt{\delta}x)^{N-1}, [-\frac{d_n}{2},\frac{d_n}{2} ] )\,,}
	hence the argument of case $(3a)$ implies that  $\Lambda_{Poi}(\xi_n)\to 0$ as $n\to\infty$, implying that  $\lamM_{K,N,D}=0$ also in case $(4a)$.

	\item[(4c)] According to  Proposition \ref{prop:RayZero2} we conclude that for $N\in (-1,0]$:
	   $$\lim_{d\to\infty} \Lambda_{Poi}(x^{N-1},  [\frac{1}{d},d])=0\,.$$ 
	   However this must also be the limit when $N=-1$, since
		we know that for $N=-1$ there is equality $\Lambda_{Poi}(1,[-\frac{d}{2},\frac{d}{2}])=\Lambda_{Poi}(x^{N-1},[\epsilon, \epsilon+d])$ for every $\epsilon>0$ and $d>0$, therefore the expression in $(4c)$ must equal the expression given in $(3c)$ for $N=-1$, which we proved to be $0$. 

	\item[(4b)] Due to Theorem \ref{thm:specCont} 
	for every $\epsilon>0$ the following identity is satisfied:
	$$\lim_{d\to\infty} \Lambda_{Poi}(\si_{\delta}(x)^{N-1}, [\epsilon, d] )= \Lambda_{Poi}(\si_{\delta}(x)^{N-1}, [\epsilon, \infty) )\,.$$
	Then by Lemma \ref{D_monotonicity} and Remark \ref{remk:D_monotonicity} we conclude that the solution to the minimization problem is given by
	$$\lamM_{K,N,D}=\lim_{\epsilon\to 0+} \Lambda_{Poi}(\si_{\delta}(x)^{N-1}, [\epsilon, \infty) )\,.$$
    \end{myitemize}

\end{enumerate}	
 \end{proof}

\section{Appendix}

We provide proofs for statements which where presented in the text, but for reasons of clarity and coherence, were embedded in this separate section.

Throughout we denote by $o(1)$ terms which go to 0 as $n\to\infty$, where $n\in\N$ denotes the index of some \mesi{}.  

\subsection{Proofs of Lemmas \ref{lem:lemma1} and \ref{lem:lemma2}}

Before we get into the proof we mention a useful simple identity which will be employed throughout.

Assume $A,B>0$, and let $\epsilon>0$, then for $\epsilon':= \frac{\epsilon B^2}{B+A+\epsilon B}$ holds the identity:
\eql{\label{identity:ratio_epsilon}\frac{A+\epsilon'}{B-\epsilon'}=\frac{A}{B}+\epsilon\,.} 
Clearly this implies that for any $0<\epsilon''\leq \epsilon'$: $\frac{A+\epsilon''}{B-\epsilon''}\leq \frac{A}{B}+\epsilon$.

\begin{proof}[Proof of Lemma \ref{lem:lemma1}]

Define a continuous extension $\tl{f}$ of $f|_{I}$ from $I$ to the entire $\R$ as in Definition \ref{defn:cont_ext}, and by $\tl{f}_s$ denote its convolution with a mollifier as in \eqref{eqn:mollified}. Considering that 
$\deriv{\xi}{m}\in L_{loc}^{\infty}(int(supp(\xi)))$, it follows from Lemma \ref{lem:MollifierConvergence} that $\tl{f}_s\to \tl{f}$ uniformly on compact subsets, and $\tl{f}'_s\to \tl{f}'$  in $L_{loc}^p(int(supp(\xi)); \xi)$ for any $p\in [1,\infty)$.

\bigskip

We define
\[ g_s(x)=\tl{f}_s(x)-c_{s} \qquad \text{where }\quad c_s:=\int \tl{f}_s(x)d\xi(x)\,. \]
Evidently $g_s\in C^{\infty}(\R)$ and $\xi(g_s)=0$ for all $s>0$. 
\bigskip

Define $A:=\int_I f'(x)^2d\xi(x)$ and $B:=\int_I f(x)^2d\xi(x)$ (which is strictly positive by assumption), and set $\epsilon':= \frac{\epsilon B^2}{B+A+\epsilon B}>0$ . 
 Let $I_s:=\{x: d(x,I)\leq s\}$; notice that since  $\tl{f}_s(x)=\tl{f}(x)$ on $ I_s^c$, it follows that $\tl{f}_s'(x)=\tl{f}'(x)=0$ on $ I_s^c$, while on $I_{s_0}$ the functions $\tl{f}_s(x) $ and $\tl{f}'_s(x)$  converge in $L^2(I_{s_0}; \xi)$  to $\tl{f}(x)$ and $\tl{f}'(x)$ respectively, if we pick $s_0>0$ sufficiently small so that $I_{s_0}\subset int(supp(\xi))$. Notice that 
$g_{s_0}'=0$ on $I_{s_0}^c$ and for $s_0>0$ sufficiently small we can ensure that the following estimates/identities hold:
\begin{enumerate}
	\item $\int_{I_{s_0}\setminus I} \tl{f}_{s_0}'(x)^2d\xi(x)\leq\epsilon'/2$. 
	\item $\int_{I} \tl{f}_{s_0}'(x)^2d\xi(x)\leq\int_{I} \tl{f}'(x)^2d\xi(x)+\epsilon'/2$.
	\item $\int_I\tl{f}_{s_0}(x)^2d\xi(x)\geq \int_I\tl{f}(x)^2d\xi(x)-\epsilon'/2$.
	\item $\int_{I_{s_0}\setminus I}\tl{f}_{s_0}(x)^2d\xi(x)\geq \int_{I_{s_0}\setminus I}\tl{f}(x)^2d\xi(x)-\epsilon'/2$.
	\item $\int\tl{f}_{s_0}(x)d\xi(x)= \int\tl{f}(x)d\xi(x)$ (by Fubini).
	\item $\int\tl{f}_{s_0}^2(x)d\xi(x)\leq \int\tl{f}(x)^2d\xi(x)+\epsilon'$.
\end{enumerate}

Notice that $\xi(g_{s_0}^2)=\xi(\tl{f}_{s_0}^2)-\xi(\tl{f}_{s_0})^2\leq \xi(\tl{f}_{s_0}^2)$ ($\xi$ by assumption is a probability measure), since $\tl{f}\in L^2(\xi)$ we conclude from (6) that $g_{s_0}\in L^2(\xi)$.

%We remark that estimate $(1)$ is justified by 
%\[ \int_{I_{s_0}\setminus I} (\tl{f}_{s_0}'(x)^2-\tl{f}'(x)^2)d\xi(x)+\cancel{\int_{I_{s_0}\setminus I}\tl{f}'(x)^2d\xi(x)}\to 0 \,.\] 
By estimates $(1)$ and $(2)$ it follows that
$$ \int g_{s_0}'(x)^2d\xi(x)=\int_{I} \tl{f}_{s_0}'(x)^2d\xi(x)+\int_{I_{s_0}\setminus I} \tl{f}_{s_0}'(x)^2d\xi(x)\leq \int_{I} \tl{f}'(x)^2d\xi(x) +\epsilon'\,.$$

Similarly by estimates $(3),(4)$ and $(5)$: 
{\footnotesize
\eq{\int g_{s_0}(x)^2d\xi(x)&= \int_I\tl{f}_{s_0}^2(x)d\xi(x)+\int_{I^c}\tl{f}_{s_0}^2(x)d\xi(x)-(\int \tl{f}_{s_0}(x)d\xi(x))^2\\
& \geq \int_I\tl{f}^2(x)d\xi(x)-\epsilon'/2+\int_{I_{s_0}\setminus I}\tl{f}^2(x)d\xi(x)-\epsilon'/2+
\int_{ I_{s_0}^c}\tl{f}^2(x)d\xi(x)-(\int \tl{f}(x)d\xi(x))^2\\&
\stackrel{\int_I\tl{f}d\xi=0}{=}\int_If^2(x)d\xi(x)+\int_{I^c}\tl{f}^2(x)d\xi(x)-(\int_{I^c} \tl{f}(x)d\xi(x))^2-\epsilon'\\&
\stackrel{Jensen}{\geq} \int_If^2(x)d\xi(x)-\epsilon'\,.
 } 
}
From the definition of $\epsilon'$ we can conclude that $g:=g_{s_0}\in C^{\infty}(\R)\cap L^2(\R)$ satisfies
$$Ray[g](\xi)\leq \frac{\int_{I} f'(x)^2 d\xi(x)+\epsilon'}{\int_{I} f(x)^2 d\xi(x)-\epsilon'}\leq Ray[f](\xi|_{I})+\epsilon\,.$$

\end{proof}
%-----------------------------------------

\begin{proof}[Proof of Lemma \ref{lem:lemma2}]

Notice that Since $g\in\F_b$ and $\xi(1)<\infty$, it holds that $g,g'\in L^2(\R; \xi)\cap\Cinf_b(\R)$. Clearly we may assume $\xi\neq 0$.

For every $ m\in \Nbb$ we associate a function $\phi_m\in C_c^{\infty}(\R)$ such that the following conditions hold:
\begin{enumerate}
	\item $\phi_m$  is supported on a compact interval $I_m$, with $\xi(\phi_m)>0$. 
	\item On $I_m$ the following estimate holds: 
		$$||\phi_{m}'||_{\infty}\leq \frac{1}{m}\,.$$
	\item $0\leq \phi_m\leq 1$ and  $\phi_m(x)\nearrow_{m\to\infty} 1$ for any $x\in \R$. 
\end{enumerate}
Of course we may also assume that the sequence $(I_m)_{m\in \Nbb}$  is a \mesi{} which exhausts $\R$ with respect to $\xi$. We define
$$  h_m(x):=(g-d_m)\cdot \phi_m \qquad \text{where} \quad d_m:=\frac{\int g\phi_md\xi}{\int \phi_md\xi}\,.$$

Notice that $\xi(h_m)=0$ for all $m\in\Nbb$. In addition since $\xi(g)=0$, $d_m\to 0$ as $m\to\infty$. Consider now the following identity
\eql{\label{ident:diff_h_m} \int h_m'^2(x)d\xi(x)&=\int g'(x)^2\phi_m(x)^2d\xi(x)+2\int g'(x)(g(x)-d_m)\phi_m(x)\phi_m'(x)d\xi(x)\\ \nonumber&+\int\phi_m'(x)^2(g(x)-d_m)^2d\xi(x)\,.
}
By the Cauchy-Schwarz inequality:
$$\Abs{\int g'(x)(g(x)-d_m)\phi_m(x)\phi_m'(x)d\xi(x)}\leq ||g'\phi_m||_{L^2(\xi)}||(g-d_m)\phi_m'||_{L^2(\xi)}\,.$$
Considering that $||g'\phi_m||_{L^2(\xi)}<\infty$, 
 since $\phi_m$ and $g'$ are bounded, 
  and that $||(g-d_m)\phi_m'||_{L^2(\xi)}\stackrel{m\to\infty}{\longrightarrow}0$, since
  $(g-d_m)$ is bounded and $||\phi'_m||_{L^2(\xi)}\leq\frac{1}{m}$, we conclude that 
$$\int g'(x)(g(x)-d_m)\phi_m(x)\phi_m'(x)d\xi(x)=o(1)$$
as $m\to\infty$.

In addition 
\eq{&\int\phi_m'(x)^2(g(x)-d_m)^2d\xi(x)\\&=\int \phi_m'(x)^2g^2(x)d\xi(x)-2d_m\int g(x)\phi_m'(x)^2d\xi(x)+d_m^2\int\phi_m'(x)^2d\xi(x)=o(1)\,,}
since $d_m\stackrel{m\to\infty}{\longrightarrow} 0$, $g\in L^2(\xi)$ and $||\phi_m'||_{\infty}\leq \frac{1}{m}\stackrel{m\to\infty}{\longrightarrow} 0$.

By property 3, $\phi_m\leq 1$, hence by identity \eqref{ident:diff_h_m}
$$ \int h_m'^2(x)d\xi(x)\leq \int g'(x)^2d\xi(x)+o(1)\,.$$
Since $d_m\stackrel{m\to\infty}{\to} 0$, $\phi_m\nearrow 1$ and $g\in L^2(\xi)$, we also conclude that
\eq{&\int h_m(x)^2 d\xi(x)=\int(g(x)-d_m)^2\phi_m(x)^2d\xi(x)\\&=\int g^2(x)\phi_m(x)^2d\xi(x)-2d_m\int g(x)\phi_m(x)^2d\xi(x)+d_m^2\int \phi_m(x)^2d\xi(x)=\int g^2(x)d\xi(x)+o(1)\,. }
Hence given $\epsilon>0$, for sufficiently large $m$, say $m=m_0$, the following inequality holds 
$$ \Ry{h_{m_0}}{\xi}\leq \Ry{g}{\xi}+\epsilon \,.$$
We conclude that  $h:=h_{m_0}\in \F_{Poi}(\xi)$ satisfies the asserted statement.

\end{proof}

\subsection{Proofs of Propositions \ref{prop:RayZero} and \ref{prop:RayZero2}}

\begin{proof}[Proof of Proposition \ref{prop:RayZero}]
Recall that $I=(-\frac{d}{2},\frac{d}{2})$ where $d\in (0,\infty]$. 
Consider the \mesi{} $(E_n)_{n\in \N}$ where $E_n=[-\frac{d_n}{2}, \frac{d_n}{2}]\subset I$, and $d_n\uparrow d$. Define $\xi_n=\xi|_{E_n}$. Let $d'$ be any fixed positive number smaller than $d_1$, and let  $f_n(x)$ be a sequence of odd functions in $\Cinf_c(\R)$ such that
\begin{itemize}
	\item  $f_n(x)=\frac{3}{d'}x$ on  $[-\frac{d'}{3},\frac{d'}{3}]$.
	\item  $|f_n(x)|=M$ on $E_n\setminus [-\frac{d'}{2},\frac{d'}{2}]$, where $M>1$ is some constant.
	\item $|f_n'(x)|$ are uniformly bounded on $I$ by a constant $l>0$.
\end{itemize} 
Notice that $\int f_n\,d\xi_n=0$ for all $n\in N$, and 
$$
\Ry{f_n}{\xi_n} \leq\frac{\prnt{\frac{3}{d'}}^2 2\cdot\frac{d'}{3}+l^2 2\cdot\frac{d'}{6}}{\int_{-d_n/2}^{-d'/2} M^2 J(x) dm   +\int^{d_n/2}_{d'/2} M^2 J(x) dm\,   }\,.
$$

In addition we notice that the nominator is a fixed number while the denominator diverges as $n\to \infty$ since $\xi(I\setminus (-\frac{d'}{2}, \frac{d'}{2})\,)=\infty$. Thus for each $k>0$ there is $n_k>0$ such that $\Ry{f_{n_k}}{\xi}<\frac{1}{k}$.  Therefore $\lim_{n\to\infty}\Ry{f_n}{\xi_n}=0$, whence $\lim_{n\to\infty}\Lambda_{Poi}(\xi_n)=0$. This was proved for the particular sequence $(E_n)_{n\in \N}$, however this would hold for any other exhausting sequence $(I_n)_{n\in \N}$; indeed, assume $(I_{n_k})_{k\in\N}$ is a sub-sequence of $(I_n)_{n\in \N}$ s.t. $\lim_{k\to\infty}\Lambda_{Poi}(\xi|_{I_{n_k}})=\uplim_{n\to\infty} \Lambda_{Poi}(\xi|_{I_{n}})$. By passing to a sub-sequence we may further assume $E_k\subset I_{n_k}$ for all $k\in \N$. By the diameter monotonicity Lemma \ref{D_monotonicity} it follows that $\Lambda(\xi|_{I_{n_k}})\leq \Lambda(\xi|_{E_k})$ for all $k\in \N$, whence $\uplim_{k\to\infty}\Lambda_{Poi}(\xi_{I_{n_k}})\leq \uplim_{k\to\infty}\Lambda_{Poi}(\xi_{E_k})=0$.

%If $d<\infty$ we define $h_{k}(x):=f_{n_k}(x)\phi_{(-d/2, d/2)}\in C_c^{\infty}(\R)$ and $\Ry{h_k}{\xi}=\Ry{f_{n_k}}{\xi}<\frac{1}{k}$. If $d=\infty$ Let $h_{k}(x):=f_{n_k}(x)\phi_{(-d_{n_k}/2, d_{n_k}/2)}$ such that $|\phi'(x)|\leq \frac{\frac{e^{-\frac{x^2}{4}}}}{J\epsilon}{M}
%According to lemma \ref{lem:lemma2} for any $\epsilon>0$ there exists $h_{k}\in C_c^{\infty}(\R)$ such that
%\[ \Ry{h_k}{\xi}\leq \Ry{f_{n_k}}{\xi}+\epsilon \,.\]
%Hence $\lim_k \Ry{h_k}{\xi}\leq \epsilon$, whence $Poin(\xi)\leq \epsilon$.   Since $\epsilon$ is arbitrary we conclude that $Poin(\xi)=0$. 
\qedhere
\end{proof}

\begin{proof}[Proof of Proposition \ref{prop:RayZero2}]
As in the previous lemma, it is enough to consider the particular \mesi{} $(E_n)_{n\in \Nbb}$, where $E_n=[\frac{1}{n}, n]$, and respectively $\xi_n=x^{N-1}1_{E_n} \cdot m$. 
Let $f_n$ be a smooth compactly supported function on $(0,\infty)$ such that $f_n(x)=(x-c_n)$ on $E_n$, where $c_n=\frac{1}{\xi_n(E_n)}\int xd\xi_n$. Notice that $\int f_nd\xi_n=0$ and furthermore:

\eq{ &\int f_n'(x)^2d\xi_n(x)=\int_{\frac{1}{n}}^nx^{N-1}dx\\
=&\begin{cases} \frac{n^{N}-(\frac{1}{n})^N}{N}=\frac{(\frac{1}{n})^N}{N}((n^2)^N-1)\stackrel{N<0}{=}-\frac{(\frac{1}{n})^N}{N}(1+ o(1)) &\mbox{ if } N\neq 0\\
2\log(n) &\mbox{ if } N=0\,, \end{cases} }
and
\eq{ \int f_n^2(x)d\xi_n(x) = \int_{\frac{1}{n}}^n(x-c_{n})^2x^{N-1}dx&= \int_{\frac{1}{n}}^nx^{N+1}dx-2c_{n}\int_{\frac{1}{n}}^nx^Ndx +c_{n}^2\int_{\frac{1}{n}}^nx^{N-1}dx\\&=\int_{\frac{1}{n}}^n x^{N+1}dx-\frac{(\int_{\frac{1}{n}}^nx^N)^2}{\int_{\frac{1}{n}}^nx^{N-1}dx}:=(\star)\,.}
For $N=0$ 
\eq{(\star)&=\half\prnt{n^2-(\frac{1}{n})^2)}-\frac{(n-\frac{1}{n})^2}{2\log(n)}=n^2(1+o(1))\,,}
and therefore $\Ry{f_n}{\xi_n}=\frac{2\log(n)}{n^2(1+o(1))}\stackrel{n\to\infty}{\longrightarrow} 0$. \\
For $N\neq 0$ 
{\small
\eq{(\star)&=
\prnt{\frac{n^{N+2}}{N+2}-\frac{(\frac{1}{n})^{N+2}}{N+2} }- \frac{ \prnt{\frac{n^{N+1}}{N+1}-\frac{(\frac{1}{n})^{N+1}}{N+1}}^2 }{\frac{n^{N}}{N}-\frac{(\frac{1}{n})^{N}}{N}}= \frac{n^{N+2}}{N+2}(1-(\frac{1}{n^2})^{N+2})-
\frac{n^{2N+2}}{(N+1)^2}\frac{N}{(\frac{1}{n})^N}   \frac{(1-(\frac{1}{n^2})^{N+1})^2  }{((n^2)^N-1) }\\&\stackrel{N+1>0}{=}\frac{n^{N+2}}{N+2}-\frac{N}{(N+1)^2}n^{3N+2}+o(1)\,.
}}
In this case $$\Ry{f_n}{\xi_n}=
\frac{1+ o(1)}{-\frac{N}{N+2}n^{2N+2}+\frac{N^2}{(N+1)^2}n^{4N+2}+o(1) }\stackrel{n\to\infty}{\longrightarrow} 0\,, $$
since $N\in (-1,0)$ and the leading term in the denominator is $n^{2N+2}$ (where $2N+2>0$). Therefore for $N\in (-1,0]$ it holds that $\lim_{n\to\infty}\Lambda_{Poi}(\xi_n)=0$. 
\end{proof}
Note that this proof does not apply to $N=-1$. 

\subsection{A stability theorem for regular SLPs}
In this sub-section we prove Theorem \ref{cont:SL_Eval} which was crucially used in the proof of the monotonicity Theorem \ref{thm:SpectralMonotonicity}. 
Let $Y_0(x)$ be a given non-negative function and let $a_0<b_0$ be two fixed points in $\R$. Let $\epsilon>0$ be such that $Y_0(x)$ is $C^1((a_0-\epsilon, b_0+\epsilon))$ and strictly positive on $[a_0-\epsilon, b_0+\epsilon]$. Given $s\in (-\epsilon,\epsilon)$ we define $Y_s(x):=Y_0(x+s)$.
For all $s\in (-\epsilon, \epsilon)$ the SLP 
\eql{\label{eq:SL_Y} (Y_s(x)u'(x))'=-\lambda Y_s(x)u(x)\qquad u'(a_0)=u'(b_0)=0\,,}
is regular. Denote by $\lambda_1(s,b)$ the function which maps $s$ and $b$ (the position of the right endpoint) to the first non-zero eigenvalue $\lambda_1$ of the problem:
\eql{\label{eq:SL_Y2} (Y_s(x)u'(x))'=-\lambda Y_s(x)u(x)\qquad u'(a_0)=u'(b)=0\,.}

By continuous dependence of the eigenvalues and the eigenfunctions of regular SLPs on the parameters we may further assume that $\epsilon$ is sufficiently small so that for all $(s, b)\in B_{(0,b_0)}(\epsilon)$ it holds that $\min_{x\in [a_0,b]}Y_s(x)>\half\min_{x\in [a_0,b_0]} Y_0(x)$ (so \eqref{eq:SL_Y2} corresponds to a regular SLP), and $\lam(s,b)\leq 2\lam(0,b_0)$. 

\begin{thm}\label{cont:SL_Eval} There exists $\epsilon'>0$ such that for any $s\in (-\epsilon', \epsilon')$ there exists $b'=b'(s)$ such that $\lambda_1(s,b'(s))=\lambda_1(0,b_0)$, and $|b'(s)-b_0|\to 0$ as $s\to 0$.
\end{thm}

\begin{proof}

The theorem is a consequence of the following two claims: 
\begin{enumerate}
	\item $\partial_s\lambda_1(s,b)$ and $\partial_b\lambda_1(s,b)$ are continuous on $B_{(0,b_0)}(\epsilon)$. 
	\item $\partial_b\lambda_1(0,b_0)\neq 0$. 
\end{enumerate}
Indeed by (1) and (2), on $B_{(0,b_0)}(\epsilon)$ the map $f:(s,b)\mapsto (s,\lambda_1(s,b))$ is continuously-differentiable and its differential at $(0,b_0)$: 
$$Df_{(0,b_0)}=
\begin{pmatrix}
	1 & 0\\
	\partial_s(\lambda_1)(0,b_0) & \partial_b(\lambda_1)(0,b_0)
\end{pmatrix}\, $$
 is invertible, hence by the inverse function theorem it is a local homeomorphism at $(0,b_0)$. For the continuity of  $\partial_s\lambda_1$ and $\partial_b\lambda_1$ at $(0,b_0)$ one can use the results of Kong and Zettl \cite{KoZet}, but for completeness we provide an independent proof. We use the following identities, valid at a general point $(s,b)\in B_{(0,b_0)}(\epsilon)$ (cf. \cite[p. 56 and p.80]{Zet}):
\begin{enumerate}
	\item $(\partial_s\lambda_1)(s,b)=\int_{a_0}^{b} \partial_s{Y_{s}}(x)\prnt{u_{s,b}'(x)^2-\lambda_1(s,b)u_{s,b}(x)^2}dx$\,,
	\item {$(\partial_b\lambda_1)(s,b)=-u_{s,b}(b)^2\lambda_1(s,b) Y_{s}(b)$\,,}
\end{enumerate} 
where $u_{s_0,b_0}$ is the normalized eigenfunction (i.e. $\int_{a_0}^{b_0}u_{s_0,b_0}(y)^2Y_{s_0}(y)dy=1$) which corresponds to $\lambda_1(s_0,b_0)$. Being a regular SLP in some small neighborhood of $(s_0,b_0)$ the eigenfunctions $u_{s,b}(x)$ as well as $\lambda_1(s,b)$ depend continuously on $s$ and $b$ \cite[p.55]{Zet}. By integration of \eqref{eq:SL_Y2}
from $a_0$ to $x$ we conclude the following identity for any $(s,b)\in B_{(0,b_0)}(\epsilon)$: 
\eql{ \label{diff_u_eqn} u_{s,b}'(x)=-\lambda(s,b)Y_s(x)^{-1} \int_{a_0}^xY_s(y)u_{s,b}(y)dy\,. }

Notice that if $z_{s,b}\in(a_0,b)$ is the point where $u_{s,b}$ vanishes, then:
{\small
\eq{ &||u_{s,b}||_{\infty}\leq\sup_{x\in [a_0,b]}\int_{z_{s,b}}^x|u_{s,b}'(y)|dy\stackrel{\text{by } C.S.}{\leq} \sqrt{|b-a_0|}\sqrt{\int_{a_0}^b|u_{s,b}'(y)|^2dy}\\&
\leq \sqrt{\frac{|b-a_0|}{\min_{y\in [a_0,b]}Y_s(y)}}\sqrt{\int_{a_0}^b|u_{s,b}'(y)|^2Y_s(y)dy}=\sqrt{\frac{|b-a_0|}{\min_{y\in [a_0,b]}Y_s(y)}}\sqrt{\lambda_1(s,b)\int_{a_0}^b|u_{s,b}(y)|^2Y_s(y)dy}\\&=\sqrt{\lambda_1(s,b)\cdot \frac{|b-a_0|}{\min_{y\in [a_0,b]}Y_s(y)}}\leq 
\sqrt{2\lambda_1(0,b_0)\cdot \frac{|b_0-a_0|+\epsilon}{\half \min_{y\in [a_0,b_0]}Y_s(y)}}:=M_{a_0,b_0,\epsilon}\,.
}}
Therefore the eigenfunctions $u_{s,b}$ are uniformly bounded by $M_{a_0,b_0,\epsilon}$ on $B_{(0,b_0)}(\epsilon)$. Thus if $(s,b) \in B_{(0,b_0)}(\epsilon)$ and $(s_n,b_n)\to (s,b)$, where w.l.o.g. we may assume $\{(s_n,b_n)\}_{n\in \Nbb}\subset B_{(0,b_0)}(\epsilon)$, by Lebesgue dominated convergence:
%
%Choose $0<\epsilon''<\epsilon$ sufficiently small so that $\min_{x\in [a_0,b]}Y_s(x)>\half\min_{x\in [a_0,b_0]} Y_0(x)$ for all $(s,b)\in B_{(s_0,b_0)}(\epsilon'')$.
%
%Define $W_{(0,b_0)}:=B_{(0,b_0)}(\frac{\epsilon}{2})$. Notice that for any $(s,b)\in W_{(0,b_0)}$ it holds that $B_{(s,b)}(\frac{\epsilon}{2})\subset B_{(0,b_0)}(\epsilon)$, hence $||u_{s,b}||_{\infty}\leq M_{a_0,b_0,\epsilon}$ (uniformly bounded).
 %Thus  $(s,b)\in W_{(0,b_0)}$:
\eq{\lim_{(s,b)\to (s_0,b_0)} u_{s,b}'(x)&=-\lambda(s_0,b_0)Y_{s_0}(x)^{-1} \int_{a_0}^{x} \lim_{(s,b)\to (s_0,b_0)}\prnt{Y_s(y)u_{s,b}(y)}dy \\&=-\lambda(s_0,b_0)Y_{s_0}(x)^{-1} \int_{a_0}^{x} Y_{s_0}(y)u_{s_0,b_0}(y)dy = u_{s_0,b_0}'(x)\,,}
showing that $u_{s,b}'(x)$ is continuous at $(s,b)$ (as a function of $s$ and $b$). Furthermore by identity \eqref{diff_u_eqn} the functions $u_{s,b}'(x)$ are also uniformly bounded on $B_{(0,b_0)}(\epsilon)$. 

Looking back at identities 1 and 2, by the continuous dependence of $\partial_s{Y_s}(x), u_{s,b}(x),u'_{s,b}(x)$ and $\lambda_1(s,b)$ on $s$ and $b$, it follows (using dominated convergence again) that $\partial_s\lambda_1(s,b)$ and $\partial_b\lambda_1(s,b)$ are continuous on $B_{(0,b_0)}(\epsilon)$. Furthermore, since $u'_{0,b_0}(b_0)=0$, a simple ODE uniqueness argument implies that 
 $|u_{0,b_0}(b_0)|\neq 0$, whence $(\partial_b\lambda_1)(0,b_0)\neq 0$.
\end{proof}

		%Due to  we know that the sequence of intervals $(I_n)_{n\in\Nbb}$ must exhaust $(0,l_{\delta})$ in the sense that $\bigcup_{n\in\Nbb}I_n=(0,\infty)$. This justifies the statement given in the theorem. 

 \newpage
 
%\usepackage{enumitem}

%In these works the reduction to 1-dimensional BVPs was based on the gradient comparison technique (which is based on maximum principles, hence restricted to compact manifolds), which had been used by P.Li and S.T.Yau in \cite{LY80} to derive the estimate $\lambda_2\geq \frac{\pi^2}{2D^2}$. After this reduction had been established, a secondary step of refining the optimal density among a class of model densities was necessary. The resolution of this second problem was achieved by completely different methods. 
%
%
%In this section our main goal is to present an improvement to the secondary refinement problem. 

\chapter{Functional Inequalities: Explicit Lower Bounds\\ \bigskip The p-\Poinc Inequality}
%\chapter{The p-\Poinc constant}
\label{chp:pPoinc}
In this chapter we derive sharp lower-bounds for the $p$-\Poinc constant for \cwrm{} which satisfy $CDD_b(K,N,D)$. 

We open with a disclaimer that this was not one of the primary goals of this work, therefore the analysis will not be as comprehensive as we provided for the $p=2$ case. We will prove a monotonicity result for $\Lambda_{Poi}^{(p)}$  for $p\in (1,\infty)$ analogous to that we established for the 2-Laplacian in Theorem  \ref{thm:SpectralMonotonicity} over the domain $\Fkndreg$.
As we saw in the previous Chapters, the solution to the minimization problem of $\xi\mapsto \Lambda_{Poi}(\xi)$ over $\Fknd^M(\R)$, turns out to be equal to the solution over the class $\Fkndreg$, due to upper semi-continuity of $\xi\mapsto \Lambda_{Poi}(\xi)$ on the space $\overline{\Fkndreg}$. Considering the analysis provided for the case $p=2$, the same should hold w.r.t. the map $\xi\mapsto \Lambda_{Poi}^{(p)}(\xi)$ with $p\in (1,\infty)$; the arguments are almost the same, therefore a proof of this property is omitted.

As in the case $p=2$, we approach the  minimization problem over $\Fkndreg$ for $p\in (1,\infty)$  via the Euler-Lagrange equation associated with the $p$-\Poinc constant. We will refer to the corresponding problem as `$p$-Sturm--Liouville BVP', and to the relevant theory as `$p$-Sturm--Liouville theory'; it bears this name due to its similarity to the case $p=2$. Many important results which are known to be valid for $p=2$, have been extended to general $p\in (1,\infty)$. Yet, at present its development is still ongoing, and we have no certainty whether certain gaps were bridged. In our approach we were taking a certain assumption regarding the validity of a specific crucial result. The specific assumption is detailed in Subsection \ref{subsec:p_SL_Theory}. It was verified for various `$p$-Sturm--Liouville BVPs' of similar but not identical form to the BVP associated with $\Lambda_{Poi}^{(p)}(\xi)$. We didn't verify the extension to our problem, knowing that it would cause an undesirable digression from the main goals.

\bigskip

Recall the localization theorem yielded Theorem \ref{thm:PoinInequality} which gave a lower bound $\lamp_{K,N,D}$ to $\Lambda^{(p)}_{Poi}$; complementing this result by the extreme points characterization Theorem  \ref{thm:ExtremePoints} yielded Corollary \ref{cor:CC} from which we conclude the following theorem about the sharp lower bound for  $\Lambda^{(p)}_{Poi}(M,\gfrak,\mu)$:
 %when $K\in \R$, $N\in(-\infty, 0]\cup [\max(n,2),\infty]$ and $D>0$:

\begin{thm}\label{LpPoinc:Estimate_gen} If $(M^n,\gfrak, \mu)$ is a \cwrm{}, which satisfies $CDD_b(K,N,D)$, where $K\in \R$ and $N\in (-\infty,0]\cup [\max(n,2),\infty]$, then subject to the proviso  $D<l_{\delta}$ if $K<0$ and $N\leq 0$:
\eql{ \label{pPoinc:Estimate} \Lambda_{Poi}^{(p)}(M,\gfrak,\mu)\geq \lamp_{K,N,D}\,,}
where 
\eql{\label{prob:pLap} \lamp_{K,N,D}:=\inf_{\xi\in \Fknd^{M}(\R)}\Lambda_{Poi}^{(p)}(\xi)\qquad \text{with}\quad \Lambda_{Poi}^{(p)}(\xi):=\inf_{f\in \F_{Poi}^{(p)}(\xi)} \left\{ \frac{\int_{\R} |f'(t)|^pd\xi(t)}{\int_{\R} |f(t)|^pd\xi(t)}\,\, \right\}\,.} 
\end{thm}

%In general solutions to the p-Laplace 'eigenvalue' equation are at least $C^{1,\alpha}$ for some $\alpha>0$ (\cite{Tol} or \cite{NaVa}).

%It is known  that for each $\xi\in\Fknd^{M}$ the first minimization 
		%\eq{		\inf \left\{ \frac{\int_{\R} |f'(t)|^p\xi(t)}{\int_{\R} |f(t)|^pd\xi(t)}:\,\,0\neq f\in C_c^{\infty}(\R)\, \mbox{ and } \, \int_{\R} h_f(t)d\xi(t)=0 \right\}   }
	%problem admits a minimizer \cite{Bre, Dac, Lin1} and we can study the problem through the associated Euler-Lagrange equation. 
	
\section{The eigenvalue equation on the line}
	Assume $\xi=J\cdot m$, $supp(\xi)=[a,b]$, and that $J(x)>0$ is continuous and positive on $[a,b]$.
Existence of minimizers realizing $\Lambda_{Poi}^{(p)}(\xi)$ is proved by direct methods \cite{Lin}, hence it is justified to study the problem via the Euler-Lagrange equations. 
	
%	By localization we have reduced from the original problem  of finding $\Lambda_{Poi}^{(p)}(M,\gfrak,\mu)$, to an optimization problem where comparison of the \Poinc constants $\Lambda_{Poi}^{(p)}(\xi)$ of measures $\xi\in \Fknd^M(\R)$ is involved. 
%	As for the 2-Laplacian it is natural to approach the problem \eqref{prob:pLap} via the associated Euler-Lagrange equations. 
	
	\bigskip
	
 Then a minimizer realizing $\Lambda_{Poi}^{(p)}(\xi)$ is a weak solution to the following BVP on $\R$: 
	\eql{\label{eqn:P_EL} \deriv{}{x}(J(x)f^{\prime(p-1)}(x))+\lambda J(x)f^{(p-1)}(x)=0 \qquad f'(a)=f'(b)=0\,,}
	where $f^{(p-1)}(x)$ stands for  
\[ \gls{f_p_1}:=|f(x)|^{p-2}f(x)=|f(x)|^{p-1}sgn(f(x))\,. \]
We may equivalently write 
	\pinka{\[ \Delta_{p,\xi}f(x) := (p-1)f^{\prime(p-2)}(x)f''(x)+\frac{J'(x)}{J(x)}f^{\prime(p-1)}(x)=-\lambda f^{(p-1)}(x)\,, \]
	where $\gls{Delta_p}$ stands for the weighted  $p$-Laplacian associated with  the measure $\xi$.} 
%As for the spectral-gap problem for the 
 For the study of the solutions to these ODEs we will rely on the same type of Pr\"{u}ffer transformation which was employed by Naber and Valtorta in \cite{Val, NaVa}. 

\subsection{Sturm-Liouville theory of equation \eqref{eqn:P_EL} }\label{subsec:p_SL_Theory}
%We refer the reader to  and references within for the definitions and results which we state here.

Equation \eqref{eqn:P_EL} has a form which is reminiscent of the Sturm-Liouville (SL) problem, which we naturally encounter when we study the eigenvalues of the 2-Laplacian on the interval. It is not a linear problem, but only half-linear; yet, an analogous SL theory for such BVPs on the interval has been developed. Our presentation  of the `$p$-Sturm--Liouville theory' is partially based on \cite[p. 108]{DrKu}.

%; indeed, a general treatment requires the inclusion of the cases in which $J(x)$ vanishes at one of the endpoints, or the interval is infinite, however in view of the arguments which were presented in Chapter \ref{chp:Poinc}, there is no loss of generality in considering only  densities $J(x)$ which satisfy the assumption.
%As we mentioned, we will not provide 'boundary measures analysis' for the p-Laplacian. 
\bigskip
In our present study we assume that $J(x)$ is smooth and strictly positive on $[a,b]$. 
By a solution to \eqref{eqn:P_EL} we understand a function $f$ s.t.  $f\in C^1(a,b)$, $f^{\prime(p-1)}=|f'|^{p-2}f'\in C^1(a,b)$, the equation \eqref{eqn:P_EL} holds at every point, and the boundary conditions are satisfied.
%, and the Dirichlet integral $\int_a^bJ(x)|u'(x)|^pdx$ is finite.

We refer to the parameter $\lambda$ as an  `eigenvalue' of the problem, if the problem has a non trivial (i.e. non-zero) solution with this $\lambda$. We refer to such a solution as the `eigenfunction associated with $\lambda$'. 
\subsubsection{Properties of the spectrum and the eigenfunctions}\label{Properties:pLap_egns}

It is known (e.g \cite{TakKuNa1,TakKuNa2,BiDr,Rei,DoRe}) that the following results apply to the Neumann BVP \eqref{eqn:P_EL} (and other similar BVPs)
\begin{enumerate}
	\item The set of all eigenvalues form an increasing sequence $(\lambda_n)_{n\in\N}$, s.t. $\lambda_1>0$ and $\lim_{n\to\infty}\lambda_{n}=+\infty$. 
	\item Every $\lambda_n$ is simple (in the sense that all eigenfunctions $u_n$ associated with $\lambda_n$ are mutually proportional). 
	\item The eigenfunction $u_n$ has precisely $n$ zeros in $(a,b)$.
\end{enumerate}
In \cite{MYZ} (and some references therein) it is stated that for a similar problem 
\eq{ \deriv{}{x}(f^{\prime(p-1)}(x))+\lambda J(x)f^{(p-1)}(x)=0\,,}
with boundary conditions which are more general than Neumann,
an analogue of Theorem \ref{SLResults2} also holds, i.e the eigenvalues are continuously differentiable in the weight (in the Fr\'{e}chet sense), and the eigenfunctions depend continuously on the weight (explicit expressions for the eigenvalue derivatives are proved in  \cite{MYZ} \pinka{in a slightly different context}). We proceed  assuming the validity of this result to the present Neumann BVP. We will refer to it as the `technical assumption'.

%
%
%We assume that 
%\begin{enumerate}
	%\item
	%\item For any $x\in (a,b)$ we have $J\in L^1(a,x)$ and $J^{1-p'}\in L^1(x,b)$, where $\frac{1}{p}+\frac{1}{p'}=1$ (e actually need not assume $J,J^{1-p'}\in L^1(a,b)$). 
%\end{enumerate}
   %
%
%
%\begin{defn}
%The SL property for \eqref{p_Lap_SL_eqn} is satisfied iff the following two conditions hold (\cite[p. 108]{DrKu})
%\eq{  &\lim_{t\to a+}\prnt{\int_a^t p(x)dx}\prnt{\int_t^bp(t')^{1-p'}(t')dt'}^{p-1}=0\\&
%\lim_{t\to b-}\prnt{\int_a^t p(x)dx}\prnt{\int_t^bp(t')^{1-p'}(t')dt'}^{p-1}=0
%}
%\end{defn}
%The reader is referred also to 

\section{Comparison of eigenvalues of equation \eqref{eqn:P_EL}}

Our goal is to derive a eigenvalue monotonicity theorem analogous to Theorem  \ref{thm:SpectralMonotonicity} for general $p\in (1,\infty)$ (which at present applies only to positive $N$ values). 
The strategy will be the same as for the 2-Laplacian. Start with the BVP \eqref{eqn:P_EL} that $f$ satisfies on $(a,b)$ (which corresponds to the eigenvalue $\lambda$). We would like to compare $\lambda$ with the eigenvalue $\tilde{\lambda}$ of  \eqref{eqn:P_EL}, where we replace $J$ with a density $\tilde{J}$, which we consider as a perturbation of $J$. 
 Rather than comparing the eigenvalues of these equations directly, we consider the following IVP:
\eql{\label{p_Lap_SL_eqn} \deriv{}{x}(\tilde{J}(x)f^{\prime(p-1)})+\lambda \tilde{J}(x)f^{(p-1)}=0 \qquad f(a)=-1,\, f'(a)=0\,.}
\begin{remk} It is known (\cite{NaVa} prop. 4.6,4.7) that $f$ and $f^{\prime(p-1)}$ are of class $C^1(\R)$, and depend continuously on the parameters $N,K$ and $a$ in the sense of local uniform convergence of $f$ and $f^{\prime(p-1)}$.
\end{remk}

\red{As we showed for the 2-Laplacian, subject to the technical assumption of Subsection \ref{subsec:p_SL_Theory},} for sufficiently small perturbations the solution to the IVP $f$ will satisfy $f'(a)=f'(b')=0$ for some $b'>a$ in a neighborhood of $b$. The principle of diameter comparison is that if $b-a\leq b'-a$ (resp. $b-a\geq b'-a$) we can conclude that $\lambda\leq \tilde{\lambda}$ ($\lambda\geq \tilde{\lambda}$). Up to this point the arguments are not different from those which led to the proof of Theorem  \ref{thm:SpectralMonotonicity}. Yet, the case $p=2$ was exceptional due to the fact that the resulting EL equation was linear, while for general $p\in(1,\infty)$ the equation is only  half-linear. As can be verified, the Liouville transformation which we applied in the case $p=2$ is not very useful for general $p$. Yet, the Pr\"{u}ffer transformation, which we have previously mentioned applies to linear as well as to half-linear 2nd order ODEs. Naber and Valtorta \cite{Val,NaVa} have successfully used it in order to refine the optimal $\xi\in\FF_{[K,n,D]}^M(\R)$ in \eqref{prob:pLap}. We will also use this transformation which greatly simplifies the problem.

\bigskip

%For the refinement problem, i.e finding the 'optimal' $\xi\in \Fknd^M$ in problem \eqref{prob:pLap}, it is sufficient (in view of the expressions \eqref{DensityExpressions} for $J_{K,N,\hfrak}$) to compare translations (i.e $x\to x+s$) of the following  four densities
%\begin{enumerate}
  %\item $J(x)=x^{N-1}$\qquad $\delta=0$
	%\item $J(x)=\cos^{N-1}(\sqrt{\delta}x)$\qquad $\delta>0$
	%\item $J(x)=\cosh^{N-1}(\sqrt{|\delta|}x)$\qquad $\delta<0$
	%\item $J(x)=\sinh^{N-1}(\sqrt{|\delta|}x)$\qquad $\delta<0$
%\end{enumerate}.

\subsection{The Pr\"{u}ffer transformation}

\begin{defn} Let $p\in (1,\infty)$, we define the positive number $\pi_p$ by
\[ \gls{pi_p}:=\int_{-1}^{1}\frac{ds}{(1-s^p)^{\frac{1}{p}}}=\frac{2\pi}{p\sin(\frac{\pi}{p})}\,,\]
and  we implicitly define the function $\sin_p(x)$ by
\eq{ &x=\int_0^{\gls{sin_p}}\frac{1}{(1-s^p)^{\frac{1}{p}}}ds \qquad \text{ if } x\in [-\pip, \pip]\\
    &\sin_p(x)=\sin_p(\pi_p-x) \qquad \text{ if } x\in [\pip, \frac{3\pi_p}{2} ] \,. 
}

This function was defined on $[-\pip, \frac{3\pi_p}{2}]$, but extends to a $C^1$ function on $\R$  by periodicity. We set $\gls{cos_p}:=\deriv{}{x}\sin_p(x)$. The usual trigonometric identity $|\sin_p(x)|^p+|\cos_p(x)|^p=1$ holds. From this one can conclude that \blue{$\cos_p^{(p-1)}(x):=|\cos_p(x)|^{p-2}\cos_p(x)\in C^1(\R)$}.

The reader is referred to \cite{Girg} for a comprehensive discussion about further differentiability properties of these functions (in particular lemma 4.3).

%From this one can conclude that $\cos_p^{(p-1)}(x)\in C^1(\R)$. 
\end{defn}
\bigskip

Consider the equation
\[ \deriv{}{x}(J(x) f^{\prime(p-1)}(x))+\lambda J(x) f^{(p-1)}(x)=0\, \]
on $[a,b]$, where $J\in C^1([a,b], \R_+^*)$. Equivalently, defining $T(x):=-(\log(J(x)))'$, we consider the equation 
\[ (p-1)f^{\prime(p-2)}(x)f''(x)-T(x)f^{\prime(p-1)}(x))+\lambda f^{(p-1)}(x)=0\,.\]
The Pr\"{u}ffer transformation turns this  equation  into two 1st order equations in two variables $(e,\phi)$ (the radial and the polar variable respectively). 

To this end set $\gls{alpha_lambda}:=\prnt{\frac{\lambda}{p-1}}^{\frac{1}{p}}$ (a fixed constant) and define $e(x)$ and $\phi(x)$ implicitly by the following relations:
\[ \alpha f(x):=e(x) \sin_p(\phi(x)) \qquad f'(x):=e(x) \cos_p(\phi(x))\,. \]
According to this definition $e(x)=(f'(x)^p+\alpha^p f(x)^p)^{\frac{1}{p}}$ and for $\phi(x)\in [-\pip, \pip]$ we can write \blue{ $\phi(x)=\arctan_p(\frac{\alpha f(x)}{\dot{f}(x)})$}. 

By taking derivatives one can verify that $\phi(x)$ and $e(x)$ satisfy the following ODEs:
\eql{\label{eqn:PolarEquations}\phi'(x)&=\alpha -T(x)\Theta(\phi(x))\qquad \text{where} \quad\Theta(\phi(x)):=\frac{1}{p-1}\cos_p^{p-1}(\phi(x))\sin_p(\phi(x))\,\,\text{and}\\
e'(x)&=\frac{T(x)e(x)}{p-1}\cos_p^p(\phi(x))\,.
}
The points $x\in [a,b]$ where $\phi(x)\in \{\frac{\pi_p}{2}+\pi_p k : k\in \mathbb{Z}\}$ correspond to the zeros of $f'(x)$.

\subsubsection{The statement and formulation of the problem}

We recall the definitions in \ref{defn:Dreg}. Given $K\in \R$ and $N\in (-\infty,0]\cup (1,\infty]$, the parametric domain of regularity is defined by
\blue{
$$\D^{reg}_{(K,N,D)}:=\{ (\hfrak, d)\in \R\times (0,D_{\delta}): \,\, [-\frac{d}{2}, \frac{d}{2}]\subset int\prnt{ \isupp\prnt{ J_{K,N,\hfrak}  }} \}\,. $$
}
We denote by $\pi_{\hfrak},\pi_d:\D^{reg}_{(K,N,D)}\to \R$ the natural projections defined by:
\eq{ &\pi_{\hfrak}(\hfrak_0, d_0)=\hfrak_0 \qquad \text{and}\qquad
\pi_{d}(\hfrak_0, d_0)=d_0 \,.}
We also define the measure valued map $\xi_{(K,N)}: \D^{reg}_{(K,N,D)}\to \Fkndreg\subset \Fknd^M$, by $\xi_{(K,N)}(\hfrak, d)=J_{K,N,\hfrak}1_{[-\frac{d}{2}, \frac{d}{2}]} \cdot m$. We set  $\lamp:\D^{reg}_{(K,N,D)}\to \R_+$  by $\lamp(\hfrak, d):=\Lambda_{Poi}^{(p)}(\xi_{(K,N)}(\hfrak, d))$. 
Recall also the definition of $\Fkndreg$ given in \red{Definition} \ref{defn:Dreg}. We have the equality

$$\Lambda^{(p)}_{Poi}\prnt{\Fkndreg}=\Lambda^{(p)}_{Poi}\prnt{\xi_{(K,N)}\prnt{\D^{reg}_{(K,N,D)}}}\qquad\prnt{=\lamp\prnt{\D^{reg}_{(K,N,D)}}}\,.$$

As for the 2-Laplacian we solve the optimization problem of finding $\inf_{\xi\in \Fkndreg} \Lambda^{(p)}_{Poi}(\xi)$, or equivalently 
$$\text{Find:}\qquad \inf_{(\hfrak, d)\in \D^{reg}_{(K,N,D)}}\lamp(\hfrak, d)\,.$$

%
%
%These observations motivate the following definitions.
%\begin{defn}[$\D^{reg}_{(K,N,D)}$,\,$\xi_{(K,N),r}$,\,$\FF^{reg}_{(K,N,D)}$]\label{defn:Dreg}
 %Assume $K\in \R$, $N\in(-\infty,\infty]$ and $D\in (0,\infty]$. We define the 'parametric domain of regularity':
%$$\D^{reg}_{(K,N,D)}:=\{ (\hfrak, d)\in \R\times (0,D]: \,\, [-\frac{d}{2}, \frac{d}{2}]\subset \prnt{ \zfrak_{-}(J_{K,N,\hfrak}),\zfrak_{+}(J_{K,N,\hfrak})} \}\,. $$
%We denote by $\pi_{\hfrak},\pi_d:\D^{reg}_{(K,N,D)}\to \R$ the natural projections defined by:
%\eq{ &\pi_{\hfrak_0}\D^{reg}_{(K,N,D)}=\{d:\,(\hfrak_0, d)\in \D^{reg}_{(K,N,D)}\}\\&
%\pi_{d_0}\D^{reg}_{(K,N,D)}=\{\hfrak:\,(\hfrak, d_0)\in \D^{reg}_{(K,N,D)}\} }
%For $r\in \R$ and $c>0$ we define $$\xi_{(K,N),r,c}: \D^{reg}_{(K,N,D)}\to \Fknd^M$$  by 
%$$\xi_{(K,N),r,c}(\hfrak, d)(x):=cJ_{K,N,\hfrak}(x+r)1_{[-\frac{d}{2}, \frac{d}{2}]}(x+r) \cdot m\,,$$
%
%and set 
%$$\FF^{reg}_{(K,N,D)}:=\R_+^*\cdot \prnt{\bigcup_{r\in\R}\xi_{(K,N),r}\prnt{\D^{reg}_{(K,N,D)}}}\,.$$ 
%where $\R_+^*:=\{c\in \R:\, c>0\}$. 
%This is a subset of $\Fknd(\R)$. We will refer to it as the 'domain of regularity'. 
%\end{defn}

\bigskip

The following result is analogous to Theorem  \ref{thm:SpectralMonotonicity} 
regarding the monotonic dependence of the 2-\Poinc constant on the parameters $\hfrak$ and $d$. 
Naber and Valtorta proved  in \cite{Val,NaVa} that for $K\leq 0$ when $N=n$ (i.e. $\mu=\mu_g$) the solution to the optimization problem will have $\hfrak=0$ (i.e. symmetric density around the origin); 

Matei \cite{Mat} proved a \blue{Lichnerowicz-type estimate} (\red{i.e.} $CD_b(K,N)$ estimate, which is independent of $D$) for the range $K>0$. The statement below generalizes the Naber-Valtorta result to $N\geq n$, $K\in\R$, $D\in\R$ ($CDD_b(K,N,D)$ conditions), and shows how $\lamp(\hfrak, d)$ depends {\bf monotonically} on $\hfrak$. The main result of this chapter is the following:

\begin{thm} \label{thm:SpectralMonotonicity_p}
\begin{enumerate}
   \item For any fixed $\hfrak_0\in \pi_{h}\D^{reg}_{(K,N,D)}$, the function $d\mapsto \lamp(\hfrak_0, d)$ on $\pi_{\hfrak}^{-1}(h_0)$ decreases as $d$ increases.
    \item
        For any fixed $d_0\in \pi_{d}\D^{reg}_{(K,N,D)}$ the function $\hfrak\mapsto \lamp(\xi_{(K,N)}(\hfrak, d_0))$ on  $\pi_{d}^{-1}(d_0)$ monotonically increases as $|\hfrak|$ increases if $N\in  (1,\infty]$.
\end{enumerate}
\end{thm}

\par{
We remind the reader that here we identify $\pi_{d}^{-1}(d_0)$ (resp. $\pi_{\hfrak}^{-1}(\hfrak_0)$) with the set of points $\pi_{\hfrak}\prnt{\pi_{d}^{-1}(d_0)}=\{\hfrak: \, (\hfrak, d_0)\in D^{reg}_{(K,N,D)}\}$ (resp. $\pi_{d}\prnt{\pi_{\hfrak}^{-1}(\hfrak_0)}=\{d: \, (\hfrak_0, d)\in D^{reg}_{(K,N,D)}\}$).
In \eqref{DensityExpressions} we interpreted the variation over the parameter $\hfrak$ as translation of a fixed density (with the exception of the case (d2)). Thanks to the technical \pinka{assumption} listed in Subsection \ref{subsec:p_SL_Theory}, specifically the continuous dependence of the eigenvalues on the weight w.r.t. the SL-BVP metric, \blue{in analogy} to the proof of Theorem \ref{thm:SpectralMonotonicity}, \blue{it suffices} to prove the statements of Theorem \ref{thm:SpectralMonotonicity_p} for $N<\infty$, and without consideration of the exceptional cases (b2) and (c3). }

Then the only dictionary we need for translating from the parameter $\hfrak$ to $s$ and vice-versa is given by the following table: 
{\footnotesize
\eql{\label{hfrak_s} 
&J_{K, N, \hfrak}(x)=(\co_{\delta}(x)+\frac{\hfrak}{N-1} \si_{\delta}(x))_+^{N-1} =g_s\cdot Y_{K,N,\hfrak}(x+s) \qquad \text{where}\\
&g_s=\begin{cases}
\frac{1}{\cos(s)^{N-1}} \quad&\text{with}\quad s=-\tan^{-1}(\frac{\hfrak}{(N-1)\sqrt{\delta}})\quad \\
 \frac{1}{s^{N-1}} \quad&\text{with}\quad s=\frac{N-1}{\hfrak}  \\    
 \frac{1}{\cosh(s)^{N-1}} \quad&\text{with}\quad s=\tanh^{-1}(\frac{\hfrak}{(N-1)\sqrt{-\delta}}) \\         
 \frac{1}{\sinh(s)^{N-1}} \quad&\text{with}\quad s=\coth^{-1}(\frac{\hfrak}{(N-1)\sqrt{-\delta}}) \\        
% \end{cases}
\end{cases}\,,
Y_{K,N,\hfrak}(x)=\begin{cases}
 \cos(\sqrt{\delta} x)^{N-1} & \,\,\mbox{(1) if } \delta >0\\
 x^{N-1}_+ & \,\,\mbox{(2) if } \delta =0\\
 \cosh(\sqrt{-\delta} x)^{N-1} & \mbox{ (3) if } \delta <0 \text{ and } |\frac{\hfrak}{N-1}|<1 \\
 \sinh(\sqrt{-\delta} x)^{N-1}_+ & \mbox{ (4) if } \delta <0 \text{ and } |\frac{\hfrak}{N-1}|> 1 \\
% \end{cases}
\end{cases}\,.
}
}
We define $T(x):=-\log\prnt{J_{K,N,\hfrak}(x)}'$; \blue{explicitly,} $T(x)$ corresponds to one of the following functions: 
\eql{\label{eqn:Tx} }
\begin{enumerate}
  \item $(N-1)\sqrt{|\delta|}\tan(\sqrt{|\delta|}x)$ with $x\in (-\frac{\pi_p}{2\sqrt{\delta}},\frac{\pi_p}{2\sqrt{\delta}})$ when $\delta>0$.
	\item $-(N-1)\frac{1}{x}$ with $x\in(0,\infty)$ when $\delta=0$.
	\item $-(N-1)\sqrt{|\delta|}\tanh(\sqrt{|\delta|}x)$ when $\delta<0$ and $|\frac{\hfrak}{N-1}|<1$.
	\item $-(N-1)\sqrt{|\delta|}\coth(\sqrt{|\delta|}x)$ with $x\in(0,\infty)$ when $\delta<0$ and $|\frac{\hfrak}{N-1}|> 1$.
\end{enumerate} 
We set $$T_s(x):=T(x+s)\,,$$ where throughout  \red{we only consider translations by $s\in\R$ s.t. }
\blue{
$$
  [s-\frac{d}{2}, s+\frac{d}{2}]\subset int\prnt{\isupp\prnt{ Y_{K,N,\hfrak}  }}\,.$$
  }
Furthermore, we will also assume $s(\hfrak)>0$ in all cases specified in \eqref{eqn:Tx}, considering that in cases 1 and 3\blue{, $T(x)$ is an even function,  and in cases 2 and 4, the function $T(x)$ is defined on $(0,\infty)$}.
   
Let $f(x)$ be the solution to the BVP:
\eql{\label{eqn:p_Poinc_f} (p-1)f^{\prime(p-2)}(x)f''(x)-T_{s}(x)f^{\prime(p-1)}(x)+\lambda f^{(p-1)}(x)=0\qquad f'(-\frac{d}{2})=f'(\frac{d}{2})=0\,.}
Using the Pr\"{u}ffer transformation we write it as \red{
\eq{&\phi'(x)=\alpha -T_{s}(x)\Theta(\phi(x))\qquad \phi(-\frac{d}{2})\stackrel{Mod\,\,\pi_p}{=}\phi(\frac{d}{2})\stackrel{Mod\,\,\pi_p}{=}\pip\,, 
\\ & e'(x)=\frac{T(x)e(x)}{p-1}\cos_{p}^p(\phi)\,,\\
&\text{where}\\
&\Theta(\phi(x))=\frac{1}{p-1}\cos_p^{p-1}(\phi(x))\sin_p(\phi(x))\,\quad \text{and}\qquad \alpha=\prnt{\frac{\lambda}{p-1}}^{\frac{1}{p}}\,.
}
}
\subsubsection{Monotonicity of the polar functions}

\bigskip 
Our analysis will consider the inverses of the polar functions $\phi(x)$, hence in order to justify this approach we prove that they  are monotonic between two consecutive zeros of $f'$.

\begin{lem}\label{lem:p_lap_phi_mono} The polar function $\phi(x)$ which is associated with $f(x)$ (which solves \eqref{eqn:p_Poinc_f}) is strictly increasing on $[-\frac{d}{2}, \frac{d}{2}]$. 
\end{lem}
\begin{proof} Denote by $S=\cprnt{x\in [-\frac{d}{2}, \frac{d}{2}]:\,\,\phi'(x)=0}$ the set of critical points of $\phi$. Assume by contradiction that $S\neq \emptyset$. By \eqref{eqn:PolarEquations} at every $x_*\in S$ holds  the equation:
\eql{\label{eqn:PolarEqnDeriv} 0=\phi'(x_*)=\alpha-T(x_*+s)\Theta(\phi(x_*))\,. }
Since $\alpha>0$ we conclude that $\Theta(\phi(x_*))\neq 0$ and $T(x_*+s)\neq 0$. In addition 
\[ \phi''(x_*)=-T'(x_*+s)\Theta(\phi(x_*))-T(x_*+s)\Theta(\phi(x_*))\phi'(x_*)=-T'(x_*+s)\Theta(\phi(x_*))\,.\]
By \eqref{eqn:PolarEqnDeriv} $\Theta(\phi(x_*))=\frac{\alpha}{T(x_*+s)}$ whence
\[ \phi''(x_*)= -\frac{\alpha}{T(x_*+s)}T'(x_*+s)\,.\]
Since  $T'(x_*+s)\neq 0$ in cases 1-4 of \eqref{eqn:Tx}, we conclude that there are no  inflection points.  By \eqref{eqn:PolarEqnDeriv} at the endpoints $\phi'(-\frac{d}{2})=\phi'(\frac{d}{2})=\alpha>0$. Set $x_0=\min S$ and $x_1=\max S$, and notice that:
 
 %, this implies that the point $x_0$ can be at most a local maximum but not a minimum. In view of \ref{eqn:Tx} we consider cases 1-4:
\begin{myitemize}
	\item[Cases 2 and 4] In case 2: $\phi''(x_*)=\frac{\alpha}{x_0+s}>0$, and in case 4: $\phi''(x_*)=\frac{\alpha}{\si_{\delta}(x_0+s)\co_{\delta}(x_*+s)}>0$, since $x_*+s>0$; in particular any critical point $x_*$ must be a local minimum; however since $\phi'(-\frac{d}{2})=\alpha>0$ it must hold that $\phi''(x_0)<0$, which is a contradiction. 
  \item[Cases 1 and 3] By \eqref{eqn:PolarEqnDeriv} we know that $-s\notin S$; in addition at a critical point $x_*$: $\phi''(x_*)=-\frac{\alpha}{\si_{\delta}(x_*+s)\co_{\delta}(x_*+s)}$. Therefore if $x_*<-s$ then $\phi''(x_*)>0$ and if $x_*>-s$ then $\phi''(x_*)<0$.
  Then
  \begin{itemize}
      \item If $0\notin (-\frac{d}{2}+s, \frac{d}{2}+s)$, considering that for every $x_*\in S$ it holds that $\phi''(x_*)<0$, we conclude that in particular $\phi''(x_1)<0$, which is impossible since  $\phi'(\frac{d}{2})=\alpha>0$. 
      \item  If $0\in (-\frac{d}{2}+s, \frac{d}{2}+s)$, then considering that for any $x_*\in S$ s.t. $x_*<-s$ it holds that $\phi''(x_*)>0$, and that
      $\phi'(-\frac{d}{2})>0$ we conclude that $\phi'(x)>0$ on $[-\frac{d}{2}, -s]$ (recall $-s\notin S$). Therefore
      $x_1>-s$, but then $\phi''(x_1)<0$, 
      which as we mentioned above, contradicts that   $\phi'(\frac{d}{2})=\alpha>0$.
  \end{itemize}
	\end{myitemize}
	We conclude that $S=\emptyset$; this implies that  $\phi(x)$  is strictly increasing on $[-\frac{d}{2}, \frac{d}{2}]$. 
 
	%\item In this case $\phi''(x_0)=-\frac{\alpha}{\sin(x+s)\cos(x+s)}=-\frac{2\alpha}{\sin(2(x+s))}$. Since $x_0$ must be a local-maximum, $x_0+s>0$, however this would imply that at any subsequent critical point $x_1>x_0$ it holds that $\phi''(x_1)<0$, which is impossible. Hence $\phi(x)$ must be decreasing for $x>x_0$, however since $\phi'(\frac{d}{2})>0$ this is impossible. 
	%\item In this case $\phi''(x_0)=-\frac{\alpha}{x_0+s}<0$ (since $x_0+s\in (0,\infty)$. Like in the previous case this implies that $\phi'(x)<0 $ on $x_0<x$, which is a contradiction to $\phi'(\frac{d}{2})>0$ .
	%\item In this case $\phi''(x_0)=\frac{\alpha\sqrt{|\delta|}}{\sinh(x_0+s)\cosh(x_0+s)}$.
%
	%\item In this case $\phi''(x_0)=-\frac{\alpha}{\sinh(x+s)\cosh(x_s)}$.
	%As we remarked, $x_0$ must be a local maximum. However since $a>0$ it holds that $x_0+s>0$, implying that $\phi''(x_0)>0$, which is a contradiction.

\end{proof}

%For each fixed $s=s_0$ we would like to determine how $T_{s_0}(x)=\lamp(J_{k,N,\hfrak(s_0)}1_{[-\frac{d}{2}, \frac{d}{2}]})$ changes as we vary $s$ in a a small neighborhood of $s_0$. 

\subsubsection{Formulation of the problem}

For $s=s_0+\epsilon>s_0$ with $\epsilon$ sufficiently small so that $(\hfrak(s_0+\epsilon),d)\in D^{reg}_{(K,N,D)}$ the solutions $f_{s}(x)$ to the IVPs (with the same $\lambda$)
\eql{\label{eqn:p_lap_s_func} (p-1)f_{s}^{\prime(p-2)}f''_{s}(x)-T_{s}(x)f_{s}^{\prime(p-1)}+\lambda f_{s}^{(p-1)}=0\qquad f_{s}'(-\frac{d}{2})=0,\, f_{s}(-\frac{d}{2})=-1}
are such that $r_{s}:=\inf\{r\in \R:\, f'_{s}(\frac{d}{2}+r)=0\}<\infty$, and $r_s\to 0$ as $\epsilon \to 0$. 

Under the technical assumption stated in \blue{Subsection} \ref{Properties:pLap_egns}, this can be justified almost exactly as for the 2-Laplacian\blue{,  following} the arguments in the proof of Theorem \ref{cont:SL_Eval}. Indeed the theorem relies on continuous differentiability of the eigenvalues in the weight, explicit expressions for the derivatives, and on continuous dependence of the eigenfunctions on the weight. These, as we stated in Subsection \ref{Properties:pLap_egns}, \blue{were} verified in  \cite{MYZ} for a similar problem\blue{, yet, the details need to be verified for our specific BVP}.

\bigskip

We will show that the p-Laplacian eigenvalues manifest the same dependence on the diameter as the 2-Laplacian; then the same strategy of studying the dependence of the 2-Laplacian eigenvalues on $\hfrak$ can be implemented here, i.e. the inequality\\ 
$\lamp(J_{K,N,\hfrak(s_0)}1_{[-\frac{d}{2},\frac{d}{2}]})\leq \lamp(J_{K,N,\hfrak(s)}1_{[-\frac{d}{2},\frac{d}{2}]})$ (resp. $''\geq''$) holds whenever $r_{s}\geq 0$ (resp. $r_{s}\leq 0$). This will hold for any $s>s_0$ sufficiently close to $s_0$, hence by differentiability of the eigenvalues w.r.t. to the weight we can \blue{determine} what is the sign of $D_{\hfrak}\prnt{\lamp(J_{K,N,\hfrak}1_{[-\frac{d}{2},\frac{d}{2}]})}$. 

\bigskip

%Using the function $\phi_{\alpha, s}(x)=\phi_{\alpha}(x+s)$ (this amounts to the coordinate transformation $x\to x-s$) the polar equation \eqref{eqn:PolarEqn} takes the form:

We rephrase the problem in terms of the polar functions: let $\phi_{\alpha,s_0}$ and $\phi_{\alpha,s}$ ($s$ in a small neighborhood of $s_0$) be the polar functions associated with the functions $f_{s_0}(x-s_0)$  and  $f_s(x-s)$ (where $f_{s_0}(x)$ and $f_s(x)$ are solutions to \eqref{eqn:p_lap_s_func} with $T_{s_0}$ and $T_s$ respectively), and $\alpha=\prnt{\frac{\lambda}{p-1}}^{\frac{1}{p}}$, i.e

\eql{\label{polar_s_0} \phi_{\alpha, s_0}'(x)&=\alpha -T(x)\Theta(\phi_{\alpha, s_0}(x))\qquad \phi_{\alpha, s_0}(s_0-\frac{d}{2})\stackrel{Mod\,\,\pi_p}{=}\phi_{\alpha, s_0}(s_0+\frac{d}{2})\stackrel{Mod\,\,\pi_p}{=}\pip\,,  }

%\eql{\label{polar_s} \phi_{\alpha, s}'(x)&=\alpha -T_{s-s_0}(x)\Theta(\phi_{\alpha, s}(x))\qquad \phi_{\alpha, s}(s_0-\frac{d}{2})\stackrel{Mod\,\,\pi_p}{=}\phi_{\alpha, s}(s_0+\frac{d}{2}+r_s)\stackrel{Mod\,\,\pi_p}{=}\pip  } 
\eql{\label{polar_s} \phi_{\alpha, s}'(x)&=\alpha -T(x)\Theta(\phi_{\alpha, s}(x))\qquad \phi_{\alpha, s}(s-\frac{d}{2})\stackrel{Mod\,\,\pi_p}{=}\phi_{\alpha, s}(s+\frac{d}{2}+r_s)\stackrel{Mod\,\,\pi_p}{=}\pip \,.} 
\red{Considering the monotonicity Lemma \ref{lem:p_lap_phi_mono}, we may assume w.l.o.g. that $\phi_{\alpha, s_0}(s_0-\frac{d}{2})=-\pip$ and $\phi_{\alpha, s_0}(s_0+\frac{d}{2})=\pip$.}
From now on we carry out the analysis via a study of the polar functions. 
%\eql{\label{eqn:PolarEqn} \dot{\phi}_s(x)&=\alpha -T_s(x)\Theta(\phi_s(x))\qquad \phi_s(-\frac{d}{2})\stackrel{Mod\,\,\pi_p}{=}\phi_s(\frac{d}{2})\stackrel{Mod\,\,\pi_p}{=}\pip  }
%and let $\phi_{s+\epsilon}$ be the polar function associated with $f_{\epsilon}$. 

\red{Due to the technical assumption of Subsection \ref{Properties:pLap_egns}, we may approach the problem as follows:}

\red{{\bf  Rephrasing the problem:}} Determine the sign of $r_{s}$ \red{for $s>s_0>0$ in a small neighborhood of $s_0$.}

%assume $\phi(-\frac{d}{2})=-\pip$ and $\phi(\frac{d}{2})=\pip$ 
%In terms of the Pr\"{u}ffer transformation polar function $\phi$ this translates to the following :
%let and by $\phi_s$ the solution to the BVP:
 %
%We consider the following problem for $s$ sufficiently small 
%\eq{ \dot{\phi}_x)&=\alpha -T_{s'}(x)\Theta(\phi_s(x))\qquad \phi_s(-\frac{d}{2})\stackrel{Mod\,\,\pi_p}{=}\phi_s(\frac{d}{2}+\upsilon_s)\stackrel{Mod\,\,\pi_p}{=}\pip  }   

%, therefore $\phi(a)=-\pip \text{ mod } \pi_p $.

%We 
%It is a straightforward consequence of the above lemma that the function $\phi(x)$ admits an inverse $\phi^{-1}:(-\pip, \pip)\to (a,b)$. 

\subsection{Proving Theorem \ref{thm:SpectralMonotonicity_p}}

%\begin{thm} \label{thm:SpectralMonotonicity:p}
%\begin{enumerate}
   %\item For any fixed $h_0\in \pi_{h}\D^{reg}_{(K,N,D)}$, the function $d\mapsto \lamp(\xi(h_0, d))$ on $\pi_{h}^{-1}(h_0)$ strictly decreases as $d$ increases.
    %\item For any fixed $d_0\in \pi_{d}\D^{reg}_{(K,N,D)}$ the function $h\mapsto \lamp(\xi_{(K,N)}(h, d_0))$ on  $\pi_{d}^{-1}(d_0)$ 
        %monotonically increases as $|h|$ increases if $N\in  (1,\infty]$
%\end{enumerate}
%\end{thm}
%
%\bigskip
\begin{proof}[Proof of Theorem \ref{thm:SpectralMonotonicity_p}]
\begin{enumerate}
	\item The statement is a straightforward consequence of Lemma  \eqref{D_monotonicity} with function $g(x,r):=(x+r)^p$; the choice of this function and the verification of the \red{l}emma conditions, are detailed in Example \ref{exmp:p_Poinc} and Remark \ref{remk:D_monotonicity}. 
	
	%; indeed, since $f'(x)$ is invariant under addition of constants $f\mapsto f+r$, and if $r$ minimizes $\inf_{r\in \R}\int |f(x)+r|^p d\xi(x)$ then $\int |f(x)+r|^{p-1}sgn(f(x)+r) d\xi(x)=\int |f(x)+r|^{p-2}(f(x)+r) d\xi(x)=0$ (to justify that the minimizer $r$ is global one can use convexity of $r\mapsto \prnt{\int |f(x)+r|^p d\xi(x)}^{\frac{1}{p}}$). We can thus make the following identification
%\eq{		\inf \left\{ \frac{\int_{\R} |f'(x)|^p\xi(x)}{\int_{\R} |f(x)|^pd\xi(x)}:\,\,0\neq f\in C_c^{\infty}(\R)\, \mbox{ and } \, \int_{\R} h_f(x)d\xi(x)=0 \right\}   =\inf_{f\in C_c^{\infty}(\R)} \left\{ \frac{\int_{\R} |f'(x)|^p\xi(t)}{\inf_{r\in\R}\int_{\R} |f(x)+r|^pd\xi(x)} \right\}   }
%which proves the claim. 

\item We consider the polar \red{equations \eqref{polar_s_0} and \eqref{polar_s}}.
%Using the function $\phi_{\alpha, s}(x)=\phi_{\alpha}(x+s)$ (this amounts to the coordinate transformation $x\to x-s$) the polar equation \eqref{eqn:PolarEqn} takes the form:
%\eq{ \phi_{\alpha, s}'(x)&=\alpha -T(x)\Theta(\phi_{\alpha, s}(x))\qquad \phi_{\alpha, s}(s-\frac{d}{2})\stackrel{Mod\,\,\pi_p}{=}\phi_{\alpha, s}(s+\frac{d}{2})\stackrel{Mod\,\,\pi_p}{=}\pip  } 
Set $\varphi_{\alpha, s}(y):=\phi_{\alpha, s}^{-1}(y)$.
%(=s+\phi_{\alpha}^{-1}(y))
 Define
\[ F(y,w):=\prnt{\alpha-T(w)\Theta(y)}^{-1}\,.\]
\red{Then $\varphi_{\alpha, s_0}$ and $\varphi_{\alpha, s}$ satisfy  }
\eql{\label{eqn:varphiODE} &\varphi_{\alpha, s_0}'(y)=\frac{1}{\phi_{\alpha, s_0}'(\phi_{\alpha, s_0}^{-1}(y))}=F(y,\varphi_{\alpha, s_0}(y))  &\varphi_{\alpha, s_0}(\pm\pip)=s_0\pm\frac{d}{2}\,,\\ \nonumber
&\varphi_{\alpha, s}'(y)=\frac{1}{\phi_{\alpha, s}'(\phi_{\alpha, s}^{-1}(y))}=F(y,\varphi_{\alpha, s}(y))  &\varphi_{\alpha, s}(-\pip)=s-\frac{d}{2}
\,.  }

\red{We remark that since $\tl{F}(y,w,s'):=F(y,w+s')$ is continuously differentiable in $w$ and $s'$ in some neighborhood of $  [-\frac{d}{2}, \frac{d}{2}]\times \{s_0\}$, considering the ODE satisfied by $\varphi_{\alpha, s}-s$, from standard ODE results regarding differentiability of solutions in the parameter (e.g \cite[p.89-93]{Sid}), we can conclude that $ \varphi_{\alpha, s}$ is continuously differentiable in $s$. }

Our problem can be phrased as finding the sign of  $D_s(r_s+d)=D_s[\varphi_{\alpha, s}(y)]-D_s[\varphi_{\alpha, s}(-y)]$ at $y=\pip$ and $s=s_0$. 

We set $\bar{\varphi}_{\alpha, s}:=-\varphi_{\alpha, s}(-y)$, then similarly
\eql{\label{eqn:barvarphiODE} \bar{\varphi}_{\alpha, {s_0}}'(y)=F(y,\bar{\varphi}_{\alpha, s_0}(y)) \qquad \bar{\varphi}_{\alpha, s_0}(-\frac{\pi}{2})=-s_0- \frac{d}{2} \,.}
\red{
$\varphi_{\alpha, s_0}$ and $\bar{\varphi}_{\alpha, s_0}$ satisfy the same ODE on $[-\pip, \pip]$ (with different initial conditions); since $\varphi_{\alpha, s_0}(-\pip)=s_0-\frac{d}{2}>-s_0-\frac{d}{2}=\bar{\varphi}_{\alpha, s}(-\pip)$,  by uniqueness theorem for 1st order ODEs it follows that
$\varphi_{\alpha, s_0}(0)\geq \bar{\varphi}_{\alpha, s_0}(0)=-\varphi_{\alpha, s_0}(0)$, in particular $\varphi_{\alpha, s_0}(0)\geq 0$. }

\red{
A similar argument shows that given $s> s_0$, due to the inequality $\varphi_{\alpha, s}(-\pip)> \varphi_{\alpha, s_0}(-\pip)$, it also follows that $\varphi_{\alpha, s}(y)\geq\varphi_{\alpha, s_0}(y)$ on $[-\pip,0]$;
in particular one can conclude that $\varphi_{\alpha, s}(0)\geq \varphi_{\alpha, s_0}(0)$, whence $D_s[\varphi_{\alpha, s}(0)]\geq 0$.
}

By \eqref{eqn:varphiODE} and \eqref{eqn:barvarphiODE},  taking derivative with respect to $s$ of \eqref{eqn:varphiODE} gives the equations
{\footnotesize
\eq{
(D_s[\varphi_{\alpha, s}](y))'_{s=s_0}&=G_1(y)\cdot D_s[\varphi_{\alpha, s}(y)]_{s=s_0} \, &\text{where} \quad &G_1(y)=F(y,\varphi_{\alpha, s_0}(y))^2T'(\varphi_{\alpha, s_0}(y))\Theta(y)\\
(D_s[\varphi_{\alpha, s}(-y)])'_{s=s_0}&=G_2(y)\cdot D_s[\varphi_{\alpha, s}(-y)]_{s=s_0}\, &\text{where} \quad &G_2(y)=F(y,\bar{\varphi}_{\alpha, s_0}(y))^2T'(\bar{\varphi}_{\alpha, s_0}(y))\Theta(y)
\,.}
}
From these equations we get the following explicit expressions for $D_s[\varphi_{\alpha, s}(y)]_{s=s_0}$ and $D_s[\varphi_{\alpha, s}(-y)]_{s=s_0}$:
\eq{
D_s[\varphi_{\alpha, s}(y)]_{s=s_0}&=C_1 exp(\int_0^y G_1(y)dy ) \\
D_s[\varphi_{\alpha, s}(-y)]_{s=s_0}&=C_2 exp(\int_0^yG_2(y)dy)=C_1 exp(\int_0^yG_2(y)dy )\,,
}
where the second equality follows due to the matching at $y=0$, since by definition $C_1=C_2=D_s[\varphi_{\alpha, s}(0)]_{s=s_0}$, which is a non-negative number as we previously showed. 

Thus the sign of  $D_s[\varphi_{\alpha, s}(y)]_{s=s_0}-D_s[\varphi_{\alpha, s}(-y)]_{s=s_0}$ at $y=\pip$ is determined by 
\eql{\label{eqn:integralG} &\int_0^{\pip} \prnt{G_1(y)-G_2(y)}dy\\ \nonumber&=\int_0^{\pip}\cprnt{F(y,\varphi_{\alpha, s_0}(y))^2T'(\varphi_{\alpha,s_0}(y))-F(y, \bar{\varphi}_{\alpha, s_0}(y))^2T'(\bar{\varphi}_{\alpha, s_0}(y))}\Theta(y) dy\,. }

It turns out that for $N\in (1,\infty)$ the function $G_1(y)-G_2(y)$ is either non-negative or either non-positive for all $y\in [0,\pip]$, what makes  determination of the sign of \eqref{eqn:integralG} straightforward.  This is a consequence of the following proposition, whose proof concludes the proof of Theorem  \ref{thm:SpectralMonotonicity_p}.

\begin{prop} Assume $N\in (1,\infty)$, then in cases $1$ and $3$: $G_1(y)\geq G_2(y)$ while in cases $2$ and $4$: $G_1(y)\leq G_2(y)$, for all $y\in [0,\pip]$.
\end{prop}
\begin{proof}

In table \ref{table:1} we provide explicit expressions for $F(x,w)$ and $T'(x)$ corresponding to the four cases of \eqref{eqn:Tx}.
\begin{changemargin}{0.0cm}{0cm} 
\begin{table}[H]
\centering
\begin{tabular}{ ||p{0.7cm}||p{2.5cm}||p{6.7cm}||p{2.3cm}||  }
% \hline
 %\multicolumn{4}{|c|}{} \\
 \hline
 Case & \text{Density} & $F(y,w)$ & $T'(w)$\\
 \hline
1 & $\cos^{N-1}(\sqrt{\delta}w)$  & $\Prnt{\alpha-(N-1)\sqrt{\delta}\tan(\sqrt{\delta}w)\Theta(y)}^{-1}$ &  $\frac{\delta(N-1)}{\cos^2(\sqrt{\delta}w)}$\\
 \hline
2 & $w^{N-1}$  & $\Prnt{\alpha+(N-1)\frac{1}{w}\Theta(y)}^{-1}$ &  $\frac{N-1}{w^2}$\\
 \hline
3 & $\cosh^{N-1}(\sqrt{|\delta|}w)$ & $\Prnt{\alpha+(N-1)\sqrt{|\delta|}\tanh(\sqrt{|\delta|}w)\Theta(y)}^{-1}$ & $-\frac{|\delta|(N-1)}{\cosh^2(\sqrt{|\delta|}w)}$\\
\hline
4 & $\sinh^{N-1}(\sqrt{|\delta|}w)$ & $\Prnt{\alpha+(N-1)\sqrt{|\delta|}\coth(\sqrt{|\delta|}w)\Theta(y)}^{-1}$ &  $\frac{|\delta|(N-1)}{\sinh^2(\sqrt{|\delta|}w)}$\\
\hline
\end{tabular}
\caption{The 4 cases of \eqref{eqn:Tx} which we need to consider. }\label{table:1}
\end{table}
\end{changemargin}
By the foregoing the following facts can be stated about the functions $\varphi_{\alpha, s_0}(y)$ and $\bar{\varphi}_{\alpha, s_0}(y)$:
%\begin{enumerate}[label=\alph*]
\begin{enumerate}[(a)]
	\item $\varphi_{\alpha, s_0}(y)$ and $\bar{\varphi}_{\alpha, s_0}(y)$ are increasing on $[0,\pip]$. 
	\item \red{$\varphi_{\alpha, s_0}(0)\geq \bar{\varphi}_{\alpha, s_0}(0)=-\varphi_{\alpha, s_0}(0)$, and $ \varphi_{\alpha, s_0}(0)\geq 0$. }
	%$\varphi_{\alpha, s}(y)\geq \bar{\varphi}_{\alpha, s}(y)$ for all $y\in [0,\pip]$, and in particular at $y=0$: $\varphi_{\alpha, s}(0)\geq 0\geq \bar{\varphi}_{\alpha, s}(0)=-\varphi_{\alpha, s}(0)$. 
	\item From (a) and (b) it follows that $|\varphi_{\alpha, s_0}(y)|\geq |\bar{\varphi}_{\alpha, s_0}(y)|$ \red{on $[0,\pip]$.}
	\item In cases 2 and 4, since $\varphi_{\alpha, s_0}(-y)> 0$  for $y\in [0,\pip]$, it follows that $\bar{\varphi}_{\alpha, s_0}(y)< 0$ for $y\in [0,\pip]$. 
\end{enumerate} 
In view of these facts the proposition now follows from the following observations: 

{\bf In case 1:} $T'(\varphi_{\alpha, s_0}(y))\geq T'(\bar{\varphi}_{\alpha, s_0}(y))>0$ and $F(y, \varphi_{\alpha, s_0}(y))\geq F(y, \bar{\varphi}_{\alpha, s_0}(y))$.\\
{\bf In case 3:} $T'(\bar{\varphi}_{\alpha, s_0}(y))\leq T'(\varphi_{\alpha, s_0}(y)) <0$ and $F(y, \bar{\varphi}_{\alpha, s_0}(y))\geq F(y, \varphi_{\alpha, s_0}(y))$.\\
{\bf In cases 2,4:} $T'(\bar{\varphi}_{\alpha, s_0}(y))\geq T'(\varphi_{\alpha, s_0}(y)) >0$ and $F(y, \bar{\varphi}_{\alpha, s_0}(y))\geq F(y, \varphi_{\alpha, s_0}(y))$ (due to observation $(d)$).
Thus (considering that $\Theta(y)\geq 0$ on $[0,\pip]$):
\eq{G_1(y)&=F(y,\varphi_{\alpha, s_0}(y))^2T'(\varphi_{\alpha, s_0}(y))\Theta(y)\,,\\
G_2(y)&=F(y,\bar{\varphi}_{\alpha, s_0}(y))^2T'(\bar{\varphi}_{\alpha, s_0}(y))\Theta(y)\,,}
satisfy the stated inequalities.
%{\bf In case 4} $T'(\bar{\varphi}_{\alpha, s}(y))\geq T'(\varphi_{\alpha, s}(y)) >0$ and $F(y, \bar{\varphi}_{\alpha, s}(y))\geq F(y, \varphi_{\alpha, s}(y))$ (due to observation $(d)$).\\
%for simplicity we set  $\omega:=\varphi_{\alpha, s}(y)$ and $\bar{\varphi}_{\alpha, s}(y)$. Notice that
%
\end{proof}
From the proposition it follows that \red{$D_s(r_s)_{s=s_0}\geq 0$} in cases 1 and 3, while \red{$D_s(r_s)_{s=s_0}\leq 0$} in cases 2 and 4. 
By diameter comparison we conclude that in cases 1 and 3 for all $s>s_0$ (sufficiently close to $s_0$)
$\lamp(\hfrak(s_0),d)\leq \lamp(\hfrak(s),d)$, and in cases 2 and 4 for all $s>s_0$ (sufficiently close to $s_0$)
$\lamp(\hfrak(s_0),d)\geq \lamp(\hfrak(s),d)$. We translate from the parameter $s$ to $\hfrak$ according to \eqref{hfrak_s}. We conclude that for $N\in (1,\infty)$:
$$\inf_{(\hfrak,d)\in D^{reg}_{(K,N,D)}}\lamp(\hfrak,d)=\lamp(0,d)\,,$$
and moreover for any fixed $d_0\in (0,D)$ (s.t. $d_0<l_{\delta}$) on $\pi_{d}^{-1}(d_0)$ it holds that $\lamp(\hfrak,d_0)$ depends monotonically on $|\hfrak|$. 
\bigskip

Lastly, \red{subject to the technical assumption of Subsection \ref{Properties:pLap_egns},} the justification for the case $N=\infty$ is due to Theorem \ref{prop:case:infty}; indeed it was proved for the 2-Laplacian, considering that its eigenvalues are continuous in the weight (more precisely, in the SL-BVP metric), but since the p-Laplacian eigenvalues for $p\in(1,\infty)$ are also continuous in the weight, the proof extends to the p-Laplacian eigenvalue, showing that also for $N=\infty$, for fixed $d_0\in (0,D)$, it holds that $\lamp(\hfrak,d_0)$ depends monotonically on $|\hfrak|$. 

\end{enumerate}
\end{proof}
Therefore as long as we consider the parameter space $\D^{reg}_{(K,N,D)}$ and $D<l_{\delta}$ 
then the infimum of $\lamp(J_{K,N,\hfrak}1_{[-\frac{d}{2},\frac{d}{2}]})$ (which is the solution to the minimization problem of $\xi\mapsto \Lambda_{Poi}^{(p)}(\xi)$ over $\Fkndreg$) corresponds to the first non-zero eigenvalue of the equation 
\pinka{
$$
(p-1)f^{\prime(p-2)}f''+\frac{J'(x)}{J(x)}f^{\prime(p-1)}=-\lambda f^{(p-1)}\qquad f'(-\frac{D}{2})=f'(\frac{D}{2})=0\,,$$ 
}
where $J(x)=J_{K,N,0}(x)$ (as defined in Definition \ref{defn:Jknh}). 
This gives the solution to the minimization problem of $\Lambda_{Poi}^{(p)}(\xi)$ over $\Fkndreg$. As for the 2-\Poinc inequality we justify that this is also the form of the solution to the minimization problem over $\Fknd^M(\R)$, by showing upper semi-continuity of $\xi\mapsto \Lambda_{Poi}^{(p)}(\xi)$; as we mentioned, the technical details are similar to showing \red{u.s.c.}  for $\xi\mapsto \Lambda_{Poi}(\xi)$ and are therefore omitted. 

%In order to justify that $\lamp_{K,N,D}$ equals this eigenvalue, in particular when $D\geq l_{\delta}$, we need to consider the augmented space of measures, which includes the boundary points, in order to cover all of $\Fknd^M(\R)$.  For the 2-\Poinc constant we have provided a detailed analysis, as we mentioned we do not repeat the same (laborious) analysis for the $p$-\Poinc constant.

In view of Theorem \ref{LpPoinc:Estimate_gen} we can conclude the following theorem:

\begin{thm}
Let $(M,\gfrak,\mu)$ be a \cwrm{} which satisfies $CDD_b(K,N,D)$, where $K\in \R$, $N\in [\max(n,2),\infty]$ and $D\in (0,\infty]$. Then under the technical assumption detailed in Subsection \ref{subsec:p_SL_Theory}: 
\[ \Lambda_{Poi}(M,\gfrak)\geq \lamp_{K,N,D}\,,\]
 where 
\begin{itemize}
	\item If $N\in [n,\infty)$ then $\lamp_{K,N,D}=\Lambda^{(p)}_{Poi}( \co_{\delta}^{N-1}(x) 1_{[-\frac{D_{\delta}}{2},\frac{D_{\delta}}{2}]}(x)\cdot m)$,\qquad $D_{\delta}:=\min(D,l_{\delta})$\,;
	\item If $N=\infty$ then $\lamp_{K,\infty,D}=\Lambda^{(p)}_{Poi}( e^{-\frac{Kx^2}{2}} 1_{[-\frac{D}{2},\frac{D}{2}]}\cdot m)$\,.
\end{itemize}
Moreover, these estimates are sharp. 
\end{thm}

%The expression for $G_2(y)$ was written in terms of $\bar{\varphi}$ on purpose, motivate by the fact that $\varphi$ and $\bar{\varphi}$ satisfy the same ODE; indeed: 
%\eql{\label{eqn:BarVarphiODE} \bar{\varphi}'(y)=F_s(y,\bar{\varphi}(y)) }
%By uniqeness theorem for 1st order ODEs if some $y_0$ there is an intersection between $\varphi$ and $\bar{\varphi}$ they must coincide on $(y_0,\pip]$. Since $\bar{\varphi}(0)=-\varphi(0)$, and it satisfies the same ODE (in particular $\bar{\varphi}'(0)=\frac{1}{\alpha}\geq 0$ on $[0,\pip]$) $\varphi$ and $\bar{\varphi}$ verify the following statements: 
%\begin{enumerate}
	%\item $\bar{\varphi}_s(0)=-\varphi_s(0)\leq 0$, since $\varphi_s(0)\geq 0$
	%
	%, moreover $\varphi_s(y)\geq \bar{\varphi}_s(y)$ for all $y\in [0,\pip]$. Indeed, since if at some $y_0\in (0,\pip]$ it holds that $\varphi_s(y_0)=\bar{\varphi}_s(y_0)$, then
	%
	%
	%
	%\item $ \varphi_s'(y), \bar{\varphi}_s'(y)> 0$ on $[0,\pip]$.
	%\item If $\varphi_s$ is non-negative on $[0,\pip]$ then $\bar{\varphi}_s$ is non-positive on $[0,\pip]$ (as in cases 2 and 4).
	%
	 %
%\end{enumerate}

%\begin{remk} The above approach does not work for $N\leq 0$, however one should notice that the condition $D_s\varphi_s(y)\geq D_s\varphi_s(-y)$ is only sufficient and definitely not a necessary for the monotonicity implication. One can require for much less; what is actually necessary for the comparison is to compare $\int_0^{\pip} G_1(y)$ and $\int_0^{\pip} G_2(y)$
%\end{remk}

\begin{remk} Sharpness is proved by explicit construction of \cwrm, or sequences of \cwrm, which the limits of their p-\Poinc constants assume these estimates (the same construction which justified sharpness of the \Poinc constant lower bound $\lam_{K,N,D}$). The reader is referred to \cite{NaVa} and \cite{Mil2} for a discussion about such constructions. 
\end{remk}

 \newpage
 \chapter{Functional Inequalities: Explicit Lower Bounds\\ \bigskip The Log-Sobolev Inequality}
\label{chp:LS}
%{cor:CC}
\section{Main goals}

For a \cwrm{} which satisfies $CDD_b(K,N,D)$, where $K\in \R$ and $N\in (-\infty,1)\cup [n,\infty]$ we obtained the estimate \eql{\label{eqn:LS_Bound_repeat}\Lambda_{LS}(M,\gfrak,\mu)\geq \rho_{K,N,D}\,,}
in Theorem \ref{thm:LogSobolevInequality}, 
where $\rho_{K,N,D}$ is defined by
\eql{
\rho_{K,N,D}:=\inf_{0\neq \xi\in \Fknd^{\Cinf}(\R)}\Lambda_{LS}(\xi) \,,} 
with
\eql{\label{eq:LS_Constant_Exp2}
  \Lambda_{LS}(\xi)=\inf_{f\in \F_{LS}(\xi)}\left\{ \frac{2\int f'(t)^2d\xi(t)}{\int f(t)^2\log \prnt{f(t)^2} d\xi(t)}\right\} \,.}

This can also be expressed in the following form (see Subsection \ref{subsec:intro_func_ineq}):
\eq{ 
&\rho_{K,N,D}=\inf_{\bar{\xi}\in \Pknd^{\Cinf}(\R)}\Lambda_{LS}(\bar{\xi})\,,}
where
\eql{\label{eq:LS_Constant_Exp}
  \Lambda_{LS}(\bar{\xi})=\inf_{f\in\tl{\F}_{LS}(\bar{\xi})}\left\{ \frac{2\int f'(t)^2d\bar{\xi}(t)}{Ent_{\bar{\xi}}(f^2)}\right\}\,. }

The extreme points characterization Theorem \ref{thm:ExtremePoints} yielded Corollary \ref{cor:CC}, from which we conclude that for $K\in \R$, $N\in (-\infty,0]\cup [2,\infty]$, subject to the proviso  $D<l_{\delta}$ if $K<0$ and $N\leq 0$,
\eql{ \label{LS:Estimate} \rho_{K,N,D}= \inf_{\xi\in \Fknd^M(\R)}\Lambda_{LS}(\xi)\,.}
%Notice that 
%we  identify $\Lambda_{LS}(\xi)$ as $\Lambda_{LS}(\bar{\xi})$ where $\bar{\xi}:=\frac{1}{\xi(1)}\xi$. 
%
 %Throughout we will also consider finite measures $\xi$ which are not probabilities; for such measures $\xi$ 

%The statement \eqref{eqn:LS_Bound_repeat} is valid also in the case $n=1$ as long as $N\in (-\infty,0]\cup [2,\infty]$, since we may still apply Corollary \ref{cor:CC}. 

Hence we can summarize:

%When $n=1$ and there is no need for localization, and it straightforwardly follows from the extreme points characterization that when $K\in \R$ and $N\in (-\infty, 0]\cup [2,\infty]$ then \eqref{LS:Estimate} is valid. We can thus summarize

\begin{thm}\label{LS:Estimate_gen} If $(M^n,\gfrak, \mu)$ is a \cwrm{}, which satisfies $CDD_b(K,N,D)$, where $K\in \R$ and $N\in (-\infty,0]\cup [\max(n,2),\infty]$, then subject to the proviso  $D<l_{\delta}$ if $K<0$ and $N\leq 0$
\eql{ \label{LS:Estimate} \Lambda_{LS}(M,\gfrak,\mu)\geq \rho_{K,N,D}\,.}
\end{thm}

\bigskip

Our main goal in this chapter is to derive explicit  lower bounds for $\Lambda_{LS}(M,\gfrak,\mu)$ which are valid when $N\in [\max(n,2),\infty]$; in this work we do not attempt to derive estimates valid for $N\in (-\infty,0]$. 
Specifically we solve the optimization problem associated with $\rho_{K,N,D}$ in the case $K\in\R$ and $N=\infty$, expressing the (sharp) constant $\rho_{K,\infty,D}$ up to universal numeric constants.  We solve it only for $N=\infty$, since the lower bound which we will obtain will also be valid to $N\in [\max(n,2),\infty)$, as for $N$ in this range $CDD_b(K,N,D)\Rightarrow CDD_b(K,\infty,D)$. In addition, in view of previous estimates, in particular the Bakry-\'{E}mery estimate (Theorem \ref{thm:BakryEmeryLS}) with the factor $\frac{N}{N-1}$, we don't expect the estimates to be highly dependent on the effective dimension $N$ (and anyway we express the solution for $\rho_{K,\infty,D}$ up to numeric constants). 
\bigskip

Our approach will not rely on the Euler-Lagrange equation  associated with $\Lambda_{LS}(\xi)$ but on a criterion of Bobkov-G\"{o}tze \cite{BobG}. Yet we precede the proof of the main result by showing that minimizers realizing $\Lambda_{LS}(\xi)$ exist for $\xi\in \Fknd^M(\R)$, with the only exception of measures $\xi$ whose density is symmetric around the center of their support. This result makes it possible to approach the optimization problem via the Euler-Lagrange equations; however, due to its non-linear nature, we have taken a more direct approach which results in expressing $\Lambda_{LS}(\xi)$ up to numeric constants, and eventually expressing $\rho_{K,\infty,D}$ up to universal numeric constants.

%\subsection{The LS-constant for $CDD(K,\infty,D)$ w.r.m with $K=-k\leq 0$ and $D<\infty$}

\section{On the existence of Log-Sobolev minimizers for measures supported on $\R$}

In contrast to the \Poinc inequality, even for measures $\xi$ supported on $\R$ the question of  whether the numbers $\Lambda_{LS}(\xi)$ are attained infima, has no immediate answer at present. 
%
%about the existence of minimizers $f\in \Cinf(\R)$ (or even its completion) for the quotient
%\eql{\label{eq:LS_Constant_Exp} \inf_{\text{const}\,\neq f\in \Cinf(\R),\, \int f^2d\xi=1}\Phi(f,\mu) \qquad \text{where}\quad   }
%has no immediate answer. 
%(assuming Neumann boundary conditions in case of non-empty boundary)  
As an instructive example consider $M=S^1=[0,2\pi]_{0\sim 2\pi}$ (with the uniform measure). It is known that $\Lambda_{Poi}(S^1)= \Lambda_{LS}(S^1)=1$. Consider the following realization of $\Lambda_{LS}(S^1)$ by a minimizing sequence: let $u(x)$ be the Laplacian eigenfunction $\frac{1}{\sqrt{\pi}}\cos(x)$. Let $\epsilon>0$ and define $f_{\epsilon}(x):=1+\epsilon u(x)$, then the family $\{f_{\epsilon}(x)\}_{\epsilon>0}$ is a minimizing sequence realizing $\Lambda_{LS}(S^1)$, as one can verify by  expansion of the entropy as in \eqref{Implication_LS_Poinc}; however 
$\lim_{\epsilon\to 0}f_{\epsilon}=1$, thus one might guess that $\Lambda_{LS}(S^1)$ could not be realized by a non-constant function. 
The same argument shows that a similar problem is encountered when considering the interval with the uniform measure. Unfortunately this is indeed the case; in \cite{Led3} it was shown that on the sphere $S^n$ with $n\geq 2$, no non-constant function realizing the log-Sobolev constant exists. As it was noted to us by the author of that work, the argument can be extended to $S^1$, a case which is also essentially equivalent to the uniform density on the interval. 

 On the other hand on $(\R,\frac{1}{\sqrt{2\pi}}e^{-\frac{x^2}{2}}m)$ holds $LS(1)$ (and also $Poi(1)$), and $\Lambda_{LS}(\R, \frac{1}{\sqrt{2\pi}}e^{-\frac{x^2}{2}}m)$ is realized by functions of the form $f(x)=e^{ax+b}$ (e.g \cite{BGL} p.259). 

%For example on $S^1$ holds sharp $Poi(1)$ as well as sharp $LS(1)$. The log-Sobolev can be realized by sequences whose limit is the constant function
%
%However a minimizer realizing $\Lambda_{LS}(S^1)$ might not exist; for example,  function. 

Let $\bar{\xi}\in \P(\R)$ and $\bar{p}(x):=\deriv{\bar{\xi}}{m}$, then if a normalized minimizer $f$ (i.e. $\bar{\xi}(f^2)=1$) realizing $\Lambda_{LS}(\bar{\xi})$ exists, then it satisfies the Euler Lagrange equation \cite{Wan1}: 
\[ (\bar{p}(x)f'(x))'=-\frac{\Lambda_{LS}(\bar{\xi})}{2} \bar{p}(x)f(x) \log f^2(x)\,.\]
If $\bar{p}(x)>0$ on $supp(\xi)=[a,b]$, a compact interval, then it is also accompanied by Neumann/periodic boundary conditions;  and conversely, a non-constant normalized solution to this equation is a minimizer realizing \eqref{eq:LS_Constant_Exp} (\cite{Wan1} or \cite{BGL} p. 273).

The main result proved in this preliminary section can informally be phrased as follows: 

for all $\xi\in \Fknd^M$, but those whose density is symmetric w.r.t. the center of their support, the constant  $\Lambda_{LS}(\xi)$ is an attained infimum; furthermore, if  $\xi_1$ is the restriction to a smaller interval of $\xi_2$ then $\Lambda_{LS}(\xi_1)\geq \Lambda_{LS}(\xi_2)$.

Proving these statements is easy but to this end some preparation is necessary. Around the 80's O.Rothaus published several pioneering works on LSI \cite{Rot2, Rot1,Rot1a,Rot3,Rot4}, and in particular he proved an interesting `alternative' regarding the existence of minimizers, which is \pink{quite} common to Sobolev inequalities. His results started with analysis over intervals and were later extended to Riemannian manifolds $(M,\mu_{\gfrak})$, but can evidently be extended also to weighted Riemannian manifolds $(M,\gfrak,\mu=U\cdot \mu_{\gfrak})$ s.t. $U=\deriv{\mu}{\mu_{\gfrak}}=e^{-V}$ where $V\in \Cinf(M;\R)$ (e.g \cite{Rot5}). We briefly review his approach to LSI. 
Let $\F$ denote one of the function spaces\footnote{In his papers Rothaus assumed the function space $\Cinf_c(\Omega)$ of functions compactly supported inside a domain $\Omega$, however the proof of $a_{\beta}$ being an attained minimum does not rely on the compact support assumption, and applies equally well to the function space $\Cinf(\Omega)$. For the details of the proof the reader is referred to \cite{Rot1, Rot1a}. }  $\Cinf(M)$ or $\Cinf_c(M)$. Given $\beta\in\R$, we define 
\eq{ a_{\beta}:=\inf_{f\in\F} \prnt{\log \int f^2d\mu+\frac{1}{\int f^2d\mu}\int [2\beta |\nab_{\gfrak} f|^2 -f^2\log f^2 ]d\mu}\,,}
which can also be identified as $\inf_{f\in\F} \int [2\beta |\nab_{\gfrak} f|^2 -f^2\log f^2 ]d\mu$ subject to the proviso $\int f^2d\mu=1$. Rothaus showed that $a_{\beta}$ is an attained infimum. Due to the inequality $|\nab_{\gfrak}|f||\leq |\nab_{\gfrak} f|$ a.e. (e.g \cite[p.152]{LiLo}) a minimizer can be assumed to be non-negative. A normalized minimizer $f_{\beta}$  satisfies the non-linear PDE
\eq{ \beta\Delta_{\gfrak,\mu}f_{\beta}+f_{\beta}\log f_{\beta}^2+a_{\beta}f_{\beta} =0\,, } 
where
\[
\Delta_{\gfrak,\mu}=\Delta_{\gfrak}-\Avg{\nab_{\gfrak} V}{\nab_{\gfrak}}\,.
\]
Notice that by definition
\eq{ 2\beta\int|\nab_{\gfrak} f|^2d\mu\geq \int f^2\log f^2d\mu-\int f^2d\mu\log\int f^2d\mu -a_{\beta}\int f^2d\mu\,, }
what motivates naming $a_{\beta}$ the `defect' \cite{Rot5}. 
\begin{remk}\label{rmk:abeta} By definition of the LS constant it follows that  $a_{\beta}=0$ for $\beta\geq \frac{1}{\Lambda_{LS}}$ (since then $\inf_{f\in\F} \int [2\beta |\nab f|^2 -f^2\log f^2 ]d\mu$ is non-negative, hence it is minimized by constants),
 and $a_{\beta}<0$ for $\beta<\frac{1}{\Lambda_{LS}}$. Hence we can identify $\frac{1}{\Lambda_{LS}}$ as the minimal point where the function $\beta\to a_{\beta}$ becomes zero. 
\end{remk}
%\begin{enumerate}
	%\item It is finite.
	%\item It is monotonically non-decreasing in $\beta$
	%\item There is $\beta_0>0$ such that $a_{\beta}=0$ for $\beta\geq \beta_0$, and $a_{\beta}<0$ for $\beta<\beta_0$. 
	%\item $\Lambda_{LS}(M,g,\mu)=\frac{1}{\beta_0}$. 
%\end{enumerate}

  %Assume $M$ is a compact oriented manifold and $\bar{\mu}_{g}$ is the normalized Riemannian measure on $M$. We denote by $|| \cdot ||$ the norm $||f||=\int_M f^2 +\int_M |\nab_g f|^2$, and let  $H_0^1$ be the 
\bigskip
We continue the discussion about log-Sobolev minimizers, restricting our attention to a measure $\xi$ supported on a compact  interval $I$, which w.l.o.g. we assume it to be $[0,1]$. Define $\Hcal([0,1])$ to be the completion of $\Cinf([0,1])$ with respect to the norm $||f||:=\sqrt{\int \prnt{f^2+f^{'2}}d\xi}$. This space can be identified as a closed subspace of $L^2(\xi)\times L^2(\xi)$, hence it is a reflexive Banach space; this implies in particular that bounded sequences of functions admit weakly convergent sub-sequences. 

The inequality $\Lambda_{LS}(\xi)\leq\Lambda_{Poi}(\xi)$ always holds \cite[p.238]{BGL}, however in the examples we presented above we actually had the equality $\Lambda_{LS}(\xi)=\Lambda_{Poi}(\xi)$.  The following theorem of Rothaus indicates on a relation between such an equality and the non-existence of minimizers. 
%We use the term 'alternative' borrowed from the Fredholm alternative theorem, however as will be apparent from the proof, the Rothaus alternative claims that 'at least one' of the statements must hold, and not 'only one', for that we use the term 'weak alternative' when it seems right to stress the distinction from  the 'only-one' alternative.  
\begin{thm}[Rothaus '80 \cite{Rot1}]\label{thm:existenceMinimizers} Assume $\bar{\xi}$ is a probability measure on $[0,1]$ which satisfies the log-Sobolev inequality with constant $\Lambda_{LS}(\bar{\xi})>0$. 
At least one of the following possibilities holds:

\begin{enumerate}
	\item there is a non-constant function $f\in \Hcal([0,1])$ realizing $\Lambda_{LS}(\bar{\xi})$, or
	\item $\Lambda_{LS}(\bar{\xi})=\Lambda_{Poi}(\bar{\xi})$. 
\end{enumerate}
\end{thm}

\begin{remk} Rothaus' Theorem implies that whenever $\Lambda_{Poi}(\xi)-\Lambda_{LS}(\xi)>0$ the infimum in \eqref{eq:LS_Constant_Exp}  is attained.
However, at present it is not known in general if the two 
possibilities are mutually exclusive. 
Thus when $\Lambda_{LS}(\xi)=\Lambda_{Poi}(\xi)$ it seems like there is not much we can say (a non-constant minimizer might still exist). Nonetheless the proof will imply the following statement: at least one minimizing sequence converges to a constant. It is worth mentioning that in \cite{Rot5} Rothaus shows that for every compact homogeneous Riemannian manifold  there are a continuum of choices of $U=\deriv{\mu}{\mu_{\gfrak}}$ for which we get $\Lambda_{LS}(M,\gfrak,\mu)=\Lambda_{Poi}(M,\gfrak,\mu)$.

\end{remk}

\begin{proof}[Proof of Theorem \ref{thm:existenceMinimizers} (following \cite{Rot1}) ]
By definition of $\Lambda_{LS}(\xi)$, given a sequence of positive numbers $\epsilon_n\to 0$, there are functions $f_n\in\Cinf([0,1])$ such that 
\eql{\label{eq:LS_inequality_fn} \prnt{\frac{\Lambda_{LS}(\bar{\xi})}{2}+\epsilon_n} Ent_{\bar{\xi}}f_n^2>\int f_n^{'2}d\bar{\xi}\,.}

We can replace $f_n$ by any multiplicative copy of itself, hence we may assume w.l.o.g. that $f_n=c_n+\theta_n$, where $c_n\geq 0$, $\int\theta_nd\bar{\xi}=0$ and $\int \theta_n^{'2}d\bar{\xi}=1$. Since $\Lambda_{LS}(\bar{\xi}) \leq \Lambda_{Poi}(\bar{\xi})$
it holds that $\Lambda_{LS}(\bar{\xi})\int \theta_n^2d\bar{\xi}\leq \int\theta_n^{'2}d\bar{\xi}=1$. Therefore $\{\theta_n\}_{n\in \Nbb}$ is bounded in $\Hcal([0,1])$. We may thus assume $\theta_n$ converges weakly to $\theta\in \Hcal([0,1]$. By weak lower-semicontinuity of the norm $\int \theta^{'2}d\bar{\xi}\leq \lim\int\theta_n^{'2}d\bar{\xi}=1$, and $0=\lim\int\theta_nd\bar{\xi}=\int\theta d\bar{\xi}$. We write $\theta_n(x)=\theta_n(\half)+z_n$, where $z_n(x)=\theta_n(x)-\theta_n(\half)$. Since whenever $x\leq y$
\[ |z_n(y)-z_n(x)|=|\int_x^y\theta_n'(t)dt|\leq \prnt{\int_0^1\theta_n^{'2}(t)dt}^{\half}\prnt{\int_x^y dt}^{\half}=\sqrt{y-x}\,, \]
we conclude that $\{z_n\}_{n\in \Nbb}$ is uniformly bounded (take $y=\half$) and equicontinuous on $[0,1]$. By Arzela-Ascoli we may further assume $z_n\to z$ uniformly on $[0,1]$. In addition since $0=\lim_{n\to\infty}\int \theta_nd\bar{\xi}=\lim_{n\to\infty} \theta_n(\half)\int d\bar{\xi}+\lim_{n\to\infty}\int z_nd\bar{\xi}$, we conclude that the limit $\lim_{n\to\infty}\theta_n(\half)$ exists as well, whence $\theta_n$ converges uniformly on $[0,1]$ to $\theta$ (by uniqueness of the weak limit). Therefore $\theta_n\to \theta$ strongly in $L^2([0,1])$, and in particular $\int\theta_n^2d\bar{\xi}\to \int\theta^2d\bar{\xi}$. 

Evidently we have the following alternative: either
\begin{enumerate}
  \item $c_n\to c$ a finite limit, or
	\item $c_n\to\infty$.
\end{enumerate}
If $c_n\to c\in R_{+}$, then $f_n\to f=c+\theta$ strongly in $L^2([0,1])$ and uniformly on $[0,1]$. 
On the one hand
\[ \frac{\Lambda_{LS}(\bar{\xi})}{2}Ent_{\bar{\xi}}f^2 \leq \int f^{'2}d\bar{\xi}\,. \]
Moreover,
\[ \frac{\Lambda_{LS}(\bar{\xi})}{2}Ent_{\bar{\xi}}f^2\geq \lim_{n\to\infty}\int f_n^{'2}d\bar{\xi}=1\,. \]
Since $\lim_{n\to\infty}\int f_n^{'2}d\bar{\xi}\geq \int f^{'2}d\bar{\xi}$ we conclude that $\int f^{'2}d\bar{\xi}=1$, and in particular $f$ is not constant.
 
\bigskip

If $c_n\to\infty$, using homogeneity of the entropy we may substitute $g_n:=c_n^{-1}f_n=1+\frac{\theta_n}{c_n}$  into  \eqref{eq:LS_inequality_fn}. Consequently for each $n\in \Nbb$, we have (after multiplication by $c_n^2$)
\[ (\frac{\Lambda_{LS}(\bar{\xi})}{2}+\epsilon_n)c_n^2\prnt{\int g_n^2\log g_n^2d\bar{\xi}-\prnt{\int g_n^2d\bar{\xi}} \log\prnt{\int g_n^2d\bar{\xi}}} >c_n^2\int g_n^{'2}d\bar{\xi}\,.\]
In the limit the RHS is $\lim_{n\to\infty}\int \theta_n^{'2}d\bar{\xi}= 1\geq\int \theta^{'2}d\bar{\xi}$. By Taylor expansions one can verify that 
$c_n^2\int g_n^2\log g_n^2 d\bar{\xi} =3\int\theta_n^2d\bar{\xi} +O(\frac{1}{c_n})$ and 
$c_n^2\int g_n^2\log\int g_n^2c_n^2=\int \theta_n^2+O(\frac{1}{c_n})$, therefore in the limit the LHS is $\Lambda_{LS}(\bar{\xi})\int\theta^2d\bar{\xi}$.  It follows that 
\[ \Lambda_{LS}(\bar{\xi})\int \theta^2d\bar{\xi}\geq \lim_{n\to\infty}\int\theta_n^{'2}d\bar{\xi}=1\geq \int\theta^{'2}d\bar{\xi}\,.\]
However since $\Lambda_{LS}(\bar{\xi})\leq \Lambda_{Poi}(\bar{\xi})$ and $\Lambda_{Poi}(\bar{\xi})=\inf_{\substack{0\neq f\in \Hcal\\\int fd\xi=0}}\frac{\int f^{'2}d\bar{\xi}}{\int f^2d\bar{\xi}}$, we conclude that  $\Lambda_{LS}(\bar{\xi})=\Lambda_{Poi}(\bar{\xi})$ (since $\Lambda_{LS}(\bar{\xi})>0$ the inequality implies $\theta\neq 0$).  
%(however $\lim_{n\to\infty}\frac{\theta_n}{c_n}$ might very well be zero, which is indeed the situation for the function $f_{\epsilon}$ in the circle example).  
\end{proof}

\bigskip
Assume $\bar{\xi}=\bar{p}\cdot m$ is a probability measure on $\R$, which is supported on a compact interval $I=[a,b]$. Recall from Chapter \ref{chp:Poinc}  that we can identify the \Poinc constant $\Lambda_{Poi}(\bar{\xi})$ as the first non-zero eigenvalue of the regular Sturm-Liouville problem 
\[ L_{\bar{\xi}}u=\lambda u\qquad u'(a)=u'(b)=0 \,.\]
Here $L_{\bar{\xi}}=-\Delta_{\bar{\xi}}$, where $\Delta_{\bar{\xi}}$ is the weighted Laplacian defined by $\Delta_{\bar{\xi}}u= \frac{1}{\bar{p}}(\bar{p}u')'$, and the equation is satisfied in the weak sense. 

The next corollary was originally formulated in \cite{Rot3} for manifolds. It relies on the previous statement, which was proved for intervals, hence for the sake of completeness and simplicity, we formulate it for intervals as well. 
\begin{cor}[Rothaus '81 ]\label{RothausLemma} If some $L_{\bar{\xi}}$ eigenfunction $u$ associated to the first non-zero eigenvalue $\lam=\Lambda_{Poi}(\bar{\xi})$ satisfies $\int u(x)^3d\bar{\xi}(x)\neq 0$ then $\Lambda_{Poi}(\bar{\xi})> \Lambda_{LS}(\bar{\xi})$. In particular there is a function realizing $\Lambda_{LS}(\bar{\xi})$.
\end{cor}

\begin{proof} The conclusion will follow from   Remark \ref{rmk:abeta}. We defined $a_{\beta}$ as the infimum over  expressions 
\[ \log \int f^2d\bar{\xi}+\frac{1}{\int f^2d\bar{\xi}}\int [2\beta |\nab f|^2 -f^2\log f^2 ]d\bar{\xi} \,.\]

Take $f=1+\epsilon u$, where $\int u^2d\bar{\xi}=1$. The zeroth and 1st order terms in $\epsilon$ vanish. The 2nd order term is $(2\beta\lam-2)\epsilon^2$, and the 3rd order term is $-\frac{2}{3}\epsilon^3\int u^3d\bar{\xi}$. Hence if $\beta\lam<1$ then $a_{\beta}<0$ (i.e. $\beta$ is below the threshold $\Lambda_{LS}(\bar{\xi})$) while if $\beta\lam=1$ and $\int u^3 d\bar{\xi}\neq 0$ we still have $a_{\beta}<0$ (implying that $\frac{1}{\lam}<\frac{1}{\Lambda_{LS}(\bar{\xi})}$). 
%
%
%
%
%For $\beta>\beta_0$ the minimizer is uniquely the constant function. If $\beta_0>\frac{1}{\Lambda_{Poi}(M,g,\mu)}$ then by the Rothaus alternative $\Lambda_{LS}(M,g,\mu)=\frac{1}{\beta_0}$ is an attained minimum. Now the proof of \ref{RothausLemma} is simple.  
\end{proof}

The first result which was informally stated at the beginning of this section is a straightforward consequence of the following corollary. 
\begin{cor}\label{cor:LSMinimizer} Let $I=(a,b)\subset \R$ be a bounded interval s.t. $a+b>0$, and let $d\xi=pdm$ be a measure on $I$, where $p\in C^1(I)$ is a positive function which satisfies one of the following conditions: 
\begin{itemize}
	\item If $0\notin I$: then $p'(x)\neq 0$ for all $x\in I$.
	\item If $0\in I$: $p(x)=p(-x)$ on $(a,-a)$ and $p'(x)\neq 0$ on $(0,b)$. 
\end{itemize}
Then a minimizer realizing the LS constant $\Lambda_{LS}(\xi)$ exists.
\end{cor}
\begin{remk} The condition $a+b>0$ is critical, as this result is incorrect when $a=-b$, considering what we have mentioned in the foregoing regarding the uniform density on the interval. 
\end{remk}
\begin{proof} Let $u$ be an eigenfunction corresponding to the first non-zero eigenvalue of $\Delta_{\xi}$:
\eql{\label{SLeqn}  (p(x)u(x)')'=-\lam p(x)u(x) \qquad \qquad u'(a)=u'(b)=0\,. }
By Corollary \ref{RothausLemma} it is sufficient to prove that $\int u(x)^3d\xi\neq 0$. 
Being the first non-constant eigenfunction,  $u'(x)\neq 0$ inside $I$ (see for example the argument in \cite[Example 1.1]{Wan1}), hence we may assume w.l.o.g. $u'(x)>0$ inside $I$. 
Using \eqref{SLeqn} and integration by parts:
{\small 
\eq{\int_{I} u(x)^3 d\xi &= \int_{I} p(x) u(x) \cdot u(x)^2dx = -\frac{1}{\lambda}\int_{I} (p(x) u(x)')'u(x)^2dx
\\&=\frac{2}{\lambda}\int_{I} p(x)(u(x)')^2u(x)dx=-\frac{2}{\lambda^2}\int_{I} (p(x) u(x)')'(u(x)')^2dx=
\frac{2}{\lambda^2}\int_{I} p(x)(u(x)')^2u(x)''dx\\&=\frac{2}{\lambda^2}\int_{I} p(x) (\frac{(u(x)')^3}{3})'dx=-\frac{2}{3\lambda^2}\int_{I}\frac{p(x)'}{p(x)^{3/2}} \prnt{p(x)^{1/2}u(x)'}^3 dx:=-\frac{2}{3\lambda^2}\int_{I} A(x) v(x)^3 dx\,,    }}
where $v(x):=p(x)^{1/2}u(x)'$ and $A(x):=\frac{p'(x)}{p(x)^{3/2}}$.
In case 1, since $p(x)$ is strictly monotone on $I$ (in particular $A(x)\neq 0$) and $v(x)>0$ inside $I$, the last integral is non-zero. 
In case 2, we notice that the transformation $v(x):=p(x)^{\half}u'(x)$ turns the BVP \eqref{SLeqn} into the following Dirichlet BVP \cite{PaWe, Kro}:
\eq{ v''(x)+H(x)v(x)=0\qquad\qquad \text{with \,\, BC } \quad v(a)=v(b)=0 \,,}
where
\eq{H(x):=
\lam+\half \prnt{ p''(x)-\frac{1}{2}\prnt{\frac{p'(x)}{p(x)}}^2 }\,. }
Since $u'(x)\neq 0$ on $(a,b)$ the function $v(x)\neq 0$ as well on $(a,b)$. We may thus assume w.l.o.g. that $v(x)>0$ on $(a,b)$.
Since $p(x)=p(-x)$ on $(a,-a)$ there is no problem defining an odd extension of it on $(-b, b)$ which we still denote by $p$. Since $H(x)=H(-x)$ on $(-b, b)$ the function $\bar{v}(x):=v(-x)$ satisfies the same ODE with BC $\bar{v}(-a)=\bar{v}(-b)=0$. Clearly $\bar{v}(0)=v(0)$, and $b>-a$ by assumption, 
unless $v(x)$ and $\bar{v}(x)$ coincide they have no additional intersection point inside $(0,b)$; indeed, if $x_1\in (0,b)$ is a second intersection point, then the function $g_{v}(x):= v(x)-\bar{v}(x)$ satisfies $$g_{v}''(x)+H(x)g_{v}(x)=0\qquad g_{v}(0)=g_{v}(x_1)=0\,,$$
then by Sturm's separation  theorem (e.g \cite[p.314]{BiRo}), since $v(x)$ and $g_{v}(x)$ are not proportional on $[0,x_1]$ (considering that $g_{v}(0)=0$),  $v(x)$ must have a zero inside $(0,x_1)$ which is a contradiction. By assumption $b>-a$ therefore $v(x)\geq \bar{v}(x)$ on $[0, b]$. 
Since they do not coincide there is one point $x\in [0,b]$ where the inequality must be strict. Notice that $A(x)=-A(-x)$ on $(a, -a)$, and by  the assumption $p'(x)\neq 0$ on $(0,b)$ it follows that $A(x)$ is non-vanishing on $(0,b)$. Therefore the sign of the integral 
\eq{ \int_{a}^{b} A(x) v(x)^3 dx=\int_{a}^{-a}A(x) v(x)^3 dx+\int_{-a}^{b} A(x) v(x)^3 dx=\\
\int_0^{-a}A(x) \prnt{v(x)^3-\bar{v}(x)^3} dx+\int_{-a}^{b} A(x) v(x)^3 dx
}
coincides with the sign of $A$ on $(0,b)$, in particular it is non-zero.  
\end{proof}
Of course the same can be concluded if we assume instead that $a+b<0$ or that there is some point $s\in I$ such that $p(x-s)$ satisfies one of the two conditions mentioned. \smallskip

The reader is referred to Definition \ref{defn:Jknh} for the definition of the functions $J_{K,N,\hfrak}$ for $N\in (-\infty,\infty]\setminus \{1\}$. 
\begin{cor} Let $K\in \R$, $N\in (-\infty,\infty]\setminus \{1\}$, $d\in (0,\infty)$, and  $\hfrak\in \R$. Define $\xi_{\hfrak, d}:=J_{K,N,\hfrak}1_{[-\frac{d}{2},\frac{d}{2}]}m\in \Fknd^M$ (where $[-\frac{d}{2},\frac{d}{2}]\subset \isupp(J_{K,N,\hfrak})$).

Set $\delta=\delta(K,N)$  as in Definition \ref{dfn:deltaSymbols}.
If 
\begin{itemize}
    \item $\delta<0$, or
    \item $\delta\geq 0$ and $\hfrak\neq 0$,
\end{itemize}
then $\Lambda(\xi_{\hfrak, d})$ is attained by a minimizing function. 
\end{cor}
\begin{remk} At present we do not know if $\Lambda(\xi_{0,d})$ is attained; we suspect that like the case $K=0$ there is no minimizer.
\end{remk}

For the \Poinc constant we know that if $\xi_1$ is a restriction of $\xi_2$ then $\Lambda_{Poi}(\xi_1)\geq \Lambda_{Poi}(\xi_2)$. We show that $\Lambda_{LS}$ manifests  this property too. The proof is based on a variational identity due to Holley and Stroock \cite{HoSt} (see also \cite[p.240]{BGL}). 

\begin{lem}[Dependence on $D$]\label{lem:Dependence_d} Assume $J\in L^{\infty}(I_0)$ where $I_0\subset \R$ is a compact interval. If $\xi_0=J\cdot m\in \M_b$ and $\xi_1$ is a restriction of $\xi_0$ to $I_1\subset I_0$, then $\Lambda_{LS}(\xi_1)\geq \Lambda_{LS}(\xi_2)$.
\end{lem}
\begin{proof} The theorem is a straightforward consequence of Lemma \ref{D_monotonicity}, whose relevance was demonstrated in  Example \ref{exmp:LS}, and whose conditions were verified in Remark \ref{remk:D_monotonicity}. For its implementation we used the function  $g(y,r)=\phi(y^2)-\phi(r)-\phi'(r)(y^2-r)$ with $\phi(r)=r\log r$; the crucial properties of this function are that by convexity of $\phi$ it follows that $g(y,r)\geq 0$, and in addition it verifies the following  variational identity:
$$\inf_{r>0}\int g(f^2(x),r)d\bar{\xi}(x)=Ent_{\bar{\xi}}(f^2)\,.$$

\end{proof}

\section{Lower bounds for $\Lambda_{LS}(M,\gfrak,\mu)$ under $CDD_b(K,\infty,D)$}

\subsection{Preliminaries}
%{eqn:LS_Bound_repeat}

Let $(M,\gfrak,\mu)$ be a \cwrm{} which satisfies $CDD_b(K,N,D)$. 
According to Theorem \ref{LS:Estimate_gen} we have the lower bound $\Lambda_{LS}(M,\gfrak,\mu)\geq \rho_{K,\infty,D}$, where 
\eql{ \label{LS_rhoknd}
&\rho_{K,\infty,D}=\inf_{0\neq \xi\in \Fkinfd^M(\R)}\Lambda_{LS}(\xi) \,,\\\nonumber  &\text{and}\\& \Lambda_{LS}(\xi)=\inf_{f\in\tl{\F}_{LS}(\bar{\xi})}\left\{ \frac{2\int f'(t)^2d\bar{\xi}(t)}{Ent_{\bar{\xi}}(f^2)}\right\}=
  \inf_{f\in \F_{LS}(\xi)}\left\{ \frac{2\int f'(t)^2d\xi(t)}{\int f(t)^2\log \prnt{f(t)^2} d\xi(t)}\right\}
\,.
}

Yet determination of $\rho_{K,\infty,D}$ is incomplete; we need to refine the measures $\xi\in\Fkinfd^M(\R)$ for which $\Lambda_{LS}(\xi)$ is minimal. One can attempt to study $\Lambda_{LS}(\xi)$ via the associated Euler-Lagrange equation, however despite being an ODE on the interval, the equation is non-linear and there does not seem to be a straightforward method to do that. Rather attempting to give a precise calculation of $\rho_{K,\infty,D}$ we provide estimates up to numeric constants; equivalently we provide lower bounds on $\Lambda_{LS}(\xi)$ for $\xi\in \Fkinfd^M(\R)$, and identify the $\bar{\xi}_{min}\in\Pkinfd^M(\R)$, for which the  lower bound on $\Lambda_{LS}(\xi)$, which is, up to numeric constants, minimal.
%Indeed apriori we don't know that the infimum over $\xi\in \Pknd^M$ is attained, and which $\xi\in\Pknd^M$ corresponds to this minimum, but if for each $\xi\in\Pknd^M$ we have estimates for $\Z(\xi)$, we can determine which $\xi$ gives the best (minimal) asymptotic estimates. 

\subsubsection{A convenient setting}

Throughout we abbreviate and write $\Fkinfd^M$ (resp. probability measures $\Pkinfd^M$) for the measure sets $\Fkinfd^M(\R)$ (resp. $\Pkinfd^M(\R)$). 
 The measures $\xi\in \Fkinfd^M$ have densities which are restrictions of the function $J_{K,\infty,\hfrak}=\exp(\hfrak x-\frac{K}{2}x^2)$ where $\hfrak\in\R$. Whenever $K\neq 0$ we can write $J_{K,\infty,\hfrak}=e^{-\half K(x+s)^2}e^{-\half \frac{\hfrak^2}{K}}$ where $s(\hfrak):=-\frac{\hfrak}{K}$.  Therefore we can take a parameter space $(s, d)\in \R\times (0,D]$, and compare the values of $s\mapsto \Lambda_{LS}(\xi_K(s,d))$; here $\xi_K(s,d)$ is the measure valued map: $\xi_K(s,d)=e^{-\half Kx^2}1_{I_s}$ where $I_s:=[s-\frac{d}{2},s+\frac{d}{2}]$. When $K=0$ we have $J_{K,\infty,\hfrak}=e^{\hfrak x}$; if we change $x\to x+a$ the density changes merely by a multiplicative constant and therefore $\Lambda_{LS}$ is unchanged, so we may always assume the measures $\xi$ are supported on $[0,d]$. In addition $\Lambda_{LS}(e^{\hfrak x}, [0,d])=\hfrak^2\Lambda_{LS}(e^x,  [0,\hfrak d])$ (this follows from simple change of coordinates in \eqref{LS_rhoknd}, and will be explicitly proved in the next section). Therefore similar to the case $K\neq 0$, we can consider a parameter space $(s, d)\in (0,\infty)\times (0,D]$, and compare the values of $s\mapsto s^2\Lambda_{LS}(\xi_0(s,d))$; here $\xi_0(s,d)$ is the measure valued map: $\xi_0(s,d):=e^{x}1_{I_s}$ where $I_s:=[0,sd]$. 
Using the parameter $s$ we have a unified formulation of the problem for all $K\in\R$: we consider a fixed density $p(x)$ on $\R$, and optimize over the   following target functions:
\begin{myitemize}
	\item[If $K\neq 0$:\, ] Target function:  $\R\ni s\mapsto  \Lambda_{LS}(p(x)1_{I_s}\cdot m)$, with $p(x)=e^{-\half Kx^2}$ and $I_s=[s-\frac{d}{2}, s+\frac{d}{2}]$\,.
	\item[If $K=0$:\, ] Target function:  $(0,\infty)\ni s\mapsto  s^2\Lambda_{LS}(p(x)1_{I_s}\cdot m)$, with $p(x)=e^{x}$ and $I_s=[0, sd]$\,.
	%\item[If $K<0$:\, ] $p(x)=e^{-\half Kx^2}$ and $I_s=[s-\frac{D}{2}, s+\frac{D}{2}]$.
\end{myitemize}
Here we substituted $d=D$, the maximal diameter, due to Lemma \ref{lem:Dependence_d}. This motivates the following definitions.

%We will firstly present the notation conventions which will bring the problem into a convenient formulation. We denote by $d\xi_s(x)=p(x)1_{I_s}dm(x)$ a family of measures, where $\{I_s\}$ is a family of intervals parametrized by a number $s\geq 0$, and $p(x)$ is a fixed positive continuous function on $\R$. As we previously discussed, the exhaustion of all values of $\Lambda_{LS}(\bar{\xi})$ over $\Pkinfd^M$, is possible by considering such a family of measures (where $s$ is interpreted as translation if $K\neq 0$ and as dilation if $K=0$).  

We define the normalization factor $H_s:=\int d\xi_s=\int_{I_s}p(x)dx$. 
 In addition we define $F_s:I_s\to [0,1]$ to be the cumulative  distribution function on $I_s$ associated with the probability measure $d\bar{\xi}_s:=\bar{p}_s(x)1_{I_s}dm$ where $\bar{p}_s:=H_s^{-1}p(x)$, i.e. $F_s(t)=\bar{\xi}_s((-\infty, t])$. Given $t\in [0,1]$ we set $a_s(t)=F_s^{-1}(t)$ and $b_s(t)=F_s^{-1}(1-t)$ the points in $I_s$ such that
\eql{ \label{def:a_sb_s}\int^{a_s(t)}_{-\infty}d\xi_s=t\,H_s \qquad \text{and}\qquad \int_{b_s(t)}^{\infty}d\xi_s=t\,H_s \,.}
%---------------------------------
\subsubsection{Reformulation of the problem}
We use the Bobkov-G\"{o}tze Estimates \cite{BobG} to evaluate $\Lambda_{LS}(\xi_s)$ up to multiplicative constants. By these estimates we have the equivalence $$\Lambda_{LS}(\xi_s)^{-1}\eqsim \B_{-}^{(s)}+\B_{+}^{(s)}\,,$$ 
where
{\small 
\eq{
\gls{B_minus}^{(s)}&=\sup_{r<m}\prnt{F_s(r)\log\frac{1}{F_s(r)}}\int_r^m\frac{1}{\bar{p}_s(x)}dx=\sup_{t\in [0,\frac{1}{2})}t\log\prnt{\frac{1}{t}}\int^{F_s^{-1}\prnt{\frac{1}{2}}}_{F_s^{-1}\prnt{t}}\frac{1}{\bar{p}_s(x)}dx\,, \qquad \text{and}\\ \gls{B_plus}^{(s)}&=\sup_{r>m}\prnt{(1-F_s(r))\log\frac{1}{(1-F_s(r))}}\int_m^r\frac{1}{\bar{p}_s(x)}dx=\sup_{t\in (\frac{1}{2},1]}(1-t)\log\prnt{\frac{1}{1-t}}\int^{F_s^{-1}\prnt{t}}_{F_s^{-1}\prnt{\frac{1}{2}}}\frac{1}{\bar{p}_s(x)}dx\\
&=\sup_{t\in [0,\frac{1}{2})}t\log\prnt{\frac{1}{t}}\int^{F_s^{-1}\prnt{1-t}}_{F_s^{-1}\prnt{\frac{1}{2}}}\frac{1}{\bar{p}_s(x)}dx\,.
 }}
Rather than estimating these terms separately we obtain estimates for
\eq{ \Us:=\sup_{t\in [0,\frac{1}{2})}t\log\prnt{\frac{1}{t}} \prnt{\int_{F_s^{-1}\prnt{t}}^{F_s^{-1}\prnt{1-t}}\frac{1}{\bar{p}_s(x)}dx}\,.} 
Since
\eq{ \Us\leq \B_{-}^{(s)}+\B_{+}^{(s)} \leq 2\Us\,,}
estimations of $\Us$ are equivalent to estimations of $\B_{-}^{(s)}+\B_{+}^{(s)}$. 
We define 
\[ \fs:=\int_{a_s(t)}^{b_s(t)}\frac{1}{p(x)}dx\qquad \text{and} \qquad \Ust:=t\log\prnt{\frac{1}{t}} \fs H_s\,.\] 
By definition $\Us:=\sup_{t\in [0,\frac{1}{2})}\Ust$ and so our goal 
is to maximize $\Us$ over $(t,s)\in [0,\half)\times S$, where we take $S=(0,\infty)$ for $CDD_b(0,\infty,D)$ and $S=[0,\infty)$ for $CDD_b(K,\infty,D)$ when $K\neq 0$ (since $p(x)=p(-x)$ when $K\neq 0$ by assumption).  

\bigskip 

The results of this section can be informally phrased as the single inequality $\UU^{(0)}\gtrsim \Us$, where the $\sim$ notation stands for validness of the statement up to numeric multiplicative constants (uniformly for all $s$ in the permissible domain $S$). By definition of $\rho_{LS}(K,N,D)$ such a statement translates into 
$$\rho_{LS}(K,N,D)\gtrsim (\UU^{(0)})^{-1}\,.$$
\bigskip

\subsubsection{General estimates }

Our problem initially started by considering a single interval of a fixed diameter $D$, and a family of measures on that interval, from which we wanted to extract the minimal log-Sobolev constant. The following lemma, known as the Holley-Stroock bounded perturbation lemma \cite{HoSt, BGL}, shows that under controlled perturbations of the densities, we have controlled perturbation of the corresponding log-Sobolev constants.

\begin{lem}[Holley-Stroock \cite{HoSt}]\label{lem:HolStrook} If $\bar{\xi}_1=\bar{p}_1\cdot m$ and $d\bar{\xi}_2=\bar{p}_2\cdot m$ are probability measures on $\Omega$, and $\frac{1}{a}\leq \deriv{\xi_2}{\xi_1}\leq b$ $[m]$ a.e. on $\Omega$ for some constants $a,b>0$, then 
 $\Lambda_{LS}(\Omega; \xi_2)\geq \frac{1}{a\cdot b}\Lambda_{LS}(\Omega; \xi_1)$. In particular $\Lambda_{LS}(\Omega; \xi_2)\eqsim \Lambda_{LS}(\Omega; \xi_1)$.
\end{lem}
If on $\Omega$ we have a class of probability measures $\{\bar{\xi}_s\}$ which are related by uniformly bounded perturbations, i.e. between any two members of the class $\bar{\xi}_{s_1}$ and $\bar{\xi}_{s_2}$ it holds that $\frac{1}{c}\leq \deriv{\xi_{s_2}}{\xi_{s_1}}\leq c$\,\, $[m]$ a.e. on $\Omega$, then the log-Sobolev constants $\{\Lambda_{LS}(\Omega; \xi_s)\}$ are all equivalent, uniformly by the same numeric constant. However,  this is not the general case, and specifically this does not apply to the class of measures in our present problem.
Hence in order to derive more general estimates we will need to use the Bobkov-G\"{o}tze criterion; however for very specific steps in the proof we will apply the Holley-Stroock lemma.

We begin with some general observations regarding the solution of the Bobkov-G\"{o}tze optimization problem for estimation of the LS constant on intervals. 
The function $t\mapsto t\log\prnt{\frac{1}{t}}$ is non-negative on $[0,1]$, it vanishes at $t=0$ and is strictly-increasing on $(0,\frac{1}{e})$ and strictly-decreasing on $(\frac{1}{e}, \half]$, while $\fs$ is strictly-decreasing in $t$ and vanishes at $t=\frac{1}{2}$. The next lemma is a simple, yet important, consequence of this observation:
%
%As a consequence there is a critical point $t_*\in (0,\frac{1}{2})$ at which a maximum of $\Ust$ is attained. 
\begin{lem} If $t_*$ is an extremum of $t\mapsto \Ust$, then $t_*\in (0,e^{-1}]$. 
\end{lem}
\begin{proof}
For $t\in (e^{-1},\half]$ both $t\mapsto t\log\prnt{\frac{1}{t}}$ and $t\mapsto \fs$ are decreasing, therefore $t_*\notin (e^{-1},\half]$.
\end{proof}

At an extremum point $t_*$ the equation $\partial_t|_{t=t_*} \Ust=0$ is equivalent to:
\eql{\label{eqn:DerZeroOptimization1}
(\log\prnt{\frac{1}{t_*}}-1)\fss + t_*\log\prnt{\frac{1}{t_*}} \partial_t|_{t=t_*}\fs =0\,.
}
According to definition \eqref{def:a_sb_s}, $a_s(t)$ and $b_s(t)$ are differentiable in $t$ and satisfy the following ODEs:
\eql{ \label{eqn:a_sb_s_deriv}
a_s'(t)=p^{-1}(a_s(t))H_s\qquad \text{and} \qquad b_s'(t)=-p^{-1}(b_s(t))H_s \,.
}
Therefore a maximum $t_*\in (0,e^{-1})$ must satisfy the following condition:
\eql{\label{eqn:DerZeroOptimization2}
\frac{t_*\log\prnt{\frac{1}{t_*}}}{\log\prnt{\frac{1}{t_*}}-1}=\Ats \qquad \text{where}\qquad \At:=\frac{\fs}{\cprnt{p^{-2}(a_s(t))+p^{-2}(b_s(t))}H_s}\,.
}
From this equation it follows that $t_*\leq \Ats$. Furthermore, if $t_*$ is sufficiently small so that $\log\prnt{\frac{1}{t_*}} \geq \frac{e}{e-1}$, then $\frac{\log\prnt{\frac{1}{t_*}}}{\log\prnt{\frac{1}{t_*}}-1}\leq e$ whence $et_*\geq \Ats$. We can thus conclude that in particular 
\[ \prnt{ \log\prnt{\frac{1}{\At}}\leq\,\, }\quad \log\prnt{\frac{1}{t_*}}\leq \max\cprnt{\frac{e}{e-1}, \log\prnt{\frac{e}{\At}} }\,.\]
Thus in view of \eqref{eqn:DerZeroOptimization2} we obtain the following upper-bound for $\Uss$:
{\small 
\eql{\label{expr:USS} \Uss&=t_*\log\prnt{\frac{1}{t_*}} \fss H_s\\ \nonumber&=(\log\prnt{\frac{1}{t_*}}-1)\prnt{\At \fss H_s}\leq \max\cprnt{\frac{1}{e-1}, \log\prnt{\frac{1}{\At}} }\prnt{\Ats \fss H_s} \,.}}
Therefore a useful strategy to get an upper bound for $\Uss$ is to bound \\$\Ats \fss H_s=\frac{\fs^2}{\cprnt{p^{-2}(a_s(t))+p^{-2}(b_s(t))}}$ from above and $\frac{1}{\Ats}$ from below. As the following lemma shows bounding the former is \pink{quite} straightforward for monotonic densities.

\begin{lem}\label{lem:BndAFH1} Assume $p(x)$ is monotonically increasing on $I_s$. For all $t\in (0,\half)$ such that $a_s(t)>0$ it holds that $\At \fs H_s\leq diam(I_s)^2$. 
\end{lem}
\begin{proof} Since $p^{-1}(x)$ is monotonically decreasing on $I_s$  for any $x\in [a_s(t), b_s(t)]$ the inequality $p^{-1}(x)\leq p^{-1}(a_s(t))$ holds, whence
\eq{ \At \fs H_s&=\frac{\prnt{\int_{a_s(t)}^{b_s(t)}p^{-1}(x)dx}^2}{\prnt{p^{-2}(a_s(t))+p^{-2}(b_s(t))}\cancel{H_s}}\cdot \cancel{H_s}\\
&\leq \frac{\cancel{p^{-2}(a_s(t))}(b_s(t)-a_s(t))^2}{\cancel{p^{-2}(a_s(t))}}\leq diam(I_s)^2 \,.     }
\qedhere

The following statement will be used to control the term $b_s(t)$ when $t\in [0,\half)$; in essence it is bounded below by its counterpart $\tl{b}(t)$, which corresponds to the uniform density.

\begin{prop}\label{est:5}
Let $I=[x_0,x_1]$ be an interval, and let $\bar{\xi}$ be an a.c.  probability measure on $I$. Set $\bar{p}(x)=\deriv{\xi}{m}$ and $\bar{u}(x)=const=\frac{1}{diam(I)}$, and let respectively $F(x)=\int_{x_0}^{x}\bar{p}(x')dx'$ and $\tilde{F}(x)=\int_{x_0}^{x}\bar{u}(x')dx'$ be their cumulative  distribution functions. Set $b(t):=F^{-1}(1-t)$ and $\tilde{b}(t):=\tilde{F}^{-1}(1-t)$. Assume $\bar{p}(x)$ is continuous monotonically increasing on $[b(\half),x_1]$. If $b(\half)> \tilde{b}(\half)$ then 
$b(t)\geq \tilde{b}(t)$ for all $t\in [0,\half]$. In particular $b(t)\geq \half (1-2t)diam(I)$ for all $t\in [0,\half]$.
\end{prop}
\begin{proof}
Since $\int_{b(\half)}^{x_1}\bar{p}(x)dx=\half=\int_{\tilde{b}(\half)}^{x_1}\bar{u}(x)dx$ and $b(\half)> \tilde{b}(\half)$, it follows by monotonicity of $\bar{p}$ that $\bar{p}(x_1)=\max_{x\in [F^{-1}(\half),x_1]}\bar{p}(x)> \bar{u}(x)=\frac{1}{diam(I)}$.
%Considering the mass balance of $\bar{p}$ and $\bar{u}$ on the interval $I_s$, the respective medians verify the following inequality: $b_s(\half)=F_s^{-1}(\half)\geq \tilde{F}_s^{-1}(\half)$.
Hence, for $t$ sufficiently small (i.e. $b(t)$ sufficiently close to $x_1$): $\int_{b(t)}^{x_1}\bar{p}(x)dx> \int_{b(t)}^{x_1}\bar{u}(x)dx$; equivalently, since $b(0)=\tilde{b}(0)=x_1$,  in some neighborhood of $0$: $b(t)\geq \tilde{b}(t)$.  Assume by contradiction that there exists $t_0\in (0,\half]$ such that $b(t_0)=\tilde{b}(t_0)$.  
Then $t_0=\int_{b(t_0)}^{x_1}\bar{p}(x)dx=\int_{\tilde{b}(t_0)}^{x_1}\bar{u}(x)dx $ and therefore,
\eql{\label{ineq:bu} \half-t_0=\int_{b(\half)}^{b(t_0)}\bar{p}(x)dx=\int_{\tilde{b}(\half)}^{\tilde{b}(t_0)}\bar{u}(x)dx
\stackrel{\substack{\tilde{b}(\half)<b(\half),\\ \tilde{b}(t_0)=b(t_0)}}{>}
\int_{b(\half)}^{b(t_0)}\bar{u}(x)dx\,. }

However since $\bar{p}(x)$ is monotone on $[b(\half), x_1]$, the inequality $\bar{p}(x_1)>\bar{u}(x_1)$ implies that $\bar{p}(b(t_0))<\bar{u}(b(t_0))$ (due to mass balance considerations), but then $\bar{p}(x)< \bar{u}(x)$ on the whole interval $[b(\half),b(t_0)]$, which contradicts inequality \eqref{ineq:bu}. 

The last part of the statement follows from the inequalities
\[ b(t)\geq \tilde{b}(t)=x_1-t(x_1-x_0)\geq \half(x_1-x_0)-t(x_1-x_0)=\half (1-2t)diam(I)\,.\]

\qedhere
\end{proof}

\subsection{Lower bounds for $\Lambda_{LS}(M,\gfrak,\mu)$ under $CDD_b(K,\infty, D)$}

%We begin with estimates for $\Lambda_{LS}(0,\infty,D)$; these can be justified by different arguments, yet for pedagogic reasons we employ the same strategy which we plan to apply to the more difficult case - $\rho(K,\infty,D)$ with $K<0$. Due to essential technical estimates the latter case is significantly more involved. 

\subsubsection{\underline{Lower estimates for $\Lambda_{LS}(M,\gfrak,\mu)$ under $CDD_b(0,\infty,D)$}}

Using the tools we have presented, we will prove the following well known proposition.
\begin{prop} 
Assume $(M^n, \gfrak, \mu)$ is a \cwrm{}, which satisfies $CDD_b(0,\infty, D)$, where $N\in [\max(n,2),\infty]$. Then  $$\Lambda_{LS}(M,\gfrak,\mu)\gtrsim\frac{1}{D^2}\,.$$
\end{prop}

The proposition follows straightforwardly from Theorem \ref{LS:Estimate_gen} and an estimate which we now derive for $\rho_{0,\infty,D}$ using the Bobkov-G\"{o}tze estimates. This can also be achieved by different methods, but we hope that this approach will have a pedagogic value, being a preparation to the more involved estimate of $\rho(K,\infty,D)$ with $K<0$.

Due to invariance under translation and inversion we may restrict to an interval $[0,D]$ and densities $p_s(x)=e^{sx}$, where $s\in (0,\infty)$; we write $\bar{\xi}_s$ for the corresponding probability measures supported on $[0,D]$.  
For $s=0$ we get the known log-Sobolev constant of the uniform density $\Lambda_{LS}([0,D], \frac{1}{D}m)=\frac{\pi^2}{D^2}$, in particular $\UU^{(0)}\eqsim D^2$. Employing the approach we have previously presented we prove the following proposition:

\begin{prop}\label{prop:Xi_00} For all $s\in (0,\infty)$ it holds that $D^2\eqsim \UU^{(0)}\gtrsim \UU^{(s)}$; in particular $\rho_{0,\infty,D}\gtrsim \frac{1}{D^2}$. 
\end{prop}

\begin{proof}
Notice that under the change of variables $y=T(x)$ with $T(x):=sx$, we get a probability measure $d\bar{\nu}_s=d(T_{\sharp}\bar{\xi}_s)$ supported on $[0,sD]$ (here $T_{\sharp}\xi$ is the push-forward by $T$ of the measure $\xi$) and the following identity holds:
\[ \inf_{f\in\F_{LS}(\bar{\xi}_s)}\frac{ \int_0^D f'(x)^2 d\bar{\xi}_s(x)}{Ent_{\bar{\xi}_s}f^2}=\inf_{g\in \F_{LS}(\bar{\nu}_s)}\frac{ s^2\int_0^{sD} g'(y)^2 d\bar{\nu}_s(y) }{Ent_{\bar{\nu}_s}g^2}\,, \]
where in the RHS (valid when $s> 0$) we made the identification $g(y)=f(\frac{y}{s})$. 
Thus as we claimed before the problem can be formulated in terms of a fixed density $p(x)=e^{x}$, a family of intervals $I_s:=[0,sD]$ and a modified target function
\eq{\Ust=\frac{1}{s^2}\sup_{t\in [0,\half)}t\log\prnt{\frac{1}{t}} \fs H_s(x)\,. }
We have $\fs:=\int_{a_s(t)}^{b_s(t)}p^{-1}(x)dx=p^{-1}(a_s(t))-p^{-1}(b_s(t))$ and $H_s(x):=\int_{I_s}p(x)dx=p(sD)-1$.
We separately derive estimates for the case $sD\leq 1$ and the case $sD>1$. 
In case $sD\leq 1$ one can apply a Holley-Stroock argument, but we rather present an alternative equivalent argument, of the same flavor. We write $\fs=\int_{a_s(t)}^{b_s(t)}p(x)^{-2}\cdot p(x)dx $ and using the identity $\int_{a_s(t)}^{b_s(t)}p(x)dx=(1-2t)H_s$,  we conclude that 
\eq{p^{-2}(sD)(1-2t)H_s\leq \int_{a_s(t)}^{b_s(t)}p(x)^{-1}dx\leq p^{-2}(a_s(t))\int_{a_s(t)}^{b_s(t)}p(x)dx\leq p^{-2}(0)(1-2t)H_s \,.}
When $sD\leq 1$ it holds that $p^{-2}(1)\leq p^{-2}(sD)$, whence
\eq{\Ust \eqsim \sup_{t\in [0,\half)} \prnt{t\log\prnt{\frac{1}{t}}(1-2t)}\frac{1}{s^2}(e^{sD}-1)^2 \eqsim D^2\,,}
where we used monotonicity and boundedness properties of the function $\frac{e^x-1}{x}$. 
%e^{-1}\prnt{\frac{e^{1}-1}{1}}^2\cdot D^2\eqsim D^2}
We will now assume $sD>1$. Clearly $\int_{a_s(t)}^{b_s(t)}p(x)^{-1}\leq e^{-a_s(t)}$, therefore $\Ust\leq \Ustvar$ where  $\Ustvar:=\frac{1}{s^2}\sup_{t\in [0,\half)}t\log\prnt{\frac{1}{t}} \fsvar H_s(x)$ with $\fsvar:=e^{-a_s(t)}$. Then by invoking Estimate \ref{expr:USS} (applied to $\Ustvar$ with $\At$ associated to $\fsvar$):
\eq{\Ussvar \leq \frac{1}{s^2}\max\cprnt{\frac{1}{e-1}, \log\prnt{\frac{1}{\At}} }\prnt{\Ats \fssvar H_s}\,. }
Explicit computation of the integrals in definition \eqref{def:a_sb_s} shows that $e^{a_s(t)}=1+tH_s$ hence $(e^{a_s(t)})'=H_s$; we can thus conclude that  $\partial_t|_{t=t_*}\fsvar=\prnt{e^{-a_s(t)}}'=-\frac{\partial_t|_{t=t_*}(e^{a_s(t)})}{(e^{a_s(t_*)})^2}=-H_s\fssvar^2$. It follows that
$\Ats=-\frac{\fssvar}{D_t|_{t=t_*}\fsvar}=\frac{\fssvar^{-1}}{H_s}=\frac{e^{a_s(t_*)}}{e^{sD}-1}$
whence $\log\frac{1}{\Ats}\leq \log\prnt{e^{sD}}=sD $. Furthermore $\Ats \fssvar H_s=1$. Since by assumption $\frac{1}{s}<D$, we conclude that 
\eq{\Ussvar \leq \frac{1}{s^2}\max\cprnt{\frac{1}{e-1}, sD }\cdot 1 \lesssim D^2 \,.}

\end{proof}
%\newpage
%\Line
%\begin{thm} For $\xi\in \Pknd^M$ with  the following estimate holds: $\Lambda_{LS}(\xi)\geq \max\{\frac{KN}{N-1},\frac{c}{D^2}\}$ . 
%\end{thm}
%

%The Cheeger Mazya inequality gives us a relation bet
%According to a result of Bobkov, since the measures $\xi\in \Pknd$ are log-concave under our assumptions on $K$ and $N$, it follows that \frac{1}{4}\inf_{t\in (0,1) \frac{I(t)}{\min(t,1-t) }
%
%By the Cheeger-Mazya inequality 
%\[ \Lambda_1\geq \frac{1}{4}h_{Che}=\frac{1}{4} } \]
%We $h_{Che}$ is Cheegr's isoperimetric constant. 

\subsubsection{\underline{Lower estimates for $\Lambda_{LS}(K,\infty,D)$ when $K\geq 0$}}
The following proposition is a straightforward consequence of the $CDD_b(0,\infty,D)$ estimate. 
\begin{prop}\label{prop:estLS_nonNeg_K} Assume $(M^n,\gfrak,\mu)$ is a \cwrm{}, and assume it satisfies $CDD_b(K,\infty,D)$ with $K\geq 0$. Then 
\eql{\label{LS:Estimate_Pos_K} \Lambda_{LS}(M,\gfrak,\mu) \gtrsim \max\cprnt{ K,\frac{1}{D^2}} \,.}
\end{prop}  
%\cite[p.270]{BGL}
\begin{proof} By the Bakry-\'{E}mery estimate (Theorem \ref{thm:BakryEmeryLS}) $\Lambda_{LS}(M,\gfrak,\mu)\geq\frac{KN}{N-1}$ whenever $N\in (1,\infty]$. By Theorem \ref{LS:Estimate_gen} we have the estimate 
$\Lambda_{LS}(M,\gfrak,\mu)\geq \rho_{K,\infty,D}$.
By Proposition \ref{prop:Xi_00}  $\rho_{K,\infty,D}\gtrsim \frac{1}{D^2}$. Hence in any case it holds that 
$\Lambda_{LS}(M,\gfrak,\mu)\gtrsim \max\{ K,\frac{1}{D^2} \}  $.
\end{proof}

\begin{remk}[Sharpness]

For the one dimensional space $([-\frac{D}{2},\frac{D}{2}], m)$, i.e. the uniform density on the interval, as we previously mentioned $\Lambda_{LS}([-\frac{D}{2},\frac{D}{2}], m)$ is not attained, hence by Theorem  \ref{thm:existenceMinimizers}:
$$\Lambda_{LS}([-\frac{D}{2},\frac{D}{2}], m)=\Lambda_{Poi}([-\frac{D}{2},\frac{D}{2}], m) =\frac{\pi^2}{D^2}\,.$$
In addition for $N\in (1,\infty]$ the Lichnerowitz estimate for the \Poinc constant\\ $\Lambda_{Poi}(M,\gfrak, \mu)\geq\frac{KN}{N-1}$ is sharp; by \eqref{LS:Estimate_Pos_K} it holds that $\Lambda_{LS}(M,\gfrak, \mu)\gtrsim K$; considering that \\$\Lambda_{LS}(M,\gfrak, \mu)\leq \Lambda_{Poi}(M,\gfrak, \mu)$, we conclude that Estimate \eqref{LS:Estimate_Pos_K} is sharp up to numeric constants.

It might also be possible to establish the sharpness (up to constants) on a \cwrm{} of arbitrary topological dimension $n\geq 2$, as in the construction in \cite{Mil2}, however we did not verify the details.
\end{remk}

\subsubsection{\underline{Lower estimates for $\Lambda_{LS}(K,\infty,D)$ when $K=-k<0$}}

We will prove the following proposition, which is the main result of this chapter.
\begin{thm}\label{thm:negK} If $(M^n,\gfrak,\mu)$ is a \cwrm{} which satisfies $CDD_b(K,\infty,D)$ with $K=-k<0$ then the following (sharp up to numeric constants) estimate holds:
\[ \Lambda_{LS}(M,\gfrak,\mu) \gtrsim \max\{ \sqrt{k}, \frac{1}{D}\}\frac{kD}{e^{k\frac{D^2}{8}}-1} \eqsim \begin{cases}
k^{\frac{3}{2}}De^{-\frac{k D^2}{8}}& \qquad \sqrt{k}D>1\\ \frac{1}{D^2}& \qquad \sqrt{k}D\leq 1\,.
\end{cases}      \]
\end{thm}
The proposition is a consequence of the estimate $\Lambda_{LS}(M,\gfrak,\mu)\geq \rho_{-k,\infty,D}$, however in contrast to the previous estimates, getting this lower bound for $\rho_{-k,\infty,D}$ will require significantly more work. 

We consider the fixed density function $p(x)=e^{\half kx^2}$ and a family of intervals $\{I_s\}$ of diameter $D$, defined by $I_s:=[s-\frac{D}{2}, s+\frac{D}{2}]$ where w.l.o.g $s\in [0,\infty]$.  
%We write $d\bar{\xi}_s=\bar{p}\,1_{I_s}dm$ for the corresponding probability measures. 

We will prove the following propositions, which straightforwardly imply Theorem \ref{thm:negK}:
\begin{prop}\label{prop:Xi_0} $\Lambda_{LS}(I_0, \xi_0)\eqsim \Upsilon_0(k,D):=\cprnt{\min\cprnt{ \frac{1}{\sqrt{k}}, D}\prnt{\frac{p(D/2)-1}{kD}}}^{-1}$.
\end{prop}
\begin{prop} \label{prop:Xi_s} For all $s\in [0,\infty)$ it holds that $\Lambda_{LS}(I_s, \xi_s)\gtrsim \Upsilon_0(k,D)$, in particular  $\rho_{-k,\infty,D}\gtrsim \Upsilon_0(k,D)$.
\end{prop}

\subsubsection{Ad-hoc estimates }
We precede the proof of Propositions \ref{prop:Xi_0} and \ref{prop:Xi_s} with the derivation of several useful estimates, which will be crucial for their proof. Throughout we identify $p(x):=e^{\frac{kx^2}{2}}$ with $k>0$. 
\begin{est}\label{est:1} Let $R>0$. Then
\[ \int_0^Re^{\half kx^2}dx\eqsim \frac{e^{\half kR^2}-1}{kR}\,,\]
uniformly for all $R>0$.
\end{est}
\begin{proof}
On $[0, R]$ the convex function $x\mapsto\half kx^2 $ is bounded above by $x\mapsto \frac{k}{2}Rx$  (a line whose graph connects $(0,0)$ to $(R, e^{\half kR^2})$), and on $[\frac{R}{2},R]$ it is bounded below by $x\mapsto kR(x-\half R)$ (whose graph is the tangent to the graph of $x\mapsto\half kx^2$ at $(R,e^{\half kR^2})$). Therefore by monotonicity of $x\mapsto e^x$:
\[ \frac{e^{\half kR^2}-1}{kR}=\int_{\half R}^{R}e^{kR(x-\half R)}dx\leq\int_0^{R}e^{\frac{kx^2}{2}}dx\leq \int_0^{R}e^{\frac{k}{2}Rx}dx=2\frac{e^{\half kR^2}-1}{kR}\,.\]
\end{proof}
%Given $R_1<0<R_2$, since $\int_{0}^{R_2}p\leq \int_{R_1}^{R_2}p\leq 2\int_{0}^{R_2}p$, therefore $\int_{R_1}^{R_2}p\eqsim \frac{p(R_2)-1}{kR_2}$ .
 
\begin{est}\label{est:2} For all $R,k>0$:
$$ \int_{0}^{R}e^{-\half kx^2}dx\eqsim \min\cprnt{\frac{1}{\sqrt{k}}, R}\,.$$

\end{est}
\begin{proof} We express the integrand as a standard normalized Gaussian density:
\[\int_{0}^{R}e^{-\half kx^2}dx=\sqrt{\frac{2\pi}{k}}\prnt{\sqrt{\frac{1}{2\pi}}\int_{0}^{\sqrt{k}R}e^{-\half w^2}dw}\,.\]
For the normalized Gaussian function $\frac{1}{\sqrt{2\pi}}e^{-\half x^2}$ we have the following estimates
\[
\begin{cases} \frac{c}{\sqrt{2\pi}}e^{-\half c^2} & \sqrt{k}R>c\\ \frac{\sqrt{k}R}{\sqrt{2\pi}}e^{-\half c^2} & \sqrt{k}R\leq c \end{cases} \leq
\int_{0}^{\sqrt{k}R}\frac{1}{\sqrt{2\pi}}e^{-\frac{w^2}{2}}dw
\leq \begin{cases} \half & \sqrt{k}R>c\\ \frac{\sqrt{k}R}{\sqrt{2\pi}} & \sqrt{k}R\leq c \end{cases}\,,
\]
which using any fixed constant $c>0$ yields the claimed estimate. 
\end{proof}

\begin{est}\label{est:3} Let $0<a<b$, then  
\[\frac{e^{\frac{kb^2}{2}}-e^{\frac{ka^2}{2}}}{kb}\leq \int_{a}^{b}e^{\frac{kx^2}{2}}dx\leq \frac{e^{\frac{kb^2}{2}}-e^{\frac{ka^2}{2}}}{ka}\,\] 
and
\[\frac{e^{\frac{-ka^2}{2}}-e^{-\frac{kb^2}{2}}}{kb}\leq \int_{a}^{b}e^{-\frac{kx^2}{2}}dx\leq \frac{e^{-\frac{ka^2}{2}}-e^{-\frac{kb^2}{2}}}{ka}\,.\] 
\end{est}
\begin{proof}
Notice that on $[a,b]$: 
\[
\int_{a}^{b}e^{\frac{kx^2}{2}}dx\leq \int_{a}^{b}\frac{x}{a}\cdot e^{\frac{kx^2}{2}}dx=\frac{e^{\frac{kb^2}{2}}-e^{\frac{ka^2}{2}}}{ka}\,.
\]
All other estimates are proved in an identical manner (multiplying by $\frac{x}{b}$ instead of $\frac{x}{a}$ to get estimates in the opposite direction). 
\end{proof}

%\begin{est}\label{est:4} Let $R_1$ be a positive number then 
%\[ \int_0^Rp^{-1}(x)dx\eqsim \frac{p(R)-1}{kR}\,.\] 
%\end{est}
\begin{est}\label{est:6} Assume $c>0$ is some fixed constant, then on $[c,\infty)$ the following inequality holds:
\[(1-e^{-\half c^2})\frac{e^{\half x^2}}{x}\leq \frac{e^{\half x^2}-1}{x}\leq \frac{e^{\half x^2}}{x}\,.\]
\end{est}
\begin{proof} 
Only the LHS of the inequality requires justification.  Notice that for $x\geq c$ it holds that $e^{\half (x^2-c^2)}\geq 1$, hence
$\frac{e^{\half x^2}-1}{x}\geq (1-e^{-\half c^2})\frac{e^{\half x^2}}{x}$.
\end{proof}

\begin{est}\label{lem:BndAFH2} If $s-\frac{D}{2}>0$ then  $\At \fs H_s\leq \frac{1}{k^2(s-\frac{D}{2})^2}$.
%\leq\frac{1-e^{-ksD}}{(s-R)^2k^2}$. 
\end{est}

\begin{proof}
The condition $s-\frac{D}{2}>0$ implies that $I_s\subset (0,\infty)$ and therefore $a_s(t)>0$ and $p(x)$ is increasing on $I_s$. Then by Definition \eqref{eqn:DerZeroOptimization2} (for $\At$) and Estimate \ref{est:3}: $\fs\leq \frac{1}{ka_s(t)}(p^{-1}(a_s(t))-p^{-1}(b_s(t)))$, and it follows that 
\eq{ \At \fs H_s&\leq \prnt{\frac{1}{ka_s(t)}(p^{-1}(a_s(t))-p^{-1}(b_s(t)))}^2\frac{1}{p^{-2}(a_s(t))+p^{-2}(b_s(t)) }\\&
\leq \frac{1}{k^2a_s(t)^2}\cdot 1\leq \frac{1}{k^2(s-\frac{D}{2})^2}\,.
%= \frac{1}{k^2a_s(t)^2 }\frac{\prnt{1-\frac{p(a_s(t))}{p(b_s(t))}}^2}{1+\frac{p(a_s(t))^2}{p(b_s(t))^2}}\\& 
%\stackrel{\substack{p(b_s(t))\leq p(s+\frac{D}{2}),\\ p(a_s(t))\geq p(s-\frac{D}{2})}}{\leq}
%\frac{1}{k^2(s-\frac{D}{2})^2 }\frac{\prnt{p(s+\frac{D}{2})-p(s-\frac{D}{2})}^2}{p(s-\frac{D}{2})^2+p(s+\frac{D}{2})^2}
%\leq \frac{1}{k^2(s-\frac{D}{2})^2}\,.
%\\& \stackrel{\frac{p(s-\frac{D}{2})}{p(s+\frac{D}{2})}=e^{-kDs}}{=}\frac{1}{k^2(s-\frac{D}{2})^2}\frac{(1-e^{-kDs})^2}{1+e^{-2kDs}}
   }
%For $p(x)=e^{\half kx^2}$ the identity $$ holds, whence
	%%=\frac{1}{k^2(\frac{D}{2}-R)^2}\frac{p(\frac{D}{2}+R)(1-e^{-kD})}{
	%\[\Ats \fss H_s\leq \frac{1}{k^2(s-\frac{D}{2})^2}\frac{1-e^{-kD}}{e^{-kD}} \]
\qedhere
\end{proof}

\begin{est} \label{lem:Bnd_A} If $s-\frac{D}{2}>0$ then $\frac{1}{\Ats}\lesssim e^{2kDs} $.
\end{est}
\begin{proof}
As in the previous estimate, $s-\frac{D}{2}>0$ implies that $a_s(t)>0$ and $p(x)$ is increasing on $I_s$. 
By Definition \eqref{eqn:DerZeroOptimization2} and Estimate \ref{est:3}
{\small 
\eql{\label{eq:AstBound} \frac{1}{\At}&\leq\frac{\prnt{p^{-2}(a_s(t))-p^{-2}(b_s(t))}H_s}{\prnt{p^{-1}(a_s(t))-p^{-1}(b_s(t))}/kb_s(t)}=\prnt{p^{-1}(a_s(t))+p^{-1}(b_s(t))}H_skb_s(t)\\ \nonumber &\leq 2p^{-1}(a_s(t))H_skb_s(t)\leq 2p^{-1}(s-\frac{D}{2})H_sk(s+\frac{D}{2})\stackrel{s>\frac{D}{2}}{\leq}  4p^{-1}(s-\frac{D}{2})H_sk s\,.
}
}

By Estimate \eqref{est:3} $H_s\leq \frac{p(s+\frac{D}{2})-p(s-\frac{D}{2})}{ks}$. 
We can thus conclude that
\[ \frac{1}{\Ats}\lesssim  \frac{p(s+\frac{D}{2})-p(s-\frac{D}{2})}{p(s-\frac{D}{2})\cancel{ks}}\cancel{ks}\stackrel{\frac{p(s+\frac{D}{2})}{p(s-\frac{D}{2})}=e^{2kDs}}{\leq} e^{2kDs}\,.\]

%2\frac{p(s+\frac{D}{2})-p(s-\frac{D}{2})}{p(s-\frac{D}{2})}=2\prnt{e^{kDs}-1}

%Since $p^{-1}(x)$ is monotone decreasing on $x>0$
%\eql{\label{eq:AstBound} \frac{1}{\Ats}=\frac{\prnt{p^{-2}(a_s(t_*))-p^{-2}(b_s(t_*))}H_s}{\prnt{p^{-1}(a_s(t_*))-p^{-1}(b_s(t_*))}/ks}\leq \frac{\frac{p(b_s(t_*))}{p(a_s(t_*))}H_s}{\prnt{p(b_s(t_*))-p(a_s(t_*))}/ks}\leq \frac{\frac{p(s+R)}{p(s-R)}H_sks}{p(b_s(t_*))-p(a_s(t_*))} 
%}
%Since $a_s(t_*)+b_s(t_*)\leq 2(s+\frac{D}{2})$ and $t_*\leq \Ats$ we have by \ref{lem:btat_est} the estimate
%\eq{&\frac{p(s+R)}{p(s-R)}H_sks\geq \frac{1}{\Ats}(p(b_s(t_*))-p(a_s(t_*)))\\&\geq \frac{1}{\Ats}\cprnt{\prnt{p(s+\frac{D}{2})-p(s-\frac{D}{2})}-kH_s\int_0^{\Ats}2(s+\frac{D}{2})dx}\\&\geq \frac{p(s+\frac{D}{2})-p(s-\frac{D}{2})}{\Ats}-kH_s2(s+\frac{D}{2})
%\geq\frac{p(s+\frac{D}{2})-p(s-\frac{D}{2})}{\Ats}-4kH_s s  }
%where the last inequality follows from the assumption $s>\frac{D}{2}$. 
%
%Substitution of this estimate into \eqref{eq:AstBound} and using the identity $\frac{p(s+\frac{D}{2})}{p(s-\frac{D}{2})}=e^{kDs}$ gives
%\eq{ &\frac{1}{\Ats}\prnt{p(s+\frac{D}{2})-p(s-\frac{D}{2})}=\frac{1}{\Ats}p(s+\frac{D}{2})\prnt{1-e^{-kDs}}\\&
%=\leq  4H_s ks+e^{kDs}H_sks\leq 2kHs(2+2e^{kDs})\lesssim ksH_se^{kDs} 
%}
%According to estimate \eqref{est:3} $H_s\leq \frac{p(s+\frac{D}{2})-p(s-\frac{D}{2})}{ks}$, we can thus conclude
%\eq{ \frac{1}{\Ats}\lesssim e^{kDs} }
\qedhere
\end{proof}

\subsubsection{Proving Propositions \ref{prop:Xi_0} and \ref{prop:Xi_s}}

\begin{proof}[Proof of Proposition \ref{prop:Xi_0}]
By Estimate \ref{est:2}
\eql{ \label{eqn:equiv_f}   \ff_0(t)=\int_{a_0(t)}^{b_0(t)}p^{-1}(x)dx= 2\int_0^{b_0(t)}p^{-1}(x)dx\eqsim\min\cprnt{\frac{1}{\sqrt{k}}, b_0(t)} \,.}
%We denote by $t_*$ the extremum point where $D_t|_{t=t_*}\ff_0(t)=0$. 
For any $t\in (0, e^{-1})$, it follows from Estimate \ref{est:5} that $b_0(t)\geq b_0(e^{-1})\geq \delta_0 \frac{D}{2}$ where
\[ \delta_0:=1-2e^{-1}\,. \]

Hence $\delta_0\cdot 2\cdot  \frac{D}{2}\leq b_s(t_*)-a_s(t_*) =2b_s(t_*) \leq 2\cdot \frac{D}{2} $. Then for $t\in(0,e^{-1}]$ we have the following equivalence:
%from lemma \ref{est:2} with  
\[      \fs=2\int_0^{b_s(t)}p^{-1}(x)dx\eqsim\min\cprnt{\frac{1}{\sqrt{k}}, D} \,.\]
In addition, according to Estimate \ref{est:1} we have the following estimate for $H_0$ 
\[ H_0\eqsim \frac{p(D/2)-1}{kD}\,. \]
Therefore 
\eq{ \sup_{t\in (0,e^{-1}]}t\log\prnt{\frac{1}{t}}\ff_0(t) H_0 &\eqsim \sup_{t\in (0,e^{-1}]}t\log\prnt{\frac{1}{t}}\min\cprnt{\frac{1}{\sqrt{k}}, D}\prnt{\frac{p(D/2)-1}{kD}}\\&=e^{-1}\min\cprnt{\frac{1}{\sqrt{k}}, D}\prnt{\frac{p(D/2)-1}{kD}}\,,}
where the last equality follows from the monotonicity of $t\mapsto t\log\prnt{\frac{1}{t}}$ on $(0,e^{-1}]$.
\end{proof}

\begin{proof}[Proof of Proposition \ref{prop:Xi_s}]  

We set again $\delta_0:=1-2e^{-1}$. One and only one of the following holds: either  $\sqrt{k} D\leq 1$ or  $\sqrt{k} D> 1$. Notice that the function
$\Upsilon_0(k,D)^{-1}=\min\cprnt{ \frac{1}{\sqrt{k}}, D}\prnt{\frac{p(D/2)-1}{kD}}$ verifies the following estimates
\eq{\Upsilon_0(k,D)^{-1}\eqsim \begin{cases} \frac{1}{k}\prnt{p(D/2)-1} &\\
\frac{1}{k^{3/2}D}\prnt{p(D/2)-1}  &
\end{cases}
\gtrsim
\begin{cases}
 D^2 & \sqrt{k} D\leq 1\\ 
\frac{1}{k}+D^2+kD^4 &\sqrt{k} D>1\,,
\end{cases}
}
where the inequalities can be justified for example by Taylor expansion.

\begin{myitemize}
	\item[If $\sqrt{k} D\leq 1$\,\,:]

			We write $x\in[s-\frac{D}{2}, s+\frac{D}{2}]$ as $x=x'+s$ where $x'\in [-\frac{D}{2},\frac{D}{2}]$. 
			Then since by assumption $\sqrt{k} D\leq 1$ the following inequality holds:
			\[ e^{0+ksx'+\frac{ks^2}{2}}\leq p_s(x'):=e^{\frac{k}{2}(x'+s)^2}\leq e^{\half\cdot \frac{1}{4}+ksx'+\frac{ks^2}{2} }\,.\]
			Consider the density $\tilde{p}_s(x'):=e^{ks(x'+\frac{s}{2})}$ and notice that :
			\[ 1=exp(0)\leq \frac{p(x')}{\tilde{p}_s(x')}\leq exp\prnt{\frac{1}{8}+ksx'+\frac{ks^2}{2}-ksx'-\frac{ks^2}{2}   }=e^{\frac{1}{8}} \,.\]
			Set $H_s=\int_{[-\frac{D}{2},\frac{D}{2}]}p_s(x')dx'$ and $\tilde{H}_s=\int_{[-\frac{D}{2},\frac{D}{2}]}\tilde{p}_s(x')dx'$ (the last estimate shows $H_s\eqsim \tilde{H}_s$). Then by the Holley-Stroock Lemma \ref{lem:HolStrook} and Proposition \ref{prop:Xi_00} it follows that: 
\eq{&\Lambda_{LS}\prnt{H_s^{-1}p_s\cdot 1_{[-\frac{D}{2},\frac{D}{2}]}m }\eqsim \Lambda_{LS}\prnt{\tilde{H}_s^{-1}\tilde{p}_s\cdot 1_{[-\frac{D}{2},\frac{D}{2}]} m }\\&\gtrsim \Lambda_{LS}\prnt{\tilde{H}_0^{-1}\tilde{p}_0\cdot 1_{[0,D]} m }\eqsim \frac{1}{D^2}\gtrsim \Upsilon_0(k,D)\,.}

  \item[If $\sqrt{k} D> 1$\,\,:]
We consider several possible cases distinguished by where $I_s$ is centered. We show that in all cases $\Uss\lesssim\frac{1}{k}+D^2+kD^4 \,\,\,(\lesssim \Upsilon_0(k,D)^{-1})$. 

\begin{enumerate}
	\item[\bf Case A: ${\bf s-\frac{D}{2}> 0}$]
According to Estimate \ref{lem:Bnd_A} :
\[ \log\prnt{\frac{1}{\Ats}}\lesssim 1+kDs\,.\]
\begin{enumerate}
	\item[If $\frac{D}{2}<s\leq D$:]  
	then by Lemmas \ref{lem:BndAFH1} and \ref{lem:Bnd_A} and :
	%\substack{s\leq D,\\ \sqrt{k}\delta_0 D\geq 1}
\[ \Uss\leq \log\prnt{\frac{1}{\Ats}}\prnt{\Ats \fss H_s}\lesssim (1+kDs)D^2 \stackrel{s\leq D}{\leq} D^2+kD^4\,. \]
	\item[If $s>D$:]  By Estimate \ref{lem:BndAFH2}, considering that $s\mapsto \frac{s}{(s-\frac{D}{2})^2}$ is bounded and decreasing for $s\in[D,\infty)$, it follows that:
	\eq{ \Uss&\leq\log\prnt{\frac{1}{\Ats}}\prnt{\Ats \fss H_s}\\&\lesssim
		\prnt{\Ats \fss H_s}+kDs\prnt{\Ats \fss H_s}\\&\stackrel{\ref{lem:BndAFH1},\ref{lem:BndAFH2}}{\lesssim} D^2+\frac{kDs}{k^2(s-\frac{D}{2})^2}\leq D^2+kD^2\frac{1}{(\frac{D}{2})^2k^2}\lesssim D^2+\frac{1}{k}\,.}
\end{enumerate}

%The proof of \ref{prop:Xi_s} will occupy the rest of this section. 

\item[\bf Case B: ${\bf s-\frac{D}{2}\leq 0}$]
We approach this case by considering several sub-cases, corresponding to the possible positions of $a_s(t_*)$.
\begin{enumerate}
	\item[If $\frac{1}{4}\delta_0 D\leq a_s(t_*)<s+\frac{D}{2}$:]
	By definition $\int_{s-\frac{D}{2}}^{a_s(t_*)}p(x)dx=t_* H_s$ hence $t_*>\frac{\int_0^{\frac{1}{4}\delta_0 D}p(x)dx}{H_s}$. 
	According to Estimate \ref{est:1} and equivalence \ref{est:6} :
	{\small
\eq{&H_s\leq H_{\frac{D}{2}}\eqsim \frac{p(D)-1}{kD} \stackrel{\sqrt{k} D> 1}{\eqsim}
\frac{p(D)}{kD}\qquad \text{and}\\   &\int_0^{\frac{1}{4}\delta_0 D}p(x)dx\eqsim \frac{p(\frac{1}{4}\delta_0 D)-1}{k\delta_0 D}\stackrel{\sqrt{k}\frac{1}{4}\delta_0 D> \frac{1}{4}\delta_0}{\eqsim} \frac{p(\delta_0 D)}{k\delta_0 D} \quad (\text{taking } c=\frac{1}{4}\delta_0\, \,\text{in \ref{est:6}})\,. } }

By \eqref{eqn:DerZeroOptimization2} $t_*\leq \Ats$ we conclude that
{\small
\eq{  
\log\prnt{\frac{1}{\Ats}}&\leq \log\prnt{\frac{1}{t_*}}  \leq \log\prnt{\frac{H_{\frac{D}{2}}}{\int_0^{\frac{1}{4}\delta_0 D}p(x)dx}}\eqsim
1+\log\prnt{ \frac{p(D)}{kD}\middle/   \frac{p(\frac{1}{4}\delta_0 D)}{k\delta_0 D}      }\\&\lesssim 1+\log\prnt{\frac{p(D)}{p(\frac{1}{4}\delta_0 D)}}=1+\half k\prnt{D^2-(\frac{1}{4}\delta_0 D)^2}\stackrel{\sqrt{k} D> 1}{\lesssim}  kD^2\,. 
}}
Since by Lemma \ref{lem:BndAFH1} we have the estimate $\Ats \fss H_s\leq D^2$ we conclude that $\Uss\lesssim kD^4$.

	\item[If $0<a_s(t_*)\leq \frac{1}{4}\delta_0 D$:]
	By Estimate \ref{est:5} \, $b_s(e^{-1})\geq \delta_0 \frac{D}{2}$; this implies that $b_s(t_*)-a_s(t_*)\geq  \frac{1}{4}\delta_0 D$. We can thus conclude (see \eqref{eqn:DerZeroOptimization2}) that
{\small
\eq{
\frac{1}{\Ats}&=\frac{p^{-2}(a_s(t_*))+p^{-2}(b_s(t_*))}{\int_{a_s(t_*}^{b_s(t_*)}p^{-1}(x)dx}\cdot H_s\leq \frac{p^{-2}(a_s(t_*))}{p^{-1}(b_s(t_*))\prnt{b_s(t_*)-a_s(t_*)} }\cdot H_s\\&\stackrel{H_s\leq H_{\frac{D}{2}}}{\lesssim} p(D)\cdot \frac{H_{\frac{D}{2}}}{\delta_0 D}
\,,}}
where the last inequality follows from $p(b_s(t_*))\leq p(s+\frac{D}{2})\leq p(D)$ and $p(a_s(t_*))\geq 1$. Using the Estimates \ref{est:1} and \ref{est:6} (recall $\sqrt{k} D> 1$ by assumption):
\[ \frac{H_{\frac{D}{2}}}{D}\lesssim \frac{p(D)}{kD^2}\stackrel{kD^2\geq 1}\leq p(D)\,.\]
It follows that
\[ \log\prnt{\frac{1}{\Ats}}\lesssim 1+\log(p(D)^2)\lesssim kD^2\,. \]
Together with the estimate $\Ats \fss H_s\leq D^2$ from Lemma \ref{lem:BndAFH1} we conclude that $\Uss\lesssim kD^4$.

\item[If $s-\frac{D}{2}< a_s(t_*)\leq 0$:]

For $t\in (0,e^{-1})$ we have the simple inequality
\[ t\log\prnt{\frac{1}{t}}=2t\log\prnt{\frac{1}{\sqrt{t}}}\leq 2t\cdot\frac{1}{\sqrt{t}}=2\sqrt{t}\,,\]
so if $t_*\leq \prnt{\frac{H_0}{H_s}}^2e^{-1}$ then
\[ \Uss=t_*\log\prnt{\frac{1}{t_*}} \fss H_s\lesssim \frac{H_0}{\cancel{H_s}}\fss \cancel{H_s}\lesssim \frac{1}{\sqrt{k}}H_0\eqsim \Upsilon_0(k,D)^{-1}\,,\]
where we used the inequality $\fss\leq \sqrt{\frac{2\pi}{k}}$. So it remains to consider the case $t_*>\prnt{\frac{H_0}{H_s}}^2e^{-1}$.

By considering the equivalence $(e^{x^2}-1)/x\eqsim e^{x^2}/x$ for $x\geq 1$ we have the bound
{\small 
\[ \frac{1}{\sqrt{t}}\leq\frac{H_s}{H_0}\eqsim \frac{p(s+\frac{D}{2})-1}{k(s+\frac{D}{2})}\frac{k\frac{D}{2}}{p(\frac{D}{2})-1}\lesssim \frac{p(s+\frac{D}{2})}{k(s+\frac{D}{2})}\frac{k\frac{D}{2}}{p(\frac{D}{2})}\leq \frac{p(s+\frac{D}{2})}{p(\frac{D}{2})}=e^{kD+\half k s^2} \,. \]}
In order to bound $t_*$ we use the identity $\int_{s-\frac{D}{2}}^{a_s(t_*)}p(x)dx=t_*H_s$, whence
\[ t_*=\frac{\int_{s-\frac{D}{2}}^{a_s(t_*)}p(x)dx}{H_s}\leq \frac{\int_{s-\frac{D}{2}}^{0}p(x)dx}{H_s}
\,.\]
Since $\int_{s-\frac{D}{2}}^{0}p(x)dx\eqsim \frac{p(\frac{D}{2}-s)-1}{k(\frac{D}{2}-s)}$ we conclude 
{\small
\eq{
\Uss&=t_*\log\prnt{\frac{1}{t_*}} \fss H_s\leq \prnt{\frac{\int_{s-\frac{D}{2}}^{0}p}{\cancel{H_s}}}\log\prnt{\frac{e}{(H_0/H_s)^2}}\fss \cancel{H_s}
\\&\lesssim \frac{p(\frac{D}{2}-s)-1}{k(\frac{D}{2}-s)}\prnt{1+k(Ds+\half s^2)}\lesssim \frac{p(\frac{D}{2}-s)-1}{k(\frac{D}{2}-s)}\prnt{1+ kDs}\,,
}}
where the last inequality is due to the assumption $s<\frac{D}{2}$. To complete the proof we prove the claim that $\frac{p(\frac{D}{2}-s)-1}{k(\frac{D}{2}-s)}\prnt{1+kDs}\lesssim \frac{p(\frac{D}{2})-1}{k\frac{D}{2}}$. To this end we show that their ratio is bounded by a constant (the same constant for all $s$). One of the following must hold: 
\begin{enumerate}
  \item $\frac{D}{2}-s\leq\frac{D}{4}$
	\item $\frac{D}{2}-s>\frac{D}{4}$	
\end{enumerate}
Case A: by monotonicity of $\frac{e^{\half kx^2}-1}{x}$ on $(0,\infty)$:
\eq{&\frac{\frac{p(\frac{D}{2}-s)-1}{k(\frac{D}{2}-s)}\prnt{1+ kDs}}{\frac{p(\frac{D}{2})-1}{kD}}\stackrel{\frac{D}{4}\leq s\leq \frac{D}{2}}{\leq} \frac{\frac{p(\frac{D}{4})-1}{k(\frac{D}{4})}\prnt{1+ \half k D^2}}{\frac{p(\frac{D}{2})-1}{kD}}
\lesssim\frac{p(\frac{D}{4})}{p(\frac{D}{2})}\prnt{1+ \half k D^2}\\ &\lesssim e^{-\frac{3}{8}k(\frac{D}{2})^2}(1+\frac{3}{8}k(\frac{D}{2})^2)\lesssim e^{-\frac{3}{8}k(\frac{D}{2})^2}e^{\frac{3}{8}k(\frac{D}{2})^2}=1\,.
}

Case B: Equivalence \ref{est:6} followed by the inequality $\frac{\frac{D}{2}}{\frac{D}{2}-s}\leq 2$ impliy that 

\eq{ &\frac{\frac{p(\frac{D}{2}-s)-1}{k(\frac{D}{2}-s)}\prnt{1+ kDs}}{\frac{p(\frac{D}{2})-1}{kD}}\stackrel{\sqrt{k}(\frac{D}{2}-s)\geq \frac{1}{4}}{\eqsim}\frac{e^{\half k((\frac{D}{2})^2-sD+s^2)}}{e^{\half k(\frac{D}{2})^2}}\prnt{\frac{\frac{D}{2}}{\frac{D}{2}-s}}(1+\frac{1}{4} kDs)\\ \lesssim &
e^{-\half kDs+\half ks^2}e^{\frac{1}{4} kDs}\stackrel{s\leq\frac{D}{2}}\leq e^{-\frac{1}{4} kDs}e^{\frac{1}{4} kDs}=1\,.
}

\end{enumerate}

%End enumeration cases A,B
\end{enumerate} 

\end{myitemize}

\end{proof}

\end{proof}

\begin{remk}[Sharpness]
Computing the log-Sobolev constant for the one dimensional space $([-\frac{D}{2},\frac{D}{2}], e^{-\half Kx^2}dx)$, we get the previous estimates. This is the extreme case, which like in the \Poinc  constant problem, corresponds to a symmetric density around $0$. 
As we mentioned regarding the case $K\geq 0$, also here it is possible to establish the sharpness (up to constants) on a \cwrm{} of arbitrary topological dimension $n\geq 2$, following the construction in \cite{Mil2}; however we did not verify the details.

\end{remk}

\subsection{Conclusions: connections with the \Poinc constant}

Considering our results from the previous sections, we know that when $(M,\gfrak,\mu)$ satisfies $CDD_b(K,\infty,D)$ then
\[  \Lambda_{Poi}(M,\gfrak,\mu)\geq \Lambda_{Poi}(\xi_K) \qquad \text{and}\qquad \Lambda_{LS}(M,\gfrak,\mu)\gtrsim \Lambda_{LS}(\xi_K)\,,\]
where $$d\xi_K(x):= \frac{1}{c_K}e^{ -\half Kx^2}1_{[-\frac{D}{2},\frac{D}{2}]}(x)dm(x) \qquad \text{with} \qquad c_K:=\int_{-\frac{D}{2}}^{\frac{D}{2}}e^{ -\half Kx^2}dx\,, $$ 
is the probability measure $\xi_K\in \Pkinfd^M(\R)$, such that the {\bf sharp} lower bounds $\lam_{K,\infty,D}$ and $\rho_{K,\infty,D}$ to $\Lambda_{Poi}(M,\gfrak,\mu)$ and $\Lambda_{LS}(M,\gfrak,\mu)$ respectively verify
\eql{\label{SharpEquivalences} &\Lambda_{Poi}(M,\gfrak,\mu)\geq\lam_{K,\infty,D}\eqsim \Lambda_{Poi}(\xi_K) \quad\mbox{and    }\\ \nonumber &\Lambda_{LS}(M,\gfrak,\mu)\geq \rho_{K,\infty,D}\eqsim \Lambda_{LS}(\xi_K)\,.}
Our main result in this concluding section is:
\begin{thm}
Under $CDD_b(K,\infty,D)$, where $K\in\R$, it holds that $$\Lambda_{LS}(\xi_K)\eqsim \Lambda_{Poi}(\xi_K)\,.$$
\end{thm}
In other words, the worst \Poinc and log-Sobolev constants, $\lam_{K,\infty, D}$ and $\rho_{K,\infty,D}$, coincide up to numeric constants. 
\bigskip

\subsubsection{`Isoperimetric type' inequalities}

\begin{defn}[Volume and Surface area] Let $(M,\gfrak,\mu=e^{-V}\mu_{\gfrak})$ be a \wrm{}. We define
$$\Omega_{sets}:=\cprnt{A\subset M: \text{A is open with smooth boundary } \partial A}\,.$$
In addition we define the following notions for the members $A$ of $\Omega_{sets}$
\begin{enumerate}
	\item The volume measure of $A$:  $\mu(A):=\int_{A}e^{-V(x)}d\mu_{\gfrak}(x)$.
	\item The surface measure of $A$:  $\mu_s(A):=\int_{\partial A}e^{-V(x)}d\Hcal^{n-1}(x)$,
	where $\Hcal^{n-1}$ stands for the $n-1$ dimensional Hausdorff measure.
	%; we can identify it as
	%$$\mu_s(A)=||\nab_{\gfrak}1_{A}||_{TV(e^{-V}d\mu_{\gfrak})}=\sup\{\int_{A}div_{\gfrak,\mu}X\,e^{-V}d\mu_{\gfrak}:\,\,X\in \Cinf_c(M; TM),\, \gfrak(X,X)\leq 1\} \,.$$
	%where $div_{\gfrak,\mu}X:=div_{\gfrak} X-\Avg{X}{\nab_{\gfrak} V}e^{-V}d\mu_{\gfrak}$ is the weighted divergence. 
	\item The Minkowski surface measure of $A$: $\mu^{+}(A):=\liminf_{\epsilon\to 0}\frac{\mu(A_{\epsilon}^{d_{\gfrak}})-\mu(A)}{\epsilon}$ where $A_{\epsilon}^{d_{\gfrak}}:=\{q\in M:\, \exists p\in A\,\, d_{\gfrak}(p,q)<\epsilon\}$ (here $d_{\gfrak}$ denotes the geodesic distance). 
\end{enumerate}
\end{defn}
It is known for $A\in \Omega_{sets}(M)$ that $\mu_s(A)$ and $\mu^+(A)$ define identical notions of surface area. Therefore for sets $A\in \Omega_{sets}(M)$ there is no ambiguity in the notion of surface area, and we always denote it by $\mu_s(A)$.

Throughout this section we assume that $\mu$ is a probability measure. 
We define the following isoperimetric functions:
\begin{defn}
The isoperimetric function $\gls{Isoper}=\Ical(M,\gfrak,\mu)$ is defined as the pointwise maximal function $\Ical:[0,1]\to\R_{+}$ so that
$\mu_s(A)\geq \Ical(\mu(A))$ for all $A\in \Omega_{sets(M)}$. 
\end{defn}
It turns out that on $\R$ there is another isoperimetric function which is important to consider.  
\begin{defn} We define $\gls{isop_flat}(t)$ as the pointwise maximal function $\Ical^{\flat}(t):[0,1]\to \R_{+}$ s.t. $\mu^+(A)\geq \Ical^{\flat}(\mu(A))$ 
for all $A\in \Omega_{rays}(\R)$ where
\[ \Omega_{rays}(\R):=\bigcup_{a\in \R}\prnt{\{(-\infty, a)\}\cup \{(a,\infty)\}}\,.\]
\end{defn}
Evidently when the underlying space $M$ is $\R$,  we have the inequality $\Ical^{\flat}(t)\geq \Ical(t)$ for any probability measure $\mu$.
However, a useful theorem of Bobkov shows that for the calculation of $\Ical(t)$ on $\R$ equipped with a log-concave measure $\mu$, it is sufficient to consider a smaller subset $\Omega_{rays}(\R)\subset \Omega_{sets}(\R)$ which contains only rays; this is formally stated in the following theorem:
\begin{thm}[Bobkov \cite{Bob2}]\label{thm:BobkovLogConc} On a space $(\R,|\cdot |, \mu)$ which satisfies \red{$CD_b(0,\infty)$} (i.e. $-\log(\deriv{\mu}{m})$ is convex) holds the identity $\Ical=\Ical^{\flat}$. 

\end{thm} 
 
\bigskip
%\subsubsection{Isoperimetric like inequalities under curvature conditions and the \Poinc/log-Sobolev Constants}

Below we briefly survey important results, relating surface and volume measure.

\begin{defn} 
We define the Cheeger constant $\gls{Cheg_const}(M,\gfrak,\mu)$ (resp. Ledoux constant $\gls{Ledo_const}(M,\gfrak,\mu)$) to be the largest constant $c\geq 0$ such that $\mu_{s}(\partial A)\geq c\cdot  \mu(A)$ (resp. 
$\mu_{s}(\partial A)\geq c\cdot \mu(A)\prnt{\log\prnt{\frac{1}{\mu(A)}}}^{\half}$) for all  $A\in \Omega_{sets}(M)$ with $\mu(A)\in (0,\half]$. Equivalently:
 \[ \hChe(M,\gfrak,\mu)=\inf_{A\in \Omega_{sets}(M):\,\,0<\mu(A)\leq \half}\frac{\mu_{s}(\partial A)}{\mu(A)}=\inf_{t\in (0,\half]} \frac{\Ical(t)}{t} \]
and
\[ \lLed(M,\gfrak,\mu)=\inf_{A\in \Omega_{sets}(M):\,\,0<\mu(A)\leq \half}\frac{\mu_{s}(\partial A)}{\mu(A)\prnt{\log\prnt{\frac{1}{\mu(A)}}}^{\half}}=\inf_{t\in (0,\half]} \frac{\Ical(t)}{t\sqrt{\log(\frac{1}{t})}}\,, \]
\end{defn}
where $\Ical=\Ical(M,\gfrak,\mu)$. 
\bigskip

%The Cheeger constant $\hChe(M,\gfrak,\mu)$ (resp. the LEis defined as the largest constant $\hfrak\geq 0$ such that 
%
%
%, or equivalently 

We briefly survey several known results relating volume and surface measure on a \cwrm{} $(M,\gfrak,\mu=e^{-V}\mu_{\gfrak})$, where throughout we assume $\mu$ is a probability measure and $\mu_s$ the associated surface measure as defined above.  

\begin{thm}[Cheeger \cite{Che}, Mazya\cite{Maz1,Maz2,Maz3}]\label{thm:Cheeger} Without  any curvature assumptions:
	\[ \half \hChe(M,\gfrak,\mu)\leq\sqrt{\Lambda_{Poi}(M,\gfrak,\mu)}\,.\]
\end{thm}

Under \red{$CD_b(K,\infty)$} the following reverse form of the Cheeger inequality can also be proved:

\begin{thm}\label{thm:BuserLedoux} 
\begin{enumerate}
%	\item (Buser \cite{Bus}, Ledoux \cite{Led2}) Under curvature conditions $CD(0,\infty)$:
%	\[ \hChe(M,\gfrak,\mu)\gtrsim \sqrt{\Lambda_{Poi}(M,\gfrak,\mu)}\,.\]
	
	\item (Buser \cite{Bus}, Ledoux \cite{Led2}) Under curvature conditions \red{$CD_b(K,\infty)$} where $K=-k\leq 0$: \[\hChe(M,\gfrak,\mu)\gtrsim E_{Poi}(M,\gfrak,\mu)\,,\] with
	   \[ \gls{Bus_Led_fcn}(M,\gfrak,\mu):=\min\prnt{ \frac{\Lambda_{Poi}(M,\gfrak,\mu)}{\sqrt k}, \sqrt{\Lambda_{Poi}(M,\gfrak,\mu)} }\,. \]
	%$\mu_{s}(\partial A)\geq \frac{1}{34}E_{LS}(M,\gfrak,\mu)\mu(A)\prnt{\log\prnt{\frac{1}{\mu(A)}}}^{\half}$ where
\end{enumerate}
\end{thm}

Moreover, similar inequalities hold for the log-Sobolev constant, conditioned that  $\lLed(M,\gfrak,\mu)$ takes the role of $\hChe(M,\gfrak,\mu)$:
\begin{thm}(Beckner, Ledoux \cite{Led4}) \label{thm:BcknerLedoux} Without any curvature assumptions
\[  \lLed(M,\gfrak,\mu)\lesssim\sqrt{\Lambda_{LS}(M,\gfrak,\mu)}\,.\]
\end{thm}
\begin{thm}(Ledoux \cite{Led2}) \label{thm:BuserLedoux2} Under curvature conditions \red{$CD_b(K,\infty)$} where $K=-k\leq 0$:
$$\lLed(M,\gfrak,\mu)\gtrsim E_{LS}(M,\gfrak,\mu)\,,$$ 
with
	\[ \gls{Led_fcn}(M,\gfrak,\mu):=\min\prnt{ \frac{\Lambda_{LS}(M,\gfrak,\mu)}{\sqrt k}, \sqrt{\Lambda_{LS}(M,\gfrak,\mu)}}  \,. \]
%\begin{enumerate}
%  \item Under curvature conditions $CD(0,\infty)$:
%	\[ \hChe(M,\gfrak,\mu)\gtrsim \sqrt{\Lambda_{LS}(M,\gfrak,\mu)}\,.\]

%	\item  

%\end{enumerate}
\end{thm}
\bigskip
Of course when $K=0$ the first expression in the above minima is interpreted as $\infty$ and plays no role.
\bigskip

In the case $K\geq 0$ it is easy to show that the previous results imply the equivalence $\Lambda_{LS}(\xi_K)\eqsim \Lambda_{Poi}(\xi_K)$, as the following theorem shows. 
\begin{thm} Under $CDD_b(K,\infty,D)$ with $K=k\geq 0$, the following equivalence holds: $$\Lambda_{LS}(\xi_K)\eqsim \Lambda_{Poi}(\xi_K)\eqsim \hfrak_{Che}(\xi_K)^2\eqsim \max(K, \frac{1}{D^2})\,.$$
\end{thm}
\begin{proof} From Theorems \ref{thm:Cheeger} and \ref{thm:BuserLedoux} complemented by Bobkov's theorem \ref{thm:BobkovLogConc} we can conclude that under curvature conditions \red{$CD_b(0,\infty)$}:
\eql{\label{PoincEquivCheeger}
\Lambda_{Poi}(\xi_K)\eqsim \hChe(\xi_K)^2=\prnt{\inf_{t\in (0,\half]}\frac{\Ical^{\flat}(t)}{t}}^2\,.
}
One can verify that $\Ical^{\flat}(t)$ (which for a probability density $\bar{p}$ with distribution function $F$ is  just $\bar{p}(F^{-1}(t))$) satisfies the ODE $I^{\flat}I^{\flat''}=-k<0$; we can thus conclude  that $\Ical^{\flat}(t)$ is a concave function on $[0,1]$. Being a non-negative concave function (vanishing only at 0 and 1) and symmetric around $\half$, we conclude that on $(0,\half]$ the slope $\frac{\Ical^{\flat}(t)}{t}$ is minimal at $t=\half$.  In view of that we conclude 
\[ \Lambda_{Poi}(\xi_K)\eqsim\prnt{ \inf_{t\in(0,\half]}\frac{\Ical^{\flat}(t)}{t}}^2=4\Ical^{\flat}(\half)^2=\frac{4}{\prnt{\int_{-\frac{D}{2}}^{\frac{D}{2}}e^{-\half kx^2}dx }^2}\,.\]

By inequality \eqref{SharpEquivalences} and Estimate \ref{est:2} we conclude $\Lambda_{Poi}(\xi_K)\eqsim\max\prnt{k,\frac{1}{D^2}}$. Considering Proposition \ref{prop:estLS_nonNeg_K} and that $\Lambda_{LS}(\xi_K)\leq \Lambda_{Poi}(\xi_K)$   we conclude $\Lambda_{LS}(\xi_K)\eqsim \max\prnt{k,\frac{1}{D^2}}$. It thus follows that  
$$\Lambda_{LS}(\xi_K)\eqsim \Lambda_{Poi}(\xi_K)\eqsim \hfrak(\xi_K)^2\eqsim \max\prnt{k,\frac{1}{D^2}}\,.$$
\end{proof}
%\ref{thm:negK}

\bigskip
Our next goal is to derive a similar result pertaining to the case $K=-k\leq 0$. Our final goal here is to prove the following analogous theorem:
\begin{thm}\label{equiv:PoiLs_xi_K} Under $CDD_b(K,\infty,D)$ with $K=-k< 0$  holds the equivalence $$\Lambda_{Poi}(\xi_K)\eqsim \Lambda_{LS}(\xi_K)\eqsim \max\{ \sqrt{k}, \frac{1}{D}\}\frac{kD}{e^{k\frac{D^2}{8}}-1}\,.$$
\end{thm}
 
Before we prove the statement we will prove some preliminary results. We express the estimates from Proposition  \ref{prop:Xi_s} as 
$$\Lambda_{LS}(M,\gfrak,\mu)\gtrsim \Lambda_{LS}(\xi_K)\eqsim f_{k,D}\,,$$ 
under $CDD_b(K,\infty,D)$ with $K=-k<0$, where
\eql{\label{def:fkd_H} f_{k,D}:=\max\prnt{\sqrt{k}, \frac{1}{D}}\frac{1}{H_{k,D}D}\qquad \text{with}\qquad H_{k,D}:=\frac{e^{\half k \prnt{\frac{D}{2}}^2}-1}{kD^2}\,. }
Notice that $H_{k,D}\eqsim\frac{1}{D}\int_{-\frac{D}{2}}^{\frac{D}{2}}e^{\half kx^2}dx$ and $H_{k,D}\gtrsim  1$ (as one can verify by considering the inequality $\frac{e^x-1}{x}\geq 1$).

The following lemma is valid without any Curvature-Dimension conditions:
\begin{lem}\label{lem:hlEstimate_a} $\lLed(M,\gfrak,\mu)\lesssim \hChe(M,\gfrak,\mu)$.
\end{lem}
\begin{proof}
$\sqrt{\log(2)} \lLed(M,\gfrak,\mu)\leq \hChe(M,\gfrak,\mu)$ since for $t\in (0,\half]$ it holds that $\frac{\log(2)}{\log\prnt{\frac{1}{t}}}\leq 1$\,.
\end{proof}
\begin{lem} 
\label{lem:hlEstimate} $\hChe(\xi_K)\lesssim \frac{1}{H_{k,D}D}$.
\end{lem}
\begin{proof}
 Since $\mathcal{I}\leq \mathcal{I}^{\flat}$:

%\eq{\hChe(M,\gfrak,\mu)\lesssim \Lambda_{Poi}(M,\gfrak,\mu)\eqsim  \Lambda_{Poi}(\xi_K)

\eq{\hChe(\xi_K)=\inf_{t\in(0,\half]}\frac{\Ical(t)}{t}\leq  \inf_{t\in(0,\half]}\frac{\Ical^{\flat}(t)}{t} \leq 2\Ical^{\flat}(\half)=\frac{2}{\int_{-\frac{D}{2}}^{\frac{D}{2}}e^{\half kx^2}dx}\eqsim \frac{1}{H_{k,D}D} \,.   }
\end{proof}

\begin{prop}\label{estim:E_LS_Lam_LS} Under $CDD_b(K,\infty,D)$ with $K=-k<0$ holds the inequality: $E_{LS}(\xi_K)\gtrsim \frac{1}{H_{k,D}D}$.
\end{prop}
\begin{proof} Recall $\Lambda_{LS}(\xi_K)\eqsim f_{k,D}$ where $f_{k,D}$ was  defined in \eqref{def:fkd_H}. By definition $E_{LS}(\xi_K)=\min\prnt{ \frac{\Lambda_{LS}(\xi_K)}{\sqrt{k}},\sqrt{\Lambda_{LS}(\xi_K)} }$, hence:
	\begin{enumerate}
	\item \pmb{ If $\sqrt{k}>\frac{1}{D}$}: then $f_{k,D}=\frac{\sqrt{k}}{H_{k,D}D}$, whence
	\eq{ &\min\prnt{ \frac{f_{k,D}}{\sqrt{k}},\sqrt{f_{k,D}} } = \min\prnt{\frac{1}{H_{k,D}D}, \sqrt{\frac{\sqrt{k}}{H_{k,D}D}}}\geq 
	\min\prnt{\frac{1}{H_{k,D}D}, \sqrt{\frac{1}{H_{k,D}D^2}}}\\&\stackrel{H_{k,D}\gtrsim 1}{\gtrsim} \min\prnt{\frac{1}{H_{k,D}D}, \sqrt{\frac{1}{H_{k,D}^2D^2}}}=\frac{1}{H_{k,D}D}\,. }
	\item \pmb{ If $\sqrt{k}\leq \frac{1}{D}$}: then $f_{k,D}=\frac{1}{H_{k,D}D^2}$, whence
	\eq{ &\min\prnt{ \frac{f_{k,D}}{\sqrt{k}},\sqrt{f_{k,D}} }=\min\prnt{\frac{1}{\sqrt{k}H_{k,D}D^2}, \sqrt{\frac{1}{H_{k,D}D^2}}} \\&\stackrel{\sqrt{k}D\leq 1,\,H_{k,D}\gtrsim 1}{\gtrsim} \min\prnt{\frac{1}{H_{k,D}D}, \sqrt{\frac{1}{H_{k,D}^2D^2}}}=\frac{1}{H_{k,D}D}\,.}
	\end{enumerate}
	Thus in either case: $\min\prnt{ \frac{\Lambda_{LS}(\xi_K)}{\sqrt{k}},\sqrt{\Lambda_{LS}(\xi_K)} }\gtrsim \frac{1}{H_{k,D}D}$.
\end{proof}

\begin{cor}\label{equivalences_Els_Epoi} Under $CDD_b(K,\infty,D)$ with $K=-k< 0$ the following equivalences are satisfied: $\hChe(\xi_K)\eqsim \lLed(\xi_K)\eqsim E_{Poi}(\xi_K)\eqsim E_{LS}(\xi_K)\eqsim \frac{1}{H_{k,D}D}$.
\end{cor}
\begin{proof} By Lemma \ref{lem:hlEstimate} and Ledoux's Theorem \ref{thm:BuserLedoux}:
\[ \frac{1}{H_{k,D}D}\gtrsim \hChe(\xi_K)\gtrsim E_{Poi}(\xi_K)\,.\]
On the other hand, by Proposition \ref{estim:E_LS_Lam_LS} and the  inequality $\Lambda_{Poi}(\xi_K)\geq \Lambda_{LS}(\xi_K)$ we also have
\[ E_{Poi}(\xi_K)\geq E_{LS}(\xi_K)\gtrsim \frac{1}{H_{k,D}D} \,.\]
We may now conclude that 
$$\hChe(\xi_K)\eqsim E_{Poi}(\xi_K)\eqsim E_{LS}(\xi_K)\eqsim \frac{1}{H_{k,D}D}\,.$$
However, according to Theorem  \ref{thm:BuserLedoux2} $\lLed(\xi_K)\gtrsim E_{Poi}(\xi_K)$, and by Lemma \ref{lem:hlEstimate_a}\\ $\lLed(M,\gfrak,\mu)\lesssim \hChe(M,\gfrak,\mu)$, so $\hChe(M,\gfrak,\mu)\gtrsim \lLed(M,\gfrak,\mu)\gtrsim E_{Poi}(\xi_K)$ implying that
\\$\lLed(M,\gfrak,\mu)\eqsim \frac{1}{H_{k,D}D}$. This completes the proof of the statement. 
\end{proof}
\bigskip

We will now prove Theorem \ref{equiv:PoiLs_xi_K}:
\begin{proof}[Proof of Theorem \ref{equiv:PoiLs_xi_K}]
According to Corollary \ref{equivalences_Els_Epoi} the following equivalences are satisfied $E_{Poi}(\xi_K)\eqsim E_{LS}(\xi_K)\eqsim \frac{1}{H_{k,D}D}$.

\begin{enumerate}
	\item If $\frac{\Lambda_{Poi}(\xi_K)}{\sqrt{k}}< \sqrt{\Lambda_{Poi}(\xi_K)}$: then $\frac{\Lambda_{Poi}(\xi_K)}{\sqrt{k}}\eqsim \frac{1}{H_{k,D}D}$, i.e. $\Lambda_{Poi}(\xi_K)\eqsim \frac{\sqrt{k}}{H_{k,D}}\frac{1}{D}$. 
\item If $\frac{\Lambda_{Poi}(\xi_K)}{\sqrt{k}}\geq \sqrt{\Lambda_{Poi}(\xi_K)}$:	then $\sqrt{\Lambda_{Poi}(\xi_K)}\eqsim \frac{1}{H_{k,D}}\frac{1}{D}$ implying that $\Lambda_{Poi}(\xi_K)\eqsim \frac{1}{H_{k,D}^2D^2}\lesssim \frac{1}{H_{k,D}D^2}$ (since $H_{k,D}\gtrsim 1$). 
\end{enumerate}
\smallskip

We conclude that in either case 
$$\Lambda_{Poi}(\xi_K)\lesssim \frac{1}{H_{k,D}D}\max\{\sqrt{k},\,\frac{1}{D}\}\,.$$
However since $\Lambda_{LS}(\xi_K)\eqsim \frac{1}{H_{k,D}D}\max\{\sqrt{k},\,\frac{1}{D}\}$ (see \eqref{def:fkd_H}) and $\Lambda_{Poi}(\xi_K)\geq \Lambda_{LS}(\xi_K)$  we conclude that $\Lambda_{Poi}(\xi_K)\eqsim \Lambda_{LS}(\xi_K)\eqsim \frac{1}{H_{k,D}D}\max\{\sqrt{k},\,\frac{1}{D}\}$.
\end{proof}

\begin{remk}
When $D\geq \frac{1}{\sqrt{|K|}}$ our derivation shows that \eq{\lLed(\xi_K)&\eqsim\frac{\Lambda_{LS}(\xi_K)}{\sqrt{|K|}}\quad (\eqsim \frac{1}{H_{k,D}D})\,, \quad \text{and}\\
\hChe(\xi_K)&\eqsim\frac{\Lambda_{Poi}(\xi_K)}{\sqrt{|K|}}\quad (\eqsim \frac{1}{H_{k,D}D})\,.}
In other words, the model measure $d\xi_K(x)=c_K\cdot e^{-\frac{Kx^2}{2}}1_{[-\frac{D}{2},\frac{D}{2}]}(x)dx$ realizes the extremal cases in the Buser-Ledoux Theorem \ref{thm:BuserLedoux}  and the Ledoux Theorem \ref{thm:BuserLedoux2}.
\end{remk}

 \chapter{Epilogue}\label{chp:epilogue}

We conclude this work with a summary of the main results and other contributions.  

\begin{enumerate}
\item Our first major result  was the extreme points characterization theorem which  complements Klartag's needle decomposition theorem \cite{Kla}. A corollary of the theorem was that for a solution to the 1d optimization problem formulated by B. Klartag on the set $\Fknd^{\Cinf}$, which characterizes the sharp lower bound on the constant associated with a given functional inequality on a \cwrm{} $(M^n,\gfrak,\mu)$ which satisfies $CDD_b(K,N,D)$, it is sufficient to consider the subclass of model measures $\Fknd^M(\R)$.
The members of this class assume a specific form, which  brings the optimization problem into a simpler tangible form.

To this end, \pink{we employed} a new approach inspired by Fradelizi-Gu\'{e}don \cite{FG},
based on functional analysis arguments. In particular, classification of extreme points of a corresponding set of measures. The development of this approach was not direct; we digressed into defining and studying an abstract class of measures $\Fknd(\R)$ which satisfy `synthetic $CDD_b(K,N,D)$ conditions'. The discussion on the synthetic class and its properties was \pink{quite}  lengthy, however it eventually led to concrete conclusions regarding our original optimization problem. 
The motivation for this approach is its generality:
\begin{itemize}
    \item it applies to diverse type of functional inequalities, in particular of Poincar\'{e}, p-Poincar\'{e} and log-Sobolev type;
    \item moreover, the reduction to the model class is almost straightforward, and requires little effort when we switch from one functional inequality to another;
    \item it applies to $CDD_b(K,N,D)$ conditions with parameters 
    $K\in\R$, $N\in (-\infty,0]\cup [\max(2,n),\infty]$ and $D\in (0,\infty]$; since it does not rely on eigenfunctions, it is not restricted to compact manifolds (i.e. $D<\infty$),  as in the previous works of \cite{Kro, BaQi, Val, NaVa} et al.
    \end{itemize}
These observations justify the lengthy discussion of Chapter \ref{chp:ExPoints}. 

\item Our first major application of the method was in proving sharp lower bounds on the \Poinc constant of manifolds which satisfy 
$CDD_b(K,N,D)$. Our most important contribution was filling the gap concerning the range $N\in (-\infty,0]$. We gave a complete characterization of the sharp lower bound on the \Poinc constant, assuming $CDD_b(K,N,D)$ conditions with $N$ in this range. In addition we also showed how the \Poinc constant of the class $\Fknd^M(\R)$ depends monotonically on the slope parameter $\hfrak$. Our study showed that 
while for $N\in (-\infty,-1)$ the dependence is similar to the previously studied range of $1\neq N\in [n,\infty]$, in the range $N\in (-1,0]$
it is reversed, and hence the characterization of the sharp lower bound is of completely different nature. 

\bigskip

\item
Following a similar approach we also made a small contribution to the p-\Poinc functional inequality. We showed that (under a certain technical assumption)  the derivation of Naber and Valtorta of a sharp lower bound for the p-\Poinc constant, under $CDD_b(K,N,D)$ with $K\leq 0$, $N=n$ and $D<\infty$, can be extended with minor efforts to $K\in \R$, $N\in [\max(2,n),\infty]$ and $D\in (0,\infty]$. As for the \Poinc constant, we showed that the p-\Poinc constant of the class $\Fknd^M(\R)$ depends monotonically on the slope parameter $\hfrak$.

\item Lastly, we expressed, up to numeric constants, the   sharp lower bound on the log-Sobolev constant of \pink{a} \cwrm{} which \pink{satisfies} $CDD_b(K,\infty,D)$; moreover, we showed that up to numeric constants it is equivalent to the sharp bound we found for the \Poinc constant of \pink{a} \cwrm{} which \pink{satisfies} $CDD_b(K,\infty,D)$. 

Our general study of the log-Sobolev constant of the class $\Pknd^M(\R)$ showed that for almost all members of the class, but those which are symmetric w.r.t. the center of their support, the log-Sobolev constant is attained. This  observation paves the way for an alternative approach for the characterization of the  sharp lower bound on the log-Sobolev constant, via the Euler-Lagrange equations; such an approach can lead to characterization of the sharp lower bound on the log-Sobolev not just up to numeric constants. 
\end{enumerate} 

% \input{bibliography}
%\bibliography{References.bib}{}
\bibliographystyle{apa}
\bibliography{bib-file-name}

\begin{bibdiv}
\begin{biblist}
\bib{AK}{book}{
   author={Aliprantis, Charalambos D.},
   author={Border, Kim C.},
   title={Infinite-dimensional analysis, A hitchhiker's guide},
   edition={2},
   publisher={Springer},
   date={1999},

}

\bib{Bac}{article}{
   author={Bacher, Kathrin},
   title={On Borell-Brascamp-Lieb inequalities on metric measure spaces},
   journal={Potential Anal.},
   volume={33},
   date={2010},
   number={1},
   pages={1--15},

}

\bib{BS}{article}{
   author={Bacher, Kathrin},
   author={Sturm, Karl-Theodor},
   title={Localization and tensorization properties of the
   curvature-dimension condition for metric measure spaces},
   journal={J. Funct. Anal.},
   volume={259},
   date={2010},
   number={1},
   pages={28--56},
}
\bib{BE1}{article}{
   author={Bakry, Dominique},
   title={L'hypercontractivit\'{e} et son utilisation en th\'{e}orie des
   semigroupes},
   language={French},
   conference={
      title={Lectures on probability theory},
      address={Saint-Flour},
      year={1992},
   },
   book={
      series={Lecture Notes in Math.},
      volume={1581},
      publisher={Springer},
   },
   date={1994},
   pages={1--114},
}	
\bib{BE2}{article}{
   author={Bakry, D.},
   author={\'{E}mery, Michel},
   title={Diffusions hypercontractives},
   language={French},
   conference={
      title={S\'{e}minaire de probabilit\'{e}s, XIX, 1983/84},
   },
   book={
      series={Lecture Notes in Math.},
      volume={1123},
      publisher={Springer},
   },
   date={1985},
   pages={177--206},
}
\bib{BGL}{book}{
   author={Bakry, Dominique},
   author={Gentil, Ivan},
   author={Ledoux, Michel},
   title={Analysis and geometry of Markov diffusion operators},
   series={Grundlehren der Mathematischen Wissenschaften [Fundamental
   Principles of Mathematical Sciences]},
   volume={348},
   publisher={Springer},
   date={2014},
}



\bib{BMZ}{article}{
   author={Barthe, Franck},
   author={Ma, Yutao},
   author={Zhang, Zhengliang},
   title={Logarithmic Sobolev inequalities for harmonic measures on spheres},
   journal={J. Math. Pures Appl. (9)},
   volume={102},
   date={2014},
   number={1},
   pages={234--248},
}
\bib{BaQi}{article}{
   author={Bakry, Dominique},
   author={Qian, Zhongmin},
   title={Some new results on eigenvectors via dimension, diameter, and
   Ricci curvature},
   journal={Adv. Math.},
   volume={155},
   date={2000},
   number={1},
   pages={98--153},
}
		
\bib{BaRo}{article}{
   author={Barthe, F.},
   author={Roberto, C.},
   title={Sobolev inequalities for probability measures on the real line},
   note={Dedicated to Professor Aleksander Pe\l czy\'{n}ski on the occasion of his
   70th birthday},
   journal={Studia Math.},
   volume={159},
   date={2003},
   number={3},
   pages={481--497},
}
\bib{Bar}{book}{
   author={Barvinok, Alexander},
   title={A course in convexity},
   series={Graduate Studies in Mathematics},
   volume={54},
   publisher={American Mathematical Society},
   date={2002},
}

	
	

\bib{Bay}{thesis}{
   author={Bayle, Vincent},
   title={Propri{\'e}t{\'e}s de concavit{\'e} du profil isop{\'e}rim{\'e}trique et applications},
   date={2003},
   school={Th\'{e}se de doctorat dirig\'{e}e par Besson, G\'{e}rard Math\'{e}matiques Université Joseph Fourier},
   language={French},
}
\bib{BiDr}{article}{
   author={Binding, P.},
   author={Dr\'{a}bek, P.},
   title={Sturm-Liouville theory for the $p$-Laplacian},
   journal={Studia Sci. Math. Hungar.},
   volume={40},
   date={2003},
   number={4},
   pages={375--396},
}
\bib{BiRo}{book}{
   author={Birkhoff, Garrett},
   author={Rota, Gian-Carlo},
   title={Ordinary differential equations},
   edition={4},
   publisher={John Wiley \& Sons, Inc.},
   date={1989},
}

\bib{Bis}{article}{
   author={Bishop, Richard L.},
   title={Infinitesimal convexity implies local convexity},
   journal={Indiana Univ. Math. J.},
   volume={24},
   date={1974/75},
   pages={169--172},
}

\bib{Bo3}{article}{
   author={Borell, Christer},
   title={Convex measures on locally convex spaces},
   journal={Ark. Mat.},
   volume={12},
   date={1974},
   pages={239--252},
}					

\bib{Bo1}{article}{
   author={Borell, Christer},
   title={Convex set functions in $d$-space},
   journal={Period. Math. Hungar.},
   volume={6},
   date={1975},
   number={2},
   pages={111--136},
}
	




\bib{Bo2}{article}{
   author={Borell, Christer},
   title={Geometric properties of some familiar diffusions in ${\bf R}^n$},
   journal={Ann. Probab.},
   volume={21},
   date={1993},
   number={1},
   pages={482--489},
}
\bib{Bob2}{article}{
   author={Bobkov, Sergey G.},
   title={Extremal properties of half-spaces for log-concave distributions},
   journal={Ann. Probab.},
   volume={24},
   date={1996},
   number={1},
   pages={35--48},
}
\bib{Bob1}{article}{
   author={Bobkov, Sergey G.},
   title={Large deviations and isoperimetry over convex probability measures
   with heavy tails},
   journal={Electron. J. Probab.},
   volume={12},
   date={2007},
   pages={1072--1100},
}
\bib{BobG}{article}{
   author={Bobkov, S. G.},
   author={G\"{o}tze, F.},
   title={Exponential integrability and transportation cost related to
   logarithmic Sobolev inequalities},
   journal={J. Funct. Anal.},
   volume={163},
   date={1999},
   number={1},
   pages={1--28},
}
			
\bib{BoLed}{article}{
   author={Bobkov, Sergey G.},
   author={Ledoux, Michel},
   title={Weighted Poincar\'{e}-type inequalities for Cauchy and other convex
   measures},
   journal={Ann. Probab.},
   volume={37},
   date={2009},
   number={2},
   pages={403--427},
}
\bib{Bou}{book}{
   author={Bourbaki, N.},
   title={Topological vector spaces. Chapters 1--5},
   series={Elements of Mathematics (Berlin)},
   note={Translated from the French by H. G. Eggleston and S. Madan},
   publisher={Springer},
   date={1987},
}				
\bib{BL}{article}{
   author={Brascamp, Herm Jan},
   author={Lieb, Elliott H.},
   title={On extensions of the Brunn-Minkowski and Pr\'{e}kopa-Leindler
   theorems, including inequalities for log concave functions, and with an
   application to the diffusion equation},
   journal={J. Functional Analysis},
   volume={22},
   date={1976},
   number={4},
   pages={366--389},
}
\bib{BGVV}{book}{
   author={Brazitikos, Silouanos},
   author={Giannopoulos, Apostolos},
   author={Valettas, Petros},
   author={Vritsiou, Beatrice-Helen},
   title={Geometry of isotropic convex bodies},
   series={Mathematical Surveys and Monographs},
   volume={196},
   publisher={American Mathematical Society},
   date={2014},
}		


\bib{Bus}{article}{
   author={Buser, Peter},
   title={A note on the isoperimetric constant},
   journal={Ann. Sci. \'{E}cole Norm. Sup. (4)},
   volume={15},
   date={1982},
   number={2},
   pages={213--230},
}
		
\bib{CFM}{article}{
   author={Caffarelli, Luis A.},
   author={Feldman, Mikhail},
   author={McCann, Robert J.},
   title={Constructing optimal maps for Monge's transport problem as a limit
   of strictly convex costs},
   journal={J. Amer. Math. Soc.},
   volume={15},
   date={2002},
   number={1},
   pages={1--26},
}	

\bib{Cav}{article}{
   author={Cavalletti, Fabio},
   title={An overview of $L^1$ optimal transportation on metric measure
   spaces},
   conference={
      title={Measure theory in non-smooth spaces},
   },
   book={
      series={Partial Differ. Equ. Meas. Theory},
      publisher={De Gruyter Open},
   },
   date={2017},
   pages={98--144},
}


\bib{CaMil}{article}{

author={Cavalletti, Fabio},
   author={Milman, Emanuel},
        title = {The Globalization Theorem for the Curvature Dimension Condition},
      journal = {arXiv e-prints},
     keywords = {Mathematics - Metric Geometry, Mathematics - Functional Analysis},
         date = {2016},
archivePrefix = {arXiv},
       eprint = {1612.07623},
 primaryClass = {math.MG},
}

\bib{CM1}{article}{
   author={Cavalletti, Fabio},
   author={Mondino, Andrea},
   title={Sharp geometric and functional inequalities in metric measure
   spaces with lower Ricci curvature bounds},
   journal={Geom. Topol.},
   volume={21},
   date={2017},
   number={1},
   pages={603--645},
}


\bib{Cha1}{book}{
   author={Chavel, Isaac},
   title={Riemannian geometry---a modern introduction},
   series={Cambridge Tracts in Mathematics},
   volume={108},
   publisher={Cambridge University Press},
   date={1993},

}
	

\bib{Che}{article}{
   author={Cheeger, Jeff},
   title={A lower bound for the smallest eigenvalue of the Laplacian},
   conference={
      title={Problems in analysis},
      address={Papers dedicated to Salomon Bochner},
      date={1969},
   },
   book={
      publisher={Princeton Univ. Press},
   },
   date={1970},
   pages={195--199},
}
		
\bib{Che1}{book}{
   author={Chen, Mufa},
   title={Eigenvalues, inequalities, and ergodic theory},
   series={Probability and its Applications},
   publisher={Springer},
   date={2005},
   pages={xiv+228},
}
\bib{Che2}{article}{
   author={Chen, Mufa},
   title={Explicit bounds of the first eigenvalue},
   journal={Sci. China Ser. A},
   volume={43},
   date={2000},
   number={10},
   pages={1051--1059},

}
		

\bib{Che3}{article}{
   author={Chen, Mufa},
   author={Scacciatelli, E.},
   author={Yao, Liang},
   title={Linear approximation of the first eigenvalue on compact manifolds},
   journal={Sci. China Ser. A},
   volume={45},
   date={2002},
   number={4},
   pages={450--461},
}

\bib{ChWa}{article}{
   author={Chen, Mufa},
   author={Wang, Fengyu},
   title={General formula for lower bound of the first eigenvalue on
   Riemannian manifolds},
   journal={Sci. China Ser. A},
   volume={40},
   date={1997},
   number={4},
   pages={384--394},
}
\bib{DHe}{article}{
   author={Dauge, Monique},
   author={Helffer, Bernard},
   title={Eigenvalues variation. I. Neumann problem for Sturm-Liouville
   operators},
   journal={J. Differential Equations},
   volume={104},
   date={1993},
   number={2},
   pages={243--262},

}
\bib{DoRe}{book}{
   author={Do\v{s}l\'{y}, Ond\v{r}ej},
   author={\v{R}eh\'{a}k, Pavel},
   title={Half-linear differential equations},
   series={North-Holland Mathematics Studies},
   volume={202},
   publisher={Elsevier Science},
   date={2005},
   pages={xiv+517},
}
		

\bib{DrKu}{article}{
   author={Dr\'{a}bek, Pavel},
   author={Kuliev, Komil},
   title={Half-linear Sturm-Liouville problem with weights},
   journal={Bull. Belg. Math. Soc. Simon Stevin},
   volume={19},
   date={2012},
   number={1},
   pages={107--119},
}
	


\bib{Eva2}{article}{
   author={Evans, Lawrence C.},
   title={Partial differential equations and Monge-Kantorovich mass
   transfer},
   conference={
      title={Current developments in mathematics, 1997 (Cambridge, MA)},
   },
   book={
      publisher={Int. Press},
   },
   date={1999},
   pages={65--126},
}
\bib{Eva}{book}{
   author={Evans, Lawrence C.},
   title={Partial differential equations},
   series={Graduate Studies in Mathematics},
   volume={19},
   edition={2},
   publisher={American Mathematical Society},
   date={2010},

}
\bib{EvGa}{book}{
   author={Evans, Lawrence C.},
   author={Gariepy, Ronald F.},
   title={Measure theory and fine properties of functions},
   series={Textbooks in Mathematics},
   edition={Revised edition},
   publisher={CRC Press, Boca Raton, FL},
   date={2015},
   pages={xiv+299},
}
		
\bib{FG}{article}{
   author={Fradelizi, Matthieu},
   author={Gu\'{e}don, Olivier},
   title={The extreme points of subsets of $s$-concave probabilities and a
   geometric localization theorem},
   journal={Discrete Comput. Geom.},
   volume={31},
   date={2004},
   number={2},
   pages={327--335},
}
% \bib{MeLi}{article}{
%   author={Garc\'{i}a Meli\'{a}n, Jorge},
%   author={Sabina de Lis, Jos\'{e}},
%   title={On the perturbation of eigenvalues for the $p$-Laplacian},
%   language={English},
%   journal={C. R. Acad. Sci. Paris S\'{e}r. I Math.},
%   volume={332},
%   date={2001},
%   number={10},
%   pages={893--898},

% }	
\bib{Gar}{article}{
   author={Gardner, R. J.},
   title={The Brunn-Minkowski inequality},
   journal={Bull. Amer. Math. Soc. (N.S.)},
   volume={39},
   date={2002},
   number={3},
   pages={355--405},
}
\bib{GMS}{article}{
   author={Ghang, Whan},
   author={Martin, Zane},
   author={Waruhiu, Steven},
   title={The sharp log-Sobolev inequality on a compact interval},
   journal={Involve},
   volume={7},
   date={2014},
   number={2},
   pages={181--186},
   issn={1944-4176},

}	


\bib{Girg}{article}{
   author={Girg, Petr},
   author={Kotrla, Luk\'{a}\v{s}},
   title={Differentiability properties of $p$-trigonometric functions},
   conference={
      title={Proceedings of the Variational and Topological Methods: Theory,
      Applications, Numerical Simulations, and Open Problems},
   },
   book={
      series={Electron. J. Differ. Equ. Conf.},
      volume={21},
      publisher={Texas State Univ.},
   },
   date={2014},
   pages={101--127},

}

\bib{GM}{article}{
   author={Gromov, M.},
   author={Milman, V. D.},
   title={Generalization of the spherical isoperimetric inequality to
   uniformly convex Banach spaces},
   journal={Compositio Math.},
   volume={62},
   date={1987},
   number={3},
   pages={263--282},
}

\bib{Gro}{article}{
   author={Gross, Leonard},
   title={Logarithmic Sobolev inequalities},
   journal={Amer. J. Math.},
   volume={97},
   date={1975},
   number={4},
   pages={1061--1083},

}



	

		
	
	

\bib{HaLi}{book}{
   author={Hardy, G. H.},
   author={Littlewood, J. E.},
   author={P\'{o}lya, G.},
   title={Inequalities},
   series={Cambridge Mathematical Library},
   note={Reprint of the 1952 edition},
   publisher={Cambridge University Press},
   date={1988},
}

\bib{HeKa}{article}{
   author={Heintze, Ernst},
   author={Karcher, Hermann},
   title={A general comparison theorem with applications to volume estimates
   for submanifolds},
   journal={Ann. Sci. \'{E}cole Norm. Sup. (4)},
   volume={11},
   date={1978},
   number={4},
   pages={451--470},
}


\bib{HoSt}{article}{
   author={Holley, Richard},
   author={Stroock, Daniel},
   title={Logarithmic Sobolev inequalities and stochastic Ising models},
   journal={J. Statist. Phys.},
   volume={46},
   date={1987},
   number={5-6},
   pages={1159--1194},

}				
\bib{Geo}{book}{
   author={Holmes, Richard B.},
   title={Geometric functional analysis and its applications},
   note={Graduate Texts in Mathematics, No. 24},
   publisher={Springer},
   date={1975},
}		
	
	

\bib{KLS}{article}{
   author={Kannan, R.},
   author={Lov\'{a}sz, L.},
   author={Simonovits, M.},
   title={Isoperimetric problems for convex bodies and a localization lemma},
   journal={Discrete Comput. Geom.},
   volume={13},
   date={1995},
   number={3-4},
   pages={541--559},
}
\bib{KaAk}{book}{	
author={Kantorovich, L.V.},
author={Akilov, G.P.},
  title={Functional Analysis},
  date={2016},
  publisher={Elsevier Science},
}	
\bib{KN}{article}{
   author={Kawai, Shigeo},
   author={Nakauchi, Nobumitsu},
   title={The first eigenvalue of the $p$-Laplacian on a compact Riemannian
   manifold},
   journal={Nonlinear Anal.},
   volume={55},
   date={2003},
   number={1-2},
   pages={33--46},
}	
\bib{Kla}{article}{
   author={Klartag, Bo'az},
   title={Needle decompositions in Riemannian geometry},
   journal={Mem. Amer. Math. Soc.},
   volume={249},
   date={2017},
   number={1180},
}

\bib{Kle}{book}{
   author={Klenke, Achim},
   title={Probability theory, A comprehensive course},
   series={Universitext},
   edition={2},
   edition={Translation from the German edition},
   publisher={Springer},
   date={2014},
}	

\bib{Mil4}{article}{
   author={Kolesnikov, Alexander V.},
   author={Milman, Emanuel},
   title={Riemannian metrics on convex sets with applications to Poincar\'{e}
   and log-Sobolev inequalities},
   journal={Calc. Var. Partial Differential Equations},
   volume={55},
   date={2016},
   number={4},
   pages={Art. 77, 36},
}
\bib{Mil5}{article}{
   author={Kolesnikov, Alexander V.},
   author={Milman, Emanuel},
   title={Brascamp-Lieb-type inequalities on weighted Riemannian manifolds
   with boundary},
   journal={J. Geom. Anal.},
   volume={27},
   date={2017},
   number={2},
   pages={1680--1702},
}			
	
\bib{Mil6}{article}{
   author={Kolesnikov, Alexander V.},
   author={Milman, Emanuel},
   title={Poincar\'{e} and Brunn-Minkowski inequalities on the boundary of
   weighted Riemannian manifolds},
   journal={Amer. J. Math.},
   volume={140},
   date={2018},
   number={5},
   pages={1147--1185},
}
\bib{KoZet}{article}{
   author={Kong, Q.},
   author={Zettl, A.},
   title={Eigenvalues of regular Sturm-Liouville problems},
   journal={J. Differential Equations},
   volume={131},
   date={1996},
   number={1},
   pages={1--19},
}

\bib{Kro}{article}{
   author={Kr\"{o}ger, Pawel},
   title={On the spectral gap for compact manifolds},
   journal={J. Differential Geom.},
   volume={36},
   date={1992},
   number={2},
   pages={315--330},
}

\bib{TakKuNa2}{article}{
   author={Kusano, Taka\^{s}i},
   author={Naito, Manabu},
   title={Sturm-Liouville eigenvalue problems for half-linear ordinary
   differential equations},
     journal={S\={u}rikaisekikenky\={u}sho Koky\={u}roku },
   number={1083},
   date={1999},
   pages={32--43},
}
\bib{TakKuNa1}{article}{
   author={Kusano, Taka\^{s}i},
   author={Naito, Manabu},
   author={Tanigawa, Tomoyuki},
   title={Second-order half-linear eigenvalue problems},
   journal={Fukuoka Univ. Sci. Rep.},
   volume={27},
   date={1997},
   number={1},
   pages={1--7},
}			

\bib{Lic}{book}{
   author={Lichnerowicz, Andr\'{e}},
   title={G\'{e}om\'{e}trie des groupes de transformations},
   language={French},
   publisher={Travaux et Recherches Math\'{e}matiques, III. Dunod, Paris},
   date={1958},
}
				

\bib{Lic1}{article}{
   author={Lichnerowicz, Andr\'{e}},
   title={Vari\'{e}t\'{e}s riemanniennes \`a tenseur C non n\'{e}gatif},
   language={French},
   journal={C. R. Acad. Sci. Paris S\'{e}r. A-B},
   volume={271},
   date={1970},
   pages={650--653},

}
\bib{Lic2}{article}{
   author={Lichnerowicz, Andr\'{e}},
   title={Vari\'{e}t\'{e}s k\"{a}hl\'{e}riennes \`a premi\`ere classe de Chern non negative et
   vari\'{e}t\'{e}s riemanniennes \`a courbure de Ricci g\'{e}n\'{e}ralis\'{e}e non negative},
   language={French},
   journal={J. Differential Geometry},
   volume={6},
   date={1971/72},
   pages={47--94},
}			
	


		

\bib{Led3}{article}{
   author={Ledoux, Michel},
   title={On an integral criterion for hypercontractivity of diffusion
   semigroups and extremal functions},
   journal={J. Funct. Anal.},
   volume={105},
   date={1992},
   number={2},
   pages={444--465},

}	
	
	\bib{Led4}{book}{
   author={Ledoux, Michel},
   title={The concentration of measure phenomenon},
   series={Mathematical Surveys and Monographs},
   volume={89},
   publisher={American Mathematical Society},
   date={2001},
}
\bib{Led2}{article}{
   author={Ledoux, Michel},
   title={Spectral gap, logarithmic Sobolev constant, and geometric bounds},
   conference={
      title={Surveys in differential geometry. Vol. IX},
   },
   book={
      series={Surv. Differ. Geom.},
      volume={9},
      publisher={Int. Press},
   },
   date={2004},
   pages={219--240},
}
\bib{Lei}{article}{
   author={Leindler, L.},
   title={On a certain converse of H\"{o}lder's inequality. II},
   journal={Acta Sci. Math.},
   volume={33},
   date={1972},
   number={3-4},
   pages={217--223},
}
	\bib{LY}{article}{
   author={Li, Peter},
   author={Yau, Shing Tung},
   title={Estimates of eigenvalues of a compact Riemannian manifold},
   conference={
      title={Geometry of the Laplace operator},
      address={Proc. Sympos. Pure Math., Univ. Hawaii},
      date={1979},
   },
   book={
      series={Proc. Sympos. Pure Math., XXXVI},
      publisher={Amer. Math. Soc.},
   },
   date={1980},
   pages={205--239},

}	
\bib{LiLo}{book}{
   author={Lieb, Elliott H.},
   author={Loss, Michael},
   title={Analysis},
   series={Graduate Studies in Mathematics},
   volume={14},
   edition={2},
   publisher={American Mathematical Society},
   date={2001},
}
	
\bib{Lin}{book}{
   author={Lindqvist, Peter},
   title={Notes on the $p$-Laplace equation},
   series={Report. University of Jyv\"{a}skyl\"{a} Department of Mathematics and
   Statistics},
   volume={102},
   publisher={University of Jyv\"{a}skyl\"{a}},
   date={2006},
}

\bib{LV}{article}{
   author={Lott, John},
   author={Villani, C\'{e}dric},
   title={Ricci curvature for metric-measure spaces via optimal transport},
   journal={Ann. of Math. (2)},
   volume={169},
   date={2009},
   number={3},
   pages={903--991},
}
\bib{LS}{article}{
   author={Lov\'{a}sz, L.},
   author={Simonovits, M.},
   title={Random walks in a convex body and an improved volume algorithm},
   journal={Random Structures Algorithms},
   volume={4},
   date={1993},
   number={4},
   pages={359--412},
}
				
	
	
	
	

\bib{Mat}{article}{
   author={Matei, Ana-Maria},
   title={First eigenvalue for the $p$-Laplace operator},
   journal={Nonlinear Anal.},
   volume={39},
   date={2000},
   number={8, Ser. A: Theory Methods},
   pages={1051--1068},
}

\bib{Maz1}{article}{
   author={Maz\cprime ja, V. G.},
   title={Classes of domains and imbedding theorems for function spaces},
   journal={Soviet Math. Dokl.},
   volume={1},
   date={1960},
   language={Russian},
   pages={882--885},
}		
	
\bib{Maz2}{article}{
   author={Maz\cprime ja, V. G.},
   title={$p$-conductivity and theorems on imbedding certain functional
   spaces into a $C$-space},
   language={Russian},
   journal={Dokl. Akad. Nauk SSSR},
   volume={140},
   date={1961},
   pages={299--302},
}
			

\bib{Maz3}{article}{
   author={Maz\cprime ja, V. G.},
   title={The negative spectrum of the higher-dimensional Schr\"{o}dinger
   operator},
   language={Russian},
   journal={Dokl. Akad. Nauk SSSR},
   volume={144},
   date={1962},
   pages={721--722},

}
		
\bib{MYZ}{article}{
   author={Meng, Gang},
   author={Yan, Ping},
   author={Zhang, Meirong},
   title={Spectrum of one-dimensional $p$-Laplacian with an indefinite
   integrable weight},
   journal={Mediterr. J. Math.},
   volume={7},
   date={2010},
   number={2},
   pages={225--248},
}
	
\bib{Mil0}{article}{
   author={Milman, Emanuel},
   title={On the role of convexity in isoperimetry, spectral gap and concentration},
   journal={Inventiones Mathematicae},
   volume={177},
   date={2009},
   number={1},
   pages={1--43},
}

\bib{Mil2}{article}{
   author={Milman, Emanuel},
   title={Sharp isoperimetric inequalities and model spaces for the
   curvature-dimension-diameter condition},
   journal={J. Eur. Math. Soc. (JEMS)},
   volume={17},
   date={2015},
   number={5},
   pages={1041--1078},
}		

\bib{Mil3}{article}{
   author={Milman, Emanuel},
   title={Beyond traditional curvature-dimension I: new model spaces for
   isoperimetric and concentration inequalities in negative dimension},
   journal={Trans. Amer. Math. Soc.},
   volume={369},
   date={2017},
   number={5},
   pages={3605--3637},
}
	
\bib{Mil7}{article}{
   author={Milman, Emanuel},
   title={Harmonic measures on the sphere via curvature-dimension},
   journal={Ann. Fac. Sci. Toulouse Math. (6)},
   volume={26},
   date={2017},
   number={2},
   pages={437--449},
}




	
\bib{Mor}{article}{
   author={Morgan, Frank},
   title={Manifolds with density},
   journal={Notices Amer. Math. Soc.},
   volume={52},
   date={2005},
   number={8},
   pages={853--858},
}


\bib{Muck}{article}{
   author={Muckenhoupt, Benjamin},
   title={Hardy's inequality with weights},
   note={Collection of articles honoring the completion by Antoni Zygmund of
   50 years of scientific activity, I},
   journal={Studia Math.},
   volume={44},
   date={1972},
   pages={31--38},

}
		



\bib{NaVa}{article}{
   author={Naber, Aaron},
   author={Valtorta, Daniele},
   title={Sharp estimates on the first eigenvalue of the $p$-Laplacian with
   negative Ricci lower bound},
   journal={Math. Z.},
   volume={277},
   date={2014},
   number={3-4},
   pages={867--891},
}

\bib{Nel}{article}{
   author={Nelson, Edward},
   title={The free Markoff field},
   journal={J. Functional Analysis},
   volume={12},
   date={1973},
   pages={211--227},
}
		

\bib{Oht1}{article}{
   author={Ohta, Shin-ichi},
   title={$(K,N)$-convexity and the curvature-dimension condition for
   negative $N$},
   journal={J. Geom. Anal.},
   volume={26},
   date={2016},
   number={3},
   pages={2067--2096},
}

\bib{Oht3}{article}{
   author={Ohta, Shin-ichi},
   title={Needle decompositions and isoperimetric inequalities in Finsler
   geometry},
   journal={J. Math. Soc. Japan},
   volume={70},
   date={2018},
   number={2},
   pages={651--693},
}
				

\bib{Oht2}{article}{
   author={Ohta, Shin-ichi},
   title={Ricci curvature, entropy, and optimal transport},
   conference={
      title={Optimal transportation},
   },
   book={
      series={London Math. Soc. Lecture Note Ser.},
      volume={413},
      publisher={Cambridge Univ. Press, Cambridge},
   },
   date={2014},
   pages={145--199},
}
	

\bib{OT1}{article}{
   author={Ohta, Shin-ichi},
   author={Takatsu, Asuka},
   title={Displacement convexity of generalized relative entropies},
   journal={Adv. Math.},
   volume={228},
   date={2011},
   number={3},
   pages={1742--1787},
}
	

\bib{PaWe}{article}{
   author={Payne, L. E.},
   author={Weinberger, H. F.},
   title={An optimal Poincar\'{e} inequality for convex domains},
   journal={Arch. Rational Mech. Anal.},
   volume={5},
   date={1960},
   pages={286--292 },

}


			

\bib{Phe}{book}{
   author={Phelps, Robert R.},
   title={Lectures on Choquet's theorem},
   series={Lecture Notes in Mathematics},
   volume={1757},
   edition={2},
   publisher={Springer},
   date={2001},
}
	
\bib{Pic}{article}{
   author={Picone, Mauro},
   title={Sui valori eccezionali di un parametro da cui dipende un'equazione
   differenziale lineare ordinaria del second'ordine},
   language={Italian},
   journal={Ann. Scuola Norm. Sup. Pisa Cl. Sci.},
   volume={11},
   date={1910},
   pages={144},
}
\bib{PiRu}{book}{
   author={Pinchover, Yehuda},
   author={Rubinstein, Jacob},
   title={An introduction to partial differential equations},
   publisher={Cambridge University Press},
   date={2005},
}
\bib{Pre1}{article}{
   author={Pr\'{e}kopa, Andr\'{a}s},
   title={Logarithmic concave measures with application to stochastic
   programming},
   journal={Acta Sci. Math.},
   volume={32},
   date={1971},
   pages={301--316},
}
			

\bib{Pre2}{article}{
   author={Pr\'{e}kopa, Andr\'{a}s},
   title={On logarithmic concave measures and functions},
   journal={Acta Sci. Math.},
   volume={34},
   date={1973},
   pages={335--343},
}


\bib{Rei}{article}{
   author={Reichel, Wolfgang},
   author={Walter, Wolfgang},
   title={Sturm-Liouville type problems for the $p$-Laplacian under
   asymptotic non-resonance conditions},
   journal={J. Differential Equations},
   volume={156},
   date={1999},
   number={1},
   pages={50--70},
}


\bib{Rot2}{article}{
   author={Rothaus, O. S.},
   title={Lower bounds for eigenvalues of regular Sturm-Liouville operators
   and the logarithmic Sobolev inequality},
   journal={Duke Math. J.},
   volume={45},
   date={1978},
   number={2},
   pages={351--362},
}	


\bib{Rot1}{article}{
   author={Rothaus, O. S.},
   title={Logarithmic Sobolev inequalities and the spectrum of
   Sturm-Liouville operators},
   journal={J. Funct. Anal.},
   volume={39},
   date={1980},
   number={1},
   pages={42--56},
}
	

\bib{Rot1a}{article}{
   author={Rothaus, O. S.},
   title={Logarithmic Sobolev inequalities and the spectrum of Schr\"{o}dinger
   operators},
   journal={J. Funct. Anal.},
   volume={42},
   date={1981},
   number={1},
   pages={110--120},
}

\bib{Rot3}{article}{
   author={Rothaus, O. S.},
   title={Diffusion on compact Riemannian manifolds and logarithmic Sobolev
   inequalities},
   journal={J. Funct. Anal.},
   volume={42},
   date={1981},
   number={1},
   pages={102--109},
}

\bib{Rot4}{article}{
   author={Rothaus, O. S.},
   title={Hypercontractivity and the Bakry-Emery criterion for compact Lie
   groups},
   journal={J. Funct. Anal.},
   volume={65},
   date={1986},
   number={3},
   pages={358--367},
}
			
\bib{Rot5}{article}{
   author={Rothaus, O. S.},
   title={Sharp log-Sobolev inequalities},
   journal={Proc. Amer. Math. Soc.},
   volume={126},
   date={1998},
   number={10},
   pages={2903--2904},
}
\bib{Roy}{book}{
   author={Royden, H. L.},
   title={Real analysis, 2ed.},
   publisher={The Macmillan Co.},
   date={1968},
}

\bib{Rud1}{book}{
   author={Rudin, Walter},
   title={Real and complex analysis},
   edition={3},
   publisher={McGraw-Hill Book Co.},
   date={1987},
}

\bib{SaC}{article}{
   author={Saloff-Coste, L.},
   title={Convergence to equilibrium and logarithmic Sobolev constant on
   manifolds with Ricci curvature bounded below},
   journal={Colloq. Math.},
   volume={67},
   date={1994},
   number={1},
   pages={109--121},
}			
\bib{Sid}{book}{
   author={Sideris, Thomas C.},
   title={Ordinary differential equations and dynamical systems},
   series={Atlantis Studies in Differential Equations},
   volume={2},
   publisher={Atlantis Press},
   date={2013},
}
%Check if appears
\bib{Sti}{article}{
   author={Strichartz, Robert S.},
   title={Analysis of the Laplacian on the complete Riemannian manifold},
   journal={J. Funct. Anal.},
   volume={52},
   date={1983},
   number={1},
   pages={48--79},
}
\bib{KTS1}{article}{
   author={Sturm, Karl-Theodor},
   title={Generalized Ricci bounds and convergence of metric measure spaces},
   journal={C. R. Math. Acad. Sci. Paris},
   volume={340},
   date={2005},
   number={3},
   pages={235--238},
}
	
\bib{KTS2}{article}{
   author={Sturm, Karl-Theodor},
   title={On the geometry of metric measure spaces. I},
   journal={Acta Math.},
   volume={196},
   date={2006},
   number={1},
   pages={65--131},
}
		

\bib{KTS3}{article}{
   author={Sturm, Karl-Theodor},
   title={On the geometry of metric measure spaces. II},
   journal={Acta Math.},
   volume={196},
   date={2006},
   number={1},
   pages={133--177},
}	
	


\bib{Tay1}{book}{
   author={Taylor, Michael E.},
   title={Partial differential equations II. Qualitative studies of linear
   equations},
   series={Applied Mathematical Sciences},
   volume={116},
   edition={2},
   publisher={Springer},
   date={2011},
}	

% \bib{Tol}{article}{
%   author={Tolksdorf, Peter},
%   title={Regularity for a more general class of quasilinear elliptic
%   equations},
%   journal={J. Differential Equations},
%   volume={51},
%   date={1984},
%   number={1},
%   pages={126--150},
% }


\bib{Val}{article}{
   author={Valtorta, Daniele},
   title={Sharp estimate on the first eigenvalue of the $p$-Laplacian},
   journal={Nonlinear Anal.},
   volume={75},
   date={2012},
   number={13},
   pages={4974--4994},
}

% \bib{MR2721045}{article}{
%   author={Figalli, Alessio},
%   title={{\it Optimal transport: old and new} [book review of MR2459454]},
%   journal={Bull. Amer. Math. Soc. (N.S.)},
%   volume={47},
%   date={2010},
%   number={4},
%   pages={723--727},
%   issn={0273-0979},
%   review={\MR{2721045}},
%   doi={10.1090/S0273-0979-10-01285-1},
% }

	
\bib{Vil}{book}{
   author={Villani, C{\'e}dric},
   title={Optimal transport: old and new},
   series={Grundlehren der mathematischen Wissenschaften},
   volume={338},
   publisher={Springer},
   date={2009},
}	
\bib{KTSR}{article}{
   author={von Renesse, Max-K.},
   author={Sturm, Karl-Theodor},
   title={Transport inequalities, gradient estimates, entropy, and Ricci
   curvature},
   journal={Comm. Pure Appl. Math.},
   volume={58},
   date={2005},
   number={7},
   pages={923--940},
}	


\bib{Wan1}{article}{
   author={Wang, Feng-Yu},
   title={Harnack inequalities for log-Sobolev functions and estimates of
   log-Sobolev constants},
   journal={Ann. Probab.},
   volume={27},
   date={1999},
   number={2},
   pages={653--663},
}	

\bib{Wan2}{article}{
   author={Wang, Feng-Yu},
   title={Logarithmic Sobolev inequalities on noncompact Riemannian
   manifolds},
   journal={Probab. Theory Related Fields},
   volume={109},
   date={1997},
   number={3},
   pages={417--424},
}
\bib{Wei}{article}{
   author={Weissler, Fred B.},
   title={Logarithmic Sobolev inequalities and hypercontractive estimates on
   the circle},
   journal={J. Funct. Anal.},
   volume={37},
   date={1980},
   number={2},
   pages={218--234},
}

\bib{Zet}{book}{
   author={Zettl, Anton},
   title={Sturm-Liouville theory},
   series={Mathematical Surveys and Monographs},
   volume={121},
   publisher={American Mathematical Society},
   date={2005},
}	
	
\bib{Zha}{article}{
   author={Zhang, Huichun},
   title={Lower bounds for the first eigenvalue of the $p$-Laplace operator
   on compact manifolds with nonnegative Ricci curvature},
   journal={Adv. Geom.},
   volume={7},
   date={2007},
   number={1},
   pages={145--155},
}

\bib{ZhY}{article}{
   author={Zhong, Jia Qing},
   author={Yang, Hong Cang},
   title={On the estimate of the first eigenvalue of a compact Riemannian
   manifold},
   journal={Sci. Sinica Ser. A},
   volume={27},
   date={1984},
   number={12},
   pages={1265--1273},
}

\bib{Zim}{book}{
   author={Zimmer, Robert J.},
   title={Essential results of functional analysis},
   series={Chicago Lectures in Mathematics},
   publisher={University of Chicago Press},
   date={1990},
}				


		
\end{biblist}
\end{bibdiv}



\section*{Supplementary material}

# אי-שוויונים פונקציונליים על יריעות רימניות ממושקלות תחת תנאי 'עקמומיות-מימד'

## חיבור על מחקר

לשם מילוי חלקי של הדרישות לקבלת תואר דוקטור לפילוסופיה

הוגש על ידי: ערן קלדרון

156

המחקר בוצע בהנחייתו של פרופ. עמנואל מילמן, בפקולטה למתמטיקה בטכניון.

# תודות

ברצוני להביע הערכתי הכנה לפרופ' עמנואל מילמן על שהיה המדריך המקצועי שלי. בפרט, על החשיפה לנושא מעניין, על העצות המועילות שהועלו בשיחות מקצועיות, על הביקורת הבונה ועל בחינה והגהה של התזה, אשר הביאו לשיפור משמעותי שלה. היה זה כבוד גדול עבורי להיות סטודנט שלו.

מעל הכל אני מביע הערכתי הכנה למשפחתי הכנה האהובה, אשר הייתה לצידי בכל מקום ובכל עת.

# תקציר

בעבודה זאת אנו מקבלים אי-שוויונים פונקציונליים חדים חדשים מסוג פואנקרה או לוג-סובולב, על יריעות רימניות ממושקלות, אשר עקמומיות ריצ'י המוכללת שלהם עם פרמטר מימד מוכללת-N חסומה מלמטה על ידי קבוע K, והקוטר שלהן חסום על ידי $D>0$. כאשר הנ"ל מתקיימים אנו אומרים שהיריעה מקיימת תנאי עקמומיות-מימד-קוטר, או בקיצור $CDD(K,N,D)$.

התנאי האמור לעיל מכליל התנאי המוכר יותר של עקמומיות ריצ'י (הרגילה) החסומה על ידי קבוע K. מסתבר כי ההכללה הנ"ל מביאה להכללות של שפע תוצאות מוכרות על יריעות רימניות עם עקמומיות ריצ'י חסומה מלמטה, ליריעות רימניות ממושקלות; בפרט ובדגש בעבודה זאת, אנו מקבלים הכללות של תוצאות אנליטיות בהן מתגלגמת אינטרקציה בין המידה, אשר באופן כללי צפיפותה רחוקה מלהיות טריוויילית, לבין הגיאומטריה, אשר באופן כללי עם עקמומיות שאינה טריוויאלית.

לשיטה בה אנו משתמשים יש היסטוריה המתחילה כבר בשנות ה-60 המוקדמות, כאשר הוכח אי שוויון פואנקרה חד על גופים קמורים במרחב אוקלידי. בשיטה המכונה 'לוקליזציה' נעשתה רדוקציה של הבעיה המקורית, של גוף קמור n מימדי, לבעיה במימד 1, אשר ניתן לאפיינה כבעיית אופטימיזציה. הוכחת הרדוקציה הסתמכה על ביצוע מספר אינסופי של חיתוכים של הגוף אשר הובילו בגבול למספר אינסופי של גופים ממימד 1. ברם, שיטה זאת הייתה מאוד מוגבלת לגופים במרחב אוקלידי, שכן באופן כללי אין דרך טבעית להגדיר איטרטיבית החיתוכים על יריעות רימניות כלליות.

לאחרונה פותח משפט על ידי ב. קלרטג, הקרוי משפט ה'לוקליזציה', המאפשר הכללה של הרדוקציה הנ"ל ליריעות ממושקלות כלליות המקיימות תנאי עקמומיות-מימד-קוטר. המשפט של קלרטג מתבסס על טכניקות 'שינוע אופטימלי' (optimal transport), ומביא לפירוק של מידת היריעה למידות שוליות הנתמכות על עקומים גאודזיים של היריעה. הפירוק הזה מביא לרדוקציית הבעיה הרב-מימדית של הוכחת האי-שוויון על היריעה, לבעיית מינימיזציה של מחלקה של מידות בעלות תומך חד-מימדי.

בעיית האופטימיזציה כוללת מינימיזציה על משפחה אינסופית של פונקציות ומידות הנתמכות על R. לשם פתרון בעיית המינימיזציה באופן כללי, אנו מפתחים שיטה המשלימה את העבודה של קלרטג, ומביאה לרדוקציה של בעיית האופטימיזציה לבעיה פשוטה יותר, מעל תת-מחלקה של מידות מסוג 'מודל'. הרדוקציה לתת מחלקה  של מידות מסוג 'מודל' מתקבלת כתוצאה של תהליך בנייה מורכב של מחלקת מידות 'סינתטית', המכילה את משפחת המידות המתקבלות ממשפט הלוקליזציה, ואיפיון של הנקודות האקסטרימליות בתוך תת-קבוצה של מרחב המידות החדש; על ידי שימוש בטכניקות מוכרות של אנליזה פונקציונלית אנו מסיקים כי המינימום, כלומר פתרון בעיית האופטימיזציה המהווה פתרון בעיית מציאת החסם האופטימלי המתאים לאי-שוויון הנדון, מתקבל על נקודות אקסטרימליות של התת-קבוצה, אשר מסתבר שהן מסוג 'מודל' (תת מחלקה של מחלקת המידות המקורית של משפט הלוקליזציה).

חשוב לציין כי רדוקציה זו הינה כללית, ואינה מצטמצמת לאי-שוויונים שצויינו בלבד; היא ניתנת ליישום לקבלת הקבועים החדים המאפיינים אי-שוויונים פונקציונליים מסוגים שונים מאלו שצויינו לעיל.

בשלושת הפרקים האחרונים של העבודה, על ידי שימוש בטכניקות אנליטיות אד-הוק, המשלימות את מסקנות השיטה שפותחה, אנו פותרים את בעיות המינימיזציה הנלוות לקבוע פואנקרה, ק-פואנקרה, ולוג-סובולב, תחת תנאי 'עקמומיות-מימד' מסויימים. באופן יותר ספציפי:

1. אנו מקבלים חסמים חדים לקבוע פואנקרה על יריעות רימניות ממושקלות המקויימות $CDD(K,N,D)$, כאשר התוצאה העיקרית היא מציאת חסם חד על קבוע פואנקרה עבור ערכי פרמטר מימד מוכלל $N$ שליליים, וכן משפט על תלות מונוטונית במימד-1 של קבוע פואנקרה של משפחת מידות ה'מודל' בפרמטר השיפוע (דבר שיובהר בגוף העבודה). באופן לא צפוי, הפתרון לבעית אפיון הקבוע הקבוע פואנקרה החד חושף התנהגות אנומלית בתת-תחום של פרמטר המימד (המוכלל), מה שמצביע על תופעה חדשה.

2. אנו מקבלים חסמים חדים לקבוע ק-פואנקרה על יריעות רימניות ממושקלות המקויימות $CDD(K,N,D)$, כאשר התוצאה העיקרית היא מציאת חסם קבוע ק-פואנקרה החד עבור ערכי פרמטר מימד מוכלל $N$ חיוביים ו-$K$ כלשהו, וכן משפט על תלות מונוטונית במימד-1 של קבוע ק-פואנקרה של משפחת מידות ה'מודל' בפרמטר השיפוע.

3. אנו מקבלים עד כדי קבועים נומריים את החסם התחתון החד על קבוע לוג-סובולב של יריעות רימניות ממושקלות המקויימות $CDD(K,\infty,D)$, ומסיקים שבתנאים אלו החסם החד לקבוע לוג סובולב שקול לחסם התחתון החד על קבוע פואנקרה.



\end{document}